\documentclass[a4paper,11pt]{amsart}
\usepackage[cm]{fullpage}
\usepackage{hyperref}

\usepackage{mymacros}

\title{Tropical contractions to integral affine manifolds with singularities}
\author{Yuto Yamamoto}
\address{
Center for Geometry and Physics, Institute for Basic Science (IBS), Pohang 37673, Korea}
\email{yuto.yamamoto@riken.jp}
\date{}
\begin{document}

\maketitle

\begin{abstract}
We consider a toric degeneration of Calabi--Yau complete intersections of Batyrev--Borisov in the Gross--Siebert program.
One can associate two types of tropical spaces with it.
One is a tropical variety obtained by tropicalization.
The other one is an integral affine manifold with singularities, which arises as the dual intersection complex of the toric degeneration.
In this article, we show that the latter is contained in the former as a subset, and construct an integral affine contraction map from the former to the latter.
We also show that the contraction preserves tropical cohomology groups, and sends the eigenwave to the radiance obstruction.
\end{abstract}

\section{Introduction}\label{sc:intro}

The aim of this article and the sequel \cite{Yam24} is to link three different notions in tropical/non-archimedean geometry and mirror symmetry:
\begin{itemize}
\item tropical varieties,
\item integral affine manifolds with singularities/the Gross--Siebert program, and
\item essential skeletons/non-archimedean SYZ fibrations.
\end{itemize}
Tropical varieties appear as the tropicalizations of algebraic varieties over valued fields.
They are polyhedral complexes equipped with some kind of affine structures.
We also sometimes enlarge the class of tropical varieties so that it includes spaces constructed by gluing polyhedral complexes in tropical affine spaces by integral affine isomorphisms (e.g.~\cite{MR2275625, MR3330789}).
Integral affine manifolds with singularities (we call them IAMS for short throughout the article) arise as the dual intersection complexes of toric degenerations in the Gross--Siebert program \cite{MR2213573, MR2669728, MR2846484}.
They are also expected to be the base spaces of Lagrangian torus fibrations (SYZ fibrations) in Strominger--Yau--Zaslow conjecture \cite{MR1429831}, and the limits of maximally degenerating families of Calabi--Yau manifolds with Ricci-flat K\"{a}hler metrics in the Gromov--Hausdorff topology \cite{MR1863732, MR2181810}.
Non-archimedean SYZ fibrations are an analogue of SYZ fibrations in non-archimedean geometry.
They are constructed by using the Berkovich retractions associated with good minimal dlt-models for maximally degenerate Calabi--Yau varieties, which are maps from the Berkovich analytifications to their subsets called the essential skeletons \cite{MR2181810, MR3946280, MPS21, PS22}.

In order to study the relations between these three notions, we construct contraction maps from tropical varieties to IAMS, which preserve integral affine structures.
We call the contraction maps \emph{tropical contractions}.
The outline of the construction is as follows:
Let $d$ and $r$ be positive integers.
Consider a free $\bZ$-module $N'$ of rank $d$, and its dual $M':=\Hom (N', \bZ)$.
Let further $e_1, \cdots, e_r$ denote the standard basis of $\bZ^r$, and $e_1^\ast, \cdots, e_r^\ast$ be the dual basis of $\lb \bZ^r \rb^\ast$.
We set
\begin{align}
N:=N' \oplus \lb \oplus_{i=1}^r \bZ e_i \rb, \quad M:=M' \oplus \lb \oplus_{i=1}^r \bZ e_i^\ast \rb,
\end{align}
and $N_\bR':=N' \otimes_\bZ \bR$, $N_\bR:=N \otimes_\bZ \bR$, $M_\bR':=M' \otimes_\bZ \bR$, $M_\bR:=M \otimes_\bZ \bR$.
Let $\lc \Delta_i \subset M'_\bR \rc_{i=1}^r$ $\lc \Deltav_i \subset N'_\bR \rc_{i=1}^r$ be lattice polytopes such that $\la m, n \ra=0$ for any $m \in \Delta_i$, $n \in \Deltav_j$ $(1 \leq i, j \leq r)$.
We associate a stable intersection of tropical hypersurfaces in a tropical toric variety with these polytopes.
We consider the cone $C$ in $N_\bR$ defined by
\begin{align}\label{eq:cone}
C:=\mathrm{cone} \lb \bigcup_{i=1}^r \Deltav_i \times \lc e_i \rc \rb,
\end{align}
and the tropical polynomials $f_i \colon N_\bR \to \bR$ $(1 \leq i \leq r)$ defined by
\begin{align}\label{eq:polynomial}
f_i(n):=\min_{m \in A_i} \la m, n \ra,
\end{align}
where $A_i:=\lc 0 \rc \cup \lb M \cap \lb \Delta_i \times \lc -e_i^\ast \rc \rb \rb$.
Let $X(f_i)^\circ \subset N_\bR$ denote the tropical hypersurface defined by $f_i$, i.e., the corner locus of the function $f_i$.
In general, a polyhedral complex in $\bR^n$ is called \emph{pure} if the dimension of every maximal polyhedron is the same.
A \emph{weighted balanced polyhedral complex} $\scrP$ in $\bR^n$ is a pure polyhedral complex such that each maximal-dimensional polyhedron in $\scrP$ is assigned a natural number called \emph{weight}, and it satisfies the \emph{balancing condition} (cf.~e.g.~\cite[Definition 3.3.1]{MR3287221}).
Recall that for two weighted balanced polyhedral complexes $\scrP_1$ and $\scrP_2$ in $\bR^n$, the \emph{stable intersection} $\scrP_1 \cap_{\mathrm{st}}  \scrP_2$ is the polyhedral complex defined by
\begin{align}
\scrP_1 \cap_{\mathrm{st}}  \scrP_2:=\lc \sigma_1 \cap \sigma_2 \relmid \sigma_1 \in \scrP_1, \sigma_2 \in \scrP_2, \dim \lb \sigma_1 + \sigma_2 \rb=n \rc.
\end{align}
By assigning multiplicities to maximal-dimensional polyhedra in $\scrP_1 \cap_{\mathrm{st}}  \scrP_2$ appropriately, the stable intersection $\scrP_1 \cap_{\mathrm{st}}  \scrP_2$ becomes a weighted balanced polyhedral complex (cf.~e.g.~\cite[Definition 3.6.5, Theorem 3.6.10]{MR3287221}).
We consider the stable intersection
\begin{align}\label{eq:intersection}
X(f_1, \cdots, f_r)^\circ :=X(f_1)^\circ \cap_{\mathrm{st}}  \cdots \cap_{\mathrm{st}}  X(f_r)^\circ \subset N_\bR,
\end{align}
and its closure in the tropical affine toric variety $X_{C}(\bT) \supset N_\bR$ associated with the cone $C$, which will be denoted by $X(f_1, \cdots, f_r)$.
See \pref{sc:toric} for the definition of tropical toric varieties.
We define the subset $B \subset N_\bR$ by
\begin{align}\label{eq:B}
B:=\bigcap_{i=1}^r \lc n \in N_\bR \relmid \exists m_i \in A_i \setminus \lc 0 \rc \ \mathrm{s.t.}\ f_i(n)=\la m_i, n \ra =0 \rc.
\end{align}
One can check $B \subset X(f_1, \cdots, f_r)^\circ$ (\pref{lm:BX}), and $B$ is a subcomplex of $X(f_1, \cdots, f_r)$.
In \pref{sc:construction}, we equip a subset $U_\xi \subset B$ with an integral affine structure with singularities by using projections to tropical torus orbits of the tropical toric variety $X_{C}(\bT)$.
The first main result of this article (\pref{th:local-cont}) is a construction of a contraction map from a subset of $X(f_1, \cdots, f_r)$ to $U_\xi \subset B$, which preserve the integral affine structures.
We call such contraction maps \emph{local models of tropical contractions} (\pref{df:local-contraction}).
By the contractions, one can produce many local models of IAMS such as focus-focus singularities, positive/negative vertices which appear in base spaces of topological $3$-torus fibrations of \cite{MR1821145}, and tropical (positive/negative) nodes introduced in \cite{MR3228462}.
See \pref{eg:ff} and \pref{eg:pos-neg}.
We also say that a map from a tropical variety to IAMS is a \emph{tropical contraction} if it is locally isomorphic to a local model of tropical contractions (\pref{df:contraction}).
The sheaves of integral affine functions on tropical spaces are regarded as the structure sheaves in tropical geometry.
In this sense, the tropical contraction is a \emph{morphism of tropical spaces} (cf.~\pref{rm:morphism}).
Tropical modifications of \cite{MR2275625} are tropical contractions of a special case (Example 3.31), and tropical contractions can be regarded as a generalization of tropical modifications of \cite{MR2275625}.

We also show that if an IAMS $B$ satisfies some condition ($B$ is \emph{quasi-simple} (\pref{df:quasi})), then the IAMS $B$ can locally be produced by a tropical contraction.
(For instance, if the IAMS $B$ is simple in the sense of Gross--Siebert program (\cite[Definition 1.60]{MR2213573}), then it is quasi-simple.)
Furthermore, if the IAMS $B$ satisfies some stronger condition ($B$ is \emph{very simple} (\pref{df:quasi})), then the tropical contraction satisfies a good condition ($\delta$ is very good (\pref{df:local-contraction})) which is necessary for \pref{cr:cohisom}.
That is the second main result.
\begin{theorem}\label{th:reconstruction}
Let $B$ be a quasi-simple IAMS.
Then for any point $x \in B$, there exists a neighborhood $U_x \subset B$ of $x$, a local model of tropical contractions $\delta \colon V \to U$, and a homeomorphism from $U_x$ to $U$, which is an integral isomorphism on the complement of the singular loci.
Furthermore, when $B$ be a very simple IAMS, the local model of tropical contraction $\delta \colon V \to U$ is very good in the sense of \pref{df:local-contraction} and $V$ is a tropical manifold.
\end{theorem}
In this article, we call a space that is locally isomorphic to the Bergman fan of a matroid a \emph{tropical manifold} (\pref{df:trop-mfd}).

Main examples that we have in mind are tropical Calabi--Yau complete intersections.
Let $R$ be a discrete valuation ring and $k$-algebra with algebraically closed residue class field $k$, and $K$ be the quotient field of $R$.
Let further $\Delta \subset M_\bR$ be a reflexive polytope, and $\Delta =\Delta_1+ \cdots + \Delta_r$ be a nef partition.
From this with some additional data, Gross \cite{MR2198802} constructed a toric degeneration $f \colon \scX \to \Spec R$ of Calabi--Yau complete intersections of Batyrev--Borisov \cite{MR1463173}, and its dual intersection complex $B^{\check{h}}_\nabla$ as a subset of $N_\bR$.
We recall the construction in \pref{sc:tcy}.
Let further $\trop \lb X \rb \subset N_\bR$ denote the tropicalization of the base change $f' \colon X \to \Spec K$ of the toric degeneration $f \colon \scX \to \Spec R$ to $K$.
The following is the third main result:

\begin{theorem}\label{th:glcontr}
The following hold:
\begin{enumerate}
\item One has 
\begin{align}
B^{\check{h}}_\nabla \subset \trop (X).
\end{align}
\item There exists a tropical contraction
\begin{align}
\delta \colon \trop \lb X \rb \to B^{\check{h}}_\nabla.
\end{align}
\end{enumerate}
\end{theorem}

In the case where $f \colon \scX \to \Spec R$ is a toric degeneration of Calabi--Yau hypersurfaces in projective spaces, Pille-Schneider \cite{PS22} showed that the essential skeleton of $X$ is isomorphic to $B^{\check{h}}_\nabla$ as piecewise integral affine manifolds, and the composition of the tropicalization map and the tropical contraction $\delta$ of \pref{th:glcontr} becomes a non-arhicmedean SYZ fibration.
We also generalize his result to the case of Calabi--Yau complete intersections in toric varieties in the sequel \cite{Yam24} of this article.
We refer the reader to \cite{Yam24} for the precise statement of the result.
In view of the relation with non-archimedean SYZ fibrations, the tropical contraction $\delta$ of \pref{th:glcontr} can be regarded as a tropical analogue of SYZ fibrations.

Next, we focus on the (co)homology theories of tropical varieties and IAMS.
Let $B$ be an IAMS, and $\iota_\ast \colon B_0 \hookrightarrow B$ be the complement of the singular locus.
We consider the local system $\check{\Lambda}$ (resp. $\Lambda$) on $B_0$ of integral cotangent vectors (resp. of integral tangent vectors).
It is known that the cohomology groups $H^q \lb B, \iota_\ast \bigwedge^p \check{\Lambda} \otimes_\bZ \bC \rb$ correspond to logarithmic Dolbeault cohomology groups of the log Calabi--Yau space associated with $B$, and the cup product of the radiance obstruction $c_B \in H^1 \lb B, \iota_\ast \Lambda \otimes_\bZ \bR \rb$ of $B$ coincides with the residue of the logarithmic extension of the Gauss--Manin connection \cite{MR2669728, MR2681794}.
A radiance obstruction is an invariant of integral affine manifolds, which was introduced in \cite{MR760977}.
See \pref{sc:radiance} for its definition.
There is a similar theory also for tropical varieties.
In \cite{MR3330789, MR3961331}, they introduced \emph{tropical cohomology groups} and \emph{tropical wave groups} for tropical varieties, which are sheaf cohomology groups of certain constructible sheaves $\scF_Q^p$ and $\scW_p^Q$ with $Q=\bZ, \bQ, \bR$ on tropical varieties.
There is also an invariant of tropical varieties, which is called an \emph{eigenwave}.
It is a class in the first cohomology group of $\scW_1^\bR$.
See \pref{sc:rat} for the definitions of these.
For a smooth tropical projective variety $V$, the tropical (co)homology groups $H^q \lb V, \scF^p_\bQ \rb$ of $V$ correspond to the grade pieces of the limiting mixed Hodge structure of the corresponding degenerating family of complex projective varieties \cite{MR3961331}, and the action of the eigenwave $c_V \in H^1 \lb V, \scW_1^\bR \rb$ of $V$ also coincides with the the residue of the logarithmic extension of the Gauss--Manin connection \cite{MR3330789}.

In \cite{MR3330789, MR3961331}, they also defined \emph{tropical homology groups} of tropical varieties, which are homology groups of certain constructible cosheaves on tropical varieties.
There are also definitions as singular and cellular homology groups.
More recently, Gross and Shokrieh \cite{MR4637248} showed that tropical homology groups as well as tropical Borel--Moore homology groups can be defined also as the hypercohomology of the Verdier dual of the sheaf $\scF_\bZ^p$.
Following their idea, we define tropical homology groups and tropical Borel--Moore homology groups of IAMS as the hypercohomology of the Verdier dual of the sheaf $\iota_\ast \bigwedge^p \check{\Lambda}$ (\pref{df:trophom}).
They are also defined as singular and cellular homology groups (\pref{df:singular}, \pref{df:cellular}).
For instance, \emph{tropical $1$-cycles} of \cite{MR4179831, MR21} and \emph{tropical $2$-cycles} of \cite{MR3228462} can be naturally regarded as singular cycles defining an element of the tropical homology group (\pref{eg:1-cycle}, \pref{eg:2-cycle}).
Tropical $1$-cycles were used to compute period integrals of toric degenerations \cite{MR4179831} and to construct deformations of tropical polarized K3 surfaces \cite{MR21}.
Tropical $2$-cycles were used to construct ordinary singular cycles which in turn govern smoothings and resolutions of the Calabi--Yau variety associated with a tropical conifold \cite{MR3228462}.

Let $\delta \colon V \to B$ be a tropical contraction.
We impose some condition for the lattice polytopes $\lc \Delta_i \subset M'_\bR \rc_{i=1}^r$ and $\lc \Deltav_i \subset N'_\bR \rc_{i=1}^r$ that we use to construct local models of tropical contractions.
When it is satisfied, we say $\delta$ is \emph{good/very good} (\pref{df:local-contraction}, \pref{df:contraction}).
One can compare the (co)homology theories for tropical varieties and IAMS via the good/very good tropical contraction $\delta$.
\begin{corollary}\label{cr:cohisom}
For a good (resp. very good) tropical contraction $\delta \colon V \to B$ and any integers $p, q \geq 0$, there exist natural isomorphisms
\begin{align}\label{eq:cohisom}
H^q \lb V, \scF^p_Q \rb \cong H^q \lb B, \iota_\ast \bigwedge^{p} \check{\Lambda} \otimes_\bZ Q \rb, &\quad 
H^q_c \lb V, \scF^p_Q \rb \cong H^q_c \lb B, \iota_\ast \bigwedge^{p} \check{\Lambda} \otimes_\bZ Q \rb, \\
\label{eq:hisom}
H_{p, q}^\mathrm{BM} \lb V, Q \rb \cong H_{p, q}^\mathrm{BM} \lb B, Q \rb, &\quad
H_{p, q} \lb V, Q \rb \cong H_{p, q} \lb B, Q \rb
\end{align}
for $Q=\bQ$ (resp. $Q=\bZ$).
\end{corollary}
Here the left hand sides of \eqref{eq:cohisom} and \eqref{eq:hisom} denote the tropical cohomology groups (with compact support), and the tropical (Borel--Moore) homology groups of $V$ respectively.
The right hand sides of \eqref{eq:hisom} denote the tropical (Borel--Moore) homology groups of $B$ mentioned above.
In \pref{rm:cy-good}, we give some sufficient conditions that make the tropical contractions $\delta$ for tropical Calabi--Yau complete intersections in \pref{th:glcontr} good/very good.
The fact that tropical (co)homology groups are preserved under tropical modifications of \cite{MR2275625} was shown in \cite[Theorem 4.13]{Sha15}, \cite[Proposition 5.6]{MR3894860}, and \cite[Proposition 4.22]{MR3903579}.

\begin{corollary}\label{cr:wave}
For a tropical contraction $\delta \colon V \to B$ and any integers $p, q \geq 0$, there is a natural group homomorphism
\begin{align}\label{eq:natu}
H^q \lb V, \scW_p^Q \rb \to H^q \lb B, \iota_\ast \bigwedge^p \Lambda \otimes_\bZ Q \rb,
\end{align}
where $Q=\bZ, \bQ, \bR$.
Furthermore, when $\delta$ is good, the map \eqref{eq:natu} with $p=q=1$ and $Q=\bR$ sends the eigenwave of $V$ to the radiance obstruction of $B$.
\end{corollary}

There is another approach by Ruddat \cite{MR4347312} to homology theory for IAMS using sheaf homology, which has some applications such as computation of period integrals \cite{MR4179831}.
In the recent work \cite{RZ21, RZ20} of constructing a topological SYZ fibration $X \to B$ for a given IAMS $B$, they also construct natural injections from the tropical sheaf homology groups of $B$ into the subquotients of the ordinary cohomology groups of $X$, which become isomorphisms when $B$ is very simple.
In general, the tropical sheaf homology groups of \cite{MR4347312} are not isomorphic to the tropical homology groups that we consider in this article.
See \pref{rm:sheaf-homology} and \pref{eg:1-cycle} where we discuss relations between these two homology groups.

The organization of this paper is as follows:
In \pref{sc:tropical}, we recall basic notions such as integral affine manifolds with singularities, tropical (co)homology groups, tropical toric varieties, and tropicalizations.
In \pref{sc:loccont}, we construct local models of tropical contractions and show their properties.
In \pref{sc:reconstruction}, we prove \pref{th:reconstruction}.
In \pref{sc:cy}, we study tropical Calabi--Yau complete intersections.
\pref{th:glcontr} is proved in this section.
In \pref{sc:coh}, we study the relations with tropical (co)homology groups.
\pref{cr:cohisom} and \pref{cr:wave} are proved in this section.
In \pref{sc:lem}, we show some technical lemmas about convex polytopes, which are used to prove main theorems of this article.

\par
{\it Acknowledgment: } 
I am greatly indebted to Grigory Mikhalkin for inspiring discussions, from which this project originates.
I learned many useful things on sheaf theory from Yuichi Ike, and on integral affine manifolds from Yuki Tsutsui.
I also learned about \cite{MR2405763} from Akiyoshi Tsuchiya.
I am grateful to them all.
I thank Johannes Rau and Andreas Gross for their kind answers to my questions on \cite{MR3894860} and \cite{MR4637248}.
I am also grateful to Mark Gross, Enrica Mazzon, and Ilia Zharkov for helpful communications.
I thank Helge Ruddat and Kazushi Ueda for many valuable comments on the draft of this paper.
This work was supported by the Institute for Basic Science (IBS-R003-D1). \\

\section{Tropical geometry}\label{sc:tropical}

\subsection{Convex geometry}\label{sc:conv}

We fix some notation in convex geometry, which we will use in this article.
Let $d$ be a positive integer.
For a subset $S \subset \bR^d$, the \emph{convex hull} $\conv \lb S \rb \subset \bR^d$, the \emph{affine hull} $\aff(S) \subset \bR^d$, the \emph{conic hull} $\cone \lb S \rb \subset \bR^d$, and the \emph{linear span} $\vspan \lb S \rb \subset \bR^d$ are defined by
\begin{align}
\conv \lb S \rb&:=\lc \sum_{i=1}^n \lambda_i s_i \relmid n \geq 1, s_i \in S, \lambda_i \in \bR_{\geq 0}, \sum_{i=1}^n \lambda_i =1 \rc \\
\aff \lb S \rb&:=\lc \sum_{i=1}^n \lambda_i s_i \relmid n \geq 1, s_i \in S, \lambda_i \in \bR, \sum_{i=1}^n \lambda_i =1 \rc \\
\cone \lb S \rb&:=\lc \sum_{i=1}^n \lambda_i s_i \relmid n \geq 0, s_i \in S, \lambda_i \in \bR_{\geq 0} \rc \\
\vspan \lb S \rb&:=\lc \sum_{i=1}^n \lambda_i s_i \relmid n \geq 0, s_i \in S, \lambda_i \in \bR \rc.
\end{align}
The relative interior of $S$ will be denoted by $\rint (S) \subset S$.
We also define
\begin{align}\label{eq:T}
T(S)&:=\lc \lambda (s_2-s_1) \relmid s_1, s_2 \in S, \lambda \in \bR \rc \\
T_\bZ(S)&:=T(S) \cap \bZ^d.
\end{align}
A \emph{rational polyhedron} in $\bR^d$ is the intersection of finitely many half spaces of the form $\lc x \in \bR^d \relmid \la m, x \ra \leq a \rc$ with $m \in \lb \bZ^d \rb^\ast$ and $a \in \bR$.
A bounded rational polyhedron is called a \emph{rational polytope}.
We also call a polytope whose vertices are all in $\bZ^d \subset \bR^d$ a \emph{lattice polytope}.
For subsets $K, K' \subset \bR^d$, the \emph{Minkowski sum} $K+K' \subset \bR^d$ and the \emph{scalar multiple} $c \cdot K \subset \bR^d$ for $c \in \bR_{\geq 0}$ are defined by
\begin{align}
K+K':=\lc k+k' \relmid k \in K, k' \in K' \rc, \quad c \cdot K:=\lc ck \relmid k \in K \rc.
\end{align}

Let $M$ denote a free $\bZ$-module of rank $d$, and $N:=\Hom_\bZ (M,\bZ)$ denote the dual lattice of $M$.
We set $M_\bR:=M \otimes_\bZ \bR$ and $N_\bR:=N \otimes_\bZ \bR=\Hom_\bZ (M,\bR)$.
We have a canonical $\bR$-bilinear pairing
\begin{align}
	\la \bullet ,\bullet \ra \colon M_\bR \times N_\bR \to \bR.
\end{align}
Let further $\Sigma$ be a complete fan in $N_\bR$.
We say that a function $h \colon N_\bR \to \bR$ is \emph{convex} if it satisfies $h \lb t n_1+(1-t)n_2 \rb \leq t h(n_1)+(1-t)h(n_2)$ for any $t \in \ld 0, 1\rd$ and $n_1, n_2 \in N_\bR$.
We also say that a convex function $h \colon N_\bR \to \bR$ is \emph{strictly convex with respect to the fan $\Sigma$} if it is convex and is linear on each cone of dimension $d$ in $\Sigma$, and distinct cones of dimension $d$ correspond to distinct linear functions.
For a strictly convex function $h \colon N_\bR \to \bR$ on $\Sigma$, the \emph{Newton polytope} $\Delta^h \subset M_\bR$ of $h$ is defined by
\begin{align}\label{eq:newton}
\Delta^h:=\lc m \in M_\bR \relmid \la m, n \ra \geq -h(n), \forall n \in N_\bR \rc.
\end{align}

Let $\Deltav \subset N_\bR$ be a rational polytope.
The \emph{support function} $\check{h}_\Deltav \colon M_\bR \to \bR$ for $\Deltav$ is defined by
\begin{align}\label{eq:supp}
\check{h}_\Deltav (m):=-\inf_{n \in \Deltav} \la m, n \ra.
\end{align}
One has
\begin{align}\label{eq:supp2}
\Deltav= \Deltav^{\check{h}_{\Deltav}}, \quad \check{h}_{\Deltav_1+\Deltav_2}=\check{h}_{\Deltav_1}+\check{h}_{\Deltav_2}, \quad \check{h}_{c \cdot \Deltav}=c \cdot \check{h}_\Deltav,
\end{align}
where $\Deltav^{\check{h}_{\Deltav}}$ is the Newton polytope of $\check{h}_{\Deltav}$, $\Deltav_1, \Deltav_2 \subset N_\bR$ are rational polytopes, and $c \in \bR_{\geq 0}$ (cf.~e.g.~\cite[Theorem A.18]{MR922894}).
Let $\scrP_{\Deltav}$ denote the set of all faces of $\Deltav$.
For an element $m \in M_\bR$ and a face $\sigma \in \scrP_{\Deltav}$ of $\Deltav$, we define
\begin{align}
\scN(m, \Deltav)&:=\lc n \in \Deltav \relmid - \la m, n \ra =\check{h}_\Deltav (m) \rc \\ \label{eq:scM}
\scM(\sigma, \Deltav)&:=\lc m \in M_\bR \relmid \scN(m, \Deltav)=\sigma \rc.
\end{align}
Then for $\sigma \in \scrP_{\Deltav}$, we can write
\begin{align}\label{eq:ss}
\sigma=\lc n \in \Deltav \relmid - \la m_0, n \ra =\check{h}_\Deltav (m_0) \rc
\end{align}
with $m_0 \in \scM(\sigma, \Deltav)$.
One also has
\begin{align}\label{eq:n-cone}
\overline{\scM(\sigma, \Deltav)}=\lc m \in M_\bR \relmid - \la m, n \ra =\check{h}_\Deltav (m), \forall n \in \sigma \rc,
\end{align}
where the overline denotes the closure in $M_\bR$.
The \emph{normal fan} $\Sigmav$ of $\Deltav$ is the collection $\lc \overline{\scM(\sigma, \Deltav)} \rc_{\sigma \in \scrP_{\Deltav}}$.
One has the natural bijection
\begin{align}\label{eq:poly-fan}
\delta_\Deltav \colon \scrP_{\Deltav} \to \Sigmav, \quad \sigma \mapsto \overline{\scM(\sigma, \Deltav)}.
\end{align}
Let $\Deltav' \subset N_\bR$ be another rational polytope whose normal fan $\Sigmav'$ is a refinement of the fan $\Sigmav$. 
We write the map $\Sigmav' \to \Sigmav$ that sends each cone $C' \in \Sigmav'$ to the minimal cone in $\Sigmav$ containing $C'$ as $\iota \colon \Sigmav' \to \Sigmav$.
We define
\begin{align}\label{eq:phi}
\phi_{\Deltav', \Deltav}:=\delta_{\Deltav}^{-1} \circ \iota \circ \delta_{\Deltav'} \colon \scrP_{\Deltav'} \to \scrP_{\Deltav},
\end{align}
where $\scrP_{\Deltav'}$ is the set of the faces of $\Deltav'$, and $\delta_{\Deltav'} \colon \scrP_{\Deltav'} \to \Sigmav'$ is the map of \pref{eq:poly-fan} for $\Deltav'$.
We have the following commutative diagram:
\begin{align}\label{eq:face-diag}
  \begin{CD}
     \scrP_{\Deltav'} @>{\phi_{\Deltav', \Deltav}}>> \scrP_{\Deltav} \\
  @V{\delta_{\Deltav'}}VV    @V{\delta_\Deltav}VV \\
     \Sigmav'  @>{\iota}>>  \Sigmav
  \end{CD}
\end{align}
The map \eqref{eq:phi} will be frequently used when we construct tropical contractions.
For $\sigma \in \scrP_{\Deltav'}$, we can see from \eqref{eq:ss} that we have 
\begin{align}\label{eq:ss2}
\phi_{\Deltav', \Deltav} \lb \sigma \rb=\lc n \in \Deltav \relmid - \la m_0, n \ra =\check{h}_\Deltav (m_0) \rc.
\end{align}
with $m_0 \in \scM(\sigma, \Deltav') \subset \scM \lb \phi_{\Deltav', \Deltav} \lb \sigma \rb, \Deltav \rb$.

\subsection{Integral affine structures with singularities}\label{sc:iass}

We recall the definition of integral affine manifolds with singularities.
We set $\mathrm{Aff}(N_\bR):= N_\bR \rtimes \GL(N)$.
Note that the linear part of $\mathrm{Aff}(N_\bR)$ is integral, while the translational part is real.

\begin{definition}
An \emph{integral affine manifold} is a real topological manifold $B$ with an atlas of coordinate charts $\psi_i \colon U_i \to N_\bR$ such that the restriction of any transition function $\psi_i \circ \psi_j^{-1} \colon U_i \cap U_j \to N_\bR$ to any connected component of $U_i \cap U_j$ is contained in $\mathrm{Aff}(N_\bR)$.
We write the sheaf on $B$ of integral affine functions (the linear part is integral, and the constant part is real) as $\mathrm{Aff}_{B}$.
\end{definition}

\begin{definition}{\rm(\cite[Definition 1.15]{{MR2213573}})}
An \emph{integral affine manifold with singularities (IAMS)} is a topological manifold $B$ with an integral affine structure on $B_0:=B \setminus \Gamma$, where $\Gamma \subset B$ is a locally finite union of locally closed submanifolds of codimension greater than or equal to $2$.
We call $\Gamma$ the \emph{discriminant locus} (or the \emph{singular locus}) of $B$.
\end{definition}

We recall the construction of IAMS using fan structures and the description of the monodromies of integral affine structures in the Gross--Siebert program.
We refer the reader to \cite[Section 1]{MR2213573} for more details.

\begin{construction}{\rm(\cite[Construction 1.26]{{MR2213573}})}\label{construction}
Suppose that we have a topological manifold $B$ (possibly with boundary) equipped with a rational polytopal structure $\scrP$, i.e., a finite set of rational polytopes satisfying $B=\bigcup_{\sigma \in \scrP} \sigma$ and the following condition:
\begin{condition}\label{cd:complex}
The following hold:
\begin{enumerate}
\item For any $\sigma \in \scrP$, all faces of $\sigma$ are also in $\scrP$. 
\item For $\sigma_1, \sigma_2 \in \scrP$, the intersection $\sigma_1 \cap \sigma_2$ is a face of both $\sigma_1$ and $\sigma_2$.
\end{enumerate}
\end{condition}
For every integer $k \geq 0$, we set $\scrP(k):=\lc \sigma \in \scrP \relmid \dim (\sigma)=k \rc$.
For an element $\tau \in \scrP$, let $\mathrm{St}(\tau) \subset B$ denote the open star of $\tau$, i.e.,
\begin{align}
\mathrm{St}(\tau):=\bigcup_{\sigma \succ \tau} \rint (\sigma).
\end{align}
A \emph{fan structure} (e.g. \cite[Definition 1.1]{{MR2846484}}) along $\tau \in \scrP$ is a continuous map $S_\tau \colon \mathrm{St}(\tau) \to \bR^k$ satisfying the following conditions:
\begin{enumerate}
\item $S_\tau^{-1}(0) = \rint (\tau)$.
\item For $\sigma \succ \tau$, the restriction $\left. S_\tau \right|_{\rint(\sigma)}$ is an integral affine submersion onto its image.
\item The collection of cones $\lc K_\sigma := \bR_{\geq 0} \cdot S_\tau \lb \sigma \cap \mathrm{St}(\tau) \rb \relmid \sigma \succ \tau \rc$ defines a finite fan in $\bR^k$.
\end{enumerate}
We say that two fan structures along the same polytope are \emph{equivalent} if they differ only by an integral linear transformation of $\bR^k$.
For $\sigma \succ \tau$, the fan structure structure $S_\tau \colon \mathrm{St}(\tau) \to \bR^k$ along $\tau$ induces the fan structure $S_\sigma \colon \mathrm{St}(\sigma) \to \bR^l$ along $\sigma$ given by
\begin{align}
\mathrm{St}(\sigma) \hookrightarrow \mathrm{St}(\tau) \xrightarrow{S_\tau} \bR^k \to \bR^k / \vspan(S_\tau \lb \rint(\sigma)\rb) \cong \bR^l.
\end{align} 
The space $(B, \scrP)$ equipped with a fan structure along every vertex $v \in \scrP(0)$, which satisfies the following condition is called a \emph{tropical manifold} in \cite[Definition 1.2]{MR2846484}.

\begin{condition}\label{cd:toric}
For any $\tau \in \scrP$ and its vertices $v_1, v_2 \prec \tau$, the fan structures along $\tau$ induced from $S_{v_1}$ and $S_{v_2}$ are equivalent.
\end{condition}

Let $(B, \scrP)$ be a tropical manifold of the above sense.
For each $\tau \in \scrP$, we take an element $a_\tau \in \rint (\tau)$, and consider the subdivision $\widetilde{\scrP}$ of $\scrP$ defined by
\begin{align}\label{eq:subdivision}
\widetilde{\scrP}:=\lc \conv \lb \lc a_{\tau_0}, a_{\tau_1}, \cdots, a_{\tau_l} \rc \rb \relmid \tau_0 \prec \tau_1 \prec \cdots \prec \tau_l, l \geq 0, \tau_i \in \scrP \rc.
\end{align}
Every fan structure $S_v$ along a vertex $v \in \scrP$ induces an integral affine structure on the open star of $v$ in $\widetilde{\scrP}$.
The interior of every maximal-dimensional polyhedron in $\scrP$ also has an integral affine structure.
We can equip $(B, \scrP)$ with an integral affine structure with singularities by giving these local integral affine structures.
The discriminant locus $\Gamma$ is given by
\begin{align}\label{eq:discriminant}
\Gamma:=\bigcup_{\substack{\tau_0 \prec \tau_1 \prec \cdots \prec \tau_l, l \geq 0, \tau_i \in \scrP \\ \dim(\tau_0) \geq 1, \dim(\tau_l) \leq \dim B-1}} \conv \lb \lc a_{\tau_0}, a_{\tau_1}, \cdots, a_{\tau_l} \rc \rb.
\end{align}
We set $B_0:=B \setminus \Gamma$.
\end{construction}

Let $\Lambda$ be the locally constant sheaf on $B_0$ of integral tangent vectors.
For each polytope $\tau \in \scrP$, let $\Lambda_\tau \subset \Lambda_x$ denote the subspace 
in the stalk at a point $x \in \rint(\tau) \cap B_0$, which consists of integral tangent vectors on $\tau$ (cf.~\cite[Definition 1.31]{MR2213573}).
Let $\omega \in \scrP(1), \rho \in \scrP(d-1)$ be polytopes such that $\omega \prec \rho$.
$\rho$ is contained in two maximal-dimensional polytopes $\sigma_{+}, \sigma_{-} \in \scrP(d)$.
Let further $v_{+}, v_{-} \in \scrP(0)$ be the vertices that $\omega$ contains.
We consider a path that starts from the vertex $v_+$, passes through the interior of $\sigma_+$, the vertex $v_-$, and the interior of $\sigma_-$ in this order, and comes back to the original point $v_+$.
Let $N$ denote the stalk of $\Lambda$ at $v_+$.
The monodromy transformation $T_{\omega}^\rho \colon N \to N$ with respect to this path has the following form:
\begin{align}\label{eq:monodromy0}
T_\omega^\rho(n)=n+\kappa_{\omega, \rho} \la \check{d}_\rho, n \ra d_\omega,
\end{align}
where $d_\omega \in \Lambda_\omega \subset N$ is the primitive integral vector pointing from $v_+$ to $v_-$, $\check{d}_\rho \in \Lambda_\rho^{\perp} \subset M:=\Hom(N, \bZ)$ is the primitive integral vector evaluating $\sigma_+$ positively, and $\kappa_{\omega, \rho} \in \bZ$ is some constant.
See \cite[Section 1.5]{MR2213573} for details.

\begin{definition}{\rm(\cite[Definition 1.54]{{MR2213573}})}\label{df:positive}
An IAMS $(B, \scrP)$ is \emph{positive} if $\kappa_{\omega, \rho}  \geq 0$ for all $\omega \in \scrP(1), \rho \in \scrP(d-1)$ such that $\omega \prec \rho$.
\end{definition}

Next, we will recall the definition of \emph{monodromy polytopes} \cite[Definition 1.58]{MR2213573}.
Assume that $(B, \scrP)$ is positive.
Let $\tau \in \scrP$ be a polytope such that $1 \leq \dim \tau \leq d-1$.
For each $\rho \in  \scrP(d-1)$ such that $\rho \succ \tau$, the monodromy polytope $\Deltav_\rho(\tau)$ is defined as follows:
Let $v_1, v_2$ be two vertices contained in $\tau$, and $\sigma_{+}, \sigma_{-}$ be the maximal-dimensional polytopes containing $\rho$.
We consider a path that starts from the vertex $v_1$, passes through the interior of $\sigma_+$, the vertex $v_2$, and the interior of $\sigma_-$ in this order, and comes back to the original point $v_1$.
The monodromy transformation $T^{\rho}_{v_1, v_2} \colon N \to N$ with respect to this path has the following form:
\begin{align}\label{eq:monodromy1}
T^{\rho}_{v_1, v_2}(n)= n + \la \check{d}_\rho, n \ra n^{\rho}_{v_1, v_2},
\end{align}
where $N$ is the stalk of $\Lambda$ at $v_+$, and $n^{\rho}_{v_1, v_2} \in \Lambda_ \tau$.
We fix a vertex $v_0 \prec \rho$, and define
\begin{align}
\Deltav_\rho(\tau):= \mathrm{conv} \lc n^{\rho}_{v_0, v} \relmid v \in \scrP(0)\ \mathrm{s.t.}\ v \prec \tau \rc \subset \Lambda_ \tau. 
\end{align}
A different choice of $v_0$ leads to a translation of $\Deltav_\rho(\tau)$.
Hence, the monodromy polytope $\Deltav_\rho(\tau)$ is well-defined up to translation.

Let $\Sigmav_\tau$ be the normal fan of $\tau$, and $\check{K}_v \in \Sigmav_\tau$ be the maximal-dimensional cone corresponding to a vertex $v \prec \tau$.
We define the piecewise linear function $\check{\psi}_{\tau, \rho}$ on $\Sigmav_\tau$ by setting $\left. \check{\psi}_{\tau, \rho} \right|_{\check{K}_v}:=-n_{v_0, v}^\rho$ for every vertex $v \prec \tau$.
When $(B, \scrP)$ is positive, the function $\check{\psi}_{\tau, \rho}$ is convex (not necessarily strictly convex) on $\Sigmav_\tau$ (cf.~\cite[Remark 1.56]{MR2213573}).
The polytope $\Deltav_\rho(\tau)$ is the Newton polytope of the function $\check{\psi}_{\tau, \rho}$, and 
the fan $\Sigmav_\tau$ is a refinement of the normal fan of $\Deltav_\rho(\tau)$ (cf.~\cite[Remark 1.59]{MR2213573}).

For each $\omega \in \scrP(1)$ such that $\omega \prec \tau$, the monodromy polytope $\Delta_\omega(\tau)$ is also defined in a similar way.
Let $v_{+}, v_{-}$ be the two vertices contained in $\omega$, and $\sigma_1, \sigma_2$ be two maximal-dimensional polytopes containing $\tau$.
We consider a path that starts from the vertex $v_+$, passes through the interior of $\sigma_1$, the vertex $v_-$, and the interior of $\sigma_2$ in this order, and comes back to the original point $v_+$.
The monodromy transformation $T_\omega^{\sigma_1, \sigma_2} \colon N \to N$ with respect to this path has the following form:
\begin{align}\label{eq:monodromy2}
T_\omega^{\sigma_1, \sigma_2}(n)= n + \la m_\omega^{\sigma_1, \sigma_2}, n \ra d_\omega,
\end{align}
where $m_\omega^{\sigma_1, \sigma_2} \in \Lambda_\tau^{\perp}$. 
We fix a maximal-dimensional polytope $\sigma_0 \succ \tau$, and define 
\begin{align}
\Delta_\omega(\tau):= \mathrm{conv} \lc m_\omega^{\sigma_0, \sigma} \relmid \sigma \in \scrP(d)\ \mathrm{s.t.}\ \sigma \succ \tau \rc \subset \Lambda_\tau^{\perp}.
\end{align}
A different choice of $\sigma_0$ leads to a translation of $\Delta_\omega(\tau)$, and the monodromy polytope $\Delta_\omega(\tau)$ is also well-defined up to translation.

\subsection{Radiance obstructions}\label{sc:radiance}

We recall the definition of radiance obstructions of \cite{MR760977}.
Let $B$ be an integral affine manifold.
We give an affine bundle structure to the tangent bundle $TB$ of $B$ as follows:
For each chart $\psi_i \colon U_i \to N_\bR$ and a point $x \in U_i$, we set an affine isomorphism
\begin{align}
\theta_{i,x} \colon T_xB \to N_\bR,\quad v \mapsto \psi_i(x)+d\psi_i(x)v,
\end{align}
and define an affine trivializations by
\begin{align}
	\theta_i \colon TU_i \to U_i \times N_\bR,\quad (x, v) \mapsto (x, \theta_{i,x}(v)),
\end{align}
where $v \in T_xB$.
This gives an affine bundle structure to $TB$.
We write $TB$ with this affine bundle structure as $T^{\mathrm{aff}}B$.
Let $\Lambda$ (resp. $\check{\Lambda}$) be the locally constant sheaf on $B$ of integral tangent vectors (resp. of integral cotangent vectors).
We set $\Lambda_\bR:=\Lambda \otimes_\bZ \bR$.

\begin{definition}{\rm(\cite{{MR760977}})}\label{df:rad}
We choose a sufficiently fine open covering $\scU:=\lc U_i \rc_i$ of $B$ so that there is a flat section $s_i \in \Gamma (U_i, T^{\mathrm{aff}}B)$ for each $U_i$.
When we set $c_B((U_i, U_j)):=s_j-s_i$ for each $1$-simplex $(U_i, U_j)$ of $\scU$, the element $c_B$ becomes a \v{C}ech $1$-cocycle for $\Lambda_\bR$.
We call $c_B \in H^1(B, \Lambda_\bR)$ the \emph{radiance obstruction} of $B$.
\end{definition}

\begin{remark}
Note that the sign convention of radiance obstructions in \pref{df:rad} is different from \cite{MR2213573}, and is the same as \cite{MR760977}.
\end{remark}

We have the exact sequence of sheaves on $B$
\begin{align}\label{eq:exaff}
0 \to \bR \to \mathrm{Aff}_{B} \to \check{\Lambda} \to 0,
\end{align}
where $\mathrm{Aff}_{B}$ is the sheaf on $B$ of integral affine functions, i.e, continuous functions $f \colon B \to \bR$ such that for every chart $\psi_i \colon U_i \to N_\bR$ of $B$, the map $f \circ \psi_i^{-1} \colon \psi_i(U_i) \to \bR$ is of the form $n \mapsto \la m, n \ra + a$ with some $m \in M$ and $a \in \bR$.

\begin{proposition}{\rm(\cite[Proposition 1.12]{MR2213573})}\label{pr:radiance}
The extension class of \eqref{eq:exaff} in $\Ext^1 \lb \check{\Lambda}, \bR \rb=H^1(B, \Lambda_\bR)$
coincides with $-c_B \in H^1(B, \Lambda_\bR)$.
\end{proposition}

Let $\lb B, \scrP \rb$ be an IAMS constructed in Construction \ref{construction}, and $\iota \colon B_0 \hookrightarrow B$ be the complement of the discriminant locus.
Let further $\Lambda$ (resp. $\check{\Lambda}$) be the locally constant sheaf on $B_0$ of integral tangent vectors (resp. of integral cotangent vectors).
Then the pushforward of the exact sequence \eqref{eq:exaff} for $B_0$ by the inclusion $\iota$ 
\begin{align}\label{eq:exaff2}
0 \to \bR \to \iota_\ast \mathrm{Aff}_{B_0} \to \iota_\ast \check{\Lambda} \to 0
\end{align}
is also exact (\cite[Proposition 1.42]{MR2213573}).
Furthermore, the radiance obstruction $c_{B_0}$ of $B_0$ is contained in $H^1(B, \iota_\ast \Lambda_\bR) \subset H^1(B_0, \Lambda_\bR)$ (\cite[Remark 1.30]{MR2213573}), and the extension class of \eqref{eq:exaff2} in 
$H^1(B, \iota_\ast \Lambda_\bR)= H^1\lb B, \cHom \lb \iota_\ast \check{\Lambda}, \bR \rb \rb \subset \Ext^1 \lb \iota_\ast \check{\Lambda}, \bR \rb$ coincides with $-c_{B_0}$ (cf.~\cite[Section 5.1]{MR2669728}).
We call $c_B:=c_{B_0} \in H^1(B, \iota_\ast \Lambda_\bR)$ the radiance obstruction of $B$.

\subsection{Tropical (co)homology groups and wave groups}\label{sc:rat}

We recall the definitions of tropical (co)homology groups and wave groups introduced in \cite{MR3330789, MR3961331}.
Let $(\bT:= \bR \cup \lc \infty \rc, \min, +)$ be the tropical number semifield.
We equip the set $\bT$ with the topology which makes it homeomorphic to a half line.
The tropical affine space $\bT^d$ has a natural stratification
\begin{align}
\bT^d=\bigsqcup_{I \subset \lc 1, \cdots, r\rc} \bR_I^d,
\end{align}
where $\bR^d_I:=\lc (x_1, \cdots, x_d) \in \bT^d \relmid x_i=\infty\ \mathrm{if\ and\ only\ if\ } i \in I \rc \cong \bR^{d-|I|}$.
A \emph{rational polyhedron} in $\bT^d$ is the closure in $\bT^d$ of a rational polyhedron in some stratum $\bR^d_I \subset \bT^d$.
Let $\overline{P} \subset \bT^d$ be the polyhedron which is the closure of a rational polyhedron $P \subset \bR^d_I$.
A \emph{finite face} of $\overline{P}$ is the closure of a face of $P$.
An \emph{infinite face} of $\overline{P}$ is the closure of a non-empty intersection $\overline{P} \cap \bR^d_J$ with some $J \supsetneq I$.
A \emph{rational polyhedral complex} $\scrP$ in $\bT^d$ is a finite set of polyhedra in $\bT^d$ satisfying \pref{cd:complex}.
We call $\tau \in \scrP$ an \emph{infinite polyhedron} if there exists a polyhedron $\sigma \in \scrP$ such that $\tau$ is an infinite face of $\sigma$.
We also call a polyhedron in $\scrP$ which is not an infinite polyhedron a \emph{finite polyhedron}.
The \emph{support} of a rational polyhedral complex $\scrP$ in $\bT^d$ is the union of all polyhedra in $\scrP$.
For every integer $k \geq 0$, we set
$\scrP(k):=\lc \sigma \in \scrP \relmid \dim (\sigma)=k \rc$.

An \emph{integral affine function} on a subset $X \subset \bT^d$ is a continuous function $f \colon X \to \bR$ that is of the form 
$x \mapsto \la m, x \ra + a$ for some $m=\lb m_1, \cdots, m_d \rb \in \lb \bZ^d \rb^\ast$ and $a \in \bR$ locally around every point in $X$.
Here we use the convention $0 \cdot \infty=0$, and around any point $x \in X \cap \bR_I^d$, we only consider functions $f=\la m, \bullet \ra+a$ with $m_i=0$ for every $i \in I$.
We write the sheaf on $X$ of integral affine functions as $\mathrm{Aff}_X$.

\begin{definition}\label{df:rat-sp}
A \emph{rational polyhedral space} $(X, \mathrm{Aff}_X)$ is a pair of a second-countable Hausdorff topological space $X$ and a sheaf $\mathrm{Aff}_X$ of continuous functions such that 
for every $x \in X$ there exists an open neighborhood $U \subset X$, an open subset $V$ of the support of a rational polyhedral complex $\scrP$ in $\bT^d$, and a homeomorphism $\varphi \colon U \to V$ that induces an isomorphism 
$\mathrm{Aff}_V \to \varphi_\ast \mathrm{Aff}_U$ via the pullback of functions.
We call $\varphi \colon U \to V$ a \emph{chart}.
\end{definition}

\begin{definition}{\rm(\cite[Definition 1.14]{MR3330789})}\label{df:trop-mfd}
A rational polyhedral space $(X, \mathrm{Aff}_X)$ is called a \emph{tropical manifold}, if every point $x \in X$ has a neighborhood that is isomorphic to an open subset of the direct product of $\bT^k$ $(k \geq 0)$ and the tropical linear space $L_M$ associated with a matroid $M$.
\end{definition}
We refer the reader to \cite[Section 4.2]{MR3287221} for details about matroids and tropical linear spaces.

Let $\scrP$ be a rational polyhedral complex in $\bT^d$.
For a subset $I \subset \lc 1, \cdots, d \rc$, we set
\begin{align}
\scrP_I:=\lc \sigma \in \scrP \relmid \rint(\sigma) \subset \bR^d_I \rc.
\end{align}
Let $p \geq 0$ be an integer.

\begin{definition}{\rm(\cite{MR3330789, MR3961331, MR4637248})}
For a polyhedron $\tau \in \scrP_I$, we define 
the \emph{$p$-th integral multi-tangent space} and 
the \emph{$p$-th integral multi-cotangent space} of $\scrP$ at $\tau$ by
\begin{align}
F_p^\bZ (\tau)&:= \lb \sum_{\sigma \in \scrP_I, \sigma \succ \tau} \bigwedge^p T(\sigma) \rb \cap \bigwedge^p \bZ^d_I\\
F^p_\bZ (\tau)&:= \Hom \lb F_p^\bZ (\tau), \bZ \rb.
\end{align}
Here $T(\sigma) \subset \bR^d_I$ is the linear subspace generated by tangent vectors on $\sigma$, which we defined in \eqref{eq:T}.
$\bZ^d_I$ is the natural lattice in $\bR^d_I$.
For polyhedra $\tau_1, \tau_2 \in \scrP$ such that $\tau_2 \prec \tau_1$, we define the map
\begin{align}\label{eq:ext}
e_{\tau_2, \tau_1} \colon F_p^\bZ (\tau_1) \to F_p^\bZ (\tau_2)
\end{align}
as follows:
If $\tau_1, \tau_2 \in \scrP_I$ for some common $I \subset \lc 1, \cdots, d \rc$, we define \eqref{eq:ext} to be the inclusion.
If not, then $\tau_1 \in \scrP_{I_1}, \tau_2 \in \scrP_{I_2}$ for some $I_1 \subset I_2 \subset \lc 1, \cdots, d \rc$.
We define the map \eqref{eq:ext} to be the map induced by the quotient $\bZ_{I_1}^d \twoheadrightarrow \bZ_{I_2}^d=\left. \bZ_{I_1}^d \middle/ \bZ_{I_2 \setminus I_1}^d \right.$.
We also define 
\begin{align}
r_{\tau_2, \tau_1} \colon F^p_\bZ (\tau_2) \to F^p_\bZ (\tau_1)
\end{align}
as the dual of \eqref{eq:ext}.
\end{definition}

\begin{definition}{\rm(\cite[Section 2.4]{{MR3330789}})}\label{df:scF}
Let $V$ be an open subset of the support of a rational polyhedral complex $\scrP$ in $\bT^d$.
For any open set $U \subset V$, we consider the poset $P(U)$ whose elements are connected components $v$
of the intersection $\tau_v \cap U$ with some polyhedron $\tau_v \in \scrP$. 
The order $\leq$ on $P(U)$ is defined by setting $v_1 \leq v_2$ if $v_1 \supset v_2$.
We set $F_p^\bZ (v):=F_p^\bZ (\tau_v)$, and for a pair $v_1 \leq v_2$, we consider $e_{\tau_{v_2}, \tau_{v_1}} \colon F_p^\bZ (\tau_{v_1}) \to F_p^\bZ (\tau_{v_2})$.
The \emph{cosheaf of $p$-th integral multi-tangent spaces} $\scF^\bZ_p$ on $V$ is defined by
\begin{align}
\scF^\bZ_p \lb U \rb:=\varinjlim_{v \in P(U)} F_p^\bZ (v),
\end{align}
where the right hand side is the limit of $\lb \lc F_p^\bZ (v) \rc_{v \in P(U)}, \lc e_{\tau_{v_2}, \tau_{v_1}} \rc_{v_1 \leq v_2} \rb$.
The \emph{sheaf of $p$-th integral multi-cotangent spaces} $\scF_\bZ^p$ on $V$ is defined by
\begin{align}
\scF_\bZ^p \lb U \rb:=\Hom \lb \scF^\bZ_p \lb U \rb, \bZ \rb.
\end{align}
For a general rational polyhedral space $(X, \mathrm{Aff}_X)$, we define the cosheaf $\scF^\bZ_p$ and the sheaf $\scF_\bZ^p$ on $X$ by gluing along charts.
For $Q=\bQ, \bR$, we also set $\scF^Q_p:=\scF^\bZ_p \otimes_\bZ Q, \scF_Q^p:=\scF_\bZ^p \otimes_\bZ Q$.
\end{definition}

Let $(X, \mathrm{Aff}_X)$ be a rational polyhedral space, $p, q \geq 0$ be integers, and $Q=\bZ, \bQ, \bR$.

\begin{definition}
The \emph{$(p,q)$-th tropical cohomology group} $H^{p,q} \lb X, Q \rb$ and 
the \emph{$(p,q)$-th tropical cohomology group with compact support} $H^{p,q}_c \lb X, Q \rb$ of $X$ are defined by
\begin{align}
H^{p,q} \lb X, Q \rb:=H^q \lb X, \scF_Q^p \rb, \quad H^{p,q}_c \lb X, Q \rb:=H^q_c \lb X, \scF_Q^p \rb.
\end{align}
\end{definition}

\begin{definition}{\rm(\cite[Definition 4.3]{{MR4637248}})}
The \emph{$(p,q)$-th tropical Borel--Moore homology group} $H_{p,q}^{\mathrm{BM}} \lb X, Q \rb$ and 
the \emph{$(p,q)$-th tropical homology group with compact support} $H_{p,q} \lb X, Q \rb$ of $X$ are defined by
\begin{align}
H_{p,q}^{\mathrm{BM}} \lb X, Q \rb:=R^{-q} \Gamma R \cHom \lb \scF^p_Q, \omega_X \rb, 
\quad H_{p,q} \lb X, Q \rb:=R^{-q} \Gamma_c R \cHom \lb \scF^p_Q, \omega_X \rb,
\end{align}
where $\omega_X \in D^b(\bZ_X)$ is the dualizing complex of $X$ (cf.~e.g.~\cite[Definition 3.1.16]{MR1299726}).
Here $\bZ_X$ denotes the constant sheaf on $X$ whose stalk is $\bZ$, and $D^b(\bZ_X)$ denotes the bounded derived category of $\bZ_X$-modules on $X$.
\end{definition}

Next, we recall the definition of tropical wave groups introduced in \cite{MR3330789}.
Although the following definition of the sheaf of wave tangent spaces looks different from that in \cite{MR3330789} at first glance, they actually agree with each other as we will see in \pref{pr:w-stalk}.

\begin{definition}\label{df:wave}
The \emph{sheaf of $p$-th wave tangent spaces} on a rational polyhedral space $(X, \mathrm{Aff}_X)$ is defined by
\begin{align}
\scW_p^Q:=\cHom \lb \cF^p_\bZ, Q \rb.
\end{align}
The \emph{$(p, q)$-th tropical wave group} of $X$ is the cohomology group $H^q \lb X, \scW_p^Q \rb$.
\end{definition}

For a rational polyhedral complex $\scrP$ in $\bT^d$, we consider the following condition:
\begin{condition}\label{cd:parent}
For any infinite polyhedron $\tau \in \scrP$, the following hold:
\begin{enumerate}
\item There uniquely exists a finite polyhedron $\tilde{\tau} \in \scrP$ such that $\tau$ is an infinite face of $\tilde{\tau}$.
\item One has
\begin{align}\label{eq:t-tau}
\lc \sigma \in \scrP \relmid \sigma \mathrm{\ is\ a\ finite\ polyhedron\ s.t.\ } \sigma \succ \tau \rc
=\lc \sigma \in \scrP \relmid \sigma \succ \tilde{\tau} \rc.
\end{align}
\end{enumerate}
\end{condition}

Let $\scrP$ be a rational polyhedral complex in $\bT^d$, which consists of the closures of rational polyhedra in $\bR^d \subset \bT^d$ and their faces.
In \cite{MR3330789}, it is said that $\scrP$ is \emph{regular at infinity} if for any subset $I \subset \lc 1, \cdots, d \rc$ and any $\sigma \in \scrP$ such that $\sigma \cap \bR^d \neq \emptyset$, the intersection $\sigma \cap \bR^d_I$ is either of dimension $\lb \dim \sigma - |I| \rb$ or empty.
The condition of being regular at infinity is imposed throughout the work \cite{MR3330789}.

\begin{lemma}
If a rational polyhedral complex $\scrP$ in $\bT^d$ is regular at infinity, then it satisfies \pref{cd:parent}.
\end{lemma}
\begin{proof}
It is mentioned in \cite[Section 1.2]{MR3330789} that a rational polyhedral complex $\scrP$ which is regular at infinity satisfies \pref{cd:parent}.1.
We show \eqref{eq:t-tau}.
For an infinite polyhedron $\tau \in \scrP$, the polyhedron $\tilde{\tau} \in \scrP$ of \pref{cd:parent}.1 is the closure of a polyhedron in $\bR^d \subset \bT^d$.
Let $\sigma \in \scrP$ be an element in the right hand side of \eqref{eq:t-tau}, i.e., a polyhedron such that $\sigma \succ \tilde{\tau}$.
Then $\sigma$ is also the closure of a polyhedron in $\bR^d \subset \bT^d$.
In particular, it is a finite polyhedron.
Since we also have $\sigma \succ \tilde{\tau} \succ \tau$, the polyhedron $\sigma$ is contained in the left hand side of \eqref{eq:t-tau}.
Thus the right hand side of \eqref{eq:t-tau} is contained in the left hand side.

We show the opposite inclusion.
Let $\sigma \in \scrP$ be a finite polyhedron such that $\sigma \succ \tau$.
From \cite[Proposition 1.6]{MR3330789}, we can see that the polyhedron $\sigma$ contains points in the relative interior of $\tilde{\tau}$.
This implies that one has $\sigma \succ \tilde{\tau}$, and $\sigma$ is contained in the right hand side of \eqref{eq:t-tau}.
We obtained the claim.
\end{proof}

\begin{example}
Let $\lc e_1, e_2 \rc$ be the standard basis of $\bZ^2 \subset \bT^2$.
The closure of $\bR (e_1+e_2)$ in $\bT^2$ satisfies \pref{cd:parent}, and is not regular at infinity.
The closure of the union of $\bR (e_1+e_2)$ and $e_2+\bR (e_1+e_2)$ in $\bT^2$ does not satisfy \pref{cd:parent}.
The space $X(f_1, \cdots, f_r)$ that we consider for local models of tropical contractions also satisfies \pref{cd:parent} (\pref{lm:parent}), although it is not regular at infinity in general.
\end{example}

\begin{proposition}\label{pr:w-stalk}
Let $\scrP$ be a rational polyhedral complex in $\bT^d$ satisfying \pref{cd:parent}, and $x$ be a point in the support $X$ of $\scrP$.
Let further $\tau_x \in \scrP$ be the polyhedron which contains $x$ in its relative interior.
We suppose $\tau_x \in \scrP_I$.
When $\tau_x$ is a finite polyhedron, the stalk of $\scW_{p}^Q$ at $x$ is 
\begin{align}\label{eq:wavestalk}
\scW_{p, x}^Q=\lc \bigwedge^p \lb \bigcap_{\substack{\sigma \in \scrP_I, \sigma \succ \tau_x \\ \sigma: \mathrm{maximal}}} T(\sigma) \rb \rc \cap \bigwedge^p \lb \bZ^d_I \otimes_\bZ Q \rb.
\end{align}
When $\tau_x$ is an infinite face, let $\tilde{\tau}_x \in \scrP$ be the finite polyhedron  such that $\tau_x$ is the closure of $\tilde{\tau}_x \cap \bR^d_I$.
($\tilde{\tau}_x$ is unique by \pref{cd:parent}.1.)
Then we have
\begin{align}\label{eq:Wpx}
\scW_{p, x}^Q=\scW_{p, \tilde{x}}^Q,
\end{align}
where $\tilde{x}$ is a point in $\rint \lb \tilde{\tau}_x \rb$.
\end{proposition}
\begin{proof}
Let $U_x \subset X$ be a small open neighborhood of $x$.
We will compute the section $\scW_p^Q \lb U_x \rb$ over $U_x$.
We write the lattice $\bZ^d \subset \bT^d$ as $N$, and its dual as $M:=\Hom \lb N, \bZ \rb$.
We also set $N_Q:=N \otimes_\bZ Q$.
We have
\begin{align}
\scF^\bZ_{p, x}&=F_p^\bZ (\tau_x)
=\lb \sum_{\sigma \in \scrP_I, \sigma \succ \tau_x} \bigwedge^p T(\sigma) \rb \cap \bigwedge^p \bZ^d_I\\
\scF_{\bZ, x}^p&=\left. \bigwedge^p M \middle/ \lb \scF^\bZ_{p, x} \rb^\perp \right.
=\left. \bigwedge^p M \middle/ \bigcap_{\sigma \in \scrP_I, \sigma \succ \tau_x} \lb \bigwedge^p T(\sigma) \rb^\perp \right. .
\end{align}

First, suppose $\tau_x$ is a finite polyhedron.
Then for any point $y \in U_x$, the natural map $\scF_{\bZ, x}^p \to \scF_{\bZ, y}^p$ which is induced by the restriction map of the sheaf $ \scF_{\bZ}^p$ is a quotient map.
It turns out that the section $\scW_p^Q \lb U_x \rb \subset \bigwedge^p N_Q$ consists of maps $\scF_{\bZ, x}^p \to Q$ that factors through $\scF_{\bZ, x}^p \to \scF_{\bZ, y}^p$ for any point $y$ in $U_x$.
Hence, we get
\begin{align}
\scW_{p}^Q \lb U_x \rb=\bigcap_{\substack{\sigma \in \scrP_I, \sigma \succ \tau_x \\ \sigma: \mathrm{maximal}}} \lb \lb \bigwedge^p T(\sigma) \rb^\perp \rb^\perp
=\lb \bigcap_{\substack{\sigma \in \scrP_I, \sigma \succ \tau_x \\ \sigma: \mathrm{maximal}}} \bigwedge^p T(\sigma) \rb \cap \bigwedge^p N_Q.
\end{align}
This is equal to the right hand side of \eqref{eq:wavestalk}.
Therefore, we have \eqref{eq:wavestalk}.

Next, suppose $\tau_x$ is an infinite face.
Take a point $\tilde{x} \in \rint \lb \tilde{\tau}_x \rb \cap U_x$.
Let further $\sigma \in \scrP$ be an arbitrary infinite polyhedron such that $\sigma \succ \tau_x$, and $\tilde{\sigma} \in \scrP$ be the unique finite polyhedron such that $\sigma$ is the closure of $\tilde{\sigma} \cap \bR^d_J$ with some $J \subset \lc 1, \cdots, r \rc$.
For any points $y \in \rint \lb \sigma \rb \cap U_x$ and $\tilde{y} \in \rint \lb \tilde{\sigma} \rb \cap U_x$, the natural map $\scF_{\bZ, y}^p \to \scF_{\bZ, \tilde{y}}^p$ induced by the restriction map of the sheaf $ \scF_{\bZ}^p$ is an inclusion.
It turns out that for any open subset $U \subset U_x$, the restriction map
\begin{align}
\scF^p_{\bZ} \lb U \rb \to \scF^p_{\bZ} \lb U \setminus \bigcup_{\substack{\sigma : \mathrm{infinite} \\ \sigma \succ \tau_x}} \rint \lb \sigma \rb \rb
\end{align}
is an inclusion.
From this, one can see $\scW_{p}^Q \lb U_x \rb =\scW_{p}^Q \lb U_x \setminus \bigcup_{\substack{\sigma : \mathrm{infinite} \\ \sigma \succ \tau_x}} \rint \lb \sigma \rb \rb$.
From \eqref{eq:t-tau}, we can also see that $\lb U_x \setminus \bigcup_{\substack{\sigma : \mathrm{infinite} \\ \sigma \succ \tau_x}} \rint \lb \sigma \rb \rb$ is a small open neighborhood of $\tilde{x}$.
Thus we obtain \eqref{eq:Wpx}.
\end{proof}

\pref{eq:wavestalk} and \pref{eq:Wpx} agree with the definition of wave tangent spaces in \cite{MR3330789}.

\begin{remark}\label{rm:MZ}
To be precise, the definition of the wave tangent space at a finite polyhedron, which was given in \cite{MR3330789} does not agree with \eqref{eq:wavestalk} in general, if it is interpreted literally.
Let $\scrP$ be a rational polyhedral complex in $\bT^d$, and $X$ be its support.
Let further $\tau \in \scrP_I$ be a finite polyhedron, and $x \in \rint (\tau)$ be a point.
We define the \emph{local cone} $\Sigma(x)$ at $x$ by
\begin{align}
\Sigma(x):=\lc y \in \bR^d_I \relmid x +\varepsilon y \in X \mathrm{\ for\ a\ sufficiently\ small\ } \varepsilon>0 \rc.
\end{align}
In \cite[Section 1.3]{MR3330789}, the first wave tangent space $\scW_{1, x}^\bR$ over $\bR$ at $x$ is defined as the intersection of all maximal linear subspaces contained in the local cone $\Sigma(x)$, and $\scW_{p, x}^\bR:=\bigwedge^p \scW_{1, x}^\bR$ for $p \geq 2$.
In general, this does not coincide with \eqref{eq:wavestalk}.
For instance, consider the support $|\Sigma|$ of the fan $\Sigma$ in $N_\bR \cong \bR e_1 \oplus \bR e_2$ given by
\begin{align}
\Sigma:= \lc \bR_{\geq 0} e_1, \bR_{\geq 0} (-e_1),  \bR_{\geq 0} (2e_1+e_2),  \bR_{\geq 0} (-2e_1+e_2),  \bR_{\geq 0} (e_1-e_2), \bR_{\geq 0} (-e_1-e_2), \lc 0 \rc \rc.
\end{align}
\eqref{eq:wavestalk} with $p=1$ and $x=0$ is $\lc 0 \rc$, whereas the only maximal linear subspace contained in $|\Sigma|$ is $\bR e_1$.
The author learned from Ilia Zharkov that what they actually meant in the definition of wave tangent spaces in \cite{MR3330789} is \eqref{eq:wavestalk}.
\end{remark}

\begin{definition}
We say that two points $x_1, x_2 \in X$ are \emph{equivalent} when there exists a path connecting $x_1$ and $x_2$, along which the dimension of the stalk of $\scW_1^\bR$ is constant. 
We also call the equivalence classes \emph{strata}.
For two strata $\scE_1, \scE_2$ of $X$, we write $\scE_1 \prec \scE_2$ when $\scE_1$ is on the boundary of $\scE_2$.
\end{definition}

Lastly, we recall the definition of eigenwaves introduced in \cite{MR3330789}.
It is represented as an element of the \v{C}ech cohomology group as follows:
\begin{definition}\label{df:eigenwave}
Let $(X, \mathrm{Aff}_X)$ be a rational polyhedral space.
Take an open covering $\scU:=\lc U_i \rc_{i \in I}$ of $X$ so that it satisfies the following:
\begin{enumerate}
\item For any subset $I_0 \subset I$ such that $\bigcap_{i \in I_0} U_i \neq \emptyset$, there exists a minimum stratum $\scE_{I_0}^{\min}$ of $\bigcap_{i \in I_0} U_i$ such that $\scE \succ \scE_{I_0}^{\min}$ for any stratum $\scE$ of $\bigcap_{i \in I_0} U_i$.
\item For any open set $U_{i} \in \scU$, there is a chart $\varphi \colon U \to V$ of $X$ such that $U$ contains $U_{i}$. 
\end{enumerate}
The covering $\scU$ is an acyclic covering for the sheaf $\scW_p^Q$ as we check in \pref{lm:acyclic} below.
We further take points $p_i \in U_i, p_{i, j} \in U_i \cap U_j$ from every open set $U_i$ and every $1$-simplex $(U_{i}, U_{j})$ of $\scU$ so that they sit in the relative interior of a finite polyhedron and $p_i \in \scE_{\lc i \rc}^{\min}$, $p_{i, j} \in \scE_{\lc i, j \rc}^{\min}$.
The vector $\lb p_{i}-p_{i, j} \rb$ in a chart containing $U_i$ can be regarded as an element in $\scW_1^\bR \lb U_{i} \cap U_{j} \rb$.
We set $c_X \lb (U_{i}, U_{j}) \rb:=\lb p_{j}-p_{i, j} \rb - \lb p_{i}-p_{i, j} \rb \in \scW_1^\bR \lb U_{i} \cap U_{j} \rb$ for every $1$-simplex $(U_{i}, U_{j})$ of $\scU$.
Then it defines a class in $H^1 \lb X, \scW_1^\bR \rb$, which we will write as $c_X \in H^1 \lb X, \scW_1^\bR \rb$.
We call it the \emph{eigenwave} of $X$.
\end{definition}

\begin{remark}
The reason why we take an intermediate point $p_{i, j} \in U_i \cap U_j$ is that there might not be a chart $\varphi \colon U \to V$ such that $U$ contains both $p_i$ and $p_j$ in general.
If there is such a chart, then the vector $\lb p_j-p_i \rb$ can be regarded as an element in $\scW_1^\bR \lb U_{i} \cap U_{j} \rb$, and $c_X \lb (U_{i}, U_{j}) \rb=p_j-p_i$.
\end{remark}

\begin{remark}
The eigenwave was originally defined as an element of the singular cohomology group.
We refer the reader to \cite[Section 5.1]{MR3330789} for its precise definition.
When a rational polyhedral space $X$ has a polyhedral structure $\scrP$ in the sense of \cite[Definition 1.10]{MR3330789}, one can see that the definition of eigenwaves in \cite{MR3330789} agrees with \pref{df:eigenwave} as follows:
One can take a triangulation $\scrT$ of $\scrP$ (cf.~\cite[Section 1.4]{MR3330789}).
The \v{C}ech cohomology group with respect to the covering $\lc U_p \rc_{p \in \scrT(0)}$ of open stars of vertices of $\scrT$ is naturally isomorphic to the simplicial cohomology group with respect to $\scrT$.
We further identify the simplicial cohomology group with the singular cohomology group.
For each $p \in \scrT(0)$, we take a mobile point $p' \in U_p$ which is mapped to $p$ by the projection along the divisorial directions of $p$.
(We refer the reader to \cite[Section 1]{MR3330789} for the terminologies \emph{mobile} and \emph{divisorial directions}.)
Then the value of the eigenwave of \cite{MR3330789} for the singular $1$-simplex corresponding to $\lb U_{p_1}, U_{p_2} \rb$ is $p_2'-p_1'$.
Here each point $p'$ is in the relative interior of a finite polyhedron and in the minimum stratum of $U_p$.
Hence, the definition in \cite{MR3330789} agrees with \pref{df:eigenwave}.
\end{remark}

\begin{lemma}\label{lm:acyclic}
The covering $\scU$ in \pref{df:eigenwave} is an acyclic covering for the sheaf $\scW_p^Q$ $(p \geq 0, Q=\bZ, \bQ, \bR)$.
\end{lemma}
\begin{proof}
This can be checked by the same argument as the one in \cite[Lemma 5.5]{MR2213573}.
For any subset $I_0 \subset I$ such that $\bigcap_{i \in I_0} U_i \neq \emptyset$, take a point $x \in \scE_{I_0}^{\min}$, a fundamental system of neighborhoods $\lc W_j \rc_j$ of $x$ in $\bigcap_{i \in I_0} U_i$, and homeomorphism $\lc \varphi_j \colon W_j \to \bigcap_{i \in I_0} U_i \rc_j$ such that $\varphi_j^{-1} \scW_p^Q=\scW_p^Q$.
One has the isomorphism
\begin{align}\label{eq:retract-stalk}
H^q \lb \bigcap_{i \in I_0} U_i, \scW_p^Q \rb \cong H^q \lb W_j, \scW_p^Q \rb
\end{align}
for all $j$.
Thus $H^q \lb W_j, \scW_p^Q \rb$ is isomorphic to the stalk of $R^q \id_\ast \scW_p^Q$ at $x$.
In general, one has $R^q \id_\ast \scA=0$ $(p \geq 1)$ for any sheaf $\scA$ of abelian groups.
Hence, we can see that \eqref{eq:retract-stalk} is trivial, and the covering $\scU$ is acyclic for the sheaf $\scW_p^Q$.
\end{proof}

Recall that one has the \emph{tropical exponential sheaf sequence} (cf.~e.g.~\cite[Section 3]{MR3894860})
\begin{align}\label{eq:expseq}
0 \to \bR \to \mathrm{Aff}_X \to \scF^1_\bZ \to 0.
\end{align}

\begin{proposition}\label{pr:eigext}
The extension class of the tropical exponential sheaf sequence \eqref{eq:expseq} in 
$H^1 \lb X, \scW_1^\bR \rb =H^1 \lb X, \cHom \lb \scF^1_\bZ, \bR \rb \rb \subset \Ext^1 \lb \scF^1_\bZ, \bR \rb$
coincides with $-c_X \in H^1 \lb X, \scW_1^\bR \rb$.
\end{proposition}
\begin{proof}
We take an open covering $\scU:=\lc U_i \rc_{i \in I}$ of $X$ and points $p_i, p_{i, j}$ as we did in \pref{df:eigenwave}.
Let $s_i \colon \scF^1_\bZ \to \mathrm{Aff}_X$ be the splitting on each open set $U_i$ such that all elements in the image of the induced homomorphism $\scF^1_{\bZ, p_i} \to \mathrm{Aff}_{X, p_i}$ between the stalks at $p_i$ take the value $0$ at $p_i$.
The extension class of \eqref{eq:expseq} is given by the $1$-cocycle $\lc s_j - s_i \colon \scF^1_{\bZ} \to \bR \rc_{i, j}$.
For any $m \in \Gamma \lb U_i \cap U_j, \scF^1_{\bZ} \rb$ and $q \in U_i \cap U_j$, we have 
\begin{align}
\lc \lb s_j - s_i \rb (m) \rc(q)
&=\la m, q-p_j \ra-\la m, q-p_i \ra \\
&=\la m, p_{i, j}-p_j \ra-\la m, p_{i, j}-p_i \ra \\
&=-c_X \lb (U_i, U_j) \rb(m).
\end{align}
Hence, the extension class of \eqref{eq:expseq} coincides with $-c_X$.
\end{proof}

\subsection{Tropical toric varieties and tropicalizations}\label{sc:toric}

We recall the definition of tropical toric varieties of \cite{MR2428356, MR2511632}.
We refer the reader to \cite[Section 1]{MR2428356} or \cite[Section 3]{MR2511632} for more details.
Let $\Sigma$ be a fan in $N_\bR$.
For each cone $C \in \Sigma$, we set
\begin{align}
	C^\vee&:= \lc m \in M_\bR \relmid \la m,n \ra \geq 0  \mathrm{\ for\ all\ } n \in C \rc \\
	C^\perp&:= \lc m \in M_\bR \relmid \la m,n \ra =0  \mathrm{\ for\ all\ } n \in C \rc.
\end{align}
For each cone $C \in \Sigma$, we define $X_C(\bT)$ as the set of monoid homomorphisms $C^\vee \cap M \to (\bT, + )$
\begin{align}
	X_C(\bT):=\Hom(C^\vee \cap M, \bT)
\end{align}
with the compact open topology.
For elements $p_1, p_2 \in X_C(\bT)$, we define $p_1+p_2 \in X_C(\bT)$ as
\begin{align}
C^\vee \cap M \to \bT, \quad m \mapsto \lb p_1+p_2 \rb (m):=p_1(m)+p_2(m).
\end{align}
Then $\lb X_C(\bT), + \rb$ also becomes a monoid.
For an element $s \in C \subset X_C(\bT)$ and $\lambda \in \bT_{\geq 0}$, we also define $\lambda s \in X_C(\bT)$ as
\begin{align}
C^\vee \cap M \to \bT, \quad m \mapsto \lambda s(m).
\end{align}
Here we use the convention $\infty \cdot 0 =0$.
For a subset $S \subset C \subset X_C(\bT)$, we define 
\begin{align}\label{eq:tcone}
\cone_\bT \lb S \rb:=\lc \sum_{i=1}^n \lambda_i s_i \in X_C(\bT) \relmid n \geq 0, s_i \in S, \lambda_i \in \bT_{\geq 0} \rc. 
\end{align}
When a cone $C_1 \in \Sigma$ is a face of a cone $C_2 \in \Sigma$, we have a natural immersion,
\begin{align}
	X_{C_1}(\bT) \to X_{C_2}(\bT),\quad \left(p \colon C_1^\vee \cap M \to \bT\right) \mapsto ({C_2}^\vee \cap M \subset C_1^\vee \cap M \xrightarrow{p} \bT).
\end{align}
By gluing $\lc X_C(\bT) \rc_{C \in \Sigma}$ together, we obtain the \emph{tropical toric variety} $X_\Sigma(\bT)$ associated with the fan $\Sigma$
\begin{align}
	X_\Sigma(\bT):=\lb \bigsqcup_{C \in \Sigma} X_C(\bT) \rb \bigg/ \sim .
\end{align}
We also set $O_C(\bT):=\Hom(C^\perp \cap M,\bR)=N_\bR / \vspan(C)$ which we call the \emph{tropical torus orbit} corresponding to the cone $C \in \Sigma$.
For a face $C' \prec C$, there is a natural inclusion
\begin{align}\label{eq:O-emb}
O_{C'}(\bT) \hookrightarrow X_C(\bT), \quad \left(p \colon {C'}^\perp \cap M \to \bR \right) \mapsto 
\lb m \mapsto 
\left\{
\begin{array}{ll}
p(m) & m \in {C'}^\perp \cap M \\
\infty & \mathrm{otherwise} \\
\end{array}
\right.
\rb.
\end{align}
We think of the set $O_C(\bT)$ as a subset of $X_C(\bT)$ by this inclusion.
One has
\begin{align}
X_\Sigma(\bT)=\bigsqcup_{C \in \Sigma} O_{C}(\bT).
\end{align}
There is also a natural projection map
\begin{align}\label{eq:proj}
\pi_C \colon X_C(\bT) \to O_C(\bT), \quad \left(p \colon C^\vee \cap M \to \bT\right) \mapsto (C^\perp \cap M \subset C^\vee \cap M \xrightarrow{p} \bT).
\end{align}

In \pref{sc:rat}, we defined polyhedra and polyhedral complexes in $\bT^d$.
In a similar spirit, we call the closure in $X_\Sigma(\bT)$ of a rational polyhedron in some tropical torus orbit $O_{C}(\bT) \subset X_\Sigma(\bT)$ $\lb C \in \Sigma \rb$ a \emph{rational polyhedron} in $X_\Sigma(\bT)$.
Let $\overline{P} \subset X_\Sigma(\bT)$ be the polyhedron which is the closure of a rational polyhedron $P \subset O_{C}(\bT)$.
We call the closure of a face of $P$ a \emph{finite face} of $\overline{P}$, and the closure of a non-empty intersection $\overline{P} \cap O_{C'}(\bT)$ with some $C' \in \Sigma$ such that $C' \succ C$ an \emph{infinite face} of $\overline{P}$.
We also call a finite set $\scrP$ of rational polyhedra in $X_\Sigma(\bT)$ satisfying \pref{cd:complex} a \emph{rational polyhedral complex} in $X_\Sigma(\bT)$.

Let $C \subset N_\bR$ be a rational cone, and $X_{C}(\bT)$ be the associated tropical toric variety.
By Gordan's lemma (cf.~e.g.~\cite[Proposition 1.2.17]{MR2810322}), the monoid $C^\vee \cap M$ is finitely generated.
Hence, there exists a surjective monoid homomorphism from $\lb \bZ_{\geq 0} \rb^n$ to $C^\vee \cap M$, which induces a closed embedding 
\begin{align}\label{eq:G-emb}
\iota \colon X_{C}(\bT)= \Hom \lb C^\vee \cap M, \bT \rb \hookrightarrow \bT^n=\Hom \lb \lb \bZ_{\geq 0} \rb^n, \bT \rb.
\end{align}

\begin{lemma}\label{lm:}
Let $C' \prec C$ be a face.
The image of the corresponding tropical torus orbit $O_{C'}(\bT)=\Hom(C'^\perp \cap M,\bR) \subset X_{C}(\bT)$ by the embedding \eqref{eq:G-emb} is contained in 
\begin{align}
\bR^n_I:=\lc (x_1, \cdots, x_n) \in \bT^n \relmid x_i=\infty\ \mathrm{if\ and\ only\ if\ } i \in I \rc
\end{align}
with some subset $I \subset \lc 1, \cdots, n \rc$, and the restriction of the embedding \eqref{eq:G-emb} to the lattice $\Hom \lb {C'}^\perp \cap M, \bZ \rb=\left. N \middle/ \lb \vspan \lb C' \rb \cap N \rb \right. \subset O_{C'}(\bT)$ is a primitive embedding into the lattice $\bZ_I^n$ of $\bR^n_I$.
\end{lemma}
\begin{proof}
Let $\lc m_1, \cdots, m_n \rc$ be generators of $C^\vee \cap M$.
We set
\begin{align}
J:= \lc j \in \lc 1, \cdots, n \rc \relmid m_j \in {C'}^\perp \rc,
\end{align}
and $I:= \lc 1, \cdots, n \rc \setminus J$.
We can see from \eqref{eq:O-emb} that the image of $O_{C'}(\bT)$ by the embedding \eqref{eq:G-emb} is contained in $\bR^n_I \subset \bT^n$.
We need to check that the lattice $\Hom \lb {C'}^\perp \cap M, \bZ \rb \subset O_{C'}(\bT)$ is primitively embedded into the lattice $\bZ_I^n \subset \bR^n_I$.
It suffices to show that its dual $\lb \bZ_I^n \rb^\ast \to {C'}^\perp \cap M$ is surjective.
It is equivalent to
\begin{align}\label{eq:CM}
\sum_{j \in J} \bZ m_j={C'}^\perp \cap M.
\end{align}
The monoid $C^\vee \cap {C'}^\perp \cap M$ is generated by $\lc m_j \rc_{j \in J}$, and \eqref{eq:CM} follows from \cite[Proposition 1.1.$\rm(i\hspace{-.08em}i\hspace{-.08em}i)$]{MR922894}.
\end{proof}

Let $\scrP$ be a rational polyhedral complex in $X_{C}(\bT)$.
The image of a rational polyhedron in $\scrP$ by the embedding \eqref{eq:G-emb} is a rational polyhedron in $\bT^n$.
By identifying each polyhedron in $\scrP$ with its image by the embedding \eqref{eq:G-emb}, one can think of $\scrP$ as a rational polyhedral complex in $\bT^n$.
We also identify each tropical torus orbit $O_{C'}(\bT)$ $\lb C' \prec C \rb$ with its image by the embedding \eqref{eq:G-emb}.
Then for each $\tau \in \scrP$ such that $\rint \lb \tau \rb \subset O_{C'}(\bT)$, one has 
\begin{align}\label{eq:Fp}
F_p^\bZ (\tau)= \lb \sum_{\sigma} \bigwedge^p T(\sigma) \rb \cap \bigwedge^p \lb N / \vspan \lb C' \rb \cap N \rb,
\end{align}
where the sum is taken over $\sigma \in \scrP$ such that $\rint \lb \sigma \rb \subset O_{C'}(\bT)$ and $\sigma \succ \tau$.
Here $T(\sigma) \subset O_{C'}(\bT)=N_\bR/ \vspan \lb C' \rb$ is the subspace generated by tangent vectors on $\sigma$, which we defined in \eqref{eq:T}.

Next, we recall tropicalizations of algebraic varieties in toric varieties.
We refer the reader to \cite[Section 2, Section 3.1]{MR3064984} for details of the following.
Let $K$ be an algebraically closed field equipped with a valuation $v_K \colon K \to \bT$.
Let further $R$ be the valuation ring, and $k$ be the residue field.
For an element $n \in N_\bR$, the tilted group ring $R \ld M \rd^n \subset K \ld M \rd$ is defined by
\begin{align}
R \ld M \rd^n := \lc \sum_m a_m x^m \in K \ld M \rd \relmid v_K \lb a_m \rb+ \la m, n \ra \geq 0\ \mathrm{for\ all}\ m \in M\ \mathrm{such\ that}\ a_m \neq 0 \rc,
\end{align}
where $x^m$ is the character associated with a point $m \in M$.
Let $T_n$ denote the closed subscheme of $\Spec R \ld M \rd^n$, which is obtained by cutting out by monomials $a_m x^m$ such that $v_K \lb a_m \rb+ \la m, n \ra>0$.
When the valuation $v_K \colon K \to \bT$ is not trivial, $\Spec R \ld M \rd^n \times_{\Spec R} \Spec K$ is naturally identified with $\Spec K \ld M \rd$, and $T_n$ is the special fiber of $\Spec R \ld M \rd^n$.
For a closed subscheme $X$ of $\Spec K \ld M \rd$, we write the closure of $X$ in $\Spec R \ld M \rd^n$ as $\scX^n$.
The \emph{initial degeneration} $X_n$ of $X$ at $n \in N_\bR$ is the closed subscheme of $T_n$ obtained by intersecting with $\scX^n$.
The \emph{tropicalization} $\trop \lb X \rb$ of $X$ is defined by
\begin{align}
\trop \lb X \rb:=\lc n \in N_\bR \relmid X_n \neq \emptyset \rc.
\end{align}
The \emph{multiplicity} $m_X(\sigma)$ of a facet $\sigma$ in $\trop \lb X \rb$ is defined as the sum of the multiplicities of the irreducible components of $X_n$ for a point $n \in \rint(\sigma)$.
The \emph{tropicalization} of a pure-dimensional cycle $C=\sum_i a_i Z_i$ $(a_i \in \bZ_{>0})$ in $\Spec K \ld M \rd$ is defined by
\begin{align}
\trop \lb C \rb:=\bigcup_i \trop \lb Z_i \rb,
\end{align}
and the multiplicity of a facet $\sigma$ is defined by
\begin{align}
m_C(\sigma)=\sum_i a_i m_{Z_i} \lb \sigma \rb.
\end{align}
We also define the \emph{tropicalization} $\trop \lb X \rb$ of a closed subscheme $X$ of the toric variety $X_\Sigma$ associated with a fan $\Sigma$ as the disjoint union of the tropicalizations of its intersections with torus orbits of $X_\Sigma$, i.e.,
\begin{align}
\trop \lb X \rb:=\bigsqcup_{C \in \Sigma} \lc n \in O_C(\bT) \relmid \lb X \cap O_C \rb_n \neq \emptyset \rc,
\end{align}
where $O_C$ is the torus orbit in $X_\Sigma$ associated with a cone $C$.

\subsection{Stable intersections of tropical hypersurfaces}\label{sc:tci}

We recall a description of stable intersections of tropical hypersurfaces.
Let $f_i \colon N_\bR \to \bR$ $(1 \leq i \leq r)$ be tropical Laurent polynomials given by
\begin{align}
f_i(n)= \min_{m \in A_i} \lc a_{i, m} + \la m, n\ra \rc,
\end{align}
where $A_i$ is a finite subset of $M$, and $a_{i, m} \in \bR$ is a non-zero constant.
Each tropical hypersurface $X(f_i)^\circ \subset N_\bR$ defined by $f_i$ defines a polyhedral subdivison $\Xi_i$ of $N_\bR$.
It is dual to the subdivision $\scrS_i$ of the Newton polytope $\Delta_i:=\conv \lb A_i \rb$ of $f_i$, which is induced by $\lc a_{i, m} \relmid m \in A_i \rc$ (cf.~e.g.~\cite[Proposition 3.1.6]{MR3287221}).
The cell $F$ in $\scrS_i$, which is dual to $\sigma \in \Xi_i$ is given by 
\begin{align}
F=\conv \lb \lc m \in A_i \relmid a_{i, m} + \la m, n\ra =f_i(n), \forall n \in \sigma \rc \rb.
\end{align}

We consider the polyhedral subdivision $\Xi$ of $N_\bR$ induced by the union of tropical hypersurfaces $X(f_i)^\circ$ $(1 \leq i \leq r)$.
Each cell $\sigma \in \Xi$ can be uniquely written as $\sigma =\bigcap_{i=1}^r \sigma_i$, where $\sigma_i$ is a cell in $\Xi_i$ that  contains $\sigma$ in its relative interior.
We also consider the mixed subdivision $\scrS$ of $\Delta:=\sum_{i=1}^r \Delta_i$ induced by $\lc a_{i, m} \relmid m \in A_i, 1 \leq i \leq r \rc$.

\begin{proposition}{\rm(\cite[Proposition 4.2]{BB07})}\label{pr:dual}
There is a one-to-one duality correspondence between $\Xi$ and $\scrS$, which reverses the inclusion relations.
If $\sigma =\bigcap_{i=1}^r \sigma_i \in \Xi$ $(\sigma_i \in \Xi_i)$ corresponds to $F \in  \scrS$, then $F$ can be written as $F=\sum_{i=1}^r F_i$, 
where 
\begin{align}
F_i &= \conv \lb \lc m \in A_i \relmid a_{i, m} + \la m, n\ra =f_i(n), \forall n \in \sigma \rc \rb \\
&= \conv \lb \lc m \in A_i \relmid a_{i, m} + \la m, n\ra =f_i(n), \forall n \in \sigma_i \rc \rb
\end{align}
is the cell in $\scrS_i$ that is dual to $\sigma_i \in \Xi_i$.
One also has $\dim \sigma+\dim F=\dim N_\bR$.
\end{proposition}

The stable intersection $X(f_1, \cdots, f_r)^\circ \subset N_\bR$ of the tropical hypersurfaces $X(f_i)^\circ \subset N_\bR$ $(1 \leq i \leq r)$ is the support of a subcomplex of $\Xi$, and whether each cell of $\Xi$ is contained in the stable intersection $X(f_1, \cdots, f_r)^\circ$ or not is determined by the following theorem:

\begin{theorem}{\rm(cf.~e.g.~\cite[Theorem 4.6.9]{MR3287221})}\label{th:st-intersection}
Let $\sigma =\bigcap_{i=1}^r \sigma_i \in \Xi$ be a cell, and $F=\sum_{i=1}^r F_i \in \scrS$ be the cell that is dual to $\sigma$.
The cell $\sigma$ is contained in the stable intersection $X(f_1, \cdots, f_r)^\circ \subset N_\bR$ if and only if 
$
\dim \lb \sum_{i \in I} F_i \rb \geq |I|
$
for any subset $I \subset \lc 1, \cdots, r \rc$.
\end{theorem}

\section{Local models of tropical contractions}\label{sc:loccont}

\subsection{Local models of IAMS}\label{sc:local-iams}

Let $\xi$ be a rational polytope in $\bR^d$.
We consider a rational polytopal complex $\scrP$ satisfying the following condition:
\begin{condition}\label{cd:loccont}
The following hold:
\begin{enumerate}
\item The whole space $U:=\bigcup_{\sigma \in \scrP} \sigma$ is a topological manifold (with boundary).
\item $\xi \in \scrP$.
\item For any $\tau \in \scrP$, there exists $\sigma \in \scrP$ satisfying $\sigma \succ \xi$ and $\sigma \succ \tau$.
\item For any $\tau \in \scrP$ such that $\tau \succ \xi$, the whole space $U$ contains $\rint \lb \tau \rb$ in its interior.
\end{enumerate}
\end{condition}
We further suppose that it is equipped with fan structures $\lc S_v \colon \mathrm{St}(v) \to \bR^d \rc_{v \prec \xi}$ satisfying \pref{cd:toric}.
For each $\tau \in \scrP$, we take an element $a_\tau \in \rint (\tau)$ and consider the associated subdivision $\widetilde{\scrP}$ of $\scrP$ that we considered in \eqref{eq:subdivision}.
We define the subset $U_{\xi} \subset U$ as the open star of the point $a_{\xi}$ in $\widetilde{\scrP}$, i.e.,
\begin{align}\label{eq:uxi}
U_{\xi} := \bigcup_{\substack{\tau_0 \prec \tau_1 \prec \cdots \prec \tau_l, \\ l \geq 0, \tau_i \in \scrP, \\ \xi \in \lc \tau_0, \cdots, \tau_l \rc}} \rint \lb \conv \lb \lc a_{\tau_0}, a_{\tau_1}, \cdots, a_{\tau_l} \rc \rb \rb,
\end{align}
and $\overline{U}_{\xi}$ as its closure in $U$.
As mentioned in \pref{sc:iass}, the fan structure $S_v \colon \mathrm{St}(v) \to \bR^d$ along a vertex $v \prec \xi$ induces an integral affine structure on $\mathrm{St}(v)$.
By restricting it to $\mathrm{St}(v) \cap U_{\xi}$, we can equip $U_{\xi}$ with an integral affine structure with singularities.

\begin{definition}\label{df:local-iams}
We call the IAMS $U_{\xi}$ constructed as above a \emph{local model of IAMS}.
\end{definition}

\begin{remark}
Let $(B, \scrP)$ be an IAMS constructed in Construction \ref{construction}.
Choose a point $a_\tau \in \rint (\tau)$ for each $\tau \in \scrP$, and consider the associated subdivision $\widetilde{\scrP}$ of $\scrP$ as we did in \eqref{eq:subdivision}.
Local models of IAMS of \pref{df:local-iams} arise as the open stars of $a_\tau$ $(\tau \in \scrP)$ in $\widetilde{\scrP}$, and the IAMS $B$ is covered by these open stars.
This is the reason why we call them the local models of IAMS.
\end{remark}

\begin{definition}\label{df:quasi}
We say that a local model of IAMS $U_{\xi}$ is \emph{quasi-simple} if $(U, \scrP)$ is positive in the sense of \pref{df:positive} and satisfies the following condition:
There exist disjoint subsets 
$\Omega_1, \cdots, \Omega_r \subset \lc \omega \in \scrP(1) \relmid \omega \prec \xi \rc$ 
and 
$R_1, \cdots, R_r \subset \lc \rho \in \scrP(d-1) \relmid \rho \succ \xi \rc$ such that
\begin{enumerate}
\item For $\omega \in \scrP(1)$ and $\rho \in \scrP(d-1)$, one has $\kappa_{\omega, \rho} = 0$ unless $\omega \in \Omega_i$ and $\rho \in R_i$ for some common $i \in \lc 1, \cdots, r \rc$.
\item For any $i \in \lc 1, \cdots, r \rc$, the monodromy polytopes $\Delta_\omega(\xi)$ and $\Deltav_\rho(\xi)$ are independent of $\omega \in \Omega_i$ and $\rho \in R_i$ up to translation respectively.
We write them as $\Delta_i(\xi)$ and $\Deltav_i(\xi)$ respectively.
\end{enumerate}
See \pref{sc:iass} for the definitions of $\kappa_{\omega, \rho}$, $\Delta_\omega(\xi)$, and $\Deltav_\rho(\xi)$.
We also say that a local model of IAMS $U_{\xi}$ is \emph{very simple} if it is quasi-simple and satisfies the following additional condition:
\begin{enumerate}
\setcounter{enumi}{2}
\item The polytopes $\Delta(\xi)$ and $\Deltav(\xi)$ defined by
\begin{align}
\Delta(\xi)&:= \mathrm{conv} \lb \bigcup_{i=1}^r \Delta_i(\xi) \times \lc e_i \rc \rb \subset \lb \Lambda_{\xi}^\perp\oplus \bZ^r \rb \otimes_\bZ \bR\\
\Deltav(\xi)&:= \mathrm{conv} \lb \bigcup_{i=1}^r \Deltav_i(\xi) \times \lc e_i \rc \rb \subset \lb \Lambda_{\xi} \oplus \bZ^r \rb \otimes_\bZ \bR
\end{align}
are standard simplices. Here $e_1, \cdots, e_r$ are the standard basis of $\bZ^r$.
\end{enumerate}
\end{definition}

\begin{definition}\label{df:simple}
We say that an IAMS $B$ is \emph{quasi-simple} (resp. \emph{very simple}) if for any point $x \in B$ there exists a quasi-simple (resp. \emph{very simple}) local model of IAMS $U_{\xi}$ and a homeomorphism from $U_{\xi}$ onto a neighborhood of the point $x$ whose restriction to the complement of the discriminant locus is an integral affine isomorphism.
\end{definition}

\begin{remark}
An IAMS $(B, \scrP)$ constructed in Construction \ref{construction} is called \emph{simple} in the Gross--Siebert program (\cite[Definition 1.60]{MR2213573}) if for any $\tau \in \scrP$ with $1 \leq \dim \tau \leq d-1$, the open star of $a_\tau$ is a quasi-simple local model of IAMS and the polytopes $\Delta(\tau)$ and $\Deltav(\tau)$ are elementary simplices.
In particular, if $(B, \scrP)$ is simple, then it is quasi-simple.
\end{remark}

\subsection{Constructions of local models of tropical contractions}\label{sc:construction}

We work on the same setup as in the second paragraph of \pref{sc:intro}.
In \pref{sc:intro}, we associated the cone $C \subset N_\bR$ of \eqref{eq:cone} and the tropical polynomials $f_i \colon N_\bR \to \bR$ $(1 \leq i \leq r)$ of \pref{eq:polynomial} with lattice polytopes $\lc \Delta_i \subset M'_\bR \rc_{i=1}^r$ $\lc \Deltav_i \subset N'_\bR \rc_{i=1}^r$ such that $\la m, n \ra=0$ for any $m \in \Delta_i$, $n \in \Deltav_j$ $(1 \leq i, j \leq r)$.
We further considered the stable intersection $X(f_1, \cdots, f_r)^\circ \subset N_\bR$ of the tropical hypersurfaces $X(f_i)^\circ \subset N_\bR$ defined by $f_i$, and its closure $X(f_1, \cdots, f_r)$ in the tropical toric variety $X_{C}(\bT)$ associated with the cone $C$.
In this subsection, we construct a local model of IAMS by contracting the space $X(f_1, \cdots, f_r)$ to the space $B \subset N_\bR$ of \eqref{eq:B}.
Some examples will be given in \pref{sc:ex}.

\begin{lemma}\label{lm:c-face}
For any face $C' \prec C$, there exist faces $\lc C_i \prec \cone \lb \Deltav_i \times \lc e_i \rc \rb \rc_{1 \leq i \leq r}$ such that $C'=\sum_{i=1}^r C_i$.
\end{lemma}
\begin{proof}
Since $C'$ is a face of $C$, there exists an element $m_0 \in M_\bR$ such that $m_0=0$ on $C'$ and $m_0> 0$ on $C \setminus C'$.
We set $C_i:=\cone\lb \Deltav_i \times \lc e_i \rc \rb \cap C'$.
Then $C_i \prec \cone\lb \Deltav_i \times \lc e_i \rc \rb$ since $m_0=0$ on $C_i$ and $m_0 > 0$ on $\cone\lb \Deltav_i \times \lc e_i \rc \rb \setminus C_i$.
It is obvious that we have $C' \supset \sum_{i=1}^r C_i$.
We check the opposite inclusion.
Take an arbitrary point $n_0 \in C'$.
Since $n_0 \in C$, there exist elements $\lc n_i \in \cone\lb \Deltav_i \times \lc e_i \rc \rb \rc_{1 \leq i \leq r}$ such that $n_0=\sum_{i=1}^r n_i$.
Here we have $\la m_0, n_0 \ra=0$ and $\la m_0, n_i \ra \geq 0$, which imply $\la m_0, n_i \ra = 0$.
Hence, we get $n_i \in C_i$ and $n_0 \in \sum_{i=1}^r C_i$.
We obtained $C' = \sum_{i=1}^r C_i$.
\end{proof}

\begin{lemma}\label{lm:BX}
One has $B \subset X(f_1, \cdots, f_r)^\circ$.
If $\sum_{i=1}^r T \lb \Delta_i \rb$ is an internal direct sum of $\lc T \lb \Delta_i \rb \rc_{i \in \lc 1, \cdots, r \rc}$, the stable intersection $X(f_1, \cdots, f_r)^\circ$ coincides with the set-theoretic intersection $\bigcap_{i=1}^r X(f_i)^\circ$.
\end{lemma}
\begin{proof}
One can easily check this by \pref{th:st-intersection}.
\end{proof}

For each $i \in \lc 1, \cdots, r \rc$ and $n \in N_\bR$, we set
\begin{align}\label{eq:F_i}
\scrM_i(n):=\lc m \in A_i \relmid \la m, n \ra =f_i \lb n \rb \rc.
\end{align}
This is the set of indices of monomials that attain the minimum of $f_i$ at the point $n$.
We consider tuples of faces $\lb F_1, \cdots, F_r, C_1, \cdots, C_r \rb$ such that $F_i \prec \conv (A_i), C_i \prec \cone(\Deltav_i \times \lc e_i \rc)$ $\lb 1 \leq i \leq r \rb$, which satisfy the following condition:

\begin{condition}
The tuple $\lb F_1, \cdots, F_r, C_1, \cdots, C_r \rb$ satisfies the following:
\begin{enumerate}
\item The polytope $\sum_{i=1}^r F_i$ is a cell in the mixed subdivision $\scrS$ of $\sum_{i=1}^r \conv (A_i)$, which is induced by $\lc a_{i, m} =0 \relmid m \in A_i, 1 \leq i \leq r \rc$.
\item For any subset $J \subset \lc 1, \cdots, r \rc$, we have $\dim \lb \sum_{j \in J} F_j \rb \geq  | J |$.
\item The cone $\sum_{i=1}^r C_i$ is a face of $C$.
\end{enumerate}
\end{condition}
For such a tuple of faces, we define the subset $\tau \lb F_1, \cdots, F_r \rb \subset X_C(\bT)$ as the closure in $X_C(\bT)$ of the cell that is dual to $\sum_{i=1}^r F_i$ in \pref{pr:dual}, i.e., the closure of
\begin{align}
\lc n \in N_\bR \relmid 
\conv \lb \scrM_i(n) \rb=F_i, \forall i \in \lc 1, \cdots, r \rc 
\rc,
\end{align}
and the subset $\tau \lb F_1, \cdots, F_r, C_1, \cdots, C_r \rb \subset X_C(\bT)$ as the closure of 
\begin{align}\label{eq:tpolyhedra}
\tau \lb F_1, \cdots, F_r\rb \cap \bigcap_{i=1}^r O_{C'} (\bT),
\end{align}
where $C':=\sum_{i=1}^r C_i$.
By \pref{th:st-intersection} again, one can see that all subsets $\tau \lb F_1, \cdots, F_r \rb$ are contained in $X(f_1, \cdots, f_r)$ and cover the whole space $X(f_1, \cdots, f_r)$.
The set of polyhedra
\begin{align}\label{eq:nat-poly}
\scrP_{X(f_1, \cdots, f_r)}:=\lc \tau \lb F_1, \cdots, F_r, C_1, \cdots, C_r \rb \relmid \tau \lb F_1, \cdots, F_r, C_1, \cdots, C_r \rb \neq \emptyset \rc
\end{align}
forms a natural polyhedral structure on $X(f_1, \cdots, f_r)$.
The subset $B$ is the support of a subcomplex of $\scrP_{X(f_1, \cdots, f_r)}$, which will be denoted by $\scrP_B \subset \scrP_{X(f_1, \cdots, f_r)}$.

We identify the space $X(f_1, \cdots, f_r) \subset X_{C}(\bT)$ with its image by the embedding \eqref{eq:G-emb}, and think of $\scrP_{X(f_1, \cdots, f_r)}$ as a rational polyhedral complex in $\bT^n$.

\begin{lemma}\label{lm:parent}
The rational polyhedral complex $\scrP_{X(f_1, \cdots, f_r)}$ in $\bT^n$ satisfies \pref{cd:parent}.
\end{lemma}
\begin{proof}
A polyhedron $\tau \lb F_1, \cdots, F_r, C_1, \cdots, C_r \rb \in \scrP_{X(f_1, \cdots, f_r)}$ is an infinite polyhedron if and only if $C_i \neq \lc 0 \rc$ for some $i \in \lc 1, \cdots r \rc$.
The unique finite polyhedron containing $\tau \lb F_1, \cdots, F_r, C_1, \cdots, C_r \rb$ as its infinite face is $\tau \lb F_1, \cdots, F_r, \lc 0 \rc, \cdots, \lc 0 \rc \rb$.
Furthermore, we have
\begin{align}\nonumber
&\lc \sigma \in \scrP_{X(f_1, \cdots, f_r)} \relmid \sigma \mathrm{\ is\ a\ finite\ polyhedron\ s.t.\ } \sigma \succ \tau \lb F_1, \cdots, F_r, C_1, \cdots, C_r \rb \rc \\
&=\lc \tau \lb F_1', \cdots, F_r', \lc 0\rc, \cdots, \lc 0 \rc \rb \in \scrP_{X(f_1, \cdots, f_r)} \relmid F_i' \subset F_i, \forall i \in \lc 1, \cdots, r\rc \rc \\
&=\lc \sigma \in \scrP_{X(f_1, \cdots, f_r)} \relmid \sigma \succ \tau \lb F_1, \cdots, F_r, \lc 0 \rc, \cdots, \lc 0 \rc \rb \rc.
\end{align}
Hence, the rational polyhedral complex $\scrP_{X(f_1, \cdots, f_r)}$ satisfies \pref{cd:parent}.
\end{proof}

We define the subset $D \subset B$ by
\begin{align}
D:&=\bigcap_{i=1}^r \lc n \in N_\bR \relmid \forall m_i \in A_i, \ \la m_i, n \ra =0 \rc,
\end{align}
where $A_i:=\lc 0 \rc \cup \lb M \cap \lb \Delta_i \times \lc -e_i^\ast \rc \rb \rb$.
This is a linear subspace of $N_\bR$.
By the assumption $\la m, n \ra=0$ for any $m \in \Delta_i$, $n \in \Deltav_j$ $(1 \leq i, j \leq r)$, we have $\Deltav_i \subset D$.
We take a rational polytope $\xi \subset D$ satisfying the following condition:
\begin{condition}\label{cd:refinement}
The normal fan $\Sigma$ of the polytope $\xi \subset D$ is a refinement of the normal fan of $\Deltav_i \subset D$ for all $i \in \lc 1, \cdots, r \rc$.
\end{condition}
For instance, one can take the Minkowski sum $\sum_{i=1}^r \Deltav_i$ as $\xi$.
(In general, the normal fan of the Minkowski sum $\sum_{i=1}^r \Deltav_i$ is a refinement of the normal fan of $\Deltav_i$ for all $i$ (cf.~e.g.~\cite[Proposition 7.12]{MR1311028})).
We consider a $d$-dimensional rational polytopal complex $\scrP$, which satisfies \pref{cd:loccont} (using the above polytope $\xi$ as the polytope $\xi$ appearing in \pref{cd:loccont}) and the following condition:
\begin{condition}\label{cd:loccont'}
Every element $\tau \in \scrP$ sits in a polyhedron in $\scrP_B$.
Let $\theta \colon \scrP \to \scrP_B$ denote the map that sends each polytope $\tau \in \scrP$ to the minimal polyhedron in $\scrP_B$ containing $\tau$. 
\end{condition}

For the space $U:=\bigcup_{\sigma \in \scrP} \sigma$, we will construct a fan structure at every vertex of $\xi$.
Let $\phi_{\xi, \Deltav_i} \colon \scrP_{\xi} \to \scrP_{\Deltav_i}$ be the map of \eqref{eq:phi} for $\xi, \Deltav_i$. 
For every face $\tau \prec \xi$, we define the subset $C_\tau \subset C$ by
\begin{align}\label{eq:ctau}
C_\tau&:=\mathrm{cone} \lb \bigcup_{i=1}^r \phi_{\xi, \Deltav_i} (\tau) \times \lc e_i \rc \rb.
\end{align}
\begin{lemma}\label{lm:cone}
For any face $\tau \prec \xi$, the subset $C_\tau$ is a face of $C$.
\end{lemma}
\begin{proof}
We show that there exists an element $m \in M_\bR$ such that $C \subset \lc m \geq 0 \rc$ and $C_\tau = C \cap \lc m = 0 \rc$.
Take an element $m_0 \in M_\bR$ such that the restriction of $m_0$ to $\xi \subset N_\bR$ attains the minimum along $\tau \prec \xi$.
Then for any $i \in \lc 1, \cdots, r \rc$, the restriction of $m_0$ to $\Deltav_i \times \lc e_i \rc \subset N_\bR$ attains the minimum along $\phi_{\xi, \Deltav_i} (\tau) \times \lc e_i \rc \prec \Deltav_i \times \lc e_i \rc$.
We set $c_i:=\min_{n \in \Deltav_i \times \lc e_i \rc} \la m_0, n \ra$ and $m_1:=m_0- \sum_{i=1}^r c_i e_i^\ast$.
For any $k_i \in \bR_{\geq 0}$ and $n_i \in \Deltav_i$ $(1 \leq i \leq r)$, one has
\begin{align}
\la m_1, \sum_{i=1}^r k_i \lb n_i + e_i \rb \ra \geq \sum_{i=1}^r c_i k_i - \sum_{i=1}^r c_i k_i=0,
\end{align}
which becomes an equality if and only if $n_i \in \phi_{\xi, \Deltav_i} (\tau)$ for all $i \in \lc 1, \cdots, r \rc$.
This implies that we have $C \subset \lc m_1 \geq 0 \rc$ and $C_\tau = C \cap \lc m_1 = 0 \rc$.
\end{proof}

For every vertex $v \prec \xi$, we consider the projection $\pi_{C_v} \colon X_{C_v}(\bT) \to O_{C_v}(\bT)$ of \eqref{eq:proj} with respect to the cone of \eqref{eq:ctau} with $\tau=v$.

\begin{lemma}\label{lm:aff-str}
For any vertex $v \prec \xi$, the restriction of the map $\pi_{C_v}$ to $B$ is a bijection.
Furthermore, the restriction of $\pi_{C_v}$ to the interior of any face of $B$ is an integral affine isomorphism onto the image.
\end{lemma}
\begin{proof}
Take a sublattice $N'' \subset N$ such that $N=N'' \oplus \lb \oplus_{i=1}^r \bZ \lb \phi_{\xi, \Deltav_i} \lb v \rb+e_i \rb \rb$.
Let $d_i^\ast \in M$ $(1 \leq i \leq r)$ be the element such that $\la d_i^\ast, \phi_{\xi, \Deltav_j} \lb v \rb+e_j \ra=\delta_{i,j}$ and $\la d_i^\ast, n \ra=0$ for all $n \in N''$.
One has
\begin{align}
B&=\lc n \in N_\bR \relmid 0=\min_{m \in A_i \setminus \lc 0 \rc} \la m, n \ra, \forall i \in \lc 1, \cdots, r \rc \rc \\
&=\lc n \in N_\bR \relmid \la d_i^\ast, n \ra = \min_{m \in A_i \setminus \lc 0 \rc} \la m+d_i^\ast, n \ra, \forall i \in \lc 1,\cdots, r \rc \rc. \label{eq:B2}
\end{align}
Since we have $\la m+d_i^\ast, \phi_{\xi, \Deltav_j} \lb v \rb+e_j \ra=\delta_{i,j}-\delta_{i,j}=0$ for any $m \in A_i \setminus \lc 0 \rc$ and $i, j \in \lc 1, \cdots, r \rc$, the polynomials $\min_{m \in A_i \setminus \lc 0 \rc} \la m+d_i^\ast, n \ra$ in \eqref{eq:B2} can be regarded as functions on 
\begin{align}
N_\bR / \oplus_{i=1}^r \bR \lb \phi_{\xi, \Deltav_i} \lb v \rb+e_i \rb \cong N''_\bR.
\end{align}
The subset $B \subset N_\bR$ can be thought of as the graph of these functions on $N''_\bR$.
On the other hand, the restriction of the map $\pi_{C_v}$ to $B$ is the projection
\begin{align}
B \hookrightarrow N_\bR \xrightarrow{\left. \pi_{C_v} \right|_{N_\bR}} 
\left. N_\bR \middle/ \vspan \lb \bigcup_{i=1}^r \phi_{\xi, \Deltav_i} \lb v \rb \times \lc e_i \rc \rb \right. 
\cong N''_\bR.
\end{align}
The claim is now obvious from these.
\end{proof}

By \pref{lm:aff-str}, it turns out that for any $v \prec \xi$, the map
\begin{align}\label{eq:fanstr}
S_v \colon \mathrm{St}(v) \hookrightarrow N_\bR \xrightarrow{\left. \pi_{C_v} \right|_{N_\bR}} 
\left. N_\bR \middle/ \vspan \lb \bigcup_{i=1}^r \phi_{\xi, \Deltav_i} \lb v \rb \times \lc e_i \rc \rb \right.
\end{align}
defines a fan structure along $v$.
Here $\mathrm{St}(v) \subset U$ is the open star of $v$ in $\scrP$.

\begin{lemma}
The set of the above fan structures $\lc S_v \rc_{v \prec \xi}$ satisfies \pref{cd:toric}.
\end{lemma}
\begin{proof}
We show that for any polyhedron $\tau \in \scrP$ such that $\tau \prec \xi$ and any vertices $v_1, v_2 \prec \tau$, the fan structures along $\tau$ induced by $S_{v_1}, S_{v_2}$ are equivalent.
Since
$\phi_{\xi, \Deltav_i} \lb v_2 \rb-\phi_{\xi, \Deltav_i}  \lb v_1 \rb \in T \lb \phi_{\xi, \Deltav_i} (\tau) \rb
\subset 
T(\tau)$ by \eqref{eq:TT},
we have 
\begin{align}
\vspan \lb \bigcup_{i=1}^r \phi_{\xi, \Deltav_i}  \lb v_1 \rb \times \lc e_i \rc \rb+ T(\tau)
=\vspan \lb \bigcup_{i=1}^r \phi_{\xi, \Deltav_i}  \lb v_2 \rb \times \lc e_i \rc \rb+ T(\tau).
\end{align}
Since the fan structure along $\tau$ induced by $S_{v_j}$ $(j=1, 2)$ is
\begin{align}
\mathrm{St}(\tau) \hookrightarrow N_\bR \to \left. N_\bR \middle/ \lb \vspan \lb \bigcup_{i=1}^r \phi_{\xi, \Deltav_i}  \lb v_j \rb \times \lc e_i \rc \rb+ T(\tau) \rb \right.,
\end{align}
these two fan structures coincide.
\end{proof}

For each polyhedron $\tau \in \scrP$, we take an element $a_\tau \in \rint (\tau)$ and consider the associated subdivision $\widetilde{\scrP}$ of $\scrP$ that we considered in \eqref{eq:subdivision}.
Let $U_{\xi}, \overline{U}_{\xi} \subset U$ denote the open star of the point $a_{\xi}$ in $\widetilde{\scrP}$ and its closure respectively.
The above fan structures $\lc S_v \rc_{v \prec \xi}$  induce an integral affine structure with singularities on $U_{\xi}$ and $\overline{U}_{\xi}$, and the space $U_{\xi}$ becomes a local model of IAMS in the sense of \pref{df:local-iams}.
We compute the monodromy of the integral affine structure.

\begin{proposition}\label{pr:monodromy}
Let $v_1, v_2 \prec \xi$ be vertices, and $\sigma_1, \sigma_2 \in \scrP(d)$ be maximal-dimensional polytopes.
We consider a loop $\gamma$ that starts from the vertex $v_1$, passes through the interior of $\sigma_1$, the vertex $v_2$, and the interior of $\sigma_2$ in this order, and comes back to the original point $v_1$.
The monodromy transformation $T_\gamma \colon N'' \to N''$ with respect to the loop $\gamma$ is given by
\begin{align}\label{eq:tgamma}
T_\gamma(n)=n +\sum_{i=1}^r \la m_i^{\sigma_2} - m_i^{\sigma_1}, n \ra \cdot \lb \phi_{\xi, \Deltav_i} \lb v_2 \rb- \phi_{\xi, \Deltav_i} \lb v_1 \rb \rb,
\end{align}
where $N'':= \left. N \middle/ \oplus_{i=1}^r \bZ \lb \phi_{\xi, \Deltav_i} \lb v_1 \rb+e_i \rb \right.$ is the integral tangent space at $v_1$, and $m_i^{\sigma_j}$ $(j=1, 2)$ is the element of $\Delta_i \cap M'$ such that $f_i=m_i^{\sigma_j}-e_i^\ast$ on $\sigma_j$.
\end{proposition}
\begin{proof}
First, notice that one can see from the last statement of \pref{pr:dual} that the element of $\Delta_i \cap M'$ such that $f_i=m_i^{\sigma_j}-e_i^\ast$ on $\sigma_j$ is unique.
We compute the parallel transport of a vector $n \in N''$.
When it arrives in $\rint \lb \sigma_1 \rb$, it becomes
\begin{align}
n + \sum_{i=1}^r \la m_i^{\sigma_1}-e_i^\ast, n \ra \cdot \lb \phi_{\xi, \Deltav_i} \lb v_1 \rb+e_i \rb.
\end{align}
For any $j \in \lc 1, \cdots, r \rc$, we have
\begin{align}
\la m_j^{\sigma_2}-e_j^\ast, n + \sum_{i=1}^r \la m_i^{\sigma_1}-e_i^\ast, n \ra \cdot \lb \phi_{\xi, \Deltav_i} \lb v_1 \rb+e_i \rb \ra
&=\la m_j^{\sigma_2}-e_j^\ast, n \ra - \la m_j^{\sigma_1}-e_j^\ast, n \ra \\
&=\la m_j^{\sigma_2} - m_j^{\sigma_1}, n \ra.
\end{align}
Hence, when the vector $n$ arrives in $\sigma_-$, it becomes 
\begin{align}
n + \sum_{i=1}^r \la m_i^{\sigma_1}-e_i^\ast, n \ra \cdot \lb \phi_{\xi, \Deltav_i} \lb v_1 \rb+e_i \rb
+ \sum_{i=1}^r \la m_i^{\sigma_2} - m_i^{\sigma_1}, n \ra \cdot \lb \phi_{\xi, \Deltav_i} \lb v_2 \rb+e_i \rb.
\end{align}
This is equal to the right hand side of \eqref{eq:tgamma} as an element of $N''= \left. N \middle/ \oplus_{i=1}^r \bZ \lb \phi_{\xi, \Deltav_i} \lb v_1 \rb+e_i \rb \right.$, the integral tangent space at $v_1$.
We obtained the claim.
\end{proof}

\begin{lemma}\label{lm:positive}
The local model of IAMS $U_{\xi}$ is positive in the sense of \pref{df:positive}.
\end{lemma}
\begin{proof}
Let $\omega \in \scrP(1)$ be an edge whose endpoints are $v_1, v_2$, 
and $\rho \in \scrP(d-1)$ be a polyhedron such that $\omega \prec \rho$ and $\sigma_1, \sigma_2 \in \scrP(d)$ are the maximal-dimensional polyhedra having $\rho$ as their face.
We show the lemma by using \pref{pr:monodromy}.
Since the monomial $m_i^{\sigma_1}-e_i^\ast$ attains the minimum of the polynomial $f_i$ on $\sigma_1$, we have $m_i^{\sigma_1}-e_i^\ast \leq m_i^{\sigma_2}-e_i^\ast$ on $\sigma_1$.
Hence, $m_i^{\sigma_2}-m_i^{\sigma_1} \geq 0$ on $\sigma_1$.
By the same reason, $m_i^{\sigma_2}-m_i^{\sigma_1} \leq 0$ on $\sigma_2$.
We obtain $m_i^{\sigma_2}-m_i^{\sigma_1} = 0$ on $\rho$.
Hence, the vector
$
\lb m_i^{\sigma_2} - m_i^{\sigma_1} \rb
$
is a non-negative multiple of the primitive cotangent vector $\check{d}_\rho \in \Lambda_\rho^{\perp}$ evaluating $\sigma_1$ positively.
Moreover, the vector $\lb \phi_{\xi, \Deltav_i} \lb v_2 \rb- \phi_{\xi, \Deltav_i} \lb v_1 \rb \rb$ is also a non-negative multiple of the primitive tangent vector $d_\omega$ on $\omega$ pointing from $v_1$ to $v_2$.
From these and \eqref{eq:tgamma}, one can see that the  local model of IAMS $U_{\xi}$ is positive.
\end{proof}

We give a sufficient condition for the IAMS $U_{\xi}$ to be quasi-simple/very simple.

\begin{lemma}\label{lm:length1}
Suppose that the following conditions hold:
\begin{enumerate}
\item The affine lengths of edges of $\Delta_i$ and $\Deltav_i$ are all $1$ for all $i \in \lc 1, \cdots, r \rc$.
\item The subspace $\sum_{i=1}^r T \lb \Delta_i \rb$ is an internal direct sum of $\lc T \lb \Delta_i \rb \rc_{i \in \lc 1, \cdots, r \rc}$, 
and the subspace $\sum_{i=1}^r T \lb \Deltav_i \rb$ is an internal direct sum of $\lc T \lb \Deltav_i \rb \rc_{i \in \lc 1, \cdots, r \rc}$.
\end{enumerate}
Then the local model of IAMS $U_{\xi}$ is quasi-simple in the sense of \pref{df:quasi}, and its monodromy polytopes $\Delta_\omega(\xi)$ and $\Deltav_\rho(\xi)$ are given by $\Delta_i$ and $\Deltav_i$ $(1 \leq i \leq r)$.
Furthermore, if the following stronger conditions 
\begin{enumerate}
\item The polytopes $\Delta_i$ and $\Deltav_i$ are standard simplices for all $i \in \lc 1, \cdots, r \rc$.
\item $\lb \sum_{i=1}^r T \lb \Delta_i \rb \rb \cap M'$ is the internal direct sum of $\lc T_\bZ \lb \Delta_i \rb \rc_{i \in \lc 1, \cdots, r \rc}$, and $\lb \sum_{i=1}^r T \lb \Deltav_i \rb \rb \cap N'$ is the internal direct sum of $\lc T_\bZ \lb \Deltav_i \rb \rc_{i \in \lc 1, \cdots, r \rc}$.
\end{enumerate}
are satisfied, then the local model of IAMS $U_{\xi}$ is very simple in the sense of \pref{df:quasi}.
\end{lemma}
\begin{proof}
Let $\omega \in \scrP(1)$ be an edge and $v_1, v_2 \in \scrP(0)$ be its endpoints.
For any $i \in \lc 1, \cdots, r \rc$, the vector appearing in \eqref{eq:tgamma}
\begin{align}\label{eq:monovect}
\phi_{\xi, \Deltav_i} \lb v_2 \rb- \phi_{\xi, \Deltav_i} \lb v_1 \rb \in T \lb \Deltav_i \rb
\end{align}
is a non-negative multiple of the primitive tangent vector $d_\omega$ on $\omega$ pointing from $v_1$ to $v_2$.
Since the subspace $\sum_{i=1}^r T \lb \Deltav_i \rb$ is an internal direct sum of $\lc T \lb \Deltav_i \rb \rc_{i \in \lc 1, \cdots, r \rc}$, there exists at most one $i \in \lc 1, \cdots, r \rc$ such that \eqref{eq:monovect} is non-zero.
For every $i \in \lc 1, \cdots, r \rc$, we define $\Omega_i \subset \scrP(1)$ to be the subset of edges $\omega \in \scrP(1)$ such that \eqref{eq:monovect} is non-zero.
When \eqref{eq:monovect} is non-zero, it turns out by the first assumption of the lemma that \eqref{eq:monovect} must coincide with the primitive tangent vector $d_\omega$.
Similarly, for $\sigma_1, \sigma_2 \in \scrP(d)$ having a common face $\rho \in \scrP(d-1)$, the vector appearing in \eqref{eq:tgamma}
\begin{align}\label{eq:monovect2}
m_i^{\sigma_2} - m_i^{\sigma_1} \in T \lb \Delta_i \rb
\end{align}
is a non-negative multiple of the primitive cotangent vector $\check{d}_\rho \in \Lambda_\rho^{\perp}$ evaluating $\sigma_1$ positively (as we saw in the proof of \pref{lm:positive}), and there exists at most one $i \in \lc 1, \cdots, r \rc$ such that \eqref{eq:monovect2} is non-zero.
For every $i \in \lc 1, \cdots, r \rc$, we also define $R_i \subset \scrP(d-1)$ to be the subset of polyhedra $\rho \in \scrP(d-1)$ such that \eqref{eq:monovect2} is non-zero.
When \eqref{eq:monovect2} is non-zero, it coincides with the primitive cotangent vector $\check{d}_\rho$.

By \pref{pr:monodromy}, for $\omega \in \Omega_i$ and $\rho \in R_j$,
the monodromy transformations of \eqref{eq:monodromy0}, \eqref{eq:monodromy1}, and \eqref{eq:monodromy2} are 
\begin{align}\label{eq:monodromy00}
T_\omega^\rho(n)&=n+\delta_{i, j} \cdot \la \check{d}_\rho, n \ra \cdot d_\omega \\ \label{eq:monodromy11}
T^{\rho}_{v_1', v_2'}(n)&= n + \la \check{d}_\rho, n \ra \cdot \lb \phi_{\xi, \Deltav_j} \lb v_2' \rb- \phi_{\xi, \Deltav_j} \lb v_1' \rb \rb \\ \label{eq:monodromy22}
T_\omega^{\sigma_1', \sigma_2'}(n)&= n + \la m_i^{\sigma_2'} - m_i^{\sigma_1'}, n \ra \cdot d_\omega,
\end{align}
where $v_1', v_2' \prec \xi$ are vertices and $\sigma_1', \sigma_2' \succ \xi$ are maximal-dimensional polyhedra in $\scrP$.
We can see by \eqref{eq:monodromy00} that the first condition for being quasi-simple is satisfied.
By \eqref{eq:monodromy11} and \eqref{eq:monodromy22}, we can also see that the monodromy polytopes $\Delta_\omega(\xi), \Deltav_\rho(\xi)$ with $\omega \in \Omega_i, \rho \in R_j$
of $U_{\xi}$ coincide with $\Delta_i$ and $\Deltav_j$ up to translation  respectively, and do not depend on the choice of $\omega \in \Omega_i$ and $\rho \in R_j$.
Now one can easily check the statement concerning being very simple too.
\end{proof}

Let $\iota \colon U_{\xi, 0} \hookrightarrow U_{\xi}$ denote the complement of the discriminant locus $\Gamma$.
Let further $\mathrm{Aff}_{U_{\xi, 0}}$ and $\check{\Lambda}$ be the sheaves on $U_{\xi, 0}$ of integral affine functions and of integral cotangent vectors respectively. 
We also consider the sheaves $\mathrm{Aff}_{X(f_1, \cdots, f_r)}$, $\scF^p_Q$, and $\scW_p^Q$ on $X(f_1, \cdots, f_r)$ of integral affine functions, $p$-th multi-cotangent spaces, and $p$-th wave tangent spaces, which we recalled in \pref{sc:rat}.

\begin{theorem}\label{th:local-cont}
Let $p \geq 0$ be any integer.
There are a subset $X \subset X(f_1, \cdots, f_r)$ and a proper continuous map $\delta \colon X \to U_{\xi}$ that 
satisfy the following:
\begin{enumerate}
\item The map $\delta$ induces an isomorphism of sheaves
\begin{align}
\iota_\ast \mathrm{Aff}_{U_{\xi, 0}} \cong \delta_\ast \mathrm{Aff}_X
\end{align}
via the pullback of functions, where $\mathrm{Aff}_X$ is the restriction of the sheaf $\mathrm{Aff}_{X(f_1, \cdots, f_r)}$ to the subset $X \subset X(f_1, \cdots, f_r)$.
\item There is an isomorphism of sheaves of graded rings
\begin{align}\label{eq:sh-gr}
\bigoplus_{p=0}^d \iota_\ast \bigwedge^{p}  \check{\Lambda} \cong \bigoplus_{p=0}^d \delta_\ast \scF^p_\bZ,
\end{align}
where the ring structures on the both sides are induced by wedge products.
\item Consider the canonical morphism
\begin{align}\label{eq:cmor}
\delta_\ast \scF^p_Q \to R \delta_\ast \scF^p_Q.
\end{align}
If $\sum_{i=1}^r T \lb \Delta_i \rb$ is the internal direct sum of $\lc T \lb \Delta_i \rb \rc_{i \in \lc 1, \cdots, r \rc}$ and $\sum_{i=1}^r T \lb \Deltav_i \rb$ is the internal direct sum of $\lc T \lb \Deltav_i \rb \rc_{i \in \lc 1, \cdots, r \rc}$, then 
\eqref{eq:cmor} is an isomorphism for $Q=\bQ$.
If $\lb \sum_{i=1}^r T \lb \Delta_i \rb \rb \cap M'$ is the internal direct sum of $\lc T_\bZ \lb \Delta_i \rb \rc_{i \in \lc 1, \cdots, r \rc}$ and $\lb \sum_{i=1}^r T \lb \Deltav_i \rb \rb \cap N'$ is the internal direct sum of $\lc T_\bZ \lb \Deltav_i \rb \rc_{i \in \lc 1, \cdots, r \rc}$, then \eqref{eq:cmor} is an isomorphism also for $Q=\bZ$.
\item The canonical morphism
\begin{align}
\delta_\ast \scW_p^Q \to R \delta_\ast \scW_p^Q
\end{align}
is also an isomorphism for $Q=\bZ, \bQ, \bR$.
\end{enumerate}
\end{theorem}

We prove \pref{th:local-cont} by constructing the subset $X$ and the map $\delta$ explicitly.
The rest of this subsection is devoted to constructing them.
The above four properties that they satisfy will be proved in \pref{sc:local-cont2},  \pref{sc:local-cont3}, and \pref{sc:local-cont1}.

For each face $\tau \prec \xi$, we define the subset $W_\tau^\circ \subset U_{\xi}$ by
\begin{align}
W_\tau^\circ :=\bigcup_{\substack{\tau \prec \tau_1 \prec \cdots \prec \tau_l, \\ l \geq 0, \tau_i \in \scrP \\ \xi \in \lc \tau, \tau_1, \cdots, \tau_l \rc}} \rint \lb \conv \lb \lc a_\tau, a_{\tau_1}, \cdots, a_{\tau_l} \rc \rb \rb.
\end{align}
Since every subset $\rint \lb \conv \lb \lc a_{\tau_0}, a_{\tau_1}, \cdots, a_{\tau_l} \rc \rb \rb \subset U_\xi$ appearing on the right hand side of \eqref{eq:uxi} is contained in the subset $W_{\tau_0}^\circ$, one has $U_\xi =\bigsqcup_{\tau \prec \xi} W_{\tau}^\circ$.
We also write the closure of $W_\tau^\circ$ as $W_\tau \subset \overline{U}_{\xi}$.
We consider the projection $\pi_{C_\tau} \colon X_{C_\tau}(\bT) \to O_{C_\tau}(\bT)$ of \eqref{eq:proj} with respect to the cone $C_\tau$ \eqref{eq:ctau} again.

\begin{lemma}\label{lm:inj}
For any face $\tau \prec \xi$, the restriction of the projection $\pi_{C_\tau} \colon X_{C_\tau}(\bT) \to O_{C_\tau}(\bT)$ to the subset $W_\tau$ is injective.
\end{lemma}
\begin{proof}
We fix a vertex $v_0 \prec \tau$.
The restriction of the map $\pi_{C_\tau}$ to $W_\tau \subset N_\bR$ is the composition
\begin{align}\label{eq:compo}
\begin{aligned}
W_\tau \hookrightarrow N_\bR \xrightarrow{\left. \pi_{C_{v_0}} \right|_{N_\bR}} \left. N_\bR \middle/ \vspan \lb \bigcup_{i=1}^r \phi_{\xi, \Deltav_i} (v_0) \times \lc e_i \rc \rb \right. \\
\xrightarrow{\pi_2} \left. N_\bR \middle/ \vspan \lb \bigcup_{i=1}^r \phi_{\xi, \Deltav_i} (\tau) \times \lc e_i \rc \rb \right. =O_{C_\tau}(\bT),
\end{aligned}
\end{align}
where $\pi_2$ is the quotient map.
By \pref{lm:aff-str}, the restriction of the first projection $\left. \pi_{C_{v_0}} \right|_{N_\bR}$ to $W_\tau$ is injective.
The kernel of the second quotient map $\pi_2$ is contained in $T(\tau)$.
It is obvious from the definition of $W_\tau$ that the subset $W_\tau$ is transversal to $T(\tau)$, and the restriction of $\pi_2$ to $\pi_{C_{v_0}} \lb W_\tau \rb$ is also injective.
Hence, the restriction $\left. \pi_{C_\tau} \right|_{W_\tau}$ is injective. 
\end{proof}

For a polyhedron $\sigma=\tau \lb F_1, \cdots, F_r, C_1, \cdots, C_r \rb \in \scrP_{X(f_1, \cdots, f_r)}$ with $F_i \prec \conv (A_i), C_i=\lc 0 \rc \prec \cone(\Deltav_i \times \lc e_i \rc)$ $\lb 1 \leq i \leq r \rb$, we set
\begin{align}\label{eq:Isig}
I_\sigma&:= \lc i \in \lc 1, \cdots r \rc \relmid f_i(n) \neq 0\ \mathrm{on}\ \rint(\sigma) \rc\\ \label{eq:til-sigma}
\tilde{\sigma}&:=\tau \lb F_1', \cdots, F_r', C_1', \cdots, C_r' \rb,
\end{align}
where $F_i'=\conv \lb F_i \cup \lc 0 \rc \rb$ and $C_i'=\lc 0 \rc$ $\lb 1 \leq i \leq r \rb$.

\begin{lemma}\label{lm:sigma}
The set $\tilde{\sigma}$ is a non-empty face of $\sigma$, which is in $\scrP_B$.
Furthermore, one has 
\begin{align}\label{eq:sigma}
\sigma \cap N_\bR = \tilde{\sigma} + \cone \lb \bigcup_{i \in I_\sigma} \Deltav_i \times \lc e_i \rc \rb.
\end{align}
\end{lemma}
\begin{proof}
First, we check $\tilde{\sigma} \neq \emptyset$.
Take a point $n_0 \in \rint \lb \sigma \rb$, and set $n_1:=n_0+\sum_{i=1}^r f_i(n_0)e_i$.
Then we can check $f_i(n_1)=0$ for all $i \in \lc 1, \cdots, r \rc$, and $\scrM_i(n_1)=\lc 0 \rc \cup \lb F_i \cap M \rb$.
Since $\conv \lb \scrM_i(n_1) \rb=\conv \lb \lc 0 \rc \cup \lb F_i \cap M \rb \rb=F_i'$, one has $n_1 \in \tilde{\sigma}$.
In particular, the set $\tilde{\sigma}$ is not empty.
It is now obvious that $\tilde{\sigma}$ is a face of $\sigma$ and is in $\scrP_B$.

Next, we try to show \eqref{eq:sigma}.
For any $i \in \lc 1, \cdots, r \rc$, the monomials of $f_i$ that attain the minimum of $f_i$ on $\rint \lb \sigma \rb$ also attain the minimum of $f_i$ at any point $n$ in $\tilde{\sigma} + \cone \lb \bigcup_{i \in I_\sigma} \Deltav_i \times \lc e_i \rc \rb$.
Hence, in \eqref{eq:sigma}, the right hand side is contained in the left hand side.
We show the opposite inclusion.
Let $n_0 \in \sigma \cap N_\bR$ be an arbitrary element.
Then $f_i(n_0)=0$ for all $i \in \lc 1, \cdots, r \rc \setminus I_\sigma$.
For each $i \in I_\sigma$, take an element $n_i \in \Deltav_i$, and set $n_0':=n_0+ \sum_{i \in I_\sigma} f_i(n_0) \cdot \lb n_i + e_i\rb$.
For any $i \in \lc 1, \cdots, r \rc$ and $m_i \in F_i \setminus \lc 0 \rc$, we have
\begin{align}
\la m_i, n_0' \ra
&=
\left\{
\begin{array}{ll}
\la m_i, n_0 \ra-f_i(n_0) & \lb i \in I_\sigma \rb \\
\la m_i, n_0 \ra & \lb i \in \lc 1, \cdots, r \rc \setminus I_\sigma \rb \\
\end{array}
\right. \\
&=0,
\end{align}
since one has $f_i(n_0)= \la m_i, n_0 \ra$.
We also have 
\begin{align}
f_i(n_0')&=\min_{m \in A_i} \la m, n_0+ \sum_{j \in I_\sigma} f_j(n_0) \cdot \lb n_j + e_j \rb \ra \\
&=
\left\{
\begin{array}{ll}
\min \lc 0, f_i(n_0)-f_i(n_0) \rc & \lb i \in I_\sigma \rb \\
f_i(n_0) & \lb i \in \lc 1, \cdots, r \rc \setminus I_\sigma \rb \\
\end{array}
\right. \\
&=0
\end{align}
for all $i \in \lc 1, \cdots, r \rc$.
It turns out from these that the element $n_0'$ is in $\tilde{\sigma}$.
Since $f_i(n_0) \leq 0$, the element $n_0=n_0'- \sum_{i \in I_\sigma} f_i(n_0) \cdot \lb n_i + e_i\rb$ is in the right hand side of \eqref{eq:sigma}.
We obtained \eqref{eq:sigma}.
\end{proof}

For each face $\tau \prec \xi$, we also set
\begin{align}
V_\tau:= \lb W_\tau + \cone_\bT \lb \bigcup_{i=1}^r \phi_{\xi, \Deltav_i} \lb \tau \rb \times \lc e_i \rc \rb \rb
\cap X(f_1, \cdots, f_r) \subset X_{C_\tau}(\bT),
\end{align}
where the operation $\cone_\bT \lb \bullet \rb$ is the one defined in \pref{eq:tcone}.

\begin{lemma}\label{lm:vtau}
For any $\tau \prec \xi$ and $\sigma \in \scrP_{X(f_1, \cdots, f_r)}$, one has
\begin{align}\label{eq:vtau}
V_{\tau} \cap \sigma=\lb W_{\tau} \cap \tilde{\sigma} \rb + \cone_\bT \lb \bigcup_{i \in I_\sigma} \phi_{\xi, \Deltav_i} (\tau) \times \lc e_i \rc \rb.
\end{align}
\end{lemma}
\begin{proof}
By taking the closure of \eqref{eq:sigma}, we obtain $\sigma = \tilde{\sigma} + \cone_\bT \lb \bigcup_{i \in I_\sigma} \Deltav_i \times \lc e_i \rc \rb$.
From this, one can get
\begin{align}
V_{\tau} \cap \sigma
=V_{\tau} \cap \lb \tilde{\sigma} + \cone_\bT \lb \bigcup_{i \in I_\sigma} \Deltav_i \times \lc e_i \rc \rb \rb
\supset \lb W_{\tau} \cap \tilde{\sigma} \rb + \cone_\bT \lb \bigcup_{i \in I_\sigma} \phi_{\xi, \Deltav_i} (\tau) \times \lc e_i \rc \rb.
\end{align}
We show the opposite inclusion.
Since for any $i \in \lc 1, \cdots r \rc \setminus I_\sigma$, the monomial $0$ attains the minimum of the polynomial $f_i$ on $\sigma$, we can see
\begin{align}
V_{\tau} \cap \sigma \subset W_\tau + \cone_\bT \lb \bigcup_{i \in I_\sigma} \phi_{\xi, \Deltav_i} (\tau) \times \lc e_i \rc \rb.
\end{align}
Hence, every element $n_0 \in V_{\tau} \cap \sigma$ can be written as $n_0=n_1+n_2$ with $n_1 \in W_\tau$ and $n_2 \in \cone_\bT \lb \bigcup_{i \in I_\sigma} \phi_{\xi, \Deltav_i} (\tau) \times \lc e_i \rc \rb$.
Since adding the element $n_2 \in \cone_\bT \lb \bigcup_{i \in I_\sigma} \phi_{\xi, \Deltav_i} (\tau) \times \lc e_i \rc \rb$ does not change the set of monomials of $f_i$ attaining the minimum except for the monomial $0$, we have $n_1 \in \tilde{\sigma}$.
Thus in \eqref{eq:vtau}, the left hand side is also contained in the right hand side.
We obtained \eqref{eq:vtau}.
\end{proof}

By \pref{lm:inj}, for each face $\tau \prec \xi$, there uniquely exists a map $t_\tau \colon \pi_{C_\tau} \lb W_\tau \rb \to W_\tau$ such that $t_\tau \circ \left. \pi_{C_\tau} \right|_{W_\tau}$ is the identity map.
We consider the map $\delta_\tau \colon V_\tau \to W_\tau$ defined as the composition
\begin{align}\label{eq:delt-tau0}
\delta_\tau \colon V_\tau \hookrightarrow X_{C_\tau}(\bT) \xrightarrow{\pi_{C_\tau}} \pi_{C_\tau}\lb W_\tau \rb \xrightarrow{t_\tau} W_\tau.
\end{align}
The restriction of this map $\delta_\tau$ to the subset $W_\tau \subset V_\tau$ is the identity map.

\begin{lemma}\label{lm:intersection}
When $V_{\tau_1} \cap V_{\tau_2} \neq \emptyset$ for faces $\tau_1, \tau_2 \prec \xi$, one has $\left. \delta_{\tau_1} \right|_{V_{\tau_1} \cap V_{\tau_2}}=\left. \delta_{\tau_2} \right|_{V_{\tau_1} \cap V_{\tau_2}}$.
\end{lemma}
\begin{proof}
First, we compute the intersection $V_{\tau_1} \cap V_{\tau_2}$.
We have
\begin{align}
V_{\tau_1} \cap V_{\tau_2} = \bigcup_{\sigma} \bigcup_{C' \prec C} V_{\tau_1} \cap V_{\tau_2} \cap \sigma \cap O_{C'} (\bT),
\end{align}
where the union concerning $\sigma$ is taken over $\sigma=\tau \lb F_1, \cdots, F_r,  C_1, \cdots, C_r \rb \in \scrP_{X(f_1, \cdots, f_r)}$ such that $C_i= \lc 0 \rc$ for all $i \in \lc 1, \cdots, r \rc$.
We set $\scrP \ld \tilde{\sigma} \rd:= \lc \tau \in \scrP \relmid \tau \subset \tilde{\sigma} \rc$, where $\tilde{\sigma}$ is the polyhedron  of \eqref{eq:til-sigma}.
Since $\xi \subset D \prec \tilde{\sigma}$, we have $\xi \in \scrP \ld \tilde{\sigma} \rd$.
By \pref{lm:vtau}, for $j=1,2$, we have
\begin{align}
V_{\tau_j} \cap \sigma \cap O_{C'} (\bT) &=
\lc \lb W_{\tau_j} \cap \tilde{\sigma} \rb + \cone_\bT \lb \bigcup_{i \in I_\sigma} \phi_{\xi, \Deltav_i} (\tau_j) \times \lc e_i \rc \rb \rc \cap O_{C'}(\bT)\\ \label{eq:vtauj}
&=\bigcup_{\substack{\tau \in \scrP \ld \tilde{\sigma} \rd\\ \tau \succ \xi}} \lc \lb W_{\tau_j} \cap \tau \rb + \cone_\bT \lb \bigcup_{i \in I_\sigma} \phi_{\xi, \Deltav_i} (\tau_j) \times \lc e_i \rc \rb \rc \cap O_{C'}(\bT).
\end{align}
By \pref{lm:c-face}, the cone $C'$ can be written as $C'= \cone \lb \bigcup_{i \in I'} F_i \times \lc e_i \rc \rb$, where $I' \subset \lc 1, \cdots, r \rc$ and $F_i \prec \Deltav_i$ is a non-empty face.
We can see from \eqref{eq:O-emb} that \eqref{eq:vtauj} is non-empty if and only if
\begin{align}
\cone \lb \bigcup_{i \in I_\sigma} \phi_{\xi, \Deltav_i} (\tau_j) \times \lc e_i \rc \rb \supset C',
\end{align}
which is equivalent to $I' \subset I_\sigma$ and $F_i \subset \phi_{\xi, \Deltav_i} (\tau_j)$ for $i \in I'$.
When this holds, \eqref{eq:vtauj} is equal to
\begin{align}
\bigcup_{\substack{\tau \in \scrP \ld \tilde{\sigma} \rd\\ \tau \succ \xi}} 
\lc \lb W_{\tau_j} \cap \tau \rb 
+ \cone_\bT \lb C' \rb  
+ \sum_{\substack{i \in I_\sigma}} \bigcup_{k_{i,j} \in \bR_{\geq 0}} k_{i, j} \cdot \lb \phi_{\xi, \Deltav_i} (\tau_j) \times \lc e_i \rc \rb \rc \cap O_{C'}(\bT).
\end{align}
We try to show $\delta_{\tau_1} = \delta_{\tau_2}$ on the intersection 
\begin{align}\label{eq:intersect}
\bigcap_{j=1,2}
\lc \lb W_{\tau_j} \cap \tilde{\tau}_j \rb 
+ \cone_\bT \lb C' \rb  
+ \sum_{\substack{i \in I_\sigma}} k_{i, j} \cdot \lb \phi_{\xi, \Deltav_i} (\tau_j) \times \lc e_i \rc \rb \rc \cap O_{C'}(\bT)
\end{align}
for polyhedra $\tilde{\tau}_j \in \scrP \ld \tilde{\sigma} \rd$ such that $\tilde{\tau}_j \succ \xi$ and $k_{i, j} \in \bR_{\geq 0}$ $(j=1, 2)$.
Since $\la A_i, F_k \times \lc e_k \rc \ra=0$ for $i \in I_\sigma \setminus I'$ and $k \in I'$, the function $f_i \colon N_\bR \to \bR$ $(i \in I_\sigma \setminus I')$ can be naturally extended to a function on $X_{C'} (\bT) \supset O_{C'}(\bT)$.
When $k_{i, 1} \neq k_{i, 2}$ for some $i \in I_\sigma \setminus I'$, \eqref{eq:intersect} is empty since the value of the polynomial $f_i$ is $-k_{i, j}$ on 
$\lc \lb W_{\tau_j} \cap \tilde{\tau}_j \rb 
+ \cone_\bT \lb C' \rb  
+ \sum_{\substack{i \in I_\sigma}} k_{i, j} \cdot \lb \phi_{\xi, \Deltav_i} (\tau_j) \times \lc e_i \rc \rb \rc$.
Furthermore, for $i \in I'$, one has
\begin{align}\label{eq:m-subset}
\begin{split}
&\lc \lb W_{\tau_j} \cap \tilde{\tau}_j \rb 
+ \cone_\bT \lb C' \rb  
+ \sum_{\substack{i \in I_\sigma}} k_{i, j} \cdot \lb \phi_{\xi, \Deltav_i} (\tau_j) \times \lc e_i \rc \rb \rc \cap O_{C'}(\bT) \\
&\subset 
\lc \lb W_{\tau_j} \cap \tilde{\tau}_j \rb 
+ \cone_\bT \lb C' \rb  
+ \sum_{\substack{i \in I_\sigma}} \max \lc k_{i, 1}, k_{i, 2} \rc \cdot \lb \phi_{\xi, \Deltav_i} (\tau_j) \times \lc e_i \rc \rb \rc \cap O_{C'}(\bT)
\end{split}
\end{align}
since $F_i \times \lc e_i \rc \subset \phi_{\xi, \Deltav_i} (\tau_j) \times \lc e_i \rc$ for $i \in I'$.
Hence, it suffices to show $\delta_{\tau_1} = \delta_{\tau_2}$ on \pref{eq:intersect} with $k_{i, 1}=k_{i, 2}$ for all $i \in I_\sigma$.
We assume $k_{i, 1}=k_{i, 2}$ for all $i \in I_\sigma$, and  write it as $k_i \in \bR_{\geq 0}$ for short.
Then \pref{eq:intersect} is equal to
\begin{align}\label{eq:intersect'}
\bigcap_{j=1,2}
\lc \pi_{C'} \lb W_{\tau_j} \cap \tilde{\tau}_j \rb 
+ \sum_{\substack{i \in I_\sigma}} k_{i} \cdot \pi_{C'}  \lb \phi_{\xi, \Deltav_i} (\tau_j) \times \lc e_i \rc \rb \rc,
\end{align}
where $\pi_{C'} \colon X_{C'} (\bT) \to O_{C'}(\bT)$ is the projection of \eqref{eq:proj}.
We will show $\delta_{\tau_1} = \delta_{\tau_2}$ on \eqref{eq:intersect'}.

First, we consider the case where $\tilde{\tau}_1=\tilde{\tau}_2=:\tilde{\tau}$.
We will use \pref{lm:tech}, a technical lemma which we will prove later in \pref{sc:lem}.
We will see in the following that by substituting
$\Deltav=\pi_{C'} \lb \tilde{\tau} \rb$, $\Deltav_0=\pi_{C'} (\xi)$, $I=\lc i \in I_\sigma \relmid k_i \neq 0 \rc$, $\Deltav_i=k_i \cdot \pi_{C'} \lb \Deltav_i \times \lc e_i \rc \rb$, and $t=1$ in \pref{lm:tech}, 
the current situation fits in the setup of \pref{lm:tech}.

There is a natural map
\begin{align}\label{eq:vp}
\varphi_{\tilde{\tau}} \colon \scrP_{\pi_{C'} (\tilde{\tau})} \to \scrP_{\tilde{\tau}}, \quad F \mapsto \pi_{C'}^{-1} (F) \cap \tilde{\tau},
\end{align}
where $\scrP_{\pi_{C'} (\tilde{\tau})}$ and $\scrP_{\tilde{\tau}}$ are the sets of faces of the polytopes $\pi_{C'} (\tilde{\tau})$ and $\tilde{\tau}$ respectively (cf.~e.g.~\cite[Lemma 7.10]{MR1311028}).
We have $\pi_{C'} \lb\varphi_{\tilde{\tau}}(F) \rb=F$ for all $F \in \scrP_{\pi_{C'} (\tilde{\tau})}$.
For each face $F \in \scrP_{\pi_{C'} (\tilde{\tau})}$, we consider the point $\pi_{C'} \lb a_{\varphi_{\tilde{\tau}}(F)} \rb \in \rint (F)$, i.e., the projection of the point that we chose for $\varphi_{\tilde{\tau}}(F) \in \scrP$.
It gives rise the set of subsets of \eqref{eq:WtF} for $\Deltav=\pi_{C'} (\tilde{\tau})$.
This includes $\pi_{C'} \lb W_{\tau_j} \cap \tilde{\tau} \rb$ since $F' \in \scrP_{\tilde{\tau}}$ such that $F' \succ \tau_j$ is contained in the image of the map $\varphi_{\tilde{\tau}}$ as we will check below.

We check that $F' \in \scrP_{\tilde{\tau}}$ such that $F' \succ \tau_j$ is contained in the image of the map $\varphi_{\tilde{\tau}}$.
It suffices to show that there is an element $m_{F'} \in {\vspan (C')}^\perp \subset M_\bR$ whose restriction to $\tilde{\tau}$ takes the minimum just along $F' \prec \tilde{\tau}$.
There exists an element $m_{F'}' \in M_\bR$ whose restriction to $\tilde{\tau}$ takes the minimum just along $F' \prec \tilde{\tau}$.
Since the value of $m_{F'}'$ is constant on $F' \succ \tau_j$, we have $m_{F'}'=0$ on $T (\tau_j)$.
Since $F_i \subset \phi_{\xi, \Deltav_i} (\tau_j)$ for $i \in I'$, we also have $T(F_i) \subset T \lb \phi_{\xi, \Deltav_i} (\tau_j) \rb \subset T(\tau_j)$ (cf.~\pref{lm:TT}.2).
We can see that the value of $m_{F'}'$ is constant also on $F_i$ $(i \in I')$, which will be denoted by $c_{i, j} \in \bR$.
For each $i \in I'$, we also take an element $m_i \in \Delta_i \cap M$ such that $(m_i -e_i^\ast)=f_i=0$ on $\tilde{\tau}$.
We set 
\begin{align}
m_{F'}:=m_{F'}'+\sum_{i \in I'} c_{i, j} (m_i -e_i^\ast).
\end{align}
Then the restriction of this to $\tilde{\tau}$ takes the minimum just along $F' \prec \tilde{\tau}$.
Furthermore, since $\la m_{F'}, F_i \times \lc e_i \rc \ra=c_{i, j}-c_{i, j}=0$ for $i \in I'$, one has $m_{F'} \in {\vspan (C')}^\perp$.
Thus $F' \succ \tau_j$ is contained in the image of the map $\varphi_{\tilde{\tau}}$.

Since the normal fans of $\pi_{C'} \lb \xi \rb$ and $\pi_{C'} \lb \Deltav_i \times \lc e_i \rc \rb$ are the restrictions of the normal fans of $\xi$ and $\Deltav_i \times \lc e_i \rc$ to ${\vspan (C')}^\perp \subset M_\bR$ respectively (cf.~e.g.~\cite[Lemma 7.11]{MR1311028}), we can see that we have the map 
$\phi_{\xi, \Deltav_i}':=\phi_{\pi_{C'} (\xi), \pi_{C'} \lb \Deltav_i \times \lc e_i \rc \rb}$
of \eqref{eq:phi} for $\pi_{C'} (\xi), \pi_{C'} \lb \Deltav_i \times \lc e_i \rc \rb$, and the commtative diagram
\begin{align}
  \begin{CD}
     \scrP_{\pi_{C'} (\xi)} @>\phi_{\xi, \Deltav_i}'>> \scrP_{\pi_{C'} \lb \Deltav_i \times \lc e_i \rc \rb} \\
  @V\varphi_{\xi}VV    @VV\varphi_{\Deltav_i \times \lc e_i \rc}V \\
     \scrP_{\xi}  @>\phi_{\xi, \Deltav_i \times \lc e_i \rc}>>  \scrP_{\Deltav_i \times \lc e_i \rc},
  \end{CD}
\end{align}
where the vertical maps are \eqref{eq:vp} for $\xi$ and $\Deltav_i \times \lc e_i \rc$.
We have
\begin{align}\label{eq:phi4}
\phi_{\Deltav_0, \Deltav_i} \lb \pi_{C'} \lb \tau_j \rb \rb
=k_{i} \cdot \phi_{\xi, \Deltav_i}' \lb \pi_{C'} \lb \tau_j \rb \rb
=k_{i} \cdot \pi_{C'} \lb \phi_{\xi, \Deltav_i\times \lc e_i \rc} \lb \tau_j \rb \rb
=k_{i} \cdot \pi_{C'} \lb \phi_{\xi, \Deltav_i} (\tau_j) \times \lc e_i \rc \rb,
\end{align}
where $\Deltav_0=\pi_{C'} (\xi)$ and $\Deltav_i=k_i \cdot \pi_{C'} \lb \Deltav_i \times \lc e_i \rc \rb$.
Here note that multiplying a scalar $k_i \in \bR_{> 0}$ to the polytope does not affect its normal fan.
\eqref{eq:phi4} coincides with the term appearing in \eqref{eq:intersect'}, and the current situation fits in the setup of \pref{lm:tech}.

It turns out by \pref{lm:tech} that the intersection \eqref{eq:intersect'} is empty if 
\begin{align}\label{eq:pp-intersect}
\bigcap_{j=1, 2} \lb \sum_{i \in I_\sigma} k_{i} \cdot \pi_{C'}  \lb \phi_{\xi, \Deltav_i} (\tau_j) \times \lc e_i \rc \rb \rb
\end{align}
is empty, and if \eqref{eq:pp-intersect} is not empty, then \eqref{eq:intersect'} is equal to
\begin{align}
\lc \bigcap_{j=1, 2} \pi_{C'} \lb W_{\tau_j} \cap \tilde{\tau} \rb \rc
+
\lc \bigcap_{j=1, 2} \lb \sum_{i \in I_\sigma} k_{i} \cdot \pi_{C'}  \lb \phi_{\xi, \Deltav_i} (\tau_j) \times \lc e_i \rc \rb \rb \rc.
\end{align}
The latter part $\bigcap_{j=1, 2} \lb \sum_{i \in I_\sigma} k_{i} \cdot \pi_{C'}  \lb \phi_{\xi, \Deltav_i} (\tau_j) \times \lc e_i \rc \rb \rb$ does not affect the values of the maps $\delta_{\tau_1}$ and $\delta_{\tau_2}$.
The former part $\bigcap_{j=1, 2} \pi_{C'} \lb W_{\tau_j} \cap \tilde{\tau} \rb$ is equal to $\pi_{C'} \lb W_{\tau_0} \cap \tilde{\tau} \rb$, where $\tau_0 \prec \xi$ is the minimum face containing both $\tau_1$ and $\tau_2$.
Let $x \in \pi_{C'} \lb W_{\tau_0} \cap \tilde{\tau} \rb$ be an arbitrary point, and $x' \in W_{\tau_0} \cap \tilde{\tau}$ be a point such that $ \pi_{C'} \lb x' \rb=x$.
Since the restriction of $\pi_{C_{\tau_j}}$ to $X_{C'} (\bT) \subset X_{C_{\tau_j}} (\bT)$ factors through the projection $\pi_{C'} \colon X_{C'} (\bT) \to O_{C'} (\bT)$, one has $\delta_{\tau_j}(x')=\delta_{\tau_j}(x)$.
Since the map $\delta_{\tau_j}$ is the identity map on $W_{\tau_j} \supset W_{\tau_0}$, one also has $\delta_{\tau_j}(x')=x'$.
Hence, we have $\delta_{\tau_1}(x)=\delta_{\tau_2}(x)=x'$.
We can see $\delta_{\tau_1} = \delta_{\tau_2}$ on $\pi_{C'} \lb W_{\tau_0} \cap \tilde{\tau} \rb=\bigcap_{j=1, 2} \pi_{C'} \lb W_{\tau_j} \cap \tilde{\tau} \rb$.
Thus we have $\delta_{\tau_1} = \delta_{\tau_2}$ also on \pref{eq:intersect'} in this case.

Next, we consider the case $\tilde{\tau}_1 \neq \tilde{\tau}_2$.
Take a vertex $v \prec \xi$, and consider the projection $\pi_{C_v} \colon X_{C_v}(\bT) \to O_{C_v}(\bT)$ of \eqref{eq:proj} with respect to the cone of \eqref{eq:ctau} with $\tau=v$.
The restriction of this projection $\pi_{C_v}$ to $B$ is bijective by \pref{lm:aff-str}.
We take an affine function $\phi \colon O_{C_v}(\bT) \to \bR$ such that
$\phi >0$ on $\pi_{C_v} \lb \tilde{\tau}_1\rb \setminus \pi_{C_v} \lb \tilde{\tau}_2\rb$, 
$\phi <0 $ on $\pi_{C_v} \lb \tilde{\tau}_2\rb \setminus \pi_{C_v} \lb \tilde{\tau}_1\rb$, 
and $\phi = 0$ on $\pi_{C_v} \lb \tilde{\tau}_1\rb \cap \pi_{C_v} \lb \tilde{\tau}_2\rb$.
We further take its pull-back $\phi':=\phi \circ \pi_{C_v}$ by the projection $\pi_{C_v}$.
Since $\pi_{C_v} \lb \tilde{\tau}_1\rb \cap \pi_{C_v} \lb \tilde{\tau}_2\rb \succ \pi_{C_v} (\xi)$, we have $\phi = 0$ on $\pi_{C_v} \lb \xi \rb$, and $\phi' = 0$ on $\xi$.
In particular, adding vectors in $T \lb \Deltav_i \rb \subset T \lb \xi \rb$ (cf.~\pref{lm:TT}.2) does not affect the value of the function $\phi'$.
Since adding vectors in $\vspan \lb C_v \rb$ also does not affect the value of the function $\phi'$, it turns out that adding vectors in $\vspan \lb C \rb$ does not affect the value of the function $\phi'$.
Hence, the function $\phi'$ can be naturally extended to a function on $X_C(\bT)$.
The values of the function are greater than or equal to $0$ on 
$\lc \pi_{C'} \lb W_{\tau_1} \cap \tilde{\tau}_1 \rb 
+ \sum_{\substack{i \in I_\sigma}} k_{i} \cdot \pi_{C'}  \lb \phi_{\xi, \Deltav_i} (\tau_1) \times \lc e_i \rc \rb \rc$, and are less than or equal to $0$ on
$\lc \pi_{C'} \lb W_{\tau_2} \cap \tilde{\tau}_2 \rb 
+ \sum_{\substack{i \in I_\sigma}} k_{i} \cdot \pi_{C'}  \lb \phi_{\xi, \Deltav_i} (\tau_2) \times \lc e_i \rc \rb \rc$.
It turns out that the intersection \pref{eq:intersect'} must be contained in
\begin{align}
\bigcap_{j=1,2}
\lc \pi_{C'} \lb W_{\tau_j} \cap \tilde{\tau}_1 \cap \tilde{\tau}_2 \rb 
+ \sum_{\substack{i \in I_\sigma}} k_{i} \cdot \pi_{C'}  \lb \phi_{\xi, \Deltav_i} (\tau_j) \times \lc e_i \rc \rb \rc.
\end{align}
We reduced to the case of $\tilde{\tau}_1=\tilde{\tau}_2$, in which we have already shown $\delta_{\tau_1} = \delta_{\tau_2}$.
Thus we obtain the lemma.
\end{proof}

We set
\begin{align}
X':=\bigcup_{\tau \prec \xi} V_\tau \subset X(f_1, \cdots, f_r).
\end{align}
By \pref{lm:intersection}, we can glue the maps $\delta_\tau$ together and obtain a continuous map $\delta' \colon X' \to \overline{U}_{\xi}$.
We set $X:={\delta'}^{-1} (U_\xi)$ and define $\delta \colon X \to U_\xi$ to be the restriction of $\delta'$ to the subset $X \subset X'$.
These give $X$ and $\delta$ in \pref{th:local-cont}.
Since $X'$ is compact and $\overline{U}_{\xi}$ is a Hausdorff space, the map $\delta' \colon X' \to \overline{U}_{\xi}$ is proper.
Therefore, the map $\delta \colon X \to U_\xi$ is also proper.

\begin{definition}\label{df:local-contraction}
We call the map $\delta \colon X \to U_{\xi}$ of \pref{th:local-cont} a \emph{local model of tropical contractions}.
If $\sum_{i=1}^r T \lb \Delta_i \rb$ is the internal direct sum of $\lc T \lb \Delta_i \rb \rc_{i \in \lc 1, \cdots, r \rc}$ and $\sum_{i=1}^r T \lb \Deltav_i \rb$ is the internal direct sum of $\lc T \lb \Deltav_i \rb \rc_{i \in \lc 1, \cdots, r \rc}$, then we say that $\delta$ is \emph{good}.
If $\lb \sum_{i=1}^r T \lb \Delta_i \rb \rb \cap M'$ is an internal direct sum of $\lc T_\bZ \lb \Delta_i \rb \rc_{i \in \lc 1, \cdots, r \rc}$ and $\lb \sum_{i=1}^r T \lb \Deltav_i \rb \rb \cap N'$ is an internal direct sum of $\lc T_\bZ \lb \Deltav_i \rb \rc_{i \in \lc 1, \cdots, r \rc}$, then we say that $\delta$ is \emph{very good}.
\end{definition}

\begin{remark}\label{rm:hyp}
The conditions for being good/very good are the ones that we imposed in \pref{th:local-cont}.3.
When $r=1$ (i.e., the case of contracting a tropical hypersurface), the condition of being good/very good is always obviously satisfied.
\end{remark}

\begin{definition}\label{df:contraction}
We say that a continuous map $\delta \colon V \to B$ from a rational polyhedral space $V$ to an IAMS $B$ is a \emph{tropical contraction} if for any point $x \in B$ there exists a neighborhood $U_x \subset B$ of $x$ such that the restriction of $\delta$ to $\delta^{-1} \lb U_x \rb$ is isomorphic to a local model of tropical contractions.
A tropical contraction is called \emph{good} (resp. \emph{very good}) if the restriction $\left. \delta \right|_{\delta^{-1} \lb U_x \rb}$ is isomorphic to a good (resp. very good) local model of tropical contractions for any $x \in B$.
\end{definition}

\begin{remark}\label{rm:morphism}
There is a notion called \emph{tropical spaces}, which was introduced in \cite{MR4562566}.
It is a generalization of rational polyhedral spaces and IAMS.
We refer the reader to Section 2 in loc.cit.~for its definition.
The local model of IAMS $U_\xi$ constructed above becomes a \emph{tropical piecewise linear space} in the sense of \cite[Definition 2.2]{MR4562566} by the embedding $U_\xi \hookrightarrow N_\bR$, and the \emph{affine structure} in the sense of \cite[Definition 2.7]{MR4562566} is given by $\iota_\ast \mathrm{Aff}_{U_{\xi, 0}}$.
The local model of IAMS $U_\xi$ can be regarded as a tropical space in their sense.
Tropical contractions are \emph{morphisms of tropical spaces} in the sense of \cite[Definition 2.16]{MR4562566}.
\end{remark}

\subsection{Examples of local models of tropical contractions}\label{sc:ex}

\begin{example}\label{eg:ff}
Suppose $d=2, r=1$.
Let $d_1, d_2$ be a basis of $N'$, and $d_1^\ast, d_2^\ast$ be its dual basis.
We consider the lattice polytopes
\begin{align}
\Delta_1:=\conv \lb \lc 0, kd_1^\ast \rc \rb \subset M_\bR', \quad \Deltav_1:=\conv \lb \lc 0, ld_2 \rc \rb \subset N_\bR',
\end{align}
where $k, l$ are positive integers.
The cone $C$ is given by
\begin{align}
C:=\cone \lb \lc e_1, e_1+l d_2 \rc \rb.
\end{align}
The dual cone $C^\vee$ is generated by $d_2^\ast$ and $(l e_1^\ast -d_2^\ast)$.
The polynomial $f_1$ is given by
\begin{align}
f_1(n)=\min \lc 0, \la -e_1^\ast, n \ra, \la -e_1^\ast+k d_1^\ast, n \ra \rc.
\end{align}
The tropical hypersurface $X(f_1) \subset X_C(\bT)$ is shown in \pref{fg:f-f}.
\begin{figure}[htbp]
\begin{center}
\includegraphics[scale=0.55]{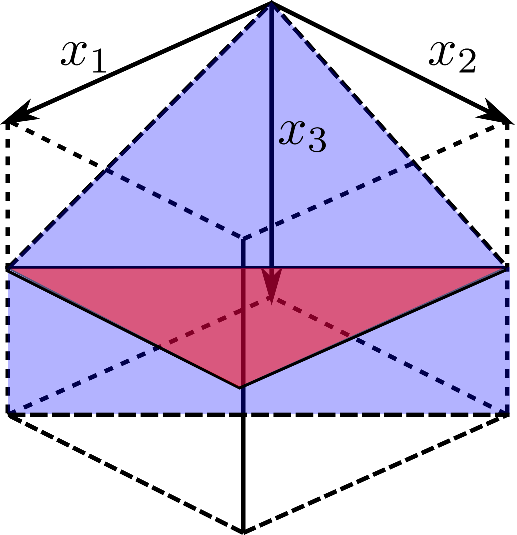}
\end{center}
\caption{The tropical hypersurface $X(f_1)$ and the subset $B$}
\label{fg:f-f}
\end{figure}
It is the union of the blue region and the red region.
The axes $x_1, x_2, x_3$ denote the coordinates on $N_\bR$ defined by $d_2^\ast, (le_1^\ast-d_2^\ast), d_1^\ast$ respectively.
The subset $B$ is the blue region.
The subset $D$ is the line $\bR d_2$, the intersection of the blue region and the red region.
For an example, we set 
\begin{align}
v_1&:=0, v_2:=d_2 \in D \\
\xi&:=\conv \lb \lc v_1, v_2 \rc \rb \subset D \\
\sigma_1&:=\conv \lb \xi \cup \lc -d_1-k e_1 \rc \rb \subset B \\
\sigma_2&:=\conv \lb \xi \cup \lc d_1 \rc \rb \subset B,
\end{align}
and let $\scrP$ be the polyhedral complex consisting of $\sigma_1, \sigma_2$ and all of their faces.
The map $\phi_{\xi, \Deltav_1} \colon \scrP_\xi \to \scrP_{\Deltav_1}$ is given by
\begin{align}
v_1 \mapsto 0, \quad v_2 \mapsto l d_2, \quad \xi \mapsto \Deltav_1,
\end{align}
and we have
\begin{align}
C_{v_1}=\cone \lb \lc e_1 \rc \rb, \quad C_{v_2}=\cone \lb \lc e_1+l d_2 \rc \rb, \quad C_{\xi}=C.
\end{align}
\begin{figure}[htbp]
\begin{center}
\includegraphics[scale=0.65]{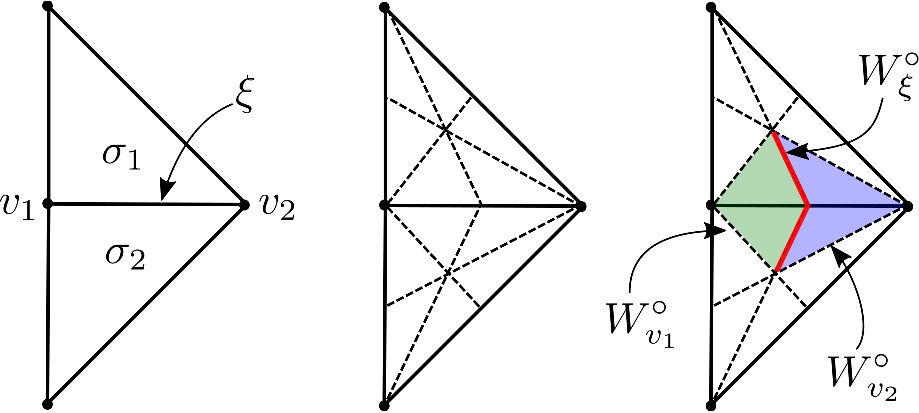}
\end{center}
\caption{The polyhedral complex $\scrP$, its subdivision $\widetilde{\scrP}$, and the subsets $W_{v_1}^\circ, W_{v_2}^\circ, W_{\xi}^\circ$}
\label{fg:scrP}
\end{figure}
The polyhedral complex $\scrP$, its subdivision $\widetilde{\scrP}$, and the subsets $W_{v_1}^\circ, W_{v_2}^\circ, W_{\xi}^\circ$ are shown in \pref{fg:scrP}.
The subset $U_\xi$ is the union of $W_{v_1}^\circ, W_{v_2}^\circ$, and $W_{\xi}^\circ$.
Let $\sigma \subset X(f_1) $ denote the closure of $X(f_1) \setminus B$, i.e., the red triangle in \pref{fg:f-f}.
\pref{fg:contr} shows the subsets $V_{v_1} \cap \sigma, V_{v_2} \cap \sigma, V_{\xi} \cap \sigma$ 
and the contractions $\delta_{v_1}, \delta_{v_2}, \delta_{\xi}$.
The monodromy that we computed in \pref{pr:monodromy} is given by
\begin{align}
T_\gamma (n)=n+kl \sum_{i=1}^r \la d_1^\ast, n \ra \cdot d_2.
\end{align}
The point $a_\xi$ that we chose from the interior of $\xi$ becomes a concentration of $kl$ focus-focus singularities.
\begin{figure}[htbp]
\begin{center}
\includegraphics[scale=0.6]{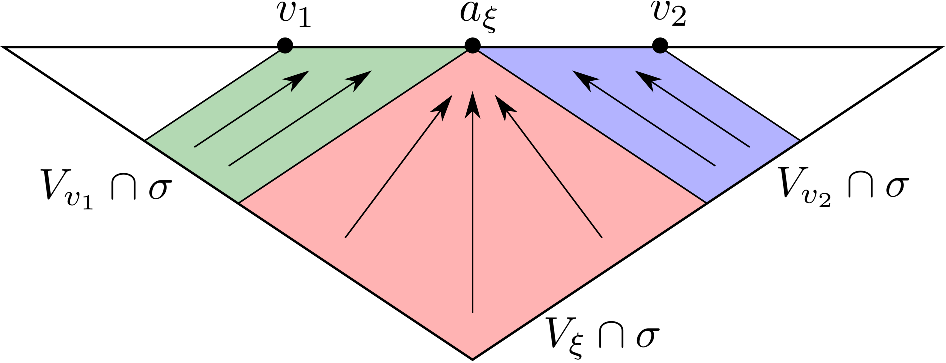}
\end{center}
\caption{The subsets $V_{v_1} \cap \sigma, V_{v_2} \cap \sigma, V_{\xi} \cap \sigma$ 
and the contractions $\delta_{v_1}, \delta_{v_2}, \delta_{\xi}$}
\label{fg:contr}
\end{figure}

\begin{remark}
Kontsevich and Soibelman constructed a $2$-sphere with an integral affine structure with singularities by contracting the Clemens polytope of a degenerating family of K3 surfaces \cite[Section 4.2.5]{MR2181810}.
Their contraction is quite similar to the above contraction of the tropical hypersurface.
Compare the local contraction given in \cite[Section 4.2.5]{MR2181810} to the above one.
\end{remark}
\end{example}

\begin{example}\label{eg:pos-neg}
Suppose $d=3, r=1$.
Let $d_1, d_2, d_3$ be a basis of $N'$, and $d_1^\ast, d_2^\ast, d_3^\ast$ be its dual basis.
We consider the lattice polytopes
\begin{align}\label{eq:3dim-poly}
\Delta_1:=\conv \lb \lc 0, d_1^\ast, d_2^\ast \rc \rb \subset M_\bR', \quad \Deltav_1:=\conv \lb \lc 0, d_3 \rc \rb \subset N_\bR'.
\end{align}
We can apply \pref{lm:length1}.
The monodromy polytopes coincide with \eqref{eq:3dim-poly}.
It turns out that the point $a_\xi$ becomes the so called \emph{positive vertex} which appears in the base spaces of topological $3$-torus fibrations of \cite{MR1821145}.
Similarly, if we consider the lattice polytopes
\begin{align}
\Delta_1:=\conv \lb \lc 0, d_1^\ast \rc \rb \subset M_\bR', \quad \Deltav_1:=\conv \lb \lc 0, d_2, d_3 \rc \rb \subset N_\bR',
\end{align}
the point $a_\xi$ becomes the \emph{negative vertex}.
If we consider the lattice polytopes
\begin{align}
\Delta_1:=\conv \lb \lc 0, d_1^\ast, d_2^\ast, d_1^\ast+d_2^\ast \rc \rb \subset M_\bR', \quad \Deltav_1:=\conv \lb \lc 0, d_3 \rc \rb \subset N_\bR',
\end{align}
the point $a_\xi$ becomes the \emph{positive node} of \cite{MR3228462}.
If we consider the lattice polytopes
\begin{align}
\Delta_1:=\conv \lb \lc 0, d_1^\ast \rc \rb \subset M_\bR', \quad \Deltav_1:=\conv \lb \lc 0, d_2, d_3, d_2+d_3 \rc \rb \subset N_\bR',
\end{align}
the point $a_\xi$ becomes the \emph{negative node}.
We refer the reader to \cite[Example 2.3, Example 2.4, Example 6.1, Example 6.2]{MR3228462} for a description of these singularities.
\end{example}

\begin{example}\label{eg:modification}
Suppose that $r=1$ and $d$ is an arbitrary positive integer.
We consider the case where the lattice polytope $\Deltav_1$ is a point $\lc 0 \rc \subset N_\bR'$.
The cone $C$ is $\bR_{\geq 0} e_1$ and the associated tropical toric variety is $X_C(\bT) = N'_\bR \times \bT$. 
The subspace $B \subset X_C(\bT) = N'_\bR \times \bT$ is 
\begin{align}
B=\lc (n, y) \in N'_\bR \times \bT \relmid y=\min_{m \in \Delta_1} \la m, n \ra \rc.
\end{align}
This is the graph of the function $f_1' \colon N_\bR \to \bT$
\begin{align}
f_1'(n):=\min_{m \in \Delta_1} \la m, n \ra.
\end{align}
The map $\phi_{\xi, \Deltav_1} \colon \scrP_\xi \to \scrP_{\Deltav_1}$ is the constant map to $\lc 0\rc$.
For any choice of the polyhedron $\xi$ and the polyhedral complex $\scrP$, we have $C_\tau=C$ and the map $\pi_{C_\tau} \colon X_{C_\tau}(\bT) \to O_{C_\tau}(\bT)$ coincides with the projection $\pi_C \colon N_\bR' \times \bT \to N_\bR'$ for any $\tau \prec \xi$.
One also has 
$X=\lb U_\xi + \bR_{\geq 0} e_1 \rb \cap X(f_1)$.
The map $t_\tau \colon \pi_{C_\tau} \lb W_\tau \rb \to W_\tau$ that we took for making the contraction $\delta_\tau$ is the restriction of
\begin{align}\label{eq:graph}
N_\bR' \to N_\bR' \times \bT, \quad n \to \lb n, f_1'(n) \rb.
\end{align}
The map \eqref{eq:graph} is a bijection to $B \subset N_\bR' \times \bT$.
By identifying $B$ with $N_\bR'$ by this map, the tropical contraction $\delta$ can be seen as the restriction of the projection $\pi_C \colon N_\bR' \times \bT \to N_\bR'$ to $X \subset X(f_1)$.
This is exactly the tropical modification of Mikhalkin \cite{MR2275625} with respect to the function $f_1'$.
We refer the reader to \cite[Section 3]{MR2275625} or \cite{Kal15} or \cite{Sha15} for details of tropical modifications.
\end{example}

\begin{example}
Suppose that $d=1$ and $r$ is an arbitrary positive integer.
Let $d_1$ be a basis of $N'$, and $d_1^\ast$ be its dual basis.
We consider the lattice polytopes 
\begin{align}
\Delta_i=\conv \lb \lc 0, d_1^\ast \rc \rb \subset M_\bR', \quad \Deltav_i=\lc 0 \rc \subset N_\bR' \quad ( 1 \leq i \leq  r ).
\end{align}
The cone $C$ is $\sum_{i=1}^r \bR_{\geq 0} e_i$, and the associated tropical toric variety is 
$X_C(\bT)=\bR d_1 + \sum_{i=1}^r \bT e_i$.
The stable intersection $X(f_1, \cdots, f_r)$ and its subspace $B$ are given by
\begin{align}
X(f_1, \cdots, f_r)&=\bR_{\geq 0}d_1 \cup \bR_{\geq 0} (-d_1-e_1- \cdots - e_r) \cup \lb \bigcup_{i=1}^r \bT_{\geq 0} e_i \rb \\
B&=\bR_{\geq 0}d_1 \cup \bR_{\geq 0} (-d_1-e_1- \cdots - e_r).
\end{align}
The subspace $D \subset B$ is $\lc 0 \rc \subset N_\bR'$, and the polytope $\xi$ also has to be $\lc 0 \rc$ in this case.
The map $\phi_{\xi, \Deltav_i} \colon \scrP_\xi \to  \scrP_{\Deltav_i}$ is just $\lc 0 \rc \mapsto \lc 0 \rc$, and we have $C_\xi=C$.
The map $\pi_C \colon X_C(\bT) \to O_C(\bT)$ is the projection $\bR d_1 + \sum_{i=1}^r \bT e_i \to \bR d_1$.
The tropical contraction $\delta$ contracts all unbounded edges $\bT_{\geq 0} e_i$ of $X(f_1, \cdots, f_r) $ to the point $\xi =\lc0 \rc$.
\pref{fg:curve} shows the contraction when $r=2$.
The blue lines are $B$ and the red lines are the unbounded edges $\bT_{\geq 0} e_1$ and $\bT_{\geq 0} e_2$.
\begin{figure}[htbp]
\begin{center}
\includegraphics[scale=0.7]{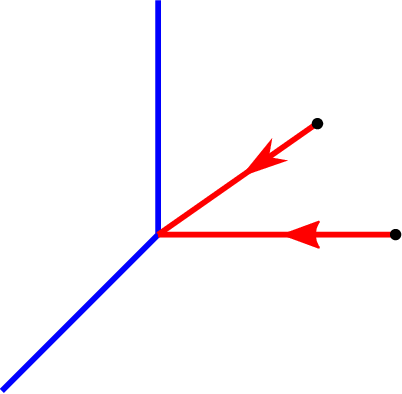}
\end{center}
\caption{A contraction of a tropical curve of valence $4$ to a line.}
\label{fg:curve}
\end{figure}
\end{example}

\subsection{Proof of \pref{th:local-cont}.2}\label{sc:local-cont2}

We prove \pref{th:local-cont}.2 by making a homomorphism $g_x \colon \lb \iota_\ast \bigwedge^{p} \check{\Lambda}\rb_x \to \lb \delta_\ast \scF^p_\bZ \rb_x$ between the stalks for every point $x \in U_{\xi}$ and checking that it defines an isomorphism $\iota_\ast \bigwedge^{p} \check{\Lambda} \to \delta_\ast \scF^p_\bZ$ of sheaves for each $p \geq 0$.
Assume $x \in \rint \lb \conv \lb \lc a_{\tau_0}, a_{\tau_1}, \cdots, a_{\tau_l} \rc \rb \rb \subset W_{\tau_0}^\circ$, where $\tau_0 \prec \tau_1 \prec \cdots \prec \tau_l \in \scrP$, $l \geq 0$, and $\xi \in \lc \tau_0, \cdots, \tau_l \rc$.
Then we have
\begin{align}\label{eq:pre-x}
\delta^{-1} (x) = \lb x + \sum_{\substack{1 \leq i \leq r \\ v \prec \tau_0}} \bT_{\geq 0} \lb \phi_{\xi, \Deltav_i} \lb v \rb + e_i \rb \rb \cap X(f_1, \cdots, f_r) \subset X_{C_{\tau_0}}(\bT).
\end{align}
We take a vertex $v_0 \prec \tau_0$, and set $N'':=\left. N \middle/ \oplus_{i=1}^r \bZ \lb \phi_{\xi, \Deltav_i} \lb v_0 \rb +e_i \rb \right.$, $M'':=\Hom \lb N'', \bZ \rb=\bigcap_{i=1}^r \lb \phi_{\xi, \Deltav_i} \lb v_0 \rb +e_i \rb^\perp$.
The stalk $\lb \iota_\ast \bigwedge^{p} \check{\Lambda}\rb_x$ is a submodule of $\bigwedge^p M''$, and the stalk $\lb \delta_\ast \scF^p_\bZ \rb_x$ is a quotient of a submodule of $\bigwedge^p M$.
We will show that the inclusion
\begin{align}\label{eq:inclv0}
\iota_{v_0} \colon \bigwedge^p M'' \hookrightarrow \bigwedge^p M
\end{align}
induces an isomorphism between the stalks $\lb \iota_\ast \bigwedge^{p} \check{\Lambda}\rb_x$ and $\lb \delta_\ast \scF^p_\bZ \rb_x$.
It is equivalent to show that for a small neighborhood $U_{\delta^{-1}} (x)$ of $\delta^{-1} (x)$ in $X$ and 
a small neighborhood $U_x$ of $x$ in $U_{\xi}$,
the projection
\begin{align}\label{eq:projv0}
\pi_{v_0} \colon \bigwedge^p N \to \bigwedge^p N''
\end{align}
induces an isomorphism between $\scF_p^\bZ \lb U_{\delta^{-1} (x)} \rb$ and $\Hom \lb \lb \iota_\ast \bigwedge^{p} \check{\Lambda} \rb \lb U_x \rb, \bZ \rb$.

First, we compute $\Hom \lb \lb \iota_\ast \bigwedge^{p} \check{\Lambda} \rb \lb U_x \rb, \bZ \rb$.
Since $\lb \iota_\ast \bigwedge^{p} \check{\Lambda} \rb \lb U_x \rb$ consists of integral cotangent vectors that are monodromy invariant on $U_x \setminus \Gamma$, and the fundamental group of $U_x \setminus \Gamma$ is generated by loops that are homotopic to the ones considered in \eqref{eq:monodromy0} with respect to $\omega \in \scrP(1), \rho \in \scrP(d-1)$ such that $\omega \prec \tau_0, \rho \succ \tau_l$, we have
\begin{align}
 \lb \iota_\ast \bigwedge^{p} \check{\Lambda} \rb \lb U_x \rb=\lc \sum_{\substack{\omega \prec \tau_0 \\ \omega \in \scrP(1)}} \sum_{\substack{\rho \succ \tau_l \\ \rho \in \scrP(d-1)}} \lb \id - \bigwedge^{p} T_\omega^\rho \rb \lb \bigwedge^{p} N'' \rb \rc^\perp,
\end{align}
where $T_\omega^\rho$ is the monodromy transformation of \eqref{eq:monodromy0}.
Hence, $\Hom \lb \lb \iota_\ast \bigwedge^{p} \check{\Lambda} \rb \lb U_x \rb, \bZ \rb$ is equal to
\begin{align}\label{eq:homp}
\left. \bigwedge^p N'' \middle/ \ld
\lc \sum_{\substack{\omega \prec \tau_0 \\ \omega \in \scrP(1)}} \sum_{\substack{\rho \succ \tau_l \\ \rho \in \scrP(d-1)}} \lb \id - \bigwedge^{p} T_\omega^\rho \rb \lb \bigwedge^{p} N'' \rb \rc \otimes_\bZ \bR \cap \bigwedge^p N'' 
\rd \right..
\end{align}
For polyhedra $\omega \in \scrP(1)$, $\rho \in \scrP(d-1)$ such that $\omega \prec \tau_0$, $\rho \succ \tau_l$, let $v_1(\omega), v_2(\omega) \in \scrP(0)$ be the vertices of $\omega$, and $\sigma_1(\rho), \sigma_2(\rho) \in \scrP(d)$ be the maximal-dimensional polyhedra containing $\rho$ as their faces. 
From \pref{pr:monodromy}, we can see
\begin{align}\label{eq:pmono}
\lb \id - \bigwedge^{p} T_{\omega}^\rho \rb \lb \bigwedge^{p} N'' \rb
=\lb \bigwedge^{p-1} \Lambda_\rho \rb \wedge \lb \sum_{i=1}^r \la m_i^{\sigma_2(\rho)} - m_i^{\sigma_1(\rho)}, N'' \ra \cdot \lb \phi_{\xi, \Deltav_i}  \lb v_2(\omega) \rb- \phi_{\xi, \Deltav_i}  \lb v_1(\omega) \rb \rb \rb,
\end{align}
where $m_i^{\sigma_j(\rho)} \in \Delta_i \cap M'$ denotes the element such that $m_i^{\sigma_j(\rho)} -e_i^\ast = f_i$ on $\sigma_j(\rho)$.
We set
\begin{align}
I(\rho):=\lc i \in \lc 1, \cdots, r \rc \relmid m_i^{\sigma_1(\rho)} \neq m_i^{\sigma_2(\rho)} \rc.
\end{align}
Then by \eqref{eq:pmono}, the denominator of \eqref{eq:homp} turns out to be equal to
\begin{align}
\lc \sum_{\substack{\rho \succ \tau_l \\ \rho \in \scrP(d-1)}} \lb \bigwedge^{p-1} \Lambda_\rho \rb \wedge \sum_{i \in I(\rho)} T \lb \phi_{\xi, \Deltav_i}  \lb \tau_0 \rb \rb \rc
\otimes_\bZ \bR \cap \bigwedge^p N'',
\end{align}
where $d_\omega$ is a primitive tangent vector on $\omega$.
Therefore, \eqref{eq:homp} is equal to
\begin{align}\label{eq:lp}
\left. \bigwedge^p N'' \middle/ \ld \lc \sum_{\substack{\rho \succ \tau_l \\ \rho \in \scrP(d-1)}} \lb \bigwedge^{p-1} \Lambda_\rho \rb \wedge \sum_{i \in I(\rho)} T \lb \phi_{\xi, \Deltav_i} \lb \tau_0 \rb \rb \rc
\otimes_\bZ \bR \cap \bigwedge^p N'' \rd
\right..
\end{align}

Next, we compute $\scF_p^\bZ \lb U_{\delta^{-1} (x)} \rb$.
\begin{lemma}\label{lm:sttl}
Let $\sigma=\tau \lb F_1, \cdots, F_r, C_1, \cdots, C_r \rb \in \scrP_{X(f_1, \cdots, f_r)}$ be a polyhedron, where $F_i \prec \conv (A_i), C_i \prec \cone(\Deltav_i \times \lc e_i \rc)$.
The following are equivalent:
\begin{enumerate}
\renewcommand{\labelenumi}{(\roman{enumi})}
\item The polyhedron $\sigma$ intersects with $\delta^{-1} (x)$.
\item The polyhedron $\sigma$ intersects with the small neighborhood $U_{\delta^{-1} (x)}$ of $\delta^{-1} (x)$.
\item $F_i \cap M \subset \scrM_i(x)$ and $C_i \prec \cone \lb \phi_{\xi, \Deltav_i} \lb \tau_0 \rb + e_i \rb$ for any $i \in \lc 1, \cdots, r \rc$.
\end{enumerate}
Here $\scrM_i(x)$ is \eqref{eq:F_i} for $n=x$.
\end{lemma}
\begin{proof}
It is obvious that $\rm(\hspace{.18em}i\hspace{.18em})$ implies $\rm(\hspace{.08em}ii\hspace{.08em})$.
We show that $\rm(\hspace{.08em}ii\hspace{.08em})$ implies $\rm(i\hspace{-.08em}i\hspace{-.08em}i)$.
If there is some $i \in \lc 1, \cdots, r \rc$ such that $F_i \cap M \not\subset \scrM_i(x)$, 
there exists $m \in F_i \cap M$ such that 
$f_i=m$ on $\sigma$ and $f_i < m$ at the point $x$.
We can see from \eqref{eq:pre-x} that we have $f_i < m$ also on $\delta^{-1} (x)$, and $U_{\delta^{-1} (x)}$.
We get $U_{\delta^{-1}(x)} \cap \sigma = \emptyset$ which contradicts $\rm(\hspace{.08em}ii\hspace{.08em})$.
Hence, we must have $F_i \cap M \subset \scrM_i(x)$ for all $i$.
We can also see from \eqref{eq:pre-x} that $\delta^{-1}(x)$ and $U_{\delta^{-1}(x)}$ intersect with $O_{C'}(\bT)$ if and only if $C' \prec \sum_{i=1}^r \cone \lb \phi_{\xi, \Deltav_i} \lb \tau_0 \rb + e_i \rb$.
We obtain $C_i \prec \cone \lb \phi_{\xi, \Deltav_i} \lb \tau_0 \rb + e_i \rb$ for all $i$.

Next, we show that $\rm(i\hspace{-.08em}i\hspace{-.08em}i)$ implies $\rm(\hspace{.18em}i\hspace{.18em})$.
Let $\lc n_{i, j} \rc_{j \in J_i}$ be the set of generators of $C_i$ $(1 \leq i \leq r)$.
Consider, for instance, the element
\begin{align}
x_0(r)&:=x +  r \cdot \lb \sum_{1 \leq i \leq r} \sum_{j \in J_i} n_{i, j} \rb,
\end{align}
where $r \in \bT_{> 0}$.
One can show $x_0(\infty) \in \sigma \cap \delta^{-1} (x)$ as follows:
One can check
\begin{align}
\scrM_i(x_0(r))
=
\left\{
\begin{array}{ll}
\scrM_i(x) & C_i=\lc 0\rc \\
\scrM_i(x) \setminus \lc 0 \rc & C_i \neq \lc 0\rc \\
\end{array}
\right.
\end{align}
for $r \in \bR_{> 0}$.
When $C_i \neq \lc 0\rc$, one has $F_i \not\ni \lc 0 \rc$.
Hence, for any $i \in \lc 1, \cdots, r \rc$ and $r \in \bR_{> 0}$, we have $\scrM_i(x_0(r)) \supset F_i \cap M$, which implies $x_0(\infty) \in \tau \lb F_1, \cdots, F_r \rb$.
One can straightforwardly check that $x_0(\infty)$ is also in $O_{C'} \lb \bT \rb$ with $C'=\sum_{i=1}^r C_i$.
It is also obvious from $C_i \prec \cone \lb \phi_{\xi, \Deltav_i} \lb \tau_0 \rb + e_i \rb$ that we have $x_0(\infty) \in \delta^{-1} (x)$.
Hence, we get $x_0(\infty) \in \sigma \cap \delta^{-1} (x)$, and $\sigma \cap \delta^{-1} (x) \neq \emptyset$.
\end{proof}

The minimal polyhedron in $\scrP_{X(f_1, \cdots, f_r)}$ containing $x$ is $\theta \lb \tau_l \rb$, where $\theta \colon \scrP \to \scrP_B$ is the map defined in Condition \ref{cd:loccont'}.
It is written as $\theta \lb \tau_l \rb=\tau \lb F_1', \cdots, F_r', C_1', \cdots, C_r' \rb$ with $F_i'= \conv \lb \scrM_i(x) \rb$ and $C_i' = \lc 0 \rc$ $(1 \leq i \leq r)$.
From \pref{lm:sttl}, we can see that a polyhedron $\sigma=\tau \lb F_1, \cdots, F_r, C_1, \cdots, C_r \rb \in \scrP_{X(f_1, \cdots, f_r)}$ with $C_i = \lc 0 \rc$ $(1 \leq i \leq r)$, which intersects with $U_{\delta^{-1} (x)}$ has $\theta \lb \tau_l \rb$ as its face.
For such a polyhedron $\sigma \succ \theta \lb \tau_l \rb$, we can see again from \pref{lm:sttl} that the maximal face $C'$ of $C$ such that $O_{C'}(\bT) \cap \sigma \cap U_{\delta^{-1} (x)} \neq \emptyset$ is 
\begin{align}
C'=\mathrm{cone} \lb \bigcup_{i \in I_\sigma} \phi_{\xi, \Deltav_i} (\tau_0) \times \lc e_i \rc \rb,
\end{align}
where $I_\sigma$ is the one defined in \eqref{eq:Isig}.
It turns out by \pref{eq:Fp} and \pref{df:scF} that $\scF_p^\bZ \lb U_{\delta^{-1} (x)} \rb$ is equal to
\begin{align}\label{eq:fp}
\left. F^\bZ_p \lb \theta \lb \tau_l \rb \rb \middle/ 
\ld \lc \sum_{\substack{\sigma \succ \theta \lb \tau_l \rb \\ \sigma \in \scrP_{X(f_1, \cdots, f_r)}}}
F^\bZ_{p-1} \lb \sigma \rb \wedge \lb \sum_{i \in I_\sigma} \bZ \lb \phi_{\xi, \Deltav_i} \lb\tau_0 \rb +e_i \rb \rb \rc \otimes_\bZ \bR \cap \bigwedge^p N  \rd \right..
\end{align}

We will check that the projection $\eqref{eq:projv0}$ induces a map from \eqref{eq:fp} to \eqref{eq:lp}.
For any polyhedron $\sigma=\tau \lb F_1, \cdots, F_r, C_1, \cdots, C_r \rb \in \scrP_{X(f_1, \cdots, f_r)}$ with $F_i \prec \conv (A_i), C_i=\lc 0 \rc \prec \cone(\Deltav_i \times \lc e_i \rc)$ $\lb 1 \leq i \leq r \rb$, one can see
\begin{align}\label{eq:Lamsig}
\Lambda_\sigma \cong \Lambda_{\tilde{\sigma}} \oplus \lb \bigoplus_{i \in I_\sigma} \bZ \lb \phi_{\xi, \Deltav_i} (v_0)+e_i\rb \rb
\end{align}
from \pref{lm:sigma}.
Notice that we have $\tilde{\sigma} \succ D \supset \Deltav_i$.

\begin{lemma}\label{lm:rho}
For any $\sigma \in \scrP_{X(f_1, \cdots, f_r)}$ such that $\sigma \succ \theta \lb \tau_l \rb$ and any $i \in I_\sigma$, there exists $\rho \in \scrP(d-1)$ such that $\rho \succ \tau_l$, $i \in I(\rho)$ and $\Lambda_\rho \supset \Lambda_{\tilde{\sigma}}$.
\end{lemma}
\begin{proof}
Since $i \in I_\sigma$, there exists $\rho' \in \scrP_B(d-1)$ such that $\rho' \succ \tilde{\sigma}$ and $m_i^{\sigma_1} \neq m_i^{\sigma_2}$ for the maximal-dimensional polyhedra $\sigma_1, \sigma_2 \succ \rho'$.
Furthermore, $\sigma \succ \theta \lb \tau_l \rb$ implies $\tilde{\sigma} \succ \theta \lb \tau_l \rb$, and $\rho' \supset \tau_l$.
Since $\xi \in \lc \tau_0, \cdots, \tau_l \rc$, we have $\tau_l \succ \xi$.
By the fourth requirement in \pref{cd:loccont} for $\tau_l$, there exists $\rho \in \scrP(d-1)$ such that $\rho \succ \tau_l$ and $\theta(\rho) = \rho'$.
Here we have $i \in I(\rho)$ and $\Lambda_\rho =\Lambda_{\rho'} \supset \Lambda_{\tilde{\sigma}}$.
\end{proof}

By \eqref{eq:Lamsig} and \pref{lm:rho}, we can see that for any $v \prec \tau_0$, $\sigma \in \scrP_{X(f_1, \cdots, f_r)}$ such that $\sigma \succ \theta \lb \tau_l \rb$ and any $i \in I_\sigma$, we have 
\begin{align}
\pi_{v_0} \lb F^\bZ_{p-1} \lb \sigma \rb \wedge \bZ \lb \phi_{\xi, \Deltav_i}(v) +e_i \rb \rb
&\subset \bigwedge^{p-1} \Lambda_{\tilde{\sigma}} \wedge  \bZ \lb \phi_{\xi, \Deltav_i}(v) - \phi_{\xi, \Deltav_i} (v_0) \rb\\
&\subset \sum_{\substack{\rho \succ \tau_l \\ \rho \in \scrP(d-1)}} \bigwedge^{p-1} \Lambda_\rho \wedge \bZ \lb \phi_{\xi, \Deltav_i}(v)-\phi_{\xi, \Deltav_i}(v_0) \rb,
\end{align}
and this is contained in the denominator of \eqref{eq:lp}. 
Hence, the map $\pi_{v_0}$ induces a map from \eqref{eq:fp} to \eqref{eq:lp}.

We will show that the induced map $\scF_p^\bZ \lb U_{\delta^{-1} (x)} \rb \to \Hom \lb \lb \iota_\ast \bigwedge^{p} \check{\Lambda} \rb \lb U_x \rb, \bZ \rb$ is an isomorphism by constructing its inverse map $\zeta$.
Take a polyhedron $\sigma \in \scrP_B(d)$ such that $\sigma \succ \theta(\tau_l)$, and consider the map
$u^\sigma \colon N'' \to N$
defined by
\begin{align}\label{eq:u-sig}
N'' \ni n \mapsto u^\sigma \lb n \rb:=\tilde{n} + \sum_{j=1}^r \la m_j^\sigma-e_j^\ast, \tilde{n} \ra \cdot \lb \phi_{\xi, \Deltav_j}(v_0)+e_j \rb,
\end{align}
where $\tilde{n} \in N$ is an element such that the projection to $N''$ is $n$.
From $\la m_i^\sigma-e_i^\ast, \phi_{\xi, \Deltav_j}(v_0)+e_j \ra=-\delta_{i, j}$, we can see that the element $u^\sigma \lb n \rb$ does not depend on the choice of $\tilde{n}$, and $u^\sigma(n) \in T(\sigma) \cap N$.
\eqref{eq:u-sig} induces the map $\bigwedge^p u^\sigma \colon \bigwedge^p N'' \to \bigwedge^p N$.
Since $\sigma \succ \theta(\tau_l)$, we have $\bigwedge^p u^\sigma \lb v \rb \in F^\bZ_p \lb \theta \lb \tau_l \rb \rb$ for any element $v \in \bigwedge^p N''$.
The equivalence class of $\bigwedge^p u^\sigma \lb v \rb$ in \eqref{eq:fp} does not depend on the choice of $\sigma \in \scrP_B(d)$.
We can check this as follows:
It suffices to show that for any $\rho \in \scrP_B(d-1)$ such that $\rho \succ \theta(\tau_l)$, the equivalence classes of $\bigwedge^p u^{\sigma_1} \lb v \rb$ and $\bigwedge^p u^{\sigma_2} \lb v \rb$ in \eqref{eq:fp} are the same for the polyhedra $\sigma_1, \sigma_2 \in \scrP_B(d)$ containing $\rho$.
We may also assume that $v$ is written as $v=n_1 \wedge \cdots \wedge n_p \in \bigwedge^p N''$ with $n_i \in T(\rho) \cap N$ $(i =2, \cdots, p)$.
It implies that for $i = 2, \cdots, p$, one has $\la m_j^{\sigma_1}-m_j^{\sigma_2}, n_i \ra =0$, and $u^{\sigma_1}(n_i)=u^{\sigma_2}(n_i)$.
Hence, we have 
\begin{align}\label{eq:vv}
\bigwedge^p u^{\sigma_1} \lb v \rb-\bigwedge^p u^{\sigma_2} \lb v \rb
&=\lb \sum_{j=1}^r \la m_j^{\sigma_1}-m_j^{\sigma_2}, n_1 \ra \cdot \lb \phi_{\xi, \Deltav_j} (v_0)+e_j \rb \rb \wedge 
\bigwedge_{i=2}^p u^{\sigma_1}(n_i) \\
&=\lb \sum_{j \in I(\rho)} \la m_j^{\sigma_1}-m_j^{\sigma_2}, n_1 \ra \cdot \lb \phi_{\xi, \Deltav_j} (v_0)+e_j \rb \rb \wedge 
\bigwedge_{i=2}^p u^{\sigma_1}(n_i).
\end{align}
For any $j \in I(\rho)$, there exists a polyhedron $\sigma^j \in \scrP_{X(f_1, \cdots, f_r)}(d)$ such that $\sigma^j \succ \rho$ and $j \in I_{\sigma^j}$.
Since $u^{\sigma_1}(n_i)=u^{\sigma_2}(n_i) \in T(\sigma_1) \cap T(\sigma_2)=T(\rho) \subset T(\sigma^j)$ for $i =2, \cdots, p$, we have 
\begin{align}
\la m_j^{\sigma_1}-m_j^{\sigma_2}, n_1 \ra \cdot \lb \phi_{\xi, \Deltav_j} (v_0)+e_j \rb \wedge 
\bigwedge_{i=2}^p u^{\sigma_1}(n_i)
\subset 
F_{p-1}^\bZ \lb \sigma^j \rb \wedge \bZ \lb \phi_{\xi, \Deltav_j} (v_0)+e_j \rb
\end{align}
for $j \in I(\rho)$.
Hence, we get $\bigwedge^p u^{\sigma_1} \lb v \rb-\bigwedge^p u^{\sigma_2} \lb v \rb=0$ in \eqref{eq:fp}.
It turned out that the map $\bigwedge^p u^{\sigma}$ from $\bigwedge^p N''$ to \eqref{eq:fp} does not depend on the choice of $\sigma \in \scrP_B(d)$.
Let $\zeta \colon \bigwedge^p N'' \to \scF_p^\bZ \lb U_{\delta^{-1} (x)} \rb$ denote the map.

Next, we check that the denominator of \eqref{eq:lp} is sent to $0$ in \eqref{eq:fp} by the map $\zeta$.
For any $\rho \in \scrP(d-1)$ such that $\rho \succ \tau_l$, and $i \in I(\rho)$, we have
\begin{align}\label{eq:xi}
\zeta \lb \lb \bigwedge^{p-1} \Lambda_\rho \rb \wedge T \lb \phi_{\xi, \Deltav_{i}} (\tau_{0}) \rb \rb=
\lb \bigwedge^{p-1} T \lb \rho \rb \rb \wedge T \lb \phi_{\xi, \Deltav_{i}} (\tau_{0}) \rb.
\end{align}
There exists a polyhedron $\sigma_0 \in \scrP_{X(f_1, \cdots, f_r)}(d)$ such that $\sigma_0 \succ \theta(\rho)$ and $i \in I_{\sigma_0}$.
Since $T \lb \sigma_0 \rb \supset T \lb \rho \rb$, \eqref{eq:xi} is contained in 
$\lb F_{p-1}^\bZ \lb \sigma_0 \rb \wedge \sum_{v \prec \tau_0} \bZ \lb \phi_{\xi, \Deltav_{i}}(v)+e_{i} \rb \rb \otimes_\bZ \bR$.
Since $\sigma_0 \succ \theta(\rho) \succ \theta \lb \tau_l \rb$, the denominator of \eqref{eq:lp} is sent to $0$ in \eqref{eq:fp} by the map $\zeta$.
It turned out that $\zeta$ induces a map from \eqref{eq:lp} to \eqref{eq:fp}, which will be denoted also by $\zeta$ in the following.

We check that the map $\zeta$ gives the inverse map.
It is obvious that we have $\pi_{v_0} \circ \zeta=\id$ by the definition of the map $\zeta$.
The equality $\zeta \circ \pi_{v_0} =\id$ can be checked as follows:
From \eqref{eq:Lamsig}, one can easily see that \eqref{eq:fp} is generated by $\lb \sum_{\substack{\sigma \succ \theta(\tau_l) \\ \sigma \in \scrP_B}} \bigwedge^p T(\sigma) \rb \cap \bigwedge^p N$.
For these elements, the map $\zeta \circ \pi_{v_0}$ is obviously identical.
Thus we get $\zeta \circ \pi_{v_0} =\id$, and see that the map $\zeta$ gives the inverse map.

We saw that the projection \eqref{eq:projv0} induces an isomorphism 
\begin{align}\label{eq:proj-ind}
\scF_p^\bZ \lb U_{\delta^{-1} (x)} \rb \to \Hom \lb \lb \iota_\ast \bigwedge^{p} \check{\Lambda} \rb \lb U_x \rb, \bZ \rb,
\end{align}
and the inclusion \eqref{eq:inclv0} induces an isomorphism between the stalks $\lb \iota_\ast \bigwedge^{p} \check{\Lambda}\rb_x$ and $\lb \delta_\ast \scF^p_\bZ \rb_x$.
We write the latter isomorphism as $g_x \colon \lb \iota_\ast \bigwedge^{p} \check{\Lambda}\rb_x \to \lb \delta_\ast \scF^p_\bZ \rb_x$.
We can also check that these maps do not depend on the choice of the vertex $v_0 \prec \tau_0$ as follows:
Let $v_0' \prec \tau_0$ be another vertex, and $N''':=\left. N \middle/ \oplus_{i=1}^r \bZ \lb \phi_{\xi, \Deltav_i} \lb v_0' \rb +e_i \rb \right.$ be the integral tangent space at $v_0'$.
Let further $\scF$ be \pref{eq:lp} (computed for the choice $v_0$), and $\scF'$ be the same one computed for the choice $v_0'$.
Take a facet $\sigma \in \scrP(d)$ such that $\sigma \succ \tau_l$.
By the parallel transport along a path in $\rint \lb \sigma \rb$, we get the identification $\bigwedge^p N'' \cong \bigwedge^p N'''$.
It induces the isomorphism $l \colon \scF \cong \scF' (\cong \Hom \lb \lb \iota_\ast \bigwedge^{p} \check{\Lambda} \rb \lb U_x \rb, \bZ \rb)$.
We can see from the definition of $\bigwedge^p u^\sigma$ that the map 
\begin{align}
\scF \xrightarrow{\zeta} \scF_p^\bZ \lb U_{\delta^{-1}(x)} \rb \xrightarrow{\pi_{v_0'}} \scF'
\end{align}
coincides with the isomorphism $l \colon \scF \to \scF'$.
This implies that the map \eqref{eq:proj-ind} does not depend on the choice of the vertex $v_0$.
Hence, the map $g_x$ also does not depend on the choice of the vertex $v_0$.

Next, we check that the maps $\lc g_x \rc_x$ induce a morphism 
$\iota_\ast \bigwedge^{p} \check{\Lambda} \to \delta_\ast \scF^p_\bZ$ of sheaves.
It suffices to check that for any point $x \in U_\xi$, any section 
$s \in \Gamma \lb U_x, \iota_\ast \bigwedge^{p} \check{\Lambda} \rb \cong \lb \iota_\ast \bigwedge^{p} \check{\Lambda} \rb_x$, 
and any point $y \in U_x$, one has 
\begin{align}\label{eq:gs}
g_y(s_y)=\lb g_x(s_x) \rb_y,
\end{align}
where $s_y, s_x$ are the values of $s$ at $y$ and $x$ respectively.
Here we regard the element $g_x(s_x) \in \lb \delta_\ast \scF^p_\bZ \rb_x=\Gamma \lb U_x, \delta_\ast \scF^p_\bZ \rb$ as a section of $\delta_\ast \scF^p_\bZ$ over $U_x$, and the right hand side denotes its value at $y$.
We will show \eqref{eq:gs}.
Since $U_x$ is a small neighborhood of $x$, there is a natural inclusion $\lb \iota_\ast \bigwedge^{p} \check{\Lambda} \rb_x \hookrightarrow \lb \iota_\ast \bigwedge^{p} \check{\Lambda} \rb_y$.
The value $s_y$ is the image of $s_x$ by this inclusion.
There is also a natural inclusion $\lb \delta_\ast \scF^p_\bZ \rb_x \hookrightarrow \lb \delta_\ast \scF^p_\bZ \rb_y$.
The value $\lb g_x(s_x) \rb_y$ is the image of $g_x(s_x)$ by this inclusion.
Assume $y \in \rint \lb \conv \lc a_{\tau_0'}, a_{\tau_1'}, \cdots, a_{\tau_k'} \rc \rb$, where $\tau_0' \prec \tau_1' \prec \cdots \prec \tau_k' \in \scrP$ and $k \geq 0$.
Then we have $\tau_0' \prec \tau_0$.
If we choose a common vertex $v_0 \prec \tau_0'$ when we think of the images of $s_x, s_y$ by the maps $g_x$, $g_y$, then it is obvious that the both sides of \eqref{eq:gs} are the image of $s_x$ by the inclusion \eqref{eq:inclv0}.
Hence, \eqref{eq:gs} holds.
We proved that we have an isomorphism $\iota_\ast \bigwedge^{p} \check{\Lambda} \to \delta_\ast \scF^p_\bZ$ of sheaves.

It is obvious that \eqref{eq:sh-gr} is an isomorphism of sheaves of graded rings since 
the inclusion $\iota_{v_0} \colon \bigoplus_{p=0}^d \bigwedge^p M'' \hookrightarrow  \bigoplus_{p=0}^d \bigwedge^p M$ inducing \eqref{eq:sh-gr} is a homomorphism of graded rings.

\subsection{Proof of \pref{th:local-cont}.3 and \pref{th:local-cont}.4}\label{sc:local-cont3}

Let $x$ be a point of $U_{\xi}$.
Assume $x \in \rint \lb \conv \lb \lc a_{\tau_0}, a_{\tau_1}, \cdots, a_{\tau_l} \rc \rb \rb$, where $\tau_0 \prec \tau_1 \prec \cdots \prec \tau_l \in \scrP$, $l \geq 0$, and $\xi \in \lc \tau_0, \cdots, \tau_l \rc$ again.
We want
\begin{align}\label{eq:wvani}
\lb R^q \delta_\ast \scW_p^Q \rb_x=H^q \lb \delta^{-1}(x), \left. \scW_p^Q \right|_{\delta^{-1}(x)} \rb=0\\
\lb R^q \delta_\ast \scF^p_Q \rb_x=H^q \lb \delta^{-1}(x), \left. \scF^p_Q \right|_{\delta^{-1}(x)} \rb=0 \label{eq:fvani}
\end{align}
for all integers $q \geq 1$.
By \eqref{eq:Wpx}, one can see that the point $x$ is the minimal stratum in $\delta^{-1} (x)$.
Hence, one can get \eqref{eq:wvani} by \pref{lm:acyclic}.

The rest of this subsection is devoted to show \eqref{eq:fvani}.
In the following, we suppose that $\delta$ is good in the sense of \pref{df:local-contraction}.
We set
\begin{align}
\scrI:= \lc I \subset \lc 1, \cdots, r \rc \relmid  \dim \lb \Delta_i \rb \geq 1 \ \mathrm{for\ all\ } i \in I \rc.
\end{align}

\begin{lemma}\label{lm:i}
For $I \subset \lc 1, \cdots, r \rc$, we consider the polyhedron $\tau \lb F_1, \cdots, F_r, C_1, \cdots, C_r \rb$ with
\begin{align}
F_i 
=
\left\{
\begin{array}{ll}
\Delta_i \times \lc -e_i^\ast \rc  & i \in I \\
\conv(A_i) & i \nin I \\
\end{array}
\right.
, \quad
C_i
=
\left\{
\begin{array}{ll}
\cone (\Deltav_i \times \lc e_i \rc)& i \in I \\
\lc 0 \rc & i \nin I.
\end{array}
\right.
\end{align}
It is non-empty if and only if $I \in \scrI$.
\end{lemma}
\begin{proof}
When $I \nin \scrI$, there exists some $i \in I$ such that $\Delta_i$ consists of a single point.
Since the tropical hypersurface defined by a tropical polynomial is defined to be the locus where more than one monomial attains the minimum, this polyhedron $\tau \lb F_1, \cdots, F_r, C_1, \cdots, C_r \rb$ is empty in this case.

Next, suppose $I \in \scrI$.
Let $\lc n_{i, j} \rc_{j \in J_i}$ be the set of vertices of $\Delta_i$.
Consider, for instance, the element
$
n_0:=\sum_{i \in I} \sum_{j \in J_i} \infty \lb n_{i, j} + e_i \rb.
$
From \pref{th:st-intersection} and the assumption that $\delta$ is good, we can see that this element $n_0$ is in $X(f_1, \cdots, f_r)$, and in $\tau \lb F_1, \cdots, F_r \rb$.
One can also straightforwardly check that it is also in $O_{C'} \lb \bT \rb$ with $C'=\sum_{i=1}^r C_i$.
Hence, one has $n_0 \in \tau \lb F_1, \cdots, F_r, C_1, \cdots, C_r \rb$, and $\tau \lb F_1, \cdots, F_r, C_1, \cdots, C_r \rb \neq \emptyset$.
\end{proof}

For $I \in \scrI$, we consider the open star $U_I$ of the polyhedron of \pref{lm:i}, i.e.,
\begin{align}
U_{I}:=\bigcup_{\tau \in \scrP_I} \rint (\tau),
\end{align}
where
\begin{align}
\scrP_{I}:=\lc \tau \lb F_1, \cdots, F_r, C_1, \cdots, C_r \rb \in \scrP_{X(f_1, \cdots, f_r)} \relmid 
\begin{array}{l}
0 \nin F_i \ \mathrm{for}\ i \in I, \\
C_i=\lc 0 \rc \ \mathrm{for}\  i \nin I\\
\end{array}
\rc.
\end{align}

\begin{lemma}\label{lm:cover}
The family of subsets $\lc U_I \rc_{I \in \scrI}$ is an open covering of $X \lb f_1, \cdots, f_r \rb$.
\end{lemma}
\begin{proof}
Let $\tau \lb F_1, \cdots, F_r, C_1, \cdots, C_r \rb \in \scrP_{X(f_1, \cdots, f_r)}$ be an arbitrary polyhedron.
We set 
\begin{align}
I_0:=\lc i \in \lc 1, \cdots r \rc \relmid C_i \neq \lc 0\rc \rc.
\end{align}
One can see from \eqref{eq:O-emb} that if $0 \in F_i$ for some $i \in I_0$, then $\tau \lb F_1, \cdots, F_r, C_1, \cdots, C_r \rb$ is empty.
Hence, we have $0 \nin F_i$ for all $i \in I_0$.
We can see $I_0 \in \scrI$.
The interior of $\tau \lb F_1, \cdots, F_r, C_1, \cdots, C_r \rb$ is contained in $U_{I_0}$.
Therefore, the family of subsets $\lc U_I \rc_{I \in \scrI}$ is an open covering of $X \lb f_1, \cdots, f_r \rb$.
\end{proof}

\begin{lemma}\label{lm:UIcap}
The following hold:
\begin{enumerate}
\item For $I_0, J_0, I_1, J_1 \in \scrI$ such that $I_1 \subset I_0 \subset J_0 \subset J_1$, one has
\begin{align}\label{eq:UIJ}
U_{I_1} \cap U_{J_1} \subset U_{I_0} \cap U_{J_0}. 
\end{align}
\item
Let $\lc I_i \rc_i$ be a subset of $\scrI$.
We set $I':= \bigcup_{i} I_i$ and $I'':=\bigcap_{i} I_i$.
Then one has
\begin{align}\label{eq:UIcap}
\bigcap_{i} U_{I_i}= U_{I'} \cap U_{I''}.
\end{align}
\end{enumerate}
\end{lemma}
\begin{proof}
For $i=0, 1$, the intersection $U_{I_i} \cap U_{J_i}$ is 
the open star of $\tau_i:=\tau \lb F_1, \cdots, F_r, C_1, \cdots, C_r \rb$ with
\begin{align}\label{eq:capFC}
F_i 
=
\left\{
\begin{array}{ll}
\Delta_i \times \lc -e_i^\ast \rc  & i \in J_i \\
\conv(A_i) &i \nin J_i \\
\end{array}
\right.
, \quad
C_i
=
\left\{
\begin{array}{ll}
\cone (\Deltav_i \times \lc e_i \rc)& i \in I_i \\
\lc 0 \rc & i \nin I_i.
\end{array}
\right.
\end{align}
Since $\tau_0 \prec \tau_1$, one has \eqref{eq:UIJ}.
The both sides of \eqref{eq:UIcap} are equal to the open star of $\tau \lb F_1, \cdots, F_r, C_1, \cdots, C_r \rb$ with
\begin{align}\label{eq:capFC}
F_i 
=
\left\{
\begin{array}{ll}
\Delta_i \times \lc -e_i^\ast \rc  & i \in \bigcup_i I_i \\
\conv(A_i) &i \nin \bigcup_i I_i \\
\end{array}
\right.
, \quad
C_i
=
\left\{
\begin{array}{ll}
\cone (\Deltav_i \times \lc e_i \rc)& i \in \bigcap_i I_i \\
\lc 0 \rc & i \nin \bigcap_i I_i.
\end{array}
\right.
\end{align}
Hence, one has \eqref{eq:UIcap}.
\end{proof}

We define $I_x \in \scrI$ by
\begin{align}
I_x:=\lc i \in \lc 1, \cdots, r \rc \relmid \scrM_i(x)\setminus \lc 0 \rc \mathrm{\ consists\ of\ more\ than\ one\ element.} \rc,
\end{align}
where $\scrM_i(x)$ is \eqref{eq:F_i} for $n=x$.
For $I \in \scrI$, we also set
\begin{align}
U_{I}^x:=\delta^{-1}(x) \cap U_I.
\end{align}
\begin{lemma}\label{lm:i2}
One has $U_{I}^x \neq \emptyset$ if and only if $I \subset I_x$.
\end{lemma}
\begin{proof}
When $I \not\subset I_x$, there exists some $i \in I$ such that $\scrM_i(x)\setminus \lc 0 \rc$ consists of a single point.
We can see that there is only one monomial that attains the minimum of $f_i$ on $U_I^x$.
Thus we have $U_{I}^x = \emptyset$.

Next, suppose $I \subset I_x$.
Consider, for instance, the element
$
x + \sum_{i \in I} \lb \phi_{\xi, \Deltav_i} \lb v \rb + e_i \rb,
$
where $v \prec \tau_0$ is some vertex.
Again from \pref{th:st-intersection} and the assumption that $\delta$ is good, we can see that this element is in $X(f_1, \cdots, f_r)$.
It is now clear that this is also in $U_I^x$.
Hence, $U_{I}^x \neq \emptyset$.
\end{proof}

We set
$\scrI_x:=\lc I \in \scrI \relmid  I \subset I_x \rc
=\lc I \in \lc 1, \cdots, r \rc \relmid  I \subset I_x \rc
$.

\begin{lemma}
The family of subsets $\lc U_I^x \rc_{I \in \scrI_x}$ is an acyclic covering of $\delta^{-1}(x)$ for the sheaf $\scF^p_Q$ $(p \geq 0, Q=\bZ, \bQ, \bR)$.
\end{lemma}
\begin{proof}
It is obvious from \pref{lm:cover} and \pref{lm:i2} that $\lc U_I^x \rc_{I \in \scrI_x}$ is an open covering of $\delta^{-1}(x)$.
We check that it is acyclic for $\scF^p_Q$.
Let $\lc I_i \rc_i$ be a subset of $\scrI_x$ such that the intersection $\bigcap_i U_{I_i}^x$ is non-empty.
As we saw in the proof of \pref{lm:UIcap}, the intersection $\bigcap_i U_{I_i}$ is the open star of $\tau \lb F_1, \cdots, F_r, C_1, \cdots, C_r \rb$ with \eqref{eq:capFC}.
One can see from this and \pref{lm:sttl} that any polyhedron in $\bigcap_i \scrP_{I_i}$ that intersects with $\delta^{-1} \lb x \rb$ contains the polyhedron $\tau \lb F_1', \cdots, F_r', C_1', \cdots, C_r' \rb$ with
\begin{align}
F_i'
=
\left\{
\begin{array}{ll}
\conv \lb \scrM_i(x) \setminus \lc 0 \rc \rb  & i \in \bigcup_i I_i \\
\conv \lb \scrM_i(x) \rb &i \nin \bigcup_i I_i \\
\end{array}
\right.
, \quad
C_i'
=
\left\{
\begin{array}{ll}
\cone (\phi_{\xi, \Deltav_i} (\tau_0) \times \lc e_i \rc)& i \in \bigcap_i I_i \\
\lc 0 \rc & i \nin \bigcap_i I_i
\end{array}
\right.
\end{align}
as its face.
Take a point $y \in \rint \lb \tau \lb F_1', \cdots, F_r', C_1', \cdots, C_r' \rb \rb \cap \delta^{-1} \lb x \rb$.
One can take a fundamental system of neighborhoods $\lc W_j \rc_j$ of the point $y$ in $\bigcap_i U_{I_i}^x$, and homeomorphisms $\lc \varphi_j \colon W_j \to \bigcap_i U_{I_i}^x \rc_j$ such that $\varphi_j^{-1} \scF^p_Q=\scF^p_Q$.
By the same argument as the one in \pref{lm:acyclic}, one can get $H^q \lb \bigcap_i U_{I_i}^x, \scF^p_Q \rb=0$ for all integers $q \geq 1$.
Hence, the covering $\lc U_I^x \rc_{I \in \scrI}$ is acyclic for the sheaf $\scF^p_Q$.
\end{proof}

In the following, we will compute the \v{C}ech cohomology group with respect to the covering $\lc U_I^x \rc_{I \in \scrI_x}$.
For an integer $p \geq 0$ and a subset $I \subset \lc 1, \cdots, r \rc$, we set
\begin{align}
G_p^Q(I)&:=\lb \sum_{\substack{\sigma \in \scrP_{I} \\ \sigma \succ \theta \lb \tau_l \rb}} 
\bigwedge^{p} T(\sigma) \rb \cap \bigwedge^p N_Q \\
\Phi_I&:=\sum_{i \in I} \vspan \lb \phi_{\xi, \Deltav_{i}} (\tau_0)+e_i \rb,
\end{align}
where $Q=\bZ, \bQ, \bR$ and $N_Q:=N \otimes_\bZ Q$.
Then similarly to \eqref{eq:fp}, we can write
\begin{align}
\scF_p^Q \lb U_I^x \rb
=
\left. 
G_p^Q(I)
\middle/
\lc \lb G_{p-1}^\bR(I) \wedge \Phi_I \rb \cap \bigwedge^p N_Q \rc
\right.
\end{align}
for $I \in \scrI_x$ and  $p \geq 1$.
We also set
\begin{align}
\scN_p \lb I', I\rb&:=G_{p-1}^\bR(I) \wedge \Phi_{I \setminus I'}
\end{align}
for an integer $p \geq 1$ and subsets $I' \subset I \subset \lc 1, \cdots, r \rc$, and
\begin{align}
H_p^Q(I_0, I_1, I_2):=
\sum_{I_1 \subset I \subset I_2} 
\lb \scN_p \lb I_0, I \rb
\cap \bigwedge^p N_Q \rb
\end{align}
for an integer $p \geq 1$ and subsets $I_0 \subset I_1 \subset I_2 \subset \lc 1, \cdots, r \rc$.
Here since
\begin{align}
G_{p-1}^\bR(I) \wedge \Phi_{I \setminus I_0} 
=G_{p-1}^\bR(I) \wedge \lb \sum_{i \in I \setminus I_0} \Phi_{\lc i \rc} \rb
\subset \sum_{i \in I \setminus I_0} G_{p-1}^\bR(I_1 \cup \lc i \rc) \wedge \Phi_{\lc i \rc}
\end{align}
for any $I_1 \subset I \subset I_2 \subset \lc 1, \cdots, r \rc$, we have
\begin{align}\label{eq:HGP}
H_p^\bR \lb I_0, I_1, I_2 \rb
=\sum_{I_1 \subset I \subset I_2} G_{p-1}^\bR(I) \wedge \Phi_{I \setminus I_0}
=\sum_{i \in I_2 \setminus I_0} G_{p-1}^\bR(I_1 \cup \lc i \rc) \wedge \Phi_{\lc i \rc}.
\end{align}

\begin{lemma}\label{lm:GH}
Let $I_0, I_1, I_2 \in \scrI_x$ be subsets of $\lc1, \cdots, r \rc$ such that $I_0 \subset I_1 \subset I_2$.
Let further $i_0 \in I_2 \setminus I_1$ be an element, and $p \geq 1$ be any integer.
Then the following hold:
\begin{enumerate}
\item One has
\begin{align}\label{eq:HpZ}
H_p^Q(I_0, I_1, I_2)=H_p^\bR(I_0, I_1, I_2) \cap \bigwedge^p N_Q
\end{align}
for $Q=\bQ$.
When $\delta$ is very good, \eqref{eq:HpZ} holds also for $Q=\bZ$.
\item When $\delta$ is very good, one has
\begin{align}\label{eq:GHN}
G_p^\bZ \lb I_1 \cup \lc i_0 \rc \rb
+
H_p^\bZ \lb I_0, I_1, I_2 \setminus \lc i_0 \rc \rb
=
\lc G_p^\bR \lb I_1 \cup \lc i_0 \rc \rb + H_p^\bR \lb I_0, I_1, I_2 \setminus \lc i_0 \rc \rb \rc \cap \bigwedge^p N.
\end{align}
\item One has
\begin{align}\label{eq:GH}
G_p^\bR \lb I_1 \cup \lc i_0 \rc \rb
\cap
H_p^\bR \lb I_0, I_1, I_2 \setminus \lc i_0 \rc \rb
=
H_p^\bR \lb I_0 \cup \lc i_0 \rc, I_1 \cup \lc i_0 \rc, I_2 \rb.
\end{align}
\end{enumerate}
\end{lemma}
\begin{proof}
First, we will check \eqref{eq:HpZ}.
It is obvious that we have $H_p^Q(I_0, I_1, I_2) \subset H_p^\bR(I_0, I_1, I_2) \cap \bigwedge^p N_Q$ for $p \geq 1$.
The opposite inclusion for $Q=\bQ$ can be checked as follows:
For every subset $I$ such that $I_1 \subset I \subset I_2$, we take a basis $\lc n_i^I \rc_i$ of the vector space $ \scN_p \lb I_0, I \rb \cap \bigwedge^p N_\bQ$ over $\bQ$.
Then this is also a basis of the vector space $\scN_p \lb I_0, I \rb$ over $\bR$.
The set of elements $\lc n_i^I \rc_{i, I}$ generates $H_p^\bR(I_0, I_1, I_2)$.
We take a subset $\lc n_i \rc_i$ of $\lc n_i^I \rc_{i, I}$ so that it forms a basis of $H_p^\bR(I_0, I_1, I_2)$.
Then we have
\begin{align}
H_p^\bR(I_0, I_1, I_2) \cap \bigwedge^p N_\bQ
=\lb \bigoplus_i \bR n_i \rb \cap \bigwedge^p N_\bQ
=\bigoplus_i \bQ n_i
\subset H_p^\bQ(I_0, I_1, I_2).
\end{align}
Hence, we get \eqref{eq:HpZ} for $Q=\bQ$.

Next, we suppose that $\delta$ is very good, and show 
$H_p^\bR(I_0, I_1, I_2) \cap \bigwedge^p N \subset H_p^\bZ(I_0, I_1, I_2)$.
By the assumption $\la m, n \ra=0$ for any $m \in \Delta_i$, $n \in \Deltav_j$ $(1 \leq i, j \leq r)$, we have
\begin{align}
\bigoplus_{i=1}^r  T \lb \Deltav_i \rb
\subset \lb \bigoplus_{i=1}^r T \lb \Delta_i \rb \rb^\perp.
\end{align}
Take sublattices $N_0, N_j' \subset N'$ $\lb 1 \leq j \leq r \rb$ so that we have
\begin{align}\label{eq:sublattice1}
N_0 \oplus \lc \lb \bigoplus_{i=1}^r T \lb \Deltav_i \rb \rb \cap N' \rc&=\lb \bigoplus_{i=1}^r T \lb \Delta_i \rb \rb^\perp \cap N', \\ \label{eq:sublattice2}
N_j' \oplus \lc \lb \bigoplus_{i=1}^r T \lb \Delta_i \rb \rb^\perp \cap N' \rc&= \lb \bigoplus_{i \neq j}^r T \lb \Delta_i \rb \rb^\perp \cap N'.
\end{align}
From the assumption $\lb \bigoplus_{i=1}^r T \lb \Delta_i \rb \rb \cap M'=\bigoplus_{i=1}^r T_\bZ \lb \Delta_i \rb$ and \eqref{eq:sublattice2}, one can see
\begin{align}
\lc \lb \bigoplus_{i=1}^r T \lb \Delta_i \rb \rb^\perp \cap N' \rc \oplus \lb \bigoplus_{j=1}^r N_j' \rb=N'.
\end{align}
By this, \eqref{eq:sublattice1}, and the assumption $\lb \bigoplus_{i=1}^r T \lb \Deltav_i \rb \rb \cap N'=\bigoplus_{i=1}^r T_\bZ \lb \Deltav_i \rb$, one can get
\begin{align}\label{eq:dsum}
N_0 \oplus \bigoplus_{i=1}^r \lb T_\bZ \lb \Deltav_i \rb \oplus N_i' \rb=N', \quad
N_0 \oplus \bigoplus_{i=1}^r \lb T_\bZ \lb \Deltav_i \rb \oplus N_i' \oplus \bZ e_i \rb=N.
\end{align}
We set $N_i:=T_\bZ \lb \Deltav_i \rb \oplus N_i' \oplus \bZ e_i$ $\lb 1 \leq i \leq r \rb$, $N_{i, \bR}:=N_i \otimes_\bZ \bR$ $\lb 0 \leq i \leq r \rb$, and $C_i:= \cone \lb \Deltav_i \times \lc e_i \rc \rb \subset N_{i, \bR}$ $\lb 1 \leq i \leq r \rb$.
Since $C=C_1 \times \cdots \times C_r$, one has
\begin{align}
X_C(\bT)=N_{0, \bR} \times X_{C_1} \lb \bT \rb \times \cdots \times X_{C_r} \lb \bT \rb,
\end{align}
where $X_{C_i} \lb \bT \rb$ denotes the tropical toric variety associated with the cone $C_i \subset N_{i, \bR}$.
Since the tropical polynomial $f_i \colon N_\bR \to \bR$ does not depend on the components $N_{j , \bR}$ $(j \neq i)$, it can be regarded as a function on $N_{i, \bR}$.
Let $X' \lb f_i \rb \subset X_{C_i} \lb \bT \rb$ denote the tropical hypersurface defined by $f_i \colon N_{i, \bR} \to \bR$.
Then we have
\begin{align}\label{eq:direct-prod}
X(f_1, \cdots, f_r)=N_{0, \bR} \times X' \lb f_1 \rb \times \cdots \times X' \lb f_r \rb.
\end{align}
Note that by \pref{lm:BX}, the stable intersection coincides with the set theoretic intersection in the current setup.
Every polyhedron $\sigma \in \scrP_{X(f_1, \cdots, f_r)}$ can be uniquely written as 
\begin{align}
\sigma = N_{0, \bR} \times \sigma_1 \times \cdots \times \sigma_r, 
\end{align}
where $\sigma_i$ is a polyhedron in the natural polyhedral structure $\scrP_{X'(f_i)}$ of $X'(f_i)$.
We write $\theta \lb \tau_l \rb= N_{0, \bR} \times \tilde{\sigma}_1 \times \cdots \times \tilde{\sigma}_r$ 
$(\tilde{\sigma}_i \in \scrP_{X'(f_i)})$.
For $1 \leq i \leq r$ and $j \geq 1$, we define $R_i^j \subset S_i^j \subset T_i^j \subset \bigwedge^{j} N_{i, \bR}$ by
\begin{align}
R_i^j&:=\lb \sum_{\substack{\sigma_i \succ \tilde{\sigma}_i \\ f_i<0\ \mathrm{on}\ \rint (\sigma_i)}} \bigwedge^{j-1} T \lb \sigma_i \rb \rb \wedge \Phi_{\lc i \rc} \\
S_i^j&:=\sum_{\substack{\sigma_i \succ \tilde{\sigma}_i \\ f_i<0\ \mathrm{on}\ \rint (\sigma_i)}} \bigwedge^{j} T \lb \sigma_i \rb \\
T_i^j&:=\sum_{\sigma_i \succ \tilde{\sigma}_i } \bigwedge^{j} T \lb \sigma_i \rb.
\end{align}
We also set $R_i^0 =0$ and $S_i^0 = T_i^0:=\bR$.
Then by using \eqref{eq:HGP}, we can write
\begin{align}\label{eq:HpR}
H_p^\bR \lb I_0, I_1, I_2 \rb
=\bigoplus_{\substack{(p_0, \cdots, p_r) \in \bZ_{\geq 0}^{r+1} \\ p_0+ \cdots +p_r=p}}
\sum_{i \in I_2 \setminus I_0}
\lb \bigwedge^{p_0} N_{0, \bR} \rb 
\wedge
R_i^{p_i}
\wedge 
\lb \bigwedge_{j \in I_1 \setminus \lc i \rc} S_j^{p_j} \rb
\wedge 
\lb 
\bigwedge_{j \in \lc 1, \cdots, r \rc \setminus \lb I_1 \cup \lc i \rc \rb} T_j^{p_j} \rb.
\end{align}
We set
\begin{align}
N_i^{(p_0, \cdots, p_r)}:=
\lb \bigwedge^{p_0} N_{0, \bR} \rb 
\wedge
R_i^{p_i}
\wedge 
\lb \bigwedge_{j \in I_1 \setminus \lc i \rc} S_j^{p_j} \rb
\wedge 
\lb 
\bigwedge_{j \in \lc 1, \cdots, r \rc \setminus \lb I_1 \cup \lc i \rc \rb} T_j^{p_j} \rb.
\end{align}
For every $i \in \lc1, \cdots, r\rc$ and $p_i \in \bZ_{\geq 0}$, take a basis $\lc n_{j}^{p_i} \rc_j$ of 
$T_i^{p_i} \cap \bigwedge^{p_i} N_i$ so that there exist its two subsets that form bases of 
$R_i^{p_i} \cap \bigwedge^{p_i} N_i$ and $S_i^{p_i} \cap \bigwedge^{p_i} N_i$ respectively.
We also choose a basis $\lc n_{j}^{p_0} \rc_j$ of $\bigwedge^{p_0} N_0$.
Consider the set of elements
\begin{align}
\scrA:=\lc \bigwedge_{i=0}^r n_{j_i}^{p_i} \relmid (p_0, \cdots, p_r) \in \bZ_{\geq 0}^{r+1}, p_0+ \cdots +p_r=p
\rc.
\end{align}
From \eqref{eq:dsum}, we can see that this forms a basis of the sublattice in $\bigwedge^p N$ generated by $\scrA$.
There also exists its subset $\scrB \subset \scrA$ that forms a basis of the intersection of \eqref{eq:HpR} and $\bigwedge^p N$.
By the construction of $\scrA$, every element of $\scrB$ is contained in $N_i^{(p_0, \cdots, p_r)} \cap \bigwedge^p N$ for some $(p_0, \cdots, p_r) \in \bZ_{\geq 0}^{r+1} $ and $i \in I_2 \setminus I_0$.
Hence, we have
\begin{align}
\begin{split}
H_p^\bR \lb I_0, I_1, I_2 \rb \cap \bigwedge^p N
=\bigoplus_{n \in \scrB} \bZ n
\subset
\sum_{i \in I_2 \setminus I_0}
\lc \lb
\bigoplus_{\substack{(p_0, \cdots, p_r) \in \bZ_{\geq 0}^{r+1} \\ p_0+ \cdots +p_r=p}} N_i^{(p_0, \cdots, p_r)}
\rb
\cap \bigwedge^p N
\rc \\
= \sum_{i \in I_2 \setminus I_0} \lc \lb G_{p-1}^\bR(I_1 \cup \lc i \rc) \wedge \Phi_{\lc i \rc} \rb \cap \bigwedge^p N \rc
\subset H_p^\bZ(I_0, I_1, I_2).
\end{split}
\end{align}
Thus we obtain \eqref{eq:HpZ} for $Q=\bZ$.

\eqref{eq:GHN} can also be checked as follows:
 $G_p^\bR \lb I_1 \cup \lc i_0 \rc \rb$ is equal to 
\begin{align}\label{eq:GpR}
\bigoplus_{\substack{(p_0, \cdots, p_r) \in \bZ_{\geq 0}^{r+1} \\ p_0+ \cdots +p_r=p}}
\lb \bigwedge^{p_0} N_{0, \bR} \rb 
\wedge 
\lb \bigwedge_{j \in I_1 \cup \lc i_0 \rc} S_j^{p_j} \rb
\wedge 
\lb \bigwedge_{j \in \lc 1, \cdots, r \rc \setminus \lb I_1 \cup \lc i_0 \rc \rb} T_j^{p_j} \rb,
\end{align}
and $H_p^\bR \lb I_0, I_1, I_2 \setminus \lc i_0 \rc \rb$ is equal to
\begin{align}\label{eq:HpR2}
\bigoplus_{\substack{(p_0, \cdots, p_r) \in \bZ_{\geq 0}^{r+1} \\ p_0+ \cdots +p_r=p}}
\sum_{i \in I_2 \setminus \lb I_0 \cup \lc i_0 \rc \rb}
\lb \bigwedge^{p_0} N_{0, \bR} \rb 
\wedge
R_i^{p_i}
\wedge 
\lb \bigwedge_{j \in I_1 \setminus \lc i \rc} S_j^{p_j} \rb
\wedge 
\lb 
\bigwedge_{j \in \lc 1, \cdots, r \rc \setminus \lb I_1 \cup \lc i \rc \rb} T_j^{p_j} \rb.
\end{align}
Hence, there also exist subsets $\scrC, \scrD \subset \scrA$ that form bases of 
$G_p^\bR \lb I_1 \cup \lc i_0 \rc \rb \cap \bigwedge^p N$ and $H_p^\bR \lb I_0, I_1, I_2 \setminus \lc i_0 \rc \rb \cap \bigwedge^p N$ respectively.
Since $G_p^\bZ \lb I_1 \cup \lc i_0 \rc \rb=G_p^\bR \lb I_1 \cup \lc i_0 \rc \rb \cap \bigwedge^p N$ and $H_p^\bZ \lb I_0, I_1, I_2 \setminus \lc i_0 \rc \rb=H_p^\bR \lb I_0, I_1, I_2 \setminus \lc i_0 \rc \rb \cap \bigwedge^p N$, one gets 
\begin{align}
\lc G_p^\bR \lb I_1 \cup \lc i_0 \rc \rb + H_p^\bR \lb I_0, I_1, I_2 \setminus \lc i_0 \rc \rb \rc 
&\cap \bigwedge^p N
=
\lb \bigoplus_{n \in \scrC \cup \scrD} \bR n \rb \cap \bigwedge^p N \\
=\bigoplus_{n \in \scrC \cup \scrD} \bZ n
=
G_p^\bZ \lb I_1 \cup \lc i_0 \rc \rb
&+
H_p^\bZ \lb I_0, I_1, I_2 \setminus \lc i_0 \rc \rb.
\end{align}
We obtained \eqref{eq:GHN}.

Lastly, we check \pref{eq:GH}.
We do not assume that $\delta$ is very good.
We take sublattices $N_0, N_j' \subset N'$ $\lb 1 \leq j \leq r \rb$ again as we did in \eqref{eq:sublattice1} and \eqref{eq:sublattice2}.
When $\delta$ is not very good, we do not have \eqref{eq:dsum} in general.
However, we still have 
$\bigoplus_{i=0}^r N_{i, \bR} =N_\bR$ and 
\begin{align}
X_C(\bT)&=N_{0, \bR} \times X_{C_1} \lb \bT \rb \times \cdots \times X_{C_r} \lb \bT \rb \\
X(f_1, \cdots, f_r)&=N_{0, \bR} \times X' \lb f_1 \rb \times \cdots \times X' \lb f_r \rb,
\end{align}
although the integral structures of both sides of these are different.
We write $\theta \lb \tau_l \rb= N_{0, \bR} \times \tilde{\sigma}_1 \times \cdots \times \tilde{\sigma}_r$ 
$(\tilde{\sigma}_i \in \scrP_{X'(f_i)})$ again.
Then $G_p^\bR \lb I_1 \cup \lc i_0 \rc \rb$ and $H_p^\bR \lb I_0, I_1, I_2 \setminus \lc i_0 \rc \rb$ are written as \eqref{eq:GpR} and \eqref{eq:HpR2}.
The intersection of these two is
\begin{align}
\bigoplus_{\substack{(p_0, \cdots, p_r) \in \bZ_{\geq 0}^{r+1} \\ p_0+ \cdots +p_r=p}}
\sum_{i \in I_2 \setminus \lb I_0 \cup \lc i_0 \rc \rb}
\lb \bigwedge^{p_0} N_{0, \bR} \rb 
\wedge
R_i^{p_i}
\wedge 
\lb \bigwedge_{j \in \lb I_1 \cup \lc i_0 \rc \rb \setminus \lc i \rc} S_j^{p_j} \rb
\wedge 
\lb 
\bigwedge_{j \in \lc 1, \cdots, r \rc \setminus \lb I_1 \cup \lc i, i_0 \rc \rb} T_j^{p_j} \rb,
\end{align}
which is equal to $H_p^\bR \lb I_0 \cup \lc i_0 \rc, I_1 \cup \lc i_0 \rc, I_2 \rb$.
We obtained \eqref{eq:GH}.
\end{proof}

\begin{lemma}\label{lm:injective}
For $I_0, I_1, I_2 \in \scrI_x$, $i_0 \in I_2 \setminus I_1$ of \pref{lm:GH}, and any integer $p \geq 1$, one has
\begin{align}
G_p^Q \lb I_1 \cup \lc i_0 \rc \rb
\cap
H_p^Q \lb I_0, I_1, I_2 \setminus \lc i_0 \rc \rb
=
H_p^Q \lb I_0 \cup \lc i_0 \rc, I_1 \cup \lc i_0 \rc, I_2 \rb
\end{align}
for $Q=\bQ$.
When $\delta$ is very good, this holds also for $Q=\bZ$.
\end{lemma}
\begin{proof}
In either case, by the first and third statements of \pref{lm:GH}, one can obtain
\begin{align}
G_p^Q \lb I_1 \cup \lc i_0 \rc \rb
\cap
H_p^Q \lb I_0, I_1, I_2 \setminus \lc i_0 \rc \rb
&=G_p^\bR \lb I_1 \cup \lc i_0 \rc \rb
\cap
H_p^\bR \lb I_0, I_1, I_2 \setminus \lc i_0 \rc \rb
\cap \bigwedge^p N_Q \\
&=H_p^\bR \lb I_0 \cup \lc i_0 \rc, I_1 \cup \lc i_0 \rc, I_2 \rb \cap \bigwedge^p N_Q \\
&=H_p^Q \lb I_0 \cup \lc i_0 \rc, I_1 \cup \lc i_0 \rc, I_2 \rb.
\end{align}
\end{proof}

\begin{lemma}\label{lm:cosec}
For any integer $p \geq 1$ and $I_0, I_1, I_2 \in \scrI_x$ such that $I_0 \subset I_1 \subset I_2$, one has 
\begin{align}\label{eq:cosections}
\scF_p^Q \lb \bigcup_{I_1 \subset I \subset I_2} U_I^x \cap U_{I \setminus I_0}^x \rb
=
\left. G_p^Q(I_1)
\middle/
H_p^Q(I_0, I_1, I_2)
\right.
\end{align}
for $Q=\bQ$ when $\delta$ is good, and for $Q=\bZ$ when $\delta$ is very good.
\end{lemma}
\begin{proof}
We check the lemma by induction on $|I_2 \setminus I_1|$.
When $|I_2 \setminus I_1|=0$, we have $I_1=I_2$, and
\begin{align}
\scF_p^Q \lb U_{I_1}^x \cap U_{I_1 \setminus I_0}^x \rb
&=
\left. G_p^Q(I_1)
\middle/
\lc \lb G_{p-1}^\bR(I_1) \wedge \Phi_{I_1 \setminus I_0} \rb \cap \bigwedge^p N_Q \rc
\right. \\
&=
\left. G_{p}^Q(I_1)
\middle/
\lb \scN_p \lb I_0, I_1 \rb \cap \bigwedge^p N_Q \rb
\right. \\
&=
\left. G_{p}^Q(I_1)
\middle/
H_p^Q(I_0, I_1, I_1)
\right. .
\end{align}
Hence, \eqref{eq:cosections} holds in this case.
Assume that the statement holds when $|I_2 \setminus I_1| \leq l$ for some integer $l \geq 0$.
We show that it also holds when $|I_2 \setminus I_1|=l+1$.
Suppose $|I_2 \setminus I_1|=l+1$ and take $i_0 \in I_2 \setminus I_1$.
We have
\begin{align}\label{eq:oset}
\bigcup_{I_1 \subset I \subset I_2} U_I^x \cap U_{I \setminus I_0}^x
=\lb \bigcup_{I_1 \subset J \subset I_2 \setminus \lc i_0 \rc} U_J^x \cap U_{J \setminus I_0}^x \rb 
\cup \lb \bigcup_{I_1 \cup \lc i_0 \rc \subset I \subset I_2} U_I^x \cap U_{I \setminus I_0}^x \rb.
\end{align}
We set
\begin{align}\label{eq:a1a2}
A_1:=\bigcup_{I_1 \subset J \subset I_2 \setminus \lc i_0 \rc} U_J^x \cap U_{J \setminus I_0}^x, \quad
A_2:=\bigcup_{I_1 \cup \lc i_0 \rc \subset I \subset I_2} U_I^x \cap U_{I \setminus I_0}^x.
\end{align}
By \pref{lm:UIcap}.2, we have
\begin{align}
A_1 \cap A_2
&=\bigcup_{I_1 \subset J \subset I_2 \setminus \lc i_0 \rc} \bigcup_{I_1 \cup \lc i_0 \rc \subset I \subset I_2} U_J^x \cap U_{J \setminus I_0}^x \cap U_I^x \cap U_{I \setminus I_0}^x \\ \label{eq:a12}
&=\bigcup_{I_1 \subset J \subset I_2 \setminus \lc i_0 \rc} \bigcup_{I_1 \cup \lc i_0 \rc \subset I \subset I_2} U_{I \cup J}^x \cap U_{\lb I \cap J \rb \setminus I_0}^x.
\end{align}
Since $I \subset I \cup J$ and $I \setminus (I_0 \cup \lc i_0 \rc) \supset \lb I \cap J \rb \setminus I_0$, we can see from \pref{lm:UIcap}.1 that \eqref{eq:a12} is contained in 
\begin{align}\label{eq:a123}
\bigcup_{I_1 \cup \lc i_0 \rc \subset I \subset I_2} U_I^x \cap U_{I \setminus (I_0 \cup \lc i_0 \rc)}^x.
\end{align}
When $J=I \setminus \lc i_0 \rc$, one has $U_{I \cup J}^x \cap U_{\lb I \cap J \rb \setminus I_0}^x=U_I^x \cap U_{I \setminus (I_0 \cup \lc i_0 \rc)}^x$.
Hence, we can see that  $A_1 \cap A_2$ is equal to \eqref{eq:a123}.
By the induction hypothesis, one has
\begin{align}\label{eq:fa12}
\scF_p^Q \lb A_1 \cap A_2\rb&=\left. G_p^Q \lb I_1 \cup \lc i_0 \rc \rb \middle/ H_p^Q \lb I_0 \cup \lc i_0 \rc, I_1 \cup \lc i_0 \rc, I_2 \rb \right. \\ \label{eq:fa1}
\scF_p^Q \lb A_1 \rb&=\left. G_p^Q \lb I_1 \rb \middle/ H_p^Q \lb I_0, I_1, I_2 \setminus \lc i_0 \rc \rb \right. \\
\scF_p^Q \lb A_2 \rb&=\left. G_p^Q \lb I_1 \cup \lc i_0 \rc \rb \middle/ H_p^Q \lb I_0, I_1 \cup \lc i_0 \rc, I_2 \rb \right. .
\end{align}
Since $H_p^Q \lb I_0 \cup \lc i_0 \rc, I_1 \cup \lc i_0 \rc, I_2 \rb \subset H_p^Q \lb I_0, I_1 \cup \lc i_0 \rc, I_2 \rb$, the extension map
$
\scF_p^Q \lb A_1 \cap A_2 \rb \to \scF_p^Q \lb A_2 \rb
$
is a quotient map.
On the other hand, it turns out by \pref{lm:injective} that the extension map
$
\scF_p^Q \lb A_1 \cap A_2 \rb \to \scF_p^Q \lb A_1 \rb
$
is injective.
Hence, we can obtain
\begin{align}
\scF_p^Q \lb \bigcup_{I_1 \subset I \subset I_2} U_I^x \cap U_{I \setminus I_0}^x \rb
&=\scF_p^Q \lb A_1 \cup A_2 \rb \\
&=
\left. G_p^Q(I_1)
\middle/ \lb H_p^Q(I_0, I_1, I_2 \setminus \lc i_0 \rc)+H_p^Q(I_0, I_1 \cup \lc i_0 \rc, I_2)\rb \right. \\
&=\left. G_p^Q(I_1)
\middle/
H_p^Q(I_0, I_1, I_2)
\right. .
\end{align}
Hence, \eqref{eq:cosections} holds when $|I_2 \setminus I_1|=l+1$.
We obtained the lemma.
\end{proof}

\begin{lemma}\label{lm:van}
Let $p, q \geq 1$ be integers, and $I_0, I_1, I_2 \in \scrI_x$ be subsets such that $I_0 \subset I_1 \subset I_2$.
One has
\begin{align}\label{eq:uicoh0}
H^q \lb \bigcup_{I_1 \subset I \subset I_2} U_I^x \cap U_{I \setminus I_0}^x,  \scF^p_Q \rb=0
\end{align}
for $Q=\bQ$ when $\delta$ is good, and for $Q=\bZ$ when $\delta$ is very good.
\end{lemma}
\begin{proof}
We show the statement again by induction on $|I_2 \setminus I_1|$.
When $|I_2  \setminus I_1|=0$, we have $I_1=I_2$, and \eqref{eq:uicoh0} is obvious for any $I_0 \subset I_1$.
Assuming that \eqref{eq:uicoh0} holds when $|I_2 \setminus I_1|=l$, we show that \eqref{eq:uicoh0} holds also when $|I_2 \setminus I_1|=l+1$.
Suppose $|I_1|=k$ and $|I_2 \setminus I_1|=l+1$.
Take an element $i_0 \in I_2 \setminus I_1$.
For the subsets $A_1, A_2$ which we considered in \eqref{eq:a1a2}, by the induction hypothesis, we have 
\begin{align}
H^q \lb A_1, \scF^p_Q \rb
=H^q \lb A_2, \scF^p_Q \rb
=H^q \lb A_1 \cap A_2, \scF^p_Q \rb 
=0
\end{align}
for any integer $q \geq 1$.
Therefore, the two open sets $A_1, A_2$ form an acyclic covering of \eqref{eq:oset}.
It is obvious that the cohomology groups $H^{\geq 2} \lb \scF_Q^p \rb$ of \eqref{eq:oset} are $0$.
In order to prove that the cohomology group $H^{1} \lb \scF_Q^p \rb$ of \eqref{eq:oset} is also $0$, we show that the restriction map
\begin{align}\label{eq:restr}
\scF^p_Q \lb A_1 \rb \to
\scF^p_Q \lb A_1 \cap A_2 \rb
\end{align}
is surjective.
It suffices to show that its dual
$
\scF_p^Q \lb A_1 \cap A_2 \rb \to
\scF_p^Q \lb A_1 \rb
$
is an injection (resp. a primitive embedding) for $Q=\bQ$ (resp. $Q=\bZ$).
It turns out by \pref{lm:cosec} that $\scF_p^Q \lb A_1 \cap A_2 \rb$ and $\scF_p^Q \lb A_1 \rb$ are given by \eqref{eq:fa12} and \eqref{eq:fa1} respectively.
It follows from \pref{lm:injective} that 
$
\scF_p^Q \lb A_1 \cap A_2 \rb \to
\scF_p^Q \lb A_1 \rb
$
is injective.
For $Q=\bZ$, it also follows from \pref{lm:GH}.2 that 
$
\scF_p^Q \lb A_1 \cap A_2 \rb \to
\scF_p^Q \lb A_1 \rb
$
is a primitive embedding.
Thus the map \eqref{eq:restr} is surjective, and the cohomology group $H^{1} \lb \scF_Q^p \rb$ of \eqref{eq:oset} is also $0$.
Hence, \eqref{eq:uicoh0} holds when $|I_2 \setminus I_1|=l+1$.
We obtained the lemma.
\end{proof}

By setting $I_0=I_1=\emptyset, I_2=I_x$ in \pref{lm:van}, we obtain \eqref{eq:fvani} for all integers $q \geq 1$.

\subsection{Proof of \pref{th:local-cont}.1}\label{sc:local-cont1}

Let $x$ be a point in $U_{\xi}$.
By \cite[Proposition 1.29]{MR2213573}, the radiance obstruction of a small neighborhood of $x$ in $U_{\xi}$ is trivial.
Since the radiance obstruction is the extension class of \eqref{eq:exaff2}, it turns out that the sheaf $\iota_\ast \mathrm{Aff}_{U_{\xi, 0}}$ is locally isomorphic to $\bR \oplus \iota_\ast \check{\Lambda}$ around the point $x$.
On the other hand, the wave cohomology group $H^1 \lb \scW_1^\bR \rb$ of the inverse image by $\delta$ of a small neighborhood of $x$ in $U_{\xi}$ is trivial by \eqref{eq:wvani}.
Hence, its eigenwave is also trivial.
By \pref{pr:eigext}, it turns out that the sheaf $\delta_\ast \mathrm{Aff}_X$ is also locally isomorphic to $\bR \oplus \delta_\ast \scF^1_\bZ$ around $x$.

Assume $x \in \rint \lb \conv \lb \lc a_{\tau_0}, a_{\tau_1}, \cdots, a_{\tau_l} \rc \rb \rb$, where $\tau_0 \prec \tau_1 \prec \cdots \prec \tau_l \in \scrP$ and $l \geq 0$.
Take a vertex $v_0 \prec \tau_0$, and set $M'':=\Hom \lb N'', \bZ \rb=\bigcap_{i=1}^r \lb \phi_{\xi, \Deltav_i} \lb v_0 \rb +e_i \rb^\perp$ again.
Around the point $x$, the pullback of functions induces the map 
\begin{align}
\bR \oplus \iota_\ast \check{\Lambda} \to \bR \oplus \scF^1_\bZ, \quad (r, m) \mapsto  \lb r, \iota_{v_0} (m) \rb,
\end{align}
where $\iota_{v_0} \colon M'' \hookrightarrow M$ is the inclusion.
As we showed in the proof of \pref{th:local-cont}.2, this is an isomorphism.
Hence, we have $\iota_\ast \mathrm{Aff}_{U_{\xi, 0}} \cong \delta_\ast \mathrm{Aff}_X$ via the pullback of functions.

\section{Proof of \pref{th:reconstruction}}\label{sc:reconstruction}

Let $B$ be a quasi-simple IAMS of dimension $d$, and $x \in B$ be an arbitrary point.
Since $B$ is quasi-simple, there exists a neighborhood $U_x$ of $x$ that is isomorphic to a quasi-simple local model of IAMS $U_{\xi}$.
We will prove \pref{th:reconstruction} by constructing a local model of tropical contractions to $U_\xi$.

Let $(U, \scrP)$ be the IAMS from which $U_\xi$ is constructed in \pref{sc:local-iams}.
Let further $\Omega_1, \cdots, \Omega_r \subset \lc \omega \in \scrP(1) \relmid \omega \prec \xi \rc$ and $R_1, \cdots, R_r \subset \lc \scrP(d-1) \relmid \rho \succ \xi \rc$ be the disjoint subsets appearing in \pref{df:quasi}.
For $\omega \in \scrP(1)$ and $\rho \in \scrP(d-1)$, the constant $\kappa_{\omega, \rho}$ is $0$ unless $\omega \in \Omega_i$ and $\rho \in R_i$ for some common $i \in \lc 1, \cdots, r \rc$.
The elements $m_{\omega, \rho}:=\kappa_{\omega, \rho} \check{d}_\rho, n_{\omega, \rho}:=\kappa_{\omega, \rho} d_\omega$ are independent of $\omega \in \Omega_i$ and $\rho \in R_i$ respectively (cf.~\cite[Remark 1.61.1]{MR2213573}).
Hence, the constant $\kappa_{\omega, \rho}$ is independent of $\omega \in \Omega_i$ and $\rho \in R_i$.
We write the constant as $\kappa_i$ $(1 \leq i \leq r)$.
We choose positive integers $\lc \alpha_i \rc_{1 \leq i \leq r}$, $\lc \beta_i \rc_{1 \leq i \leq r}$ so that $\kappa_i=\alpha_i \beta_i$ for any $i \in \lc 1, \cdots, r \rc$.
For instance, one can always choose $\alpha_i =1, \beta_i =\kappa_i$ or $\alpha_i =\kappa_i, \beta_i =1$.

Let $v_1, v_2$ be vertices of $\xi$, and $\sigma_1, \sigma_2 \in \scrP(d)$ be maximal-dimensional polytopes.
The elements $n^{\rho}_{v_1, v_2} \in \Lambda_ \xi$ and $m_\omega^{\sigma_1, \sigma_2} \in \Lambda_\xi^{\perp}$ appearing in \pref{eq:monodromy1} and \pref{eq:monodromy2} are also independent of $\rho \in R_i$ and $\omega \in \Omega_i$ respectively again by \cite[Remark 1.61.1]{MR2213573}.
We set
\begin{align}
n_{i}^{v_1, v_2}:= \frac{1}{\alpha_i} n^{\rho}_{v_1, v_2}, \quad
m_i^{\sigma_1, \sigma_2}:=\frac{1}{\beta_i} m_\omega^{\sigma_1, \sigma_2},
\end{align}
where $\rho \in R_i$ and $\omega \in \Omega_i$.
We fix a vertex $v_0$ of $\xi$ and a maximal-dimensional polytope $\sigma_0 \in \scrP(d)$, and define
\begin{align}\label{eq:delv}
\Deltav_i&:= \mathrm{conv} \lc n_i^{v_0, v} \relmid v \in \scrP(0)\ \mathrm{s.t.}\ v \prec \xi \rc \subset \Lambda_ \xi \\ \label{eq:del}
\Delta_i&:= \mathrm{conv} \lc m_i^{\sigma_0, \sigma} \relmid \sigma \in \scrP(d)
\rc \subset \Lambda_\xi^{\perp}.
\end{align}
Since the monodromy polytopes $\Deltav_i(\xi), \Delta_i(\xi)$ are the convex hulls of 
$n^\rho_{v_0, v}=\alpha_i n_i^{v_0, v}$ and $m_\omega^{\sigma_0, \sigma}=\beta_i m_i^{\sigma_0, \sigma}$ respectively, one has $\Deltav_i(\xi)=\alpha_i \cdot \Deltav_i$ and $\Delta_i(\xi)=\beta_i \cdot \Delta_i$.
As mentioned in \pref{sc:iass}, the normal fan of $\xi$ is a refinement of the normal fan of $\Deltav_i(\xi)=\alpha_i \cdot \Deltav_i$. 
The map $\phi_{\xi, \Deltav_i}$ of \eqref{eq:phi} for $\xi, \Deltav_i$ is given by
\begin{align}\label{eq:phin1}
\phi_{\xi, \Deltav_i} \colon \scrP_\xi \to \scrP_{\Deltav_i}, \quad \tau \mapsto 
\mathrm{conv} \lc n_i^{v_0, v} \relmid v \in \scrP(0)\ \mathrm{s.t.}\ v \prec \tau \rc.
\end{align}

Let $N'$ denote the integral tangent space at $v_0$.
We identify the integral tangent space of $\rint(\sigma_0)$ with $N'$ via the parallel transport in the chart of the fan structure at $v_0$.
Let further $e_1, \cdots, e_r$ be the standard basis of $\bZ^r$, and $e_1^\ast, \cdots, e_r^\ast$ be the dual basis of $\lb \bZ^r \rb^\ast$.
We set
\begin{align}
N:=N' \oplus \lb \oplus_{i=1}^r \bZ e_i \rb, \quad M:=M' \oplus \lb \oplus_{i=1}^r \bZ e_i^\ast \rb,
\end{align}
where $M':=\Hom (N', \bZ)$.
The polytopes $\Deltav_i$ and $\Delta_i$ can be thought to sit in $N_\bR':=N' \otimes_\bZ \bR$ and $M_\bR':=M' \otimes_\bZ \bR$ respectively.
We consider the stable intersection of tropical hypersurfaces associated with the polytopes $\lc \Delta_i \rc_{i=1}^r$ and $\lc \Deltav_i  \rc_{i=1}^r$ as we did in the second paragraph of \pref{sc:intro}.
Namely, let $C$ be the cone in $N_\bR$ defined by \eqref{eq:cone} and $f_i \colon N_\bR \to \bR$ $(1 \leq i \leq r)$ be the tropical polynomials defined by \eqref{eq:polynomial} with $A_i:=\lc 0 \rc \cup \lb M \cap \lb \Delta_i \times \lc -e_i^\ast \rc \rb \rb$.
We consider the stable intersection \eqref{eq:intersection} and its closure $X(f_1, \cdots, f_r)$ in the tropical toric variety $X_{C}(\bT) \supset N_\bR$ associated with the cone $C$.
We consider the subset $B \subset N_\bR$ defined by \eqref{eq:B} too.

By the fan structure at $v_0$, the space $U_{\xi}$ can be embedded into $N_\bR' \cong N_\bR / \oplus_{i=1}^r \bR e_i$.
In the following, we think of the space $U_{\xi}$ and every polytope $\tau \in \scrP$ also as subsets of $B$ by 
\begin{align}
U_{\xi} \cong \lb \left. \pi_{C_{v_0}} \right|_B \rb^{-1} \lb U_{\xi} \rb \subset B, \quad 
\tau \cong \lb \left. \pi_{C_{v_0}} \right|_B \rb^{-1} \lb \tau \rb \subset B
\end{align}
respectively, where $C_{v_0} \prec C$ is the cone of \eqref{eq:ctau} with $\tau=v_0$, and the map $\left. \pi_{C_{v_0}} \right|_B$ is the composition
\begin{align}
B \hookrightarrow N_\bR \xrightarrow{\left. \pi_{C_{v_0}} \right|_{N_\bR}} 
\left. N_\bR \middle/ \vspan \lb \bigcup_{i=1}^r \phi_{\xi, \Deltav_i} \lb v_0 \rb \times \lc e_i \rc \rb \right. 
=N_\bR / \oplus_{i=1}^r \bR e_i
\cong N_\bR'.
\end{align}

\begin{lemma}\label{lm:subdiv}
Let $\sigma \in \scrP(d)$ be a maximal-dimensional polytope.
When we think of $\sigma$ as a subset of $B$ as above, one has
\begin{align}
\sigma \subset \lc n \in B \relmid \la m_i^{\sigma_0, \sigma} - e_i^\ast, n \ra =0, 1 \leq i \leq r \rc.
\end{align}
\end{lemma}
\begin{proof}
Take an edge $\omega \in \Omega_i$, and consider the fan structure $S_\omega \colon \mathrm{St} \lb \omega \rb \to N'_\bR / \Lambda_\omega$ along $\omega$ induced by the fan structure along a vertex.
The collection of sets
\begin{align}
\Sigma_\omega:=\lc K_\sigma :=\bR_{\geq 0} \cdot S_\omega \lb \sigma \rb \relmid \sigma \succ \omega \rc
\end{align}
is a complete rational fan in $N'_\bR / \Lambda_\omega$ (\cite[Definition 1.35]{MR2213573}).
We define a piecewise linear function $\psi_\omega'$ on the fan $\Sigma_\omega$ by setting $\left. \psi_\omega' \right|_{K_\sigma}:=-m_i^{\sigma_0, \sigma}$ for each maximal dimensional polyhedron $\sigma \succ \tau$.
($\beta_i \psi_\omega'$ coincides with the function $\psi_\omega$ defined in \cite[Remark 1.56]{MR2213573}.)
When $U_\xi$ is positive, the function $\psi_\omega'$ is convex and equal to
\begin{align}
-\inf_{m \in \Delta_i} \la m, \bullet \ra
\end{align}
(cf.~\cite[Remark 1.59.1]{MR2213573}).
This implies that $m:=m_i^{\sigma_0, \sigma}-e_i^\ast$ attains the minimum of the tropical polynomial $\min_{m \in A_i \setminus \lc 0 \rc} \la m+e_i^\ast, - \ra \colon N'_\bR \to \bR$ on $\sigma \subset N_\bR'$.
Since $B$ is the graph of the functions $\min_{m \in A_i \setminus \lc 0 \rc} \la m+e_i^\ast, - \ra$ $(1 \leq i \leq r)$ on $N'_\bR$ as we saw in the proof of \pref{lm:aff-str}, we get
\begin{align}
\sigma \cong \lb \left. \pi_{C_{v_0}} \right|_B \rb^{-1} \lb \sigma \rb \subset \lc n \in B \relmid \la e_i^\ast, n \ra=\la m_i^{\sigma_0, \sigma}, n \ra, 1 \leq i \leq r \rc.
\end{align}
\end{proof}

It turns out by \pref{lm:subdiv} that $U_\xi \subset B$ satisfies \pref{cd:loccont'}.
We can use the local model of tropical contractions $\delta \colon X \to U_\xi$ that we constructed in \pref{sc:construction}.
In order to finish the proof of the statement of \pref{th:reconstruction} for a quasi-simple IAMS $B$, we need to check that the integral affine structure which we constructed for $U_\xi$ in \pref{sc:construction} agrees with its original one of $U_\xi$.

\begin{lemma}\label{lm:aff-str2}
For any vertex $v \prec \xi$, the fan structure along $v$ given by
\begin{align}
S_v' \colon \mathrm{St}(v) \hookrightarrow N_\bR \xrightarrow{\left. \pi_{C_{v}} \right|_{N_\bR}} 
\left. N_\bR \middle/ \vspan \lb \bigcup_{i=1}^r \phi_{\xi, \Deltav_i} \lb v \rb \times \lc e_i \rc \rb \right. 
\end{align}
is is equivalent to the original fan structure $S_v$ of $U_\xi$.
\end{lemma}
\begin{proof}
We show the statement by induction on the number of edges of $\xi$ that are necessary to connect $v_0$ and $v$.
It is obvious that we have $S_{v_0}=S_{v_0}'$ for the vertex $v_0$.
Assume that the statement holds for any vertex of $\xi$ that can be connected to $v_0$ by $k$ edges.
We will show the statement for all vertices $v \prec \xi$ that can be connected to $v_0$ by $k+1$ edges.
There exists a vertex $v' \prec \xi$ that can be connected to $v_0$ by $k$ edges and to $v$ by an edge $\omega \prec \xi$.
For any polyhedron $\rho \in \scrP(d-1)$ such that $\rho \succ \xi$, we consider the monodromy transformation $T_\omega^\rho$ that we considered in \eqref{eq:monodromy0} for $\omega, \rho$.
Let $\sigma_+, \sigma_-$ be the maximal-dimensional polyhedra containing $\rho$.
When we consider the integral affine structure determined by $\lc S_v' \rc_{v}$, by \pref{pr:monodromy} and \pref{lm:subdiv}, the monodromy transformation $T_\omega^\rho$ is given by
\begin{align}
T_\omega^\rho(n)
&=n +\sum_{i=1}^r \la m_i^{\sigma_0, \sigma_-}-m_i^{\sigma_0, \sigma_+}, n \ra \cdot \lb \phi_{\xi, \Deltav_i} \lb v \rb- \phi_{\xi, \Deltav_i} \lb v' \rb \rb\\ \label{eq:phin2}
&=n +\sum_{i=1}^r \la m_i^{\sigma_0, \sigma_-}-m_i^{\sigma_0, \sigma_+}, n \ra \cdot 
\lb n_i^{v_0, v}- n_i^{v_0, v'} \rb \\
&=n +\sum_{i=1}^r \la m_i^{\sigma_+, \sigma_-}, n \ra \cdot n_i^{v', v} \\
&=
\left\{ \begin{array}{ll}
    n + \la  \alpha_{i_0} \check{d}_\rho, n \ra \beta_{i_0} d_\omega & \omega \in \Omega_{i_0}, \rho \in R_{i_0} \mathrm{\ for\ some\ common\ } i_0 \in \lc 1, \cdots, r \rc \\
    n & \mathrm{otherwise}.
  \end{array} 
\right.
\end{align}
For \eqref{eq:phin2}, we use \eqref{eq:phin1}.
This coincides with the monodromy of the original integral affine structure determined by the fan structures $\lc S_v \rc_{v}$.
Since the fan structures $S_{v'}$ and $S_{v'}'$ at $v'$ are equivalent by the induction assumption, this implies that the fan structures $S_{v}$ and $S_{v}'$ at $v$ are also equivalent.
Therefore, the claim holds.
\end{proof}

Lastly, we show the statement of \pref{th:reconstruction} for a very simple IAMS $B$.
Suppose $U_\xi$ is very simple.
Then one has $\kappa_i=\alpha_i=\beta_i=1$ and $\Deltav_i(\xi)=\Deltav_i, \Delta_i(\xi)=\Delta_i$.
Since $\Deltav_i, \Delta_i$ are the polytopes that define the local model of tropical contractions, we can easily see from the definition of being very simple that the local model of tropical contractions $\delta$ is very good in the sense of \pref{df:local-contraction}.
Furthermore, as we saw in \eqref{eq:direct-prod}, the space $X \lb f_1, \cdots, f_r \rb (\supset V)$ is isomorphic to the direct product
\begin{align}
N_{0, \bR} \times X' \lb f_1 \rb \times \cdots \times X' \lb f_r \rb,
\end{align}
where $X' \lb f_i \rb$ is the tropical hypersurface defined by $f_i \colon N_{i, \bR} \to \bR$ in the tropical toric variety $X_{C_i} \lb \bT \rb$ with $C_i:=\cone \lb \Deltav_i \times \lc e_i \rc \rb \subset N_{i, \bR}$.
The Newton polytope of $f_i$ is the convex hull of $\lc 0 \rc \cup \lb \Delta_i \times \lc -e_i^\ast \rc \rb$.
$\Deltav_i$ and $\Delta_i$ are standard simplices, and it is easy to see that the tropical hypersurface $X' \lb f_i \rb \subset X_{C_i} \lb \bT \rb$ is a tropical manifold.
Since the direct product of tropical manifolds is also a tropical manifold, the space $X \lb f_1, \cdots, f_r \rb$ is also a tropical manifold.
Thus the domain $V$ of the local model of tropical contractions $\delta$ is a tropical manifold too.

\begin{remark}
There is a condition for IAMS called \emph{semi-simple polytopal}, which was introduced in \cite{RZ20, RZ21}.
We refer the reader to \cite[Section 2]{RZ20} for its definition.
Let $U_\xi$ be a quasi-simple local model of IAMS.
If \eqref{eq:delv} and \eqref{eq:del} constructed for $U_\xi$ are polytopes such that $\sum_{i=1}^r T \lb \Delta_i \rb$ is the internal direct sum of $\lc T \lb \Delta_i \rb \rc_{i \in \lc 1, \cdots, r \rc}$ and $\sum_{i=1}^r T \lb \Deltav_i \rb$ is the internal direct sum of $\lc T \lb \Deltav_i \rb \rc_{i \in \lc 1, \cdots, r \rc}$, then they play the role of the collections of polytopes appearing in \cite[Definition 4]{RZ20}, and it turns out that $U_\xi$ is semi-simple polytopal in their sense.
(In this case, the tropical contraction constructed above is good in the sense of \pref{df:local-contraction}.)
In particular, if a local model of IAMS $U_\xi$ is either simple or very simple, then it is also semi-simple polytopal.
On the other hand, in \pref{df:simple}, we do not impose the \emph{compatibility} in the sense of \cite[Definition 2]{RZ20} on the quasi-simple/very simple IAMS $B$.
\end{remark}

\section{Contractions of tropical Calabi--Yau varieties}\label{sc:cy}

\subsection{Toric degenerations of Calabi--Yau varieties}\label{sc:tcy}

We recall the construction by Gross \cite{MR2198802} of toric degenerations of Calabi--Yau varieties and their dual intersection complexes.
Let $d$ and $r$ be positive integers.
Consider a free $\bZ$-module $M$ of rank $d+r$ and its dual lattice $N:=\Hom(M, \bZ)$.
We set $M_Q:=M \otimes_\bZ Q$ and $N_Q:=N \otimes_\bZ Q=\Hom(M, Q)$ for $Q=\bQ, \bR$ again.
Let $\Delta \subset M_\bR$ be a reflexive polytope, and $\Delta^\ast := \lc n \in N_\bR \relmid \la \Delta, n \ra \geq -1 \rc$ be its polar polytope.
For a face $F \prec \Delta$, its dual face $F^\ast \prec \Delta^\ast$ is defined by
\begin{align}
F^\ast := \lc n \in \Delta^\ast \relmid \la F, n \ra = -1 \rc.
\end{align}
Let $\Sigma \subset N_\bR$ be the normal fan of $\Delta$.
Consider the convex piecewise linear function $\varphi \colon N_\bR \to \bR$ that corresponds to the anti-canonical sheaf on the toric variety associated with the fan $\Sigma$.
Let $\lc e_1, \cdots, e_l \rc$ be the set of primitive generators of one-dimensional cones of the fan $\Sigma$.
Then we have $\varphi(e_i)=1$ for any $i \in \lc 1, \cdots, r \rc$.
A Minkowski decomposition $\Delta =\Delta_1+ \cdots + \Delta_r$ is called a \emph{nef-partition} if the induced decomposition $\varphi = \varphi_1 + \cdots + \varphi_r$ satisfies $\varphi_i(e_j) \in \lc 0, 1\rc$ for any $i, j \in \lc 1, \cdots, r\rc$.
Let $\nabla_i \subset N_\bR$ be the convex hull of $0 \in N_\bR$ and all $e_j$ such that $\varphi_i(e_j) =1$.
Then the Minkowski sum $\nabla=\nabla_1 + \cdots + \nabla_r$ becomes a reflexive polytope, and one has
\begin{align}
\nabla^\ast=\conv \lc \Delta_1, \cdots, \Delta_r \rc, \quad \Delta^\ast=\conv \lc \nabla_1, \cdots, \nabla_r \rc,
\end{align}
where $\nabla^\ast \subset M_\bR$ denotes the polar polytope of $\nabla$ (cf.~\cite[Theorem 4.10]{MR1463173}).
We also know
\begin{align}\label{eq:nabla}
\nabla_i&= \lc n \in N_\bR \relmid \la \Delta_j, n \ra \geq -\delta_{i, j}, \forall j \in \lc 1, \cdots, r \rc \rc \\ \label{eq:lattice-pts}
\partial \Delta^\ast \cap N&=\bigsqcup_{i=1}^r \lb \nabla_i \cap N \setminus \lc 0 \rc \rb, 
\quad \partial \nabla^\ast \cap M=\bigsqcup_{i=1}^r \lb \Delta_i \cap M \setminus \lc 0 \rc \rb
\end{align}
\cite[Proposition 3.13, Corollary 3.23]{MR2405763}.
Let $\Sigmav \subset M_\bR$ be the normal fan of $\nabla$, and $\varphiv \colon M_\bR \to \bR$ be the piecewise linear function corresponding to the anti-canonical sheaf on the toric variety associated with the fan $\Sigmav$.
We also write the decomposition of $\varphiv$ induced by $\nabla=\nabla_1 + \cdots + \nabla_r$ as $\varphiv = \varphiv_1 + \cdots + \varphiv_r$.
For lattice polytopes $\mu \subset \partial \Delta^\ast$ and $\eta \subset \partial \nabla^\ast$, we define
\begin{align}
\beta_i^\ast(\mu)&:= \lc n \in \mu \relmid \varphi_i(n)=1 \rc \subset \nabla_i \cap \partial \Delta^\ast, \quad \rotatebox[origin=c]{180}{$\beta$} (\mu):=\sum_{i=1}^r \beta_i^\ast(\mu) \subset \nabla, \\
\rotatebox[origin=c]{180}{$\beta$} _i^\ast(\eta)&:= \lc m \in \eta \relmid \check{\varphi}_i(m)=1 \rc \subset \Delta_i \cap \partial \nabla^\ast, \quad \beta (\eta):=\sum_{i=1}^r \rotatebox[origin=c]{180}{$\beta$} _i^\ast(\eta) \subset \Delta.
\end{align}
The subsets $\beta_i^\ast(\mu), \rotatebox[origin=c]{180}{$\beta$} _i^\ast(\eta)$ are faces of $\mu, \eta$ respectively.
Since
\begin{align}\label{eq:polytope-b}
\nabla_i \cap \partial \Delta^\ast=\lc n \in \partial \Delta^\ast \relmid \varphi_i(n)=1 \rc, \quad
\Delta_i \cap \partial \nabla^\ast=\lc n \in \partial \nabla^\ast \relmid \check{\varphi}_i(n)=1 \rc
\end{align}
\cite[Proposition 3.19]{MR2405763}, we have
\begin{align}\label{eq:beta_i}
\beta_i^\ast(\mu)=\mu \cap \nabla_i, \quad \rotatebox[origin=c]{180}{$\beta$} _i^\ast(\eta)=\eta \cap \Delta_i.
\end{align}

We take coherent subdivisions $\Sigma' \subset N_\bR, \Sigmav' \subset M_\bR$ of the fans $\Sigma, \Sigmav$ whose fan polytopes (i.e., the convex hulls of primitive generators of all one-dimensional cones) remain to be $\Delta^\ast, \nabla^\ast$ respectively, and consider an integral piecewise linear function $\check{h} \colon M_\bR \to \bR$ that is strictly convex on the fan $\Sigmav'$.
We also assume that the function $\check{h}':=\check{h}-\varphiv \colon M_\bR \to \bR$ is convex (not necessarily strict convex) on $\Sigmav'$.
Let further $\nabla^{\check{h}}, \nabla^{\check{h}'} \subset N_\bR$ denote the Newton polytopes \eqref{eq:newton} of $\check{h}$ and $\check{h}'$.
We set
\begin{align}
\scrR^{\check{h}}_{\Delta^\ast}&:= \lc (F_1, F_2) \relmid F_1 \prec \Delta^\ast,  F_2 \prec \nabla^{\check{h}'}\ \mathrm{such\ that}\ \rotatebox[origin=c]{180}{$\beta$} (F_1) \neq \emptyset, F_1+F_2 \prec \Delta^\ast + \nabla^{\check{h}'}\rc\\
\scrP^{\check{h}}_{\Delta^\ast}&:= \lc \rotatebox[origin=c]{180}{$\beta$} (F_1)+F_2 \relmid (F_1, F_2) \in \scrR^{\check{h}}_{\Delta^\ast} \rc \\
B^{\check{h}}_\nabla&:=\bigcup_{(F_1, F_2) \in \scrR^{\check{h}}_{\Delta^\ast}} \rotatebox[origin=c]{180}{$\beta$} (F_1)+F_2 \subset N_\bR.
\end{align}
One has $B^{\check{h}}_\nabla \subset \partial \nabla^{\check{h}}$ (\cite[Corollary 3.3]{MR2198802}).
We also consider
\begin{align}\label{eq:t-Delta-i}
\tilde{\Delta}_i:=\lc (m, l) \in M_\bR \oplus \bR \relmid m \in \Delta_i, l \geq \check{h}'(m) \rc, \quad 
\tilde{\Delta}:=\sum_{i=1}^r \tilde{\Delta}_i
\end{align}
and the normal fan $\tilde{\Sigma} \subset N_\bR \oplus \bR$ of $\tilde{\Delta}$.
The fan $\tilde{\Sigma}$ is the union of
\begin{enumerate}
\item $\lc \cone \lb F_1 \rb \times \lc 0 \rc \relmid F_1 \prec \Delta^\ast \rc$,
\item $\lc \cone(F_1) \times \lc 0 \rc + \cone \lb F_2 \times \lc 1 \rc \rb \relmid F_1 \prec \Delta^\ast,  F_2 \prec \nabla^{\check{h}'}, F_1+F_2 \prec \Delta^\ast + \nabla^{\check{h}'} \rc$, and 
\item $\lc \cone \lb F_2 \times \lc 1 \rc \rb \relmid F_2 \prec \nabla^{\check{h}'} \rc$
\end{enumerate}
(\cite[Proposition 3.7]{MR2198802}).
The polytope $\tilde{\Delta}$ determines a line bundle on the toric variety associated with $\tilde{\Sigma}$.
It is induced by the piecewise linear function $\tilde{\varphi}$ on $\tilde{\Sigma}$ defined by
\begin{align}
\tilde{\varphi} (\tilde{n}):=-\inf_{\tilde{m}\in \tilde{\Delta}} \la \tilde{m},\tilde{n}\ra,
\end{align}
where $\tilde{n} \in N_\bR \oplus \bR$.
The Minkowski decomposition $\tilde{\Delta}=\sum_{i=1}^r \tilde{\Delta}_i$ induces a decomposition $\tilde{\varphi}=\sum_{i=1}^r \tilde{\varphi}_i$ with $\tilde{\varphi}_i \lb \lb n , 0 \rb \rb=\varphi_i(n)$ and $\tilde{\varphi}_i \lb \lb n , 1 \rb \rb=0$ for any $n \in \nabla^{\check{h}'}$.
We take a subdivision $\tilde{\Sigma}'$ of the fan $\tilde{\Sigma} \subset N_\bR \oplus \bR$ so that it satisfies the following conditions:
\begin{enumerate}
\item The fan $\lc C \cap (N_\bR \oplus \lc 0\rc) \relmid C \in \tilde{\Sigma}' \rc$ coincides with the fan $\Sigma'$.
\item Every $1$-dimensional cone of $\tilde{\Sigma}'$ not contained in $N_\bR \oplus \lc 0\rc$ is generated by a primitive vector $(n, 1)$ with $n \in \nabla^{\check{h}'} \cap N$.
\end{enumerate}
Such a subdivision $\tilde{\Sigma}'$ is called \emph{good} in \cite[Definition 3.8]{MR2198802}.
The fan $\tilde{\Sigma}'$ consists of the following three sorts of cones (\cite[Observation 3.9]{MR2198802}):
\begin{enumerate}
\item cones of the form $\cone(\mu) \times \lc 0 \rc$ for $\cone(\mu) \in \Sigma'$ with $\mu \subset \partial \Delta^\ast$,
\item cones of the form $\cone(\mu) \times \lc 0 \rc + \cone \lb \nu \times \lc 1 \rc \rb$ where $\cone(\mu) \in \Sigma'$ for $\mu \subset \partial \Delta^\ast,  \nu \subset \partial \nabla^{\check{h}'}$, and $\mu + \nu$ is contained in a face of $\Delta^\ast + \nabla^{\check{h}'}$, and
\item cones contained in $\cone(\nabla^{\check{h}'} \times \lc 1\rc)$.
\end{enumerate}
Cones of the second type with $\rotatebox[origin=c]{180}{$\beta$}(\mu) \neq \emptyset$ are called \textit{relevant}.
We set
\begin{align}
\scrR(\tilde{\Sigma}')&:=\lc (\mu, \nu) \relmid 
\begin{array}{l}
\mu \subset \partial \Delta^\ast, \nu \subset \partial \nabla^{\check{h}'} \\
\cone(\mu) \times \lc 0 \rc + \cone \lb \nu \times \lc 1 \rc \rb \mathrm{\ is\ a\ relevant\ cone\ in\ } \tilde{\Sigma}'
\end{array}
\rc \\
\scrP(\tilde{\Sigma}')&:=\lc \rotatebox[origin=c]{180}{$\beta$} (\mu) + \nu \relmid (\mu, \nu) \in \scrR(\tilde{\Sigma}') \rc.
\end{align}
Then one has
\begin{align}
B^{\check{h}}_\nabla=\bigcup_{(\mu, \nu) \in \scrR(\tilde{\Sigma}')} \rotatebox[origin=c]{180}{$\beta$} (\mu) + \nu,
\end{align}
and $\scrP(\tilde{\Sigma}')$ is a polyhedral decomposition of $B^{\check{h}}_\nabla$ \cite[Proposition 3.12, Definition 3.13]{MR2198802}.

For each $\tau \in \scrP(\tilde{\Sigma}')$, we take an element $a_\tau \in \rint (\tau)$, and consider the associated subdivision $\widetilde{\scrP} ( \tilde{\Sigma}' )$ of $\scrP ( \tilde{\Sigma}' )$ that we considered in \eqref{eq:subdivision}.
Let further $\Gamma(\tilde{\Sigma}') \subset B^{\check{h}}_\nabla$ denote the discriminant locus defined by \eqref{eq:discriminant}.
For each vertex $v \in \scrP(\tilde{\Sigma}')$, let $W_v^\circ$ be the union of interiors of all simplices of $\widetilde{\scrP}(\tilde{\Sigma}')$ containing $v$.
The vertex $v$ is written as $v=\rotatebox[origin=c]{180}{$\beta$} (\mu_v) + \nu_v$ with $(\mu_v, \nu_v) \in \scrR(\tilde{\Sigma}')$.
We define a chart
\begin{align}\label{eq:gfanstr}
\psi_v \colon W_v^\circ \to N_\bR / \vspan (\beta^\ast_1(\mu_v), \cdots, \beta^\ast_r(\mu_v))
\end{align}
via the projection.
For each maximal-dimensional face $\sigma \in \scrP(\tilde{\Sigma}')$ of $B$, we define a chart
\begin{align}
\psi_ \sigma \colon \rint(\sigma) \hookrightarrow \aff(\sigma)
\end{align}
via the inclusion.
These charts $\psi_v$ and $\psi_\sigma$ define an integral affine structure on $B^{\check{h}}_\nabla \setminus \Gamma(\tilde{\Sigma}')$, and we obtain the IAMS of dimension $d$, which is equipped with the polyhedral decomposition $(B^{\check{h}}_\nabla, \scrP(\tilde{\Sigma}'))$.

Let $R$ be a discrete valuation ring and $k$-algebra with algebraically closed residue class field $k$.
Let further $X_{\tilde{\Sigma}'}$ be the toric variety over $k$ associated with the fan $\tilde{\Sigma}'$, and $t$ be the regular function on $X_{\tilde{\Sigma}'}$ corresponding to $(0, 1) \in M \oplus \bZ$, which defines the morphism $f \colon X_{\tilde{\Sigma}'} \to \bA^1$.
Let further $\scL_i$ be the line bundle on $X_{\tilde{\Sigma}'}$ associated with the polytope $\tilde{\Delta}_i$, and take a general section $s_i \in \Gamma \lb X_{\tilde{\Sigma}'},  \scL_i \rb$ for every $i \in \lc 1, \cdots, r \rc$.
We consider the variety $\scX \subset X_{\tilde{\Sigma}'}$ defined by
\begin{align}\label{eq:tss}
ts_1+s_1^0= \cdots =ts_r+s_r^0=0,
\end{align}
where $s_i^0 \in \Gamma \lb X_{\tilde{\Sigma}'},  \scL_i \rb$ is the section defined by $(0, 0) \in \tilde{\Delta}_i$.
By restricting the morphism $f \colon X_{\tilde{\Sigma}'} \to \bA^1$ to $\scX$ and taking the base change to $R$ with uniformizing parameter $t$, we obtain a morphism $\scX \to \Spec R$.
By abuse of notation, we also write it as $f \colon \scX \to \Spec R$.
This is a toric degeneration of Calabi--Yau varieties \cite[Theorem 3.10]{MR2198802}, and the IAMS equipped with the polyhedral decomposition $(B^{\check{h}}_\nabla, \scrP(\tilde{\Sigma}'))$ forms its dual intersection complex \cite[Proposition 3.14]{MR2198802}.
We refer the reader to \cite[Definition 4.1]{MR2213573} for the definitions of toric degenerations and their dual intersection complexes.

\begin{remark}
There is also another construction of integral affine spheres with singularities for Calabi--Yau complete intersections by Haase and Zharkov, which was discovered independently \cite{HZ02, MR2187503}.
\end{remark}

\subsection{Tropical Calabi--Yau varieties}\label{sc:tcyci}

The aim of this subsection is to show some propositions that are necessary to prove \pref{th:glcontr}. 
We consider the tropical Calabi--Yau variety corresponding to the toric degeneration that we recalled in the previous subsection.
We work on the same setup as in the previous subsection.
Let $f_i \colon N_\bR \to \bR$ be the tropical polynomial defined by
\begin{align}\label{eq:trop}
f_i (n) :=\min_{m \in \Delta_i \cap M} \lc \check{h}(m)+\la m, n \ra \rc,
\end{align}
and $X(f_i)^\circ \subset N_\bR$ be the tropical hypersurface defined by $f_i$.
We consider the stable intersection
\begin{align}
X(f_1, \cdots, f_r)^\circ :=X(f_1)^\circ \cap_{\mathrm{st}}  \cdots \cap_{\mathrm{st}}  X(f_r)^\circ
\end{align}
and its closure $X(f_1, \cdots, f_r)$ in the tropical toric variety $X_{\Sigma'}(\bT)$ associated with the fan $\Sigma'$.

\begin{proposition}\label{pr:trop-stable}
Let $K$ be the quotient field of $R$, and $f' \colon X \to \Spec K$ be the base change of the toric degeneration $f \colon \scX \to \Spec R$ to $K$.
Then one has
\begin{align}\label{eq:trop-stable}
\trop(X)=X(f_1, \cdots, f_r),
\end{align}
where $\trop(X)$ denotes the tropicalization of $X$.
\end{proposition}
\begin{proof}
Since tropicalizations are invariant under field extensions (cf.~e.g.~\cite[Proposition 3.7]{MR3088913}), we may replace the field $K$ with its algebraic closure.
By abuse of notation, let $K$ denote the algebraic closure of the quotient field of $R$, and we work over it in the following.

We follow the proofs of \cite[Proposition 2.7.8, Theorem 5.3.3]{MR3064984}.
Let $\scX_i \subset X_{\tilde{\Sigma}'}$ be the hypersurface defined by $ts_i+s_i^0=0$, and $X_i$ be its base change to $K$.
Let further $T:=\Spec K \ld M \rd$ denote the dense torus.
For $1 \leq r' \leq r$, we will first show
\begin{align}\label{eq:t-stable}
\trop \lb \bigcap_{i=1}^{r'} X_i \cap T \rb=\trop \lb X_1 \cap T \rb \cap_{\mathrm{st}}  \cdots  \cap_{\mathrm{st}} \trop \lb X_{r'} \cap T \rb,
\end{align}
by induction on $r'$.
When $r'=1$, it is obvious.
We assume that it holds when $r' = r_0-1$ with some $r_0 \geq 2$, and show it for $r'=r_0$.
We set $Y^\circ:=\bigcap_{i=1}^{r_0-1} X_i \cap T$ and $X^\circ_{r_0}:=X_{r_0} \cap T$.
If we show
\begin{align}\label{eq:intersect-r_0}
\trop \lb Y^\circ \cap X^\circ_{r_0} \rb=\trop \lb Y^\circ \rb \cap_{\mathrm{st}} \trop \lb X^\circ_{r_0} \rb,
\end{align}
then we can obtain \eqref{eq:t-stable} for $r'=r_0$ by the induction hypothesis.
We will try to prove \pref{eq:intersect-r_0}.
In the following, we use the notations in \pref{sc:toric}.

Let $n \in N_\bQ \cap \trop \lb Y^\circ \rb \cap \trop \lb X_{r_0}^\circ \rb$ be an arbitrary point, and $\trop \lb Y^\circ_n \rb, \trop \lb X_{r_0, n}^\circ \rb$ be the tropicalizations of the initial degenerations $Y^\circ_n$ and $X_{r_0, n}^\circ$ respectively.
Here we think of $Y^\circ_n$ and $X_{r_0, n}^\circ$ as algebraic varieties over the residue class field $k$ with the trivial valuation.
They are rational fans in $N_\bR$.
First, we will show
\begin{align}\label{eq:trivial-in}
\trop \lb Y^\circ_n \cap X_{r_0, n}^\circ \rb=
\trop \lb Y^\circ_n \rb \cap_\mathrm{st} \trop \lb X_{r_0, n}^\circ \rb,
\end{align}
which will be used to show \pref{eq:intersect-r_0}.
Recall that one has
\begin{align}
H^0 \lb X_{\tilde{\Sigma}'}, \scL_i \rb \cong \bigoplus_{(m, z) \in \tilde{\Delta}_i \cap \lb M \times \bZ \rb} k x^m t^z.
\end{align}
We write the section $s_{r_0}$ as $\sum_{m, z} c_{m, z} x^m t^z$ $\lb c_{m, z} \in k \rb$.
Then since $X_{r_0}^\circ$ is defined by $ts_{r_0}+s_{r_0}^0=0$ and the polytope $\tilde{\Delta}_i$ is defined by \eqref{eq:t-Delta-i}, we can see that the initial degeneration $X_{r_0, n}^\circ$ of $X_{r_0}^\circ$ at $n$ is defined in the special fiber $T_n$ of $\Spec R \ld M \rd^n$ by
\begin{align}\label{eq:df-in}
\left\{ \begin{array}{ll}
\lb \sum_{m} c_{m, \check{h}'(m)} x^m t^{\check{h}(m)} \rb t^{-f_{r_0}(n)} & (f_{r_0}(n)<0) \\
1+\sum_{m} c_{m, \check{h}'(m)} x^m t^{\check{h}(m)} & (f_{r_0}(n)=0),
\end{array} 
\right.
\end{align}
where the sum is taken over $m \in \lb \Delta_i \cap M \rb \setminus \lc 0 \rc$ such that $\check{h}(m)+\la m, n \ra=f_{r_0}(n)$.
We take a complete unimodular fan $\Sigma_n$ in $N_\bR$ containing $\trop \lb Y^\circ_n \rb$ and $\trop \lb X_{r_0, n}^\circ \rb$ as its subfans.
Let further $\overline{Y^\circ_{n}}$ and $\overline{X_{r_0, n}^\circ}$ denote the closures of $Y^\circ_{n}$ and $X_{r_0, n}^\circ$ in the toric variety $X_{\Sigma_n}$ associated with $\Sigma_n$.
It turns out by \eqref{eq:df-in} that if we choose $c_{m, \check{h}'(m)} \in k$ for $m \in \lb \Delta_i \cap M \rb \setminus \lc 0 \rc$ such that $\check{h}(m)+\la m, n \ra=f_{r_0}(n)$ to be general when we take the section $s_{r_0}$, 
the initial degeneration $X_{r_0, n}^\circ$ intersects with $Y^\circ_n$ properly, or $X_{r_0, n}^\circ \cap Y^\circ_n =\emptyset$.
Furthermore, for every closed torus-invariant subvariety $V$ in $X_{\Sigma_n}$, $\overline{X_{r_0, n}^\circ} \cap V$ and $\overline{Y^\circ_{n}} \cap V$ intersect with each other properly in $V$, or are disjoint.
For any points $n, n' \in N_\bQ$ sitting in the relative interior of a common face of $\trop \lb Y^\circ \rb \cap \trop \lb X_{r_0}^\circ \rb$, we have the isomorphisms $Y^\circ_n \cong Y^\circ_{n'}$ and $X_{r_0, n}^\circ \cong X_{r_0, n'}^\circ$ induced by an isomorphism $T_n \cong T_{n'}$ (cf.~\cite[Theorem 2.2.1]{MR2707751}), and there are only finitely many faces in $\trop \lb Y^\circ \rb \cap \trop \lb X_{r_0}^\circ \rb$.
Hence, when we take a general section $s_{r_0}$, 
not only $X_{r_0}^\circ$ and $Y^\circ$ but also $X_{r_0, n}^\circ$ and $Y_n^\circ$, $\overline{X_{r_0, n}^\circ} \cap V$ and $\overline{Y^\circ_{n}} \cap V$ intersect properly or are disjoint for any $n \in N_\bQ \cap \trop \lb Y^\circ \rb \cap \trop \lb X_{r_0}^\circ \rb$.

Let $\tau$ be a face of $ \trop \lb Y^\circ_n \rb \cap \trop \lb X_{r_0, n}^\circ \rb$ of codimension $r_0$ in $N_\bR$.
By \cite[Proposition 2.7.7]{MR3064984} and \cite[Example 8.2.7]{MR1644323}, the multiplicity of $\trop \lb Y^\circ_n \rb \cap_\mathrm{st} \trop \lb X^\circ_{r_0, n} \rb  $ along $\tau$ is equal to
\begin{align}\label{eq:tau-multi}
\sum_{Z} i \lb Z, Y^\circ_n \cdot X_{r_0, n}^\circ; T_n \rb m_Z(\tau)
=\sum_{Z} \mathrm{length} \lb \scO_{Z, Y^\circ_n \cap X_{r_0, n}^\circ} \rb m_Z(\tau),
\end{align}
where the sum is over components $Z$ of $Y^\circ_n \cap X_{r_0, n}^\circ$ such that $\trop(Z)$ contains $\tau$.
$Y^\circ_n \cdot X_{r_0, n}^\circ$ is the refined intersection cycle, and $i \lb Z, Y^\circ_n \cdot X_{r_0, n}^\circ; T_n \rb$ is the intersection multiplicity of $Z$ in $Y^\circ_n \cdot X_{r_0, n}^\circ$.
We refer the reader to \cite[Section 8]{MR1644323} for details on these notions.
$m_Z(\tau)$ is the multiplicity of $\trop(Z)$ along $\tau$.
Since the fundamental cycle of $Y^\circ_n \cap X_{r_0, n}^\circ$ is defined to be
\begin{align}
\sum_{Z} \mathrm{length} \lb \scO_{Z, Y^\circ_n \cap X_{r_0, n}^\circ} \rb Z,
\end{align}
it turns out that \eqref{eq:tau-multi} is the multiplicity of the tropicalization of the fundamental cycle of $Y^\circ_n \cap X_{r_0, n}^\circ$ along $\tau$, which is equal to the multiplicity of $\trop \lb Y^\circ_n \cap X_{r_0, n}^\circ \rb$ along $\tau$ by \cite[Corollary 4.4.6]{MR3064984}.
Hence, we obtain \eqref{eq:trivial-in}.

We now go back to showing \eqref{eq:intersect-r_0}.
We first check \eqref{eq:intersect-r_0} as subsets in $N_\bR$.
Let $n \in N_\bQ \cap \trop \lb Y^\circ \rb \cap \trop \lb X_{r_0}^\circ \rb$ be a point.
If $n \nin \trop \lb Y^\circ \rb \cap_{\mathrm{st}} \trop \lb X_{r_0}^\circ \rb$, 
then $\trop \lb Y^\circ_n \rb \cap_\mathrm{st} \trop \lb X_{r_0, n}^\circ \rb$ is empty, since the stars of $n$ in $\trop \lb Y^\circ \rb$ and $\trop \lb X_{r_0}^\circ \rb$ are naturally identified with $\trop \lb Y^\circ_n \rb$ and $\trop \lb X_{r_0, n}^\circ \rb$ respectively (cf.~\cite[Proposition 2.2.3]{MR2707751}).
By \eqref{eq:trivial-in}, one gets $Y^\circ_n \cap X_{r_0, n}^\circ = \emptyset$, which implies $n \nin \trop \lb Y^\circ \cap X_{r_0}^\circ \rb$.
If $n \in \trop \lb Y^\circ \rb \cap_{\mathrm{st}} \trop \lb X_{r_0}^\circ \rb$, then $Y^\circ_n \cap X_{r_0, n}^\circ \neq \emptyset$.
By \cite[Theorem 4.1.3]{MR3064984}, one can see that there exists a point in $Y^\circ \cap X_{r_0}^\circ$, which specializes to a point in $Y^\circ_n \cap X_{r_0, n}^\circ$.
This implies that the initial degeneration $\lb Y^\circ \cap X_{r_0}^\circ \rb_n$ is non-empty.
Hence, we get $n \in \trop \lb Y^\circ \cap X_{r_0}^\circ \rb$.
Thus we obtain 
\begin{align}
\trop \lb Y^\circ \cap X^\circ_{r_0} \rb \cap N_\bQ=\trop \lb Y^\circ \rb \cap_{\mathrm{st}} \trop \lb X^\circ_{r_0} \rb \cap N_\bQ.
\end{align}
By taking the closures, we obtain \eqref{eq:intersect-r_0} as subsets in $N_\bR$.

One can also check that the multiplicities of facets are also equal in \eqref{eq:intersect-r_0} as follows:
Let $\tau$ be a facet of $\trop \lb Y^\circ  \rb \cap_\mathrm{st} \trop \lb X_{r_0}^\circ \rb$, and $n$ be a point in $N_\bQ \cap \rint \lb \tau \rb$.
The multiplicity of $\trop \lb Y^\circ  \rb \cap_\mathrm{st} \trop \lb X_{r_0}^\circ \rb$ along $\tau$ is equal to the multiplicity of $\trop \lb Y^\circ_n \rb \cap_\mathrm{st} \trop \lb X_{r_0, n}^\circ \rb $ along $\tau_n:=\bR_{\geq 0}(\tau-n)$.
As we saw above, it is equal to
\begin{align}
\sum_{Z} i \lb Z, Y^\circ_n \cdot X_{r_0, n}^\circ; T_n \rb m_Z(\tau_n),
\end{align}
where the sum is over components $Z$ of $Y^\circ_n \cap X_{r_0, n}^\circ$ such that $\trop(Z)$ contains $\tau_n$.
By \cite[Theorem 4.4.5]{MR3064984}, it turns out that this is equal to
\begin{align}\label{eq:multi}
\sum_{Z} \lb \sum_{\widetilde{Z}} m ( Z, \widetilde{Z} ) i \lb \widetilde{Z}, Y^\circ \cdot X_{r_0}^\circ; T \rb \rb m_Z(\tau_n),
\end{align}
where the sum of $\widetilde{Z}$ is over components $\widetilde{Z}$ of $Y^\circ \cap X_{r_0}^\circ$ whose closures contain $Z$, and $m ( Z, \widetilde{Z} )$ is the multiplicity of $Z$ in the special fiber of the closure of $\widetilde{Z}$.
Since we have
\begin{align}
m_{\tilde{Z}}(\tau)=\sum_{Z} m ( Z, \widetilde{Z} ) m_Z(\tau_n),
\end{align}
where the sum ranges over components $Z$ of the special fiber of the closure of $\widetilde{Z}$, \eqref{eq:multi} is equal to
\begin{align}
\sum_{\widetilde{Z}} i \lb \widetilde{Z}, Y^\circ \cdot X_{r_0}^\circ; T \rb m_{\tilde{Z}}(\tau)
=\sum_{\widetilde{Z}} \mathrm{length} \lb \scO_{\widetilde{Z}, Y^\circ \cap X_{r_0}^\circ} \rb m_{\tilde{Z}}(\tau) 
=m_{Y^\circ \cap X_{r_0}^\circ} \lb \tau \rb,
\end{align}
where the sum is over components $\widetilde{Z}$ of $Y^\circ \cap X_{r_0}^\circ$ such that $\trop(\widetilde{Z})$ contains $\tau$.
Here we used \cite[Example 8.2.7]{MR1644323} and \cite[Corollary 4.4.6]{MR3064984} again.
We showed \eqref{eq:intersect-r_0} including multiplicities, and \eqref{eq:t-stable}.

When we think of $ts_i+s_i^0$ as an element of $K \ld M \rd$ and tropicalize it, we obtain \eqref{eq:trop}.
By Kapranov's theorem (cf.~e.g.~\cite[Theorem 3.1.3]{MR3287221}), one has
\begin{align}\label{eq:kap}
\trop \lb X_i \cap T \rb=X(f_i)^\circ.
\end{align}
Since the sections $s_i$ $(1 \leq i \leq r)$ are general, the variety $X$ is irreducible and intersects with the dense torus $T$.
By \cite[Lemma 3.1.1]{MR3064984}, the tropicalization of $X$ coincides with the closure of the tropicalization of $X \cap T$.
Hence, we can obtain \eqref{eq:trop-stable} by taking the closure of \eqref{eq:t-stable} with $r'=r$ and using \eqref{eq:kap}.
\end{proof}

We also consider the subset $X(f_1, \cdots, f_r)^c \subset X(f_1, \cdots, f_r)$ where the monomial $0$ attains the minimums of all tropical polynomials $f_i \ (1 \leq i \leq r)$, i.e.,
\begin{align}\label{eq:compact}
X(f_1, \cdots, f_r)^c:= \bigcap_{i=1}^r \lc n \in N_\bR \relmid \exists m_i \in \lb \Delta_i \cap M \rb \setminus \lc 0 \rc \ \mathrm{s.t.}\ f_i(n)=\check{h}(m_i)+\la m_i, n \ra =0 \rc.
\end{align}
This is also a subset of $\partial \nabla^{\check{h}}$.
We consider the maps and the diagram of \pref{eq:poly-fan}, \pref{eq:face-diag} for $\nabla^{\check{h}}, \nabla$:
\begin{align}\label{eq:nabla-diag}
  \begin{CD}
     \scrP_{\nabla^{\check{h}}} @>{\phi}>> \scrP_{\nabla} \\
  @V{\delta_{\check{h}}}VV    @V{\delta_{\check{\varphi}}}VV \\
     \Sigmav'   @>{\iota}>>  \Sigmav
  \end{CD}
\end{align}
Here we write the maps $\phi_{\nabla^{\check{h}}, \nabla}, \delta_{\nabla^{\check{h}}}, \delta_{\nabla}$ as $\phi, \delta_{\check{h}}, \delta_{\check{\varphi}}$ respectively for short.
For every face $G \prec \nabla^{\check{h}}$ and $i \in \lc 1, \cdots, r \rc$, we set
\begin{align}\label{eq:Mi}
M_i(G):=M \cap \rotatebox[origin=c]{180}{$\beta$}_i^\ast \lb \delta_{\check{h}} \lb G \rb \cap \partial \nabla^\ast \rb.
\end{align}
By \eqref{eq:n-cone} and \eqref{eq:beta_i}, one has
\begin{align}
M_i(G)
&=M \cap \rotatebox[origin=c]{180}{$\beta$}_i^\ast \lb
\lc m \in \partial \nabla^\ast \relmid \check{h}(m)+\la m, n \ra =0, \forall n \in G \rc
\rb \\
&=\lc m \in \Delta_i \cap \partial \nabla^\ast \cap M \relmid \check{h}(m)+\la m, n \ra =0, \forall n \in G \rc \\ \label{eq:mbetai}
&=\lc m \in \lb \Delta_i \cap M \rb \setminus \lc 0 \rc \relmid \check{h}(m)+\la m, n \ra =0, \forall n \in G \rc.
\end{align}
Since the tropical polynomial $f_i$ is equal to $0$ on $G \prec \nabla^{\check{h}}$, \eqref{eq:mbetai} implies that the set $M_i(G)$ corresponds to the monomials of $f_i$ that attain the minimum of $f_i$ on $G$ (aside from the monomial $0$).
We define
\begin{align}
\scrP_{X(f_1, \cdots, f_r)^c} 
:= \lc G \prec \nabla^{\check{h}} \relmid 
M_i(G) \neq \emptyset, \forall i \in \lc 1, \cdots, r \rc \rc.
\end{align}

\begin{lemma}
$\scrP_{X(f_1, \cdots, f_r)^c}$ is a natural polyhedral structure of $X(f_1, \cdots, f_r)^c$.
\end{lemma}
\begin{proof}
First, we check $\bigcup_{G \in \scrP_{X(f_1, \cdots, f_r)^c}} G=X(f_1, \cdots, f_r)^c$.
For any $G_0 \in \scrP_{X(f_1, \cdots, f_r)^c}$, by \eqref{eq:mbetai}, there exists $m_i \in \lb \Delta_i \cap M \rb \setminus \lc 0 \rc$ such that $\check{h}(m_i)+\la m_i, n \ra =0=f_i(n)$ for any $n \in G_0$.
Hence, one has $G_0 \subset X(f_1, \cdots, f_r)^c$, and $\bigcup_{G \in \scrP_{X(f_1, \cdots, f_r)^c}} G \subset X(f_1, \cdots, f_r)^c$.
On the other hand, for any $n_0 \in X(f_1, \cdots, f_r)^c$, there exists $G_0 \prec \nabla^{\check{h}}$ such that $n_0 \in \rint \lb G_0 \rb$, since $X(f_1, \cdots, f_r)^c \subset \partial \nabla^{\check{h}}$.
For any $i \in \lc 1, \cdots, r \rc$, there exists $m_i \in \lb \Delta_i \cap M \rb \setminus \lc 0 \rc$ such that $f_i(n_0)=\check{h}(m_i)+\la m_i, n_0 \ra =0$.
Since the monomials of $f_i$ that attain the minimum of $f_i$ at $n_0$ also attain the minimum at any point of $G_0$, we also have $f_i(n)=\check{h}(m_i)+\la m_i, n \ra =0$ for any $n \in G_0$.
It turns out by this and \eqref{eq:mbetai} that we have
$G_0 \in \scrP_{X(f_1, \cdots, f_r)^c}$.
We obtained $\bigcup_{G \in \scrP_{X(f_1, \cdots, f_r)^c}} G = X(f_1, \cdots, f_r)^c$.

One can also see that $\scrP_{X(f_1, \cdots, f_r)^c}$ forms a complex as follows:
For $G_1 \prec G_2 \prec \nabla^{\check{h}}$, one has 
$M_i \lb G_1 \rb \supset M_i \lb G_2 \rb$ by \eqref{eq:mbetai}.
Hence, if $G_2 \in \scrP_{X(f_1, \cdots, f_r)^c}$ and $G_1 \prec G_2$, then $G_1 \in \scrP_{X(f_1, \cdots, f_r)^c}$.
If $G_1, G_2 \in \scrP_{X(f_1, \cdots, f_r)^c}$, then $G_1 \cap G_2 \in \scrP_{X(f_1, \cdots, f_r)^c}$.
\end{proof}

\begin{proposition}\label{pr:cpart}
One has $\scrP^{\check{h}}_{\Delta^\ast}=\scrP_{X(f_1, \cdots, f_r)^c}$.
In particular, $X(f_1, \cdots, f_r)^c=B^{\check{h}}_\nabla$.
\end{proposition}
\begin{proof}
First, we show $\scrP^{\check{h}}_{\Delta^\ast} \subset \scrP_{X(f_1, \cdots, f_r)^c}$, i.e., $\rotatebox[origin=c]{180}{$\beta$}(F_1)+F_2 \in \scrP_{X(f_1, \cdots, f_r)^c}$ for any $(F_1, F_2) \in \scrR^{\check{h}}_{\Delta^\ast}$.
By \cite[Lemma 3.2]{MR2198802}, for any $(F_1, F_2) \in \scrR^{\check{h}}_{\Delta^\ast}$, the subset $\rotatebox[origin=c]{180}{$\beta$}(F_1)+F_2$ is a face of $\nabla^{\check{h}}$.
We will check the condition $M_{i} \lb \rotatebox[origin=c]{180}{$\beta$}(F_1)+F_2 \rb \neq \emptyset$.
By the \emph{Claim} in the proof of \cite[Lemma 3.2]{MR2198802}, we have
\begin{align}\label{eq:lem32}
F_1^\ast \cap \delta_{\check{h}'}(F_2)=\lc m \in \cone \lb \rotatebox[origin=c]{180}{$\beta$}(F_1)^\ast \rb \cap \delta_{\check{h}'}(F_2) \relmid \check{\varphi}_i(m)=1, \forall i \in \lc 1, \cdots, r \rc \rc.
\end{align}
Here one has
\begin{align}\label{eq:cones}
\cone \lb \rotatebox[origin=c]{180}{$\beta$}(F_1)^\ast \rb \cap \delta_{\check{h}'}(F_2)
=\delta_{\check{h}} \lb \rotatebox[origin=c]{180}{$\beta$}(F_1)+F_2 \rb,
\end{align}
since
\begin{align}
\cone \lb \rotatebox[origin=c]{180}{$\beta$}(F_1)^\ast \rb
&=\cone \lb \lc m \in \nabla^\ast \relmid \la m, n \ra=-1, \forall n \in \rotatebox[origin=c]{180}{$\beta$}(F_1) \rc \rb \\
&=\cone \lb \lc m \in M_\bR \relmid 
\begin{array}{l}
\la m, n \ra=-1, \forall n \in \rotatebox[origin=c]{180}{$\beta$}(F_1) \\
\la m, n \ra \geq -1, \forall n \in \nabla
\end{array}
\rc \rb \\
&=\delta_{\check{\varphi}} \lb \rotatebox[origin=c]{180}{$\beta$}(F_1) \rb
\end{align}
and $\delta_{\check{\varphi}} \lb \rotatebox[origin=c]{180}{$\beta$}(F_1) \rb \cap \delta_{\check{h}'}(F_2)=\delta_{\check{h}} \lb \rotatebox[origin=c]{180}{$\beta$}(F_1)+F_2 \rb$.
The primitive generators of $1$-dimensional faces of the cone $\delta_{\check{h}} \lb \rotatebox[origin=c]{180}{$\beta$}(F_1)+F_2 \rb \in \Sigmav'$ are elements $m \in \partial \nabla^\ast \cap M$ such that $\check{h}(m)+\la m, n \ra =0$ for all $n \in \rotatebox[origin=c]{180}{$\beta$}(F_1)+F_2$.
Suppose that there is some $i_0 \in \lc 1, \cdots, r \rc$ such that 
$M_{i_0} \lb \rotatebox[origin=c]{180}{$\beta$}(F_1)+F_2 \rb) = \emptyset$.
Then by \eqref{eq:mbetai} and \eqref{eq:lattice-pts}, we can see that \eqref{eq:cones} is generated by elements in 
$\lb \partial \nabla^\ast \cap M \rb \setminus \lb \Delta_{i_0} \cap M \setminus \lc 0 \rc \rb=\bigsqcup_{i \neq i_0} \lb \Delta_i \cap M \setminus \lc 0 \rc\rb$.
This implies $\check{\varphi}_{i_0}=0$ on \eqref{eq:cones}, and it turns out that the right hand side of \eqref{eq:lem32} is empty.
On the other hand, by \cite[Lemma 3.1]{MR2198802}, the left hand side of \eqref{eq:lem32} is non-empty.
Since we obtained a contradiction, one can conclude $M_{i} \lb \rotatebox[origin=c]{180}{$\beta$}(F_1)+F_2 \rb \neq \emptyset$ for all $i \in \lc 1, \cdots, i \rc$.
We obtained $\rotatebox[origin=c]{180}{$\beta$}(F_1)+F_2 \in \scrP_{X(f_1, \cdots, f_r)^c}$.

Next, we check $\scrP^{\check{h}}_{\Delta^\ast} \supset \scrP_{X(f_1, \cdots, f_r)^c}$.
Take an element $G \in \scrP_{X(f_1, \cdots, f_r)^c}$.
Since, in general, a face of the Minkowski sum of polytopes can be uniquely written as the Minkowski sum of faces of the polytopes (cf.~e.g.~\cite[Theorem 3.1.2]{Wei07}), the face $G \prec \nabla^{\check{h}}=\nabla+\nabla^{\check{h}'}$ is uniquely decomposed into a Minkowski sum $G=G_1+G_2$ with $G_1 \prec \nabla, G_2 \prec \nabla^{\check{h}'}$.
We have
\begin{align}\label{eq:delta-dual}
\delta_{\check{\varphi}} \lb G_1 \rb \cap \partial \nabla^\ast
=\lc m \in \partial \nabla^\ast \relmid \check{\varphi}(m)+\la m, n \ra =0, \forall n \in G_1 \rc 
=G_1^\ast.
\end{align}
From this, we can see
\begin{align}
\delta_{\check{h}} \lb G \rb \cap \partial \nabla^\ast
=\delta_{\check{\varphi}} \lb G_1 \rb \cap \delta_{\check{h}'} \lb G_2 \rb \cap \partial \nabla^\ast
=G_1^\ast \cap \delta_{\check{h}'} \lb G_2 \rb
\subset G_1^\ast.
\end{align}
Hence, we have 
$\rotatebox[origin=c]{180}{$\beta$}_i^\ast \lb G_1^\ast \rb \supset 
\rotatebox[origin=c]{180}{$\beta$}_i^\ast \lb \delta_{\check{h}} \lb G \rb \cap \partial \nabla^\ast \rb$.
Since $G$ is an element in $\scrP_{X(f_1, \cdots, f_r)^c}$, we also have $\rotatebox[origin=c]{180}{$\beta$}_i^\ast \lb \delta_{\check{h}} \lb G \rb \cap \partial \nabla^\ast \rb \neq \emptyset$.
We obtain $\rotatebox[origin=c]{180}{$\beta$}_i^\ast \lb G_1^\ast \rb \neq \emptyset$ for all $i \in \lc 1, \cdots, r \rc$.
By substituting $\nabla, G_1$ to $\Delta, \sigma$ in \cite[Lemma 2.7.(a)]{MR2198802} respectively, it turns out that $\sum_{i=1}^r \rotatebox[origin=c]{180}{$\beta$}_i^\ast \lb G_1^\ast \rb$ is a face of $\Delta$.
We set
\begin{align}
F_1:=\lb \sum_{i=1}^r \rotatebox[origin=c]{180}{$\beta$}_i^\ast \lb G_1^\ast \rb \rb^\ast, \quad F_2:=G_2.
\end{align}
We will show that the pair $(F_1, F_2)$ is in $\scrR^{\check{h}}_{\Delta^\ast}$, and satisfies
$\rotatebox[origin=c]{180}{$\beta$}(F_1)+F_2=G$.
From \cite[Lemma 2.7.(a)]{MR2198802} again, we can get $\rotatebox[origin=c]{180}{$\beta$} \lb F_1 \rb=G_1 (\neq \emptyset)$.
Hence, we have $\rotatebox[origin=c]{180}{$\beta$}(F_1)+F_2=G_1+G_2=G$.
The only thing left to be checked is $F_1+F_2 \prec \Delta^\ast + \nabla^{\check{h}'}$.
By \cite[Lemma 3.1]{MR2198802}, this is equivalent to $F_1^\ast \cap \delta_{\check{h}'}(F_2) \neq \emptyset$.
We will show that the right hand side of \eqref{eq:lem32} is non-empty.
By \eqref{eq:cones}, we have 
\begin{align}\label{eq:bf12g}
\cone \lb \rotatebox[origin=c]{180}{$\beta$}(F_1)^\ast \rb \cap \delta_{\check{h}'}(F_2)=\delta_{\check{h}} \lb G \rb.
\end{align}
Take an element $m_i \in \rotatebox[origin=c]{180}{$\beta$}_i^\ast \lb \delta_{\check{h}} \lb G \rb \cap \partial \nabla^\ast \rb$ for every $i \in \lc 1, \cdots, r \rc$, and set $m_0 :=\sum_{i=1}^r m_i$.
Then we have $m_0 \in \delta_{\check{h}} \lb G \rb$ and 
$\check{\varphi}_j \lb m_0 \rb=\sum_{i=1}^r \check{\varphi}_j (m_i)=\sum_{i=1}^r \delta_{i, j}=1$ for any $j \in \lc 1, \cdots, r \rc$.
By combining these and \eqref{eq:bf12g}, we can see that the right hand side of \eqref{eq:lem32} contains the element $m_0$ and is non-empty.
We obtained $F_1+F_2 \prec \Delta^\ast + \nabla^{\check{h}'}$, and $G=\rotatebox[origin=c]{180}{$\beta$}(F_1)+F_2 \in \scrP^{\check{h}}_{\Delta^\ast}$.
\end{proof}

We also consider the map 
\begin{align}\label{eq:theta2}
\theta \colon \scrP(\tilde{\Sigma}') \to \scrP^{\check{h}}_{\Delta^\ast}=\scrP_{X(f_1, \cdots, f_r)^c} \subset \scrP_{\nabla^{\check{h}}}, \quad 
\tau=\rotatebox[origin=c]{180}{$\beta$} (\mu_\tau) + \nu_\tau 
\mapsto
\theta(\tau):= \rotatebox[origin=c]{180}{$\beta$} (F_1) + F_2,
\end{align}
where $F_1 \prec \Delta^\ast,  F_2 \prec \nabla^{\check{h}'}$ are the faces such that $\cone(F_1) \times \lc 0 \rc + \cone \lb F_2 \times \lc 1 \rc \rb$ is the minimal cone in $\tilde{\Sigma}$ containing $C(\mu_\tau) \times \lc 0 \rc + C \lb \nu_\tau \times \lc 1 \rc \rb$.
Note that we have $\mu_\tau \subset F_1, \nu_\tau \subset F_2$.

\begin{lemma}\label{lm:beta}
For any polyhedron $\sigma =\rotatebox[origin=c]{180}{$\beta$} (F_1)+F_2 \in \scrP^{\check{h}}_{\Delta^\ast} \subset  \scrP_{\nabla^{\check{h}}}$, one has 
\begin{align}\label{eq:betas}
\la \rotatebox[origin=c]{180}{$\beta$}_i^\ast (\delta_{\check{h}} \lb \sigma \rb \cap \partial \nabla^\ast),
\beta_j^\ast \lb F_1 \rb \ra=-\delta_{i,j} \quad (1 \leq i, j \leq r).
\end{align}
For any polyhedron $\tau =\rotatebox[origin=c]{180}{$\beta$} (\mu_\tau) + \nu_\tau \in \scrP(\tilde{\Sigma}')$, one also has 
\begin{align}\label{eq:betat}
\la \rotatebox[origin=c]{180}{$\beta$}_i^\ast (\delta_{\check{h}} \lb \theta \lb \tau \rb \rb \cap \partial \nabla^\ast),
\beta_j^\ast \lb \mu_{\tau} \rb \ra=-\delta_{i,j} \quad (1 \leq i, j \leq r).
\end{align}
\end{lemma}
\begin{proof}
First we show \eqref{eq:betas}.
Since a face of the Minkowski sum of polytopes can be uniquely written as the Minkowski sum of faces of the polytopes, \pref{lm:minksum} implies 
\begin{align}\label{eq:phi-b}
\phi \lb \sigma \rb=\rotatebox[origin=c]{180}{$\beta$} (F_1).
\end{align}
From this and \eqref{eq:delta-dual} for $G_1=\phi \lb \sigma \rb$, we can see
\begin{align}
\delta_{\check{h}} \lb \sigma \rb \cap \partial \nabla^\ast 
&\subset \iota \lb \delta_{\check{h}} \lb \sigma \rb \rb \cap \partial \nabla^\ast
= \delta_{\check{\varphi}} \lb \phi \lb \sigma \rb \rb \cap \partial \nabla^\ast
= \phi \lb \sigma \rb^\ast
=\rotatebox[origin=c]{180}{$\beta$} (F_1)^\ast.
\end{align}
Since we have $\la \rotatebox[origin=c]{180}{$\beta$} (F_1)^\ast, \rotatebox[origin=c]{180}{$\beta$} (F_1) \ra=-1$, we get
\begin{align}
\la \rotatebox[origin=c]{180}{$\beta$}_i^\ast (\delta_{\check{h}} \lb \sigma \rb \cap \partial \nabla^\ast),
\beta_1^\ast \lb F_1 \rb+\cdots + \beta_r^\ast \lb F_2 \rb\ra=-1.
\end{align}
From this and \eqref{eq:nabla}, we obtain \eqref{eq:betas}.

Next, we show \eqref{eq:betat}.
When we write $\theta(\tau)=\rotatebox[origin=c]{180}{$\beta$} (F_1)+F_2$, we have $\beta_j^\ast \lb \mu_\tau \rb \subset \beta_j^\ast \lb F_1 \rb$ for all $j \in \lc 1, \cdots, r \rc$, since $\mu_\tau \subset F_1$.
\eqref{eq:betat} follows from this and \eqref{eq:betas} for $\sigma=\theta(\tau)$.
\end{proof}

\begin{lemma}\label{lm:bpoints}
Let $v=\rotatebox[origin=c]{180}{$\beta$} (\mu_v) + \nu_v \in \scrP(\tilde{\Sigma}')$ be a vertex, and $\sigma \in \scrP(\tilde{\Sigma}')$ be a maximal-dimensional polyhedron.
 Then the following hold:
 \begin{enumerate}
 \item Both of the subsets $\beta_i^\ast \lb \mu_{v} \rb \subset \nabla_i \cap N$ and 
 $M_i \lb \theta (\sigma) \rb \subset \Delta_i \cap M$ consist of a single point.
\item Both $\lc \beta_i^\ast \lb \mu_{v} \rb \relmid 1 \leq i \leq r \rc$ and $\lc M_i \lb \theta (\sigma) \rb \relmid 1 \leq i \leq r \rc$ are linearly independent and generate a primitive sublattice in $N$ and $M$ respectively.
\item When $v \prec \sigma$, one has
\begin{align}\label{eq:mibetaj}
\la M_i \lb \theta (\sigma) \rb,
\beta_j^\ast \lb \mu_{v} \rb \ra=-\delta_{i,j} \quad (1 \leq i, j \leq r).
\end{align}
\end{enumerate}
\end{lemma}
\begin{proof}
First, we check \eqref{eq:mibetaj}.
Suppose $v \prec \sigma=\rotatebox[origin=c]{180}{$\beta$} (\mu_\sigma) + \nu_\sigma$.
Then \eqref{eq:mibetaj} follows from \eqref{eq:betat} for $\sigma$, since $M_i \lb \theta (\sigma) \rb \subset \rotatebox[origin=c]{180}{$\beta$}_i^\ast (\delta_{\check{h}} \lb \theta \lb \sigma \rb \rb \cap \partial \nabla^\ast)$ and $\mu_v \subset \mu_\sigma$.

Next, we check the statement for $\beta_i^\ast \lb \mu_{v} \rb$.
Since we have $v=\rotatebox[origin=c]{180}{$\beta$} (\mu_v) + \nu_v=\beta_1^\ast \lb \mu_{v} \rb + \cdots + \beta_r^\ast \lb \mu_{v} \rb +\nu_v$, it is obvious that the subset $\beta_i^\ast \lb \mu_{v} \rb \subset \nabla_i$ is a point.
As stated in \cite{MR2198802} just before Proposition 3.14 in loc.cit., the set $\mu_{v}$ is a convex hull of $r$ points $n_1, \cdots, n_r \in N$ such that $\varphi_i(n_j)= \delta_{i, j}$.
It turns out that the subset $\beta_i^\ast \lb \mu_{v} \rb \subset \nabla_i$ is the point $n_i \in N$.
For the vertex $v$, take a maximal-dimensional polyhedron $\sigma \in \scrP(\tilde{\Sigma}')$ such that $\sigma \succ v$.
From $\theta (\sigma) \in \scrP^{\check{h}}_{\Delta^\ast}$ and \pref{pr:cpart}, one can see $M_i \lb \theta (\sigma) \rb \neq \emptyset$.
By this and \eqref{eq:mibetaj}, one can see that $\lc \beta_i^\ast \lb \mu_{v} \rb \relmid 1 \leq i \leq r \rc$ is linearly independent and generates a primitive sublattice in $N$. 

The statement for $M_i \lb \theta (\sigma) \rb$ can also be checked as follows:
We know $M_i \lb \theta (\sigma) \rb \neq \emptyset$.
As mentioned just after \eqref{eq:mbetai}, the set of indices $M_i \lb \theta (\sigma) \rb \cup \lc 0 \rc$ corresponds to the monomials of $f_i$ that attain the minimum of $f_i$ on $\theta (\sigma)$.
Since $\theta (\sigma)$ is a maximal-dimensional polyhedron, by the last statement of \pref{pr:dual} and \eqref{eq:mibetaj}, we can see that $M_i \lb \theta (\sigma) \rb$ consists of a single point.
It is also obvious again from \eqref{eq:mibetaj} that $\lc M_i \lb \theta (\sigma) \rb \relmid 1 \leq i \leq r \rc$ are linearly independent and generates a primitive sublattice in $M$.
\end{proof}

In the same way as we did in the local case \pref{eq:nat-poly}, one can construct a natural polyhedral structure also for the space $X(f_1, \cdots, f_r)$, which will be denoted by $\scrP_{X(f_1, \cdots, f_r)}$.
Each polyhedron $\sigma$ in $\scrP_{X(f_1, \cdots, f_r)}$ is labeled by the set of indices of monomials attaining the minimum of $f_i$ on $\sigma$ and the cone corresponding to the tropical torus orbit containing $\rint \lb \sigma \rb$.
For a polyhedron $\sigma \in \scrP_{X(f_1, \cdots, f_r)}$ such that $\sigma \cap N_\bR \neq \emptyset$, we set
\begin{align}
\scrM_i (\sigma) =\lc m \in \Delta_i \cap M \relmid \check{h}(m)+\la m, n \ra =f_i(n), \forall n \in \sigma \cap N_\bR \rc
\quad \lb i \in \lc 1, \cdots, r \rc \rb.
\end{align}
This is the set of indices of monomials of $f_i$ that attain the minimum of $f_i$ on $\sigma \cap N_\bR$.
As in \eqref{eq:Isig}, we also set
\begin{align}\label{eq:Isig'}
I_\sigma:= \lc i \in \lc 1, \cdots r \rc \relmid f_i(n) \neq 0\ \mathrm{on}\ \rint(\sigma) \rc.
\end{align}

\begin{proposition}\label{pr:t-sigma}
For a polyhedron $\sigma \in \scrP_{X(f_1, \cdots, f_r)}$ satisfying $\sigma \cap N_\bR \neq \emptyset$, there exists a polyhedron $\tilde{\sigma} \in \scrP^{\check{h}}_{\Delta^\ast}=\scrP_{X(f_1, \cdots, f_r)^c}$ such that 
$\tilde{\sigma} \prec \sigma$ and $\scrM_i (\tilde{\sigma}) = \scrM_i (\sigma) \cup \lc 0 \rc$ for all $i \in \lc 1, \cdots, r \rc$.
Furthermore, when we write $\tilde{\sigma}$ as $\tilde{\sigma}=\rotatebox[origin=c]{180}{$\beta$} (F_1)+F_2$ with $\lb F_1, F_2 \rb \in \scrR^{\check{h}}_{\Delta^\ast}$, one has
\begin{align}\label{eq:sigma2}
\sigma \cap N_\bR = \tilde{\sigma} + \cone \lb \bigcup_{i \in I_\sigma} \beta_i^\ast \lb F_1 \rb \rb.
\end{align}
\end{proposition}
\begin{proof}
First, we show the former claim.
Recall that the polyhedral subdivision $\Xi$ of $N_\bR$ induced by the union of tropical hypersurfaces $X(f_i)^\circ$ $(1 \leq i \leq r)$ is dual to the mixed subdivision $\scrS$ of $\Delta$ induced by $\check{h}$ (\pref{pr:dual}).
The stable intersection $X(f_1, \cdots, f_r)^\circ \subset N_\bR$ is the support a subcomplex of $\Xi$, and whether each cell of $\Xi$ is contained in the stable intersection $X(f_1, \cdots, f_r)^\circ$ or not is determined by \pref{th:st-intersection}.
We write the element in the mixed subdivision $\scrS$ corresponding to $\sigma \cap N_\bR \in \Xi$ as $F:=\sum_{i=1}^r F_i \subset \Delta$ $\lb F_i \subset \Delta_i, \dim F_i \geq 1 \rb$.
Then $F_i \cap M=\scrM_i(\sigma)$.
Finding the polyhedron $\tilde{\sigma}$ in the claim is equivalent to find an element $\tilde{F}=\sum_{i=1}^r \tilde{F_i} \in \scrS$ $\lb \tilde{F_i} \subset \Delta_i \rb$ such that $\tilde{F_i} \cap M=(F_i \cap M) \cup \lc 0 \rc$.
We will try to find such an element $\tilde{F}$.

Let $\scrL$ be the subdivision of $\nabla^\ast$ induced by $\check{h}$.
The following claim is shown in the proof of \cite[Proposition 2.4]{MR2187503}.
\begin{claim}\label{cl:HZ05}
The mixed subdivision $\scrS$ of $\Delta$ consists of 
$\lc \sum_{i=1}^r \lb \eta \cap \Delta_i \rb \relmid \eta \in \scrL, \eta \ni 0 \rc$
and their faces.
\end{claim}
It turns out by this that there exists an element $\eta \in \scrL$ such that $F=\sum_{i=1}^r F_i \prec \sum_{i=1}^r \lb \eta \cap \Delta_i \rb$.
We have $F_i \prec \lb \eta \cap \Delta_i \rb$.
There exists an element $\eta' \in \scrL$ such that $\eta' \subset \partial \nabla^\ast$ and $\eta =\conv \lb \eta' \cup \lc 0 \rc \rb$.
Take a facet $G \prec \nabla^\ast$ such that $\eta' \subset G$, and let $n_0 \in N$ be the vertex of $\nabla$ that is dual to $G$.
Then we have $\la \eta', n_0 \ra=-1$.
Since $\dim \lb \eta \cap \Delta_i \rb \geq \dim F_i \geq 1$, we can see that the set $\eta \cap \Delta_i$ contains a point in $\eta'$.
It turns out that the restriction of $n_0$ to $\eta \cap \Delta_i$ takes the minimum along $\eta' \cap \Delta_i$.
Hence, we have
\begin{align}\label{eq:f1}
\sum_{i=1}^r \lb \eta' \cap \Delta_i \rb \prec \sum_{i=1}^r \lb \eta \cap \Delta_i \rb.
\end{align}
Let $F_i' \prec F_i$ be the face along which the restriction of $n_0$ to $F_i$ takes the minimum.
We have 
\begin{align}\label{eq:f2}
\sum_{i=1}^r F_i' \prec \sum_{i=1}^r F_i \prec \sum_{i=1}^r \lb \eta \cap \Delta_i \rb.
\end{align}
Since $\dim F_i \geq 1$, the set $F_i$ also contains a point in $\eta'$.
It turns out that we have $F_i' \subset \lb \eta' \cap \Delta_i \rb$.
By combining this, \eqref{eq:f1}, and \eqref{eq:f2}, one gets
\begin{align}\label{eq:F'}
\sum_{i=1}^r F_i' \prec \sum_{i=1}^r \lb \eta' \cap \Delta_i \rb.
\end{align}
Since $\lb F_i \cap M \rb \subset \lb \eta \cap M \rb = \lb \eta' \cap M \rb \cup \lc 0 \rc$, we also have
\begin{align}\label{eq:FiM}
\lb F_i' \cap M \rb \cup \lc 0 \rc = (F_i \cap M) \cup \lc 0 \rc.
\end{align}

We set $F':=\conv \lb \bigcup_{i=1}^r F_i' \rb \subset \eta'$.
By \eqref{eq:beta_i}, we have
\begin{align}\label{eq:F''}
F' \cap \Delta_i=\lc m \in F' \relmid \check{\varphi}_i (m)=1 \rc=F_i'
\end{align}
since $\check{\varphi}_i=\delta_{i, j}$ on $F_j' \subset \lb \eta' \cap \Delta_j \rb$.
We also have $F' \prec \eta'$ which can be shown as follows:
We will make an element $n_0' \in N_\bR$ such that the restriction $\left. n_0' \right|_{\eta'}$ takes the minimum along $F'$.
By \eqref{eq:F'}, there exists an element $n_0'' \in M_\bR$ such that the restriction $\left. n_0'' \right|_{\eta' \cap \Delta_i}$ takes the minimum along $F_i' \prec \lb \eta' \cap \Delta_i \rb$ for all $i \in \lc 1, \cdots, r \rc$.
Let $c_i \in \bR$ denote the minimum value of $\left. n_0'' \right|_{\eta' \cap \Delta_i}$ $(1 \leq i \leq r)$.
Since $\eta'$ is contained in a face of $\nabla^\ast$, there are also elements $n_i \in N_\bR$ $(1 \leq i \leq r)$ such that $\left. \check{\varphi}_i \right|_{\eta'}=n_i$.
The value of $n_i$ is constantly $1$ on $\eta' \cap \Delta_i$, and is $0$ on $\eta' \cap \Delta_j$ for $j \neq i$.
We set $n_0':=n_0''-\sum_{i=1}^r c_i n_i \in N_\bR$.
Then the restriction $\left. n_0' \right|_{\eta' \cap \Delta_i}$ takes the minimum value $0$ along $F_i'$.
On the other hand, by \eqref{eq:lattice-pts}, we have 
\begin{align}\label{eq:lattice-ptse}
\eta' \cap M = \bigsqcup_{i=1}^r \lb \eta' \cap \Delta_i \cap M \rb.
\end{align}
Let $\eta'' \prec \eta'$ be the face along which the restriction $\left. n_0' \right|_{\eta'}$ takes the minimum.
We will see $F' = \eta''$.
Since $\eta''$ is a lattice polytope, it contains a point in the lattice $M$.
Hence, it turns out by \eqref{eq:lattice-ptse} that the minimum value of the restriction $\left. n_0' \right|_{\eta'}$ is $0$, and $F' \subset \eta''$.
By \eqref{eq:lattice-ptse} again, one can also obtain 
\begin{align}
\eta'= \conv \lb \eta' \cap M \rb
=\conv \lb \bigcup_{i=1}^r \lb \eta' \cap \Delta_i \cap M \rb \rb
=\conv \lb \bigcup_{i=1}^r \lb \eta' \cap \Delta_i \rb \rb,
\end{align}
from which we can see $F' = \eta''$.
Therefore, one can conclude $F' \prec \eta'$.
In particular, one has $F' \in \scrL$.

We set $F'':=\conv \lb F' \cup \lc 0 \rc \rb \in \scrL$, $\tilde{F_i}:=F'' \cap \Delta_i$, and $\tilde{F}:=\sum_{i=1}^r \tilde{F_i}$.
By \pref{cl:HZ05} again, one can get $\tilde{F} \in \scrS$.
Furthermore, by \eqref{eq:F''} and \eqref{eq:FiM}, one has 
\begin{align}
\tilde{F_i} \cap M = \lb F' \cap \Delta_i \cap M \rb \cup \lc 0 \rc=\lb F_i' \cap M \rb \cup \lc 0 \rc = (F_i \cap M) \cup \lc 0 \rc.
\end{align}
Thus one can conclude the former claim of the proposition.

Next, we show \eqref{eq:sigma2}.
When $I_\sigma = \emptyset$, one has $\tilde{\sigma}=\sigma$ and the claim is clear.
We suppose $I_\sigma \neq \emptyset$ in the following.
In general, for a polyhedron $P=\lc x \in \bR^n \relmid A x \geq b \rc$, where $A$ is a $m \times n$ matrix and $b \in \bR^m$ is a vector, a $0$-dimensional face (vertex) of $P$ is called an \emph{extreme point}, and the cone 
$\rec (P):=\lc x \in \bR^n \relmid p+x \in P, \forall p \in P \rc$
is called the \emph{recession cone} of $P$. 
One has $\rec (P)=\lc x \in \bR^n \relmid A x \geq 0 \rc$ (cf.~e.g.~\cite[Proposition 1.12 $\rm(\hspace{.18em}i\hspace{.18em})$]{MR1311028}).
In order to prove the lemma, we use the following theorem:

\begin{theorem}{\rm(cf.~e.g.~\cite[Theorem 4.24]{MR3060144})}\label{th:lau}
Let $P$ be a polyhedron in $\bR^n$, and $\ext (P) \subset P$ denote the set of extreme points of $P$. 
When $\ext (P) \neq \emptyset$, one has 
\begin{align}
P=\conv \lb \ext (P) \rb+\rec (P).
\end{align}
\end{theorem}

We will compute $\conv \lb \ext \lb \sigma \cap N_\bR \rb \rb$ and $\rec \lb \sigma \cap N_\bR \rb$.
First, we compute the former one.
If a point $v \in \sigma \cap N_\bR$ is an extreme point, then the polyhedron $\tilde{v} \in \scrP_{X \lb f_1, \cdots, f_r \rb^c}$ of the former claim of the proposition for $v$ is equal to $v$ since $\tilde{v} \prec v$ and $\dim v=0$.
Therefore, we have $\scrM_i(v)=\scrM_i \lb \tilde{v} \rb \ni 0$.
Since $v \prec \sigma \cap N_\bR$, we also have $\scrM_i(v) \supset \scrM_i(\sigma)$.
By combining these, we obtain $\scrM_i(v) \supset \scrM_i \lb \sigma \rb \cup \lc 0 \rc=\scrM_i \lb \tilde{\sigma} \rb$, which implies $v \prec \tilde{\sigma}$.
On the contrary, since $\tilde{\sigma}$ is a face of $\sigma$, all vertices of $\tilde{\sigma}$ are extreme points of $\sigma$.
Therefore, we get
\begin{align}\label{eq:ext-s}
\conv \lb \ext \lb \sigma \cap N_\bR \rb \rb=\tilde{\sigma}.
\end{align}

Next, we compute $\rec \lb \sigma \cap N_\bR \rb$.
By \eqref{eq:mbetai} and $\scrM_i (\tilde{\sigma}) = \scrM_i (\sigma) \cup \lc 0 \rc$, one can see
\begin{align}\label{eq:MM}
\scrM_i \lb \tilde{\sigma} \rb=M_i \lb \tilde{\sigma} \rb \cup \lc 0 \rc, 
\quad \scrM_i \lb \sigma \rb=
\left\{ 
\begin{array}{ll}
M_i \lb \tilde{\sigma} \rb & i \in I_\sigma \\
M_i \lb \tilde{\sigma} \rb \cup \lc 0 \rc & i \nin I_\sigma.
  \end{array} 
\right.
\end{align}
From this, we can get
\begin{align}
\begin{split}
\sigma \cap N_\bR = 
&\bigcap_{i \in I_\sigma} \lc n \in N_\bR \relmid 
\begin{array}{l}
\check{h}(m) +\la m, n \ra \leq \check{h}(m') +\la m', n \ra \\ 
\forall m \in M_i \lb \tilde{\sigma} \rb, 
\forall m' \in \Delta_i \cap M
\end{array}
\rc \\
&\cap
\bigcap_{i \nin I_\sigma} \lc n \in N_\bR \relmid 
\begin{array}{l}
0= \check{h}(m) +\la m, n \ra \leq \check{h}(m') +\la m', n \ra \\ 
\forall m \in M_i \lb \tilde{\sigma} \rb, 
\forall m' \in \Delta_i \cap M
\end{array}
\rc.
\end{split}
\end{align}
Therefore, the recession cone of $\sigma \cap N_\bR$ is given by
\begin{align}\label{eq:rec0}
\bigcap_{i \in I_\sigma} \lc n \in N_\bR \relmid 
\begin{array}{l}
\la m, n \ra \leq \la m', n \ra \\ 
\forall m \in M_i \lb \tilde{\sigma} \rb, 
\forall m' \in \Delta_i \cap M
\end{array}
\rc \cap
\bigcap_{i \nin I_\sigma} \lc n \in N_\bR \relmid 
\begin{array}{l}
0= \la m, n \ra \leq \la m', n \ra \\ 
\forall m \in M_i \lb \tilde{\sigma} \rb, 
\forall m' \in \Delta_i \cap M
\end{array}
\rc.
\end{align}
One can show that this is equal to
\begin{align}\label{eq:rec1}
\cone \lb 
\bigcup_{i \in I_\sigma} \lc n \in N_\bR \relmid 
\begin{array}{l}
-1= \la m, n \ra \leq \la m', n \ra \\ 
\forall m \in M_i \lb \tilde{\sigma} \rb, \forall m' \in \Delta_i \cap M \\
0= \la m'', n \ra \leq \la m''', n \ra \\
\forall m'' \in M_j \lb \tilde{\sigma} \rb, \forall m''' \in \Delta_j \cap M, \forall j \neq i \\
\end{array}
\rc
\rb 
\end{align}
as follows:
It is obvious that \eqref{eq:rec1} is contained in \eqref{eq:rec0}.
We will check that \eqref{eq:rec0} is contained in \eqref{eq:rec1}.
The cone \eqref{eq:rec0} contains
\begin{align}\label{eq:rec-1}
\bigcap_{i=1}^r \lc n \in N_\bR \relmid 
\begin{array}{l}
0= \la m, n \ra \leq \la m', n \ra \\ 
\forall m \in M_i \lb \tilde{\sigma} \rb, 
\forall m' \in \Delta_i \cap M
\end{array}
\rc
=\rec \lb \tilde{\sigma} \rb
\end{align}
as its face.
Furthermore, since $\tilde{\sigma} \subset \nabla^{\check{h}}$ is bounded, we also have $\rec \lb \tilde{\sigma} \rb=\lc 0 \rc$.
It turns out that the cone \eqref{eq:rec0} has $\lc 0 \rc$ as its face.
Hence, the cone \eqref{eq:rec0} is generated by generators of its $1$-dimensional faces (cf.~e.g.~\cite[Lemma 1.2.15]{MR2810322}).
Therefore, it suffices to show that all $1$-dimensional faces of \eqref{eq:rec0} are contained in \eqref{eq:rec1}.
Let $F$ be a $1$-dimensional face of \eqref{eq:rec0}.
There exists a face $\tau \prec \sigma$ such that $\rec \lb \tau \cap N_\bR \rb=F$ (cf.~\cite[Lemma 3.5]{MR2846179}).
Here we assume $\dim \tau=1$ by replacing $\tau$ with its face if necessary.
We will show that $\rec \lb \tau \cap N_\bR \rb$ is contained in \eqref{eq:rec1}.
We set $I_\tau:=\lc i \in \lc 1, \cdots, r \rc \relmid f_i (n) < 0 \mathrm{\ on \ } \rint \lb \tau \rb \rc \subset I_\sigma$, and let $\tilde{\tau} \in \scrP^{\check{h}}_{\Delta^\ast}=\scrP_{X(f_1, \cdots, f_r)^c}$ be the polyhedron of the former claim of the proposition for $\tau$.
Note that $\tilde{\tau}$ is a point. 
By \eqref{eq:MM} for $\tau$, we have
\begin{align}\label{eq:scrM}
\sum_{i=1}^r \conv \lb \scrM_i \lb \tau \rb \rb
=
\lb \sum_{i \in I_\tau} \conv \lb M_i \lb \tilde{\tau} \rb \rb \rb+
\lb \sum_{i \nin I_\tau} \conv \lb M_i \lb \tilde{\tau} \rb \cup \lc 0 \rc \rb \rb.
\end{align}
This must be of codimension $1$ in $M_\bR$ by the last statement of \pref{pr:dual}, since $\dim \tau=1$.
When we write $\tilde{\tau}=\rotatebox[origin=c]{180}{$\beta$} (F_1')+F_2'$ with $\lb F_1', F_2' \rb \in \scrR^{\check{h}}_{\Delta^\ast}$, the set $\beta_i^\ast (F_1')$ $(1 \leq i \leq r)$ consists of a single point by \pref{lm:bpoints}.1.
We can see from \eqref{eq:betas} for $\tilde{\tau}$ that $\beta_i^\ast (F_1')$ $(i \in I_\tau)$ is constant on \eqref{eq:scrM}.
Since $\lc \beta_i^\ast (F_1') \relmid 1 \leq i \leq r \rc$ is linearly independent by \pref{lm:bpoints}.2, the number of elements in $I_\tau$ must be less than $2$.
If $I_\tau = \emptyset$, then we get $\tau=\tilde{\tau}$ which contradicts $1=\dim \tau > \dim \tilde{\tau}=0$.
Hence, we can conclude that the set $I_\tau$ consists of a single element, which will be denoted by $i_0 \in I_\sigma$.
Since the recession cone of $\tau \cap N_\bR$ is 
\begin{align}
\lc n \in N_\bR \relmid 
\begin{array}{l}
\la m, n \ra \leq \la m', n \ra \\ 
\forall m \in M_{i_0} \lb \tilde{\tau} \rb, 
\forall m' \in \Delta_{i_0} \cap M
\end{array}
\rc \cap
\bigcap_{i \neq i_0} \lc n \in N_\bR \relmid 
\begin{array}{l}
0= \la m, n \ra \leq \la m', n \ra \\ 
\forall m \in M_i \lb \tilde{\tau} \rb, 
\forall m' \in \Delta_i \cap M
\end{array}
\rc
\end{align}
and $M_{i} \lb \tilde{\tau} \rb \supset M_{i} \lb \tilde{\sigma} \rb$, one can see that $\rec \lb \tau \cap N_\bR \rb =F$ is contained in \eqref{eq:rec1}.
Thus \eqref{eq:rec0} is equal to \eqref{eq:rec1}.

By \eqref{eq:nabla} and \eqref{eq:lattice-pts}, we can see that \eqref{eq:rec1} is equal to
\begin{align}\label{eq:n-nabla} \nonumber
&\cone \lb 
\bigcup_{i \in I_\sigma} \lc n \in \nabla_i \relmid 
\begin{array}{l}
-1= \la m, n \ra, \forall m \in \delta_{\check{h}} \lb \tilde{\sigma} \rb \cap \partial \nabla^\ast \cap \Delta_i \cap M \\
0= \la m', n \ra, \forall m' \in \delta_{\check{h}} \lb \tilde{\sigma} \rb \cap \partial \nabla^\ast \cap \Delta_j \cap M, \forall j \neq i
\end{array}
\rc
\rb \\
&=\cone \lb 
\bigcup_{i \in I_\sigma} \lc n \in \nabla_i \relmid 
\begin{array}{l}
-\check{\varphi}_i (m)= \la m, n \ra, \forall m \in \delta_{\check{h}} \lb \tilde{\sigma} \rb \cap \partial \nabla^\ast \cap M
\end{array}
\rc
\rb \\
&=\cone \lb 
\bigcup_{i \in I_\sigma} \lc n \in \nabla_i \relmid 
\begin{array}{l}
-\check{\varphi}_i (m)= \la m, n \ra, \forall m \in \iota \lb \delta_{\check{h}} \lb \tilde{\sigma} \rb \rb \cap \partial \nabla^\ast
\end{array}
\rc
\rb,
\end{align}
where $\iota \colon \Sigmav' \to \Sigmav$ is the map appearing in \eqref{eq:nabla-diag}.
The last equality holds since the normal fan $\Sigmav$ of $\nabla=\sum_{i=1}^r \nabla_i$ is finer than the normal fan of $\nabla_i$.
On the other hand, by \eqref{eq:phi-b} and \eqref{eq:delta-dual}, one can get
\begin{align}
\iota \lb \delta_{\check{h}} \lb \tilde{\sigma} \rb \rb \cap \partial \nabla^\ast 
=\delta_{\check{\varphi}} \lb \phi \lb \tilde{\sigma} \rb \rb \cap \partial \nabla^\ast
=\delta_{\check{\varphi}} \lb \rotatebox[origin=c]{180}{$\beta$}(F_1) \rb \cap \partial \nabla^\ast
=\rotatebox[origin=c]{180}{$\beta$} (F_1)^\ast.
\end{align}
Furthermore, by \cite[Lemma 2.7.(b)]{MR2198802}, one also has
\begin{align}
\beta_i^\ast \lb F_1\rb
&=\lc n \in \nabla_i \relmid \la m , n \ra= -\check{\varphi}_i(m), \forall m \in \rotatebox[origin=c]{180}{$\beta$} (F_1)^\ast \rc.
\end{align}
By these, we obtain
$
\rec \lb \sigma \cap N_\bR \rb=\cone \lb \bigcup_{i \in I_\sigma} \beta_i^\ast \lb F_1 \rb \rb.
$
By combining this, \eqref{eq:ext-s}, and \pref{th:lau}, we obtain \eqref{eq:sigma2}.
\end{proof}

\subsection{Proof of \pref{th:glcontr}}\label{sc:glcontr}

We have shown $\trop \lb X \rb=X(f_1, \cdots, f_r)$ (\pref{pr:trop-stable}) and $B^{\check{h}}_\nabla=X(f_1, \cdots, f_r)^c$ (\pref{pr:cpart}).
These imply $B^{\check{h}}_\nabla \subset \trop \lb X \rb$, i.e., \pref{th:glcontr}(1).
We will prove \pref{th:glcontr}(2) by constructing a map 
\begin{align}
\delta \colon 
X(f_1, \cdots, f_r) \to X(f_1, \cdots, f_r)^c
\end{align}
concretely.

For $\tau =\rotatebox[origin=c]{180}{$\beta$} (\mu_\tau) + \nu_\tau \in \scrP(\tilde{\Sigma}')$, we fix a vertex $v_0=\rotatebox[origin=c]{180}{$\beta$} (\mu_{v_0}) + \nu_{v_0}$ and a maximal-dimensional polyhedron $\sigma_0=\rotatebox[origin=c]{180}{$\beta$} (\mu_{\sigma_0}) + \nu_{\sigma_0} \in \scrP(\tilde{\Sigma}')$ such that $v_0 \prec \tau$ and $\sigma_0 \succ \tau$.
For $i \in \lc 1, \cdots, r \rc$, we set 
\begin{align}
e_i&:=\beta_i^\ast \lb \mu_{v_0} \rb \in \nabla_i \cap N \\
e_i^\ast&:=-M_i \lb \theta \lb \sigma_0 \rb \rb \in \Delta_i \cap M,
\end{align}
where $\theta$ is the map defined in \eqref{eq:theta2}, and $M_i \lb \theta \lb \sigma_0 \rb \rb$ is \eqref{eq:Mi} with $G=\theta \lb \sigma_0 \rb$.
These are points that satisfy $\la e_i^\ast, e_j \ra=\delta_{i, j}$ and generate a primitive sublattice in $N$ and $M$ respectively by \pref{lm:bpoints}.
We also set $M':=\bigcap_{i=1}^r e_i^\perp \subset M$ and $N':=\bigcap_{i=1}^r \lb e_i^\ast \rb ^\perp \subset N$.
Then we have $M \cong M' \oplus \lb \oplus_{i=1}^r \bZ e_i^\ast \rb$ and $N \cong N' \oplus \lb \oplus_{i=1}^r \bZ e_i \rb$.
We define
\begin{align}
\Delta_{\tau, i}&:=\rotatebox[origin=c]{180}{$\beta$}_i^\ast (\delta_{\check{h}} \lb \theta \lb \tau \rb \rb \cap \partial \nabla^\ast)+e_i^\ast \subset M_\bR \\
\Deltav_{\tau, i}&:=\beta_i^\ast \lb \mu_{\tau} \rb- e_i \subset N_\bR.
\end{align}
We also write $N_\bR':=N' \otimes_\bZ \bR$ and $M_\bR':=M' \otimes_\bZ \bR$.

\begin{lemma}\label{lm:2prop}
One has 
\begin{enumerate}
\item $\Delta_{\tau, i} \subset M_\bR'$ and $\Deltav_{\tau, i} \subset N_\bR'$ $(1 \leq i \leq r)$.
\item $\la m, n \ra=0$ for any $m \in \Delta_{\tau, i}$, $n \in \Deltav_{\tau, j}$ $(1 \leq i, j \leq r)$.
\end{enumerate}
\end{lemma}
\begin{proof}
By \eqref{eq:betat} and $e_j=\beta_j^\ast \lb \mu_{v_0} \rb \subset \beta_j^\ast \lb \mu_{\tau} \rb$, one has
\begin{align}
\la \Delta_{\tau, i}, e_j \ra
=\la \rotatebox[origin=c]{180}{$\beta$}_i^\ast (\delta_{\check{h}} \lb \theta \lb \tau \rb \rb \cap \partial \nabla^\ast)+e_i^\ast, e_j \ra=-\delta_{i,j}+\delta_{i,j}=0
\end{align}
for any $j \in \lc 1, \cdots, r \rc$.
By \eqref{eq:betat} for $\sigma_0$,
$e_j^\ast = -M_j \lb \theta \lb \sigma_0 \rb \rb \subset -\rotatebox[origin=c]{180}{$\beta$}_j^\ast (\delta_{\check{h}} \lb \theta \lb \sigma_0 \rb \rb \cap \partial \nabla^\ast)$, and $\beta_i^\ast \lb \mu_{\tau} \rb \subset \beta_i^\ast \lb \mu_{\sigma_0} \rb$, one has
\begin{align}
\la e_j^\ast, \Deltav_{\tau, i} \ra
=\la e_j^\ast, \beta_i^\ast \lb \mu_{\tau} \rb- e_i \ra=\delta_{i,j}-\delta_{i,j}=0
\end{align}
for any $j \in \lc 1, \cdots, r \rc$.
Thus we get $\Delta_{\tau, i} \subset M_\bR'$ and $\Deltav_{\tau, i} \subset N_\bR'$.
By \eqref{eq:betat} for $\tau$, we can also get
\begin{align}
\la \Delta_{\tau, i}, \Deltav_{\tau, j} \ra
=\la \rotatebox[origin=c]{180}{$\beta$}_i^\ast (\delta_{\check{h}} \lb \theta \lb \tau \rb \rb \cap \partial \nabla^\ast)+e_i^\ast, \beta_j^\ast \lb \mu_{\tau} \rb- e_j \ra=0.
\end{align}
\end{proof}

\begin{lemma}
For any polyhedron $\tau =\rotatebox[origin=c]{180}{$\beta$} (\mu_\tau) + \nu_\tau \in \scrP(\tilde{\Sigma}')$, one has 
\begin{align}
\cone \lb \mu_\tau \rb =\cone \lb \bigcup_{i=1}^r \Deltav_{\tau, i} \times \lc e_i \rc \rb,
\end{align}
where $\Deltav_{\tau, i} \times \lc e_i \rc \subset N_\bR' \times \lc e_i \rc \subset N_\bR=N_\bR' \oplus \lb \oplus_{i=1}^r \bR e_i \rb$.
\end{lemma}
\begin{proof}
We will prove
\begin{align}\label{eq:mu}
\mu_\tau= \conv \lb \bigcup_{i=1}^r \beta_i^\ast \lb \mu_\tau \rb \rb,
\end{align}
from which the lemma follows by taking $\cone$.
By \eqref{eq:beta_i}, it is obvious that the right hand side is contained in the left hand side.
We check the opposite inclusion.
Since $\mu_\tau$ is a lattice polytope, it suffices to check that all vertices of $\mu_\tau$ are contained in $\bigcup_{i=1}^r \beta_i^\ast \lb \mu_\tau \rb$.
Let $v \prec \mu_\tau$ be an arbitrary vertex.
Since $\mu_\tau \subset \partial \Delta^\ast$, \eqref{eq:lattice-pts} implies that the vertex $v$ sits in $\lb \nabla_i \cap N \setminus \lc 0 \rc \rb$ for some $i \in \lc 1, \cdots, r \rc$.
By \eqref{eq:beta_i}, we get $v \in \beta_i^\ast \lb \mu_\tau \rb$.
We obtained \eqref{eq:mu}.
\end{proof}

For every polyhedron $\tau=\rotatebox[origin=c]{180}{$\beta$} (\mu_\tau) + \nu_\tau \in \scrP(\tilde{\Sigma}')$, we set
\begin{align}\label{eq:ctau'}
C_\tau:=\cone \lb \bigcup_{i=1}^r \Deltav_{\tau, i} \times \lc e_i \rc \rb=\cone \lb \mu_\tau \rb \in \Sigma'.
\end{align}
Let $U_\tau \subset X(f_1, \cdots, f_r)^c$ denote the open star of $a_{\tau}$ in $\widetilde{\scrP}(\tilde{\Sigma}')$, i.e.,
\begin{align}
U_\tau := \bigcup_{\substack{\tau_0 \prec \tau_1 \prec \cdots \prec \tau_l, \\ l \geq 0, \tau_i \in \scrP(\tilde{\Sigma}'), \\ \tau \in \lc \tau_0, \cdots, \tau_l \rc}} \rint \lb \conv \lb \lc a_{\tau_0}, a_{\tau_1}, \cdots, a_{\tau_l} \rc \rb \rb.
\end{align}
For each face $\tau' \prec \tau$, we define
\begin{align}\label{eq:Wtautau}
W_{\tau', \tau}^\circ&:=\bigcup_{\substack{\tau' \prec \tau_1 \prec \cdots \prec \tau_l, \\ l \geq 0, \tau_i \in \scrP(\tilde{\Sigma}'), \\ \tau \in \lc \tau', \tau_1, \cdots, \tau_l \rc}} \rint \lb \conv \lb \lc a_{\tau'}, a_{\tau_1}, \cdots, a_{\tau_l} \rc \rb \rb \subset U_\tau.
\end{align}
Then one has $U_\tau =\bigsqcup_{\tau' \prec \tau} W_{\tau', \tau}^\circ$.
We also define 
\begin{align}\label{eq:Vtautau}
V_{\tau', \tau}^\circ&:=\lb W_{\tau', \tau}^\circ + \cone_\bT \lb \bigcup_{i=1}^r \beta_i^\ast \lb \mu_{\tau'} \rb \rb
 \rb 
 \cap X(f_1, \cdots, f_r) \subset X_{C_{\tau'}}(\bT) \subset X_{\Sigma'}(\bT) \\
X_\tau^\circ&:=\bigcup_{\tau' \prec \tau} V_{\tau', \tau}^\circ \subset X(f_1, \cdots, f_r) \cap X_{C_\tau}(\bT)\\
f_{\tau, i}(n)&:=\min_{m \in A_{\tau, i}} \lc \check{h}(m)+\la m, n \ra \rc,
\end{align}
where $A_{\tau, i}:= \lc 0 \rc \cup \lb M \cap \lb \Delta_{\tau, i} \times \lc -e_i^\ast \rc \rb \rb=\lc 0 \rc \cup M_i \lb \theta (\tau) \rb$.

\begin{lemma}\label{lm:domi-mono}
Let $\tau=\rotatebox[origin=c]{180}{$\beta$} (\mu_\tau) + \nu_\tau \in \scrP(\tilde{\Sigma}')$ be a polyhedron.
For $i \in \lc1, \cdots, r \rc$, we set 
\begin{align}
f_{\tau, i}'(n):=\min_{m \in A_{\tau, i}} \la m, n \ra.
\end{align}
Then the pair $\lb X_\tau^\circ, U_\tau \rb$ is isomorphic to the pair $\lb X, U_\xi \rb$ of a local model of tropical contractions associated with the stable intersection $X \lb f_{\tau, 1}', \cdots, f_{\tau, r}' \rb$ of the tropical hypersurfaces defined by $f_{\tau, i}'$ $(1 \leq i \leq r)$ in the tropical toric variety $X_{C_\tau}(\bT)$ via translation by some vector $n_0 \in N_\bR$.
\end{lemma}
\begin{proof}
First, we show that the subset $X_\tau^\circ \subset X(f_1, \cdots, f_r) $ coincides with a subset of the stable intersection $X \lb f_{\tau, 1}, \cdots, f_{\tau, r} \rb$ of the tropical hypersurfaces $X(f_{\tau, i})$ $(1 \leq i \leq r)$ in the tropical toric variety $X_{C_\tau}(\bT)$.
It suffices to show that all the indices of monomials of $f_i$ that attain the minimum of $f_i$ at some point in $X_\tau^\circ$ are contained in $A_{\tau, i}$.
Recall that $A_{\tau, i}=\lc 0 \rc \cup M_i \lb \theta (\tau) \rb \subset \Delta_i \cap M$ is the set of the indices of monomials of $f_i$ that attain the minimum of $f_i$ on $\theta \lb \tau \rb$.
Since any point in $U_\tau$ sits in $\rint (F)$ for some $F \in \scrP_{\nabla^{\check{h}}}$ such that $F \succ \theta \lb \tau \rb$, 
all the indices of monomials of $f_i$ that attain the minimum of $f_i$ at some point in $U_\tau$ are contained in $A_{\tau, i}$.
Furthermore, since we have $\la \Delta_{\tau, i} -e_i^\ast, \beta_j^\ast \lb \mu_\tau \rb \ra=-\delta_{i, j}$ by \eqref{eq:betat}, and 
$\la \Delta_i, \beta_j^\ast \lb \mu_\tau \rb \ra \geq -\delta_{i, j}$ by \eqref{eq:nabla}, it turns out that all the indices of monomials of $f_i$ that attain the minimum at some point in $X_\tau^\circ \setminus U_\tau$ are also contained in $A_{\tau, i}$.

Next, we show that the stable intersection $X \lb f_{\tau, 1}, \cdots, f_{\tau, r} \rb$ is isomorphic to the stable intersection $X \lb f_{\tau, 1}', \cdots, f_{\tau, r}' \rb$.
Since $\check{h}$ is linear on every cone in $\Sigmav'$ and $\delta_{\check{h}} \lb \theta \lb \tau \rb \rb \in \Sigmav'$, there exists $n_0 \in N_\bR$ such that $\check{h}(m)=\la m, n_0 \ra$ for any $m \in \delta_{\check{h}} \lb \theta \lb \tau \rb \rb$.
Since $\delta_{\check{h}} \lb \theta \lb \tau \rb \rb \supset A_{\tau, i}$, one has
\begin{align}
f_{\tau, i}(n)=\min_{m \in A_{\tau, i}} \lc \la m, n_0 \ra +\la m, n \ra \rc=\min_{m \in A_{\tau, i}} \la m, n+n_0 \ra=f_{\tau, i}'(n+n_0)
\end{align}
for any $i \in \lc 1, \cdots, r \rc$.
Hence, the stable intersections $X \lb f_{\tau, 1}, \cdots, f_{\tau, r} \rb$ and $X \lb f_{\tau, 1}', \cdots, f_{\tau, r}' \rb$ are isomorphic via translation $+n_0$.

We set $X:=X_\tau^\circ+n_0$ and $U_\xi:=U_\tau+n_0$.
Since $U_\tau \subset X(f_1, \cdots, f_r)^c$, the monomial $0$ attains the minimum of every polynomial $f_{i}$ $(1 \leq i \leq r)$ on $U_\tau$.
Hence, the monomial $0$ also attains the minimum of every polynomial $f_{\tau, i}'$ $(1 \leq i \leq r)$ on $U_\xi$.
Two properties stated in \pref{lm:2prop} are the ones that we imposed for polytopes $\Delta_i$ and $\Deltav_i$ when we constructed local models of tropical contractions.
Furthermore, since 
$\tau=\rotatebox[origin=c]{180}{$\beta$} (\mu_\tau) + \nu_\tau =\sum_{i=1}^r \beta_i^\ast \lb \mu_{\tau} \rb +\nu_\tau$ and $\Deltav_{\tau, i}:=\beta_i^\ast \lb \mu_{\tau} \rb- e_i$,
the normal fan of $\tau \in \scrP(\tilde{\Sigma}')$ is a refinement of the normal fan of $\Deltav_{\tau, i}$.
This is the condition that we imposed for $\xi$ in \pref{cd:refinement}.
We have the map \eqref{eq:phi} for $\tau, \Deltav_{\tau, i}$
\begin{align}
\phi_{\tau, \Deltav_{\tau, i}} \colon \scrP_{\tau} \to \scrP_{\Deltav_{\tau, i}}.
\end{align}
By \pref{lm:minksum}, it turns out that for all $\tau' =\rotatebox[origin=c]{180}{$\beta$} (\mu_{\tau'}) + \nu_{\tau'} \in \scrP_{\tau}$, one has 
\begin{align}\label{eq:phi-tau}
\phi_{\tau, \Deltav_{\tau, i}}  \lb \tau' \rb = \beta_i^\ast \lb \mu_{\tau'} \rb-e_i=\Deltav_{\tau', i}.
\end{align}
From these and the construction of $U_\tau$ and $X_\tau^\circ$, we can see that $\lb X, U_\xi \rb$ is a pair of a local model of tropical contractions, and the pair $\lb X_\tau^\circ, U_\tau \rb$ is isomorphic to it via translation by $n_0 \in N_\bR$.
\end{proof}

By \pref{lm:domi-mono}, there is a map
\begin{align}\label{eq:del-tau}
\delta_\tau \colon X_\tau^\circ \to U_\tau
\end{align}
that is isomorphic to a local model of tropical contractions via the translation $+n_0 \in N_\bR$.
Its restriction to $V_{\tau', \tau}^\circ$ is the composition
\begin{align}\label{eq:del-tau'}
V_{\tau', \tau}^\circ \hookrightarrow X_{C_{\tau'}}(\bT) \xrightarrow{\pi_{C_{\tau'}}} \pi_{C_{\tau'}}\lb W_{\tau', \tau}^\circ \rb \xrightarrow{t_{\tau'}} W_{\tau', \tau}^\circ,
\end{align}
where $t_{\tau'} \colon \pi_{C_{\tau'}} \lb W_{\tau', \tau}^\circ \rb \to W_{\tau', \tau}^\circ$ is a map such that 
$t_{\tau'} \circ \pi_{C_{\tau'}}$ is the identity map on $W_{\tau', \tau}^\circ$.
The fan structures \eqref{eq:fanstr} that are constructed for the local model $U_\xi$ also agree with the fan structures defined in \pref{eq:gfanstr}, since $\phi_{\tau, \Deltav_{\tau, i}}  \lb v \rb +e_i= \beta_i^\ast \lb \mu_{v} \rb$ for every vertex $v$ of $\tau$ by \eqref{eq:phi-tau}.

We set $V_{\tau}^\circ:=V_{\tau, \tau}^\circ$ and $W_{\tau}^\circ:=W_{\tau, \tau}^\circ$ for short, and write their closures as $V_{\tau}$ and $W_{\tau}$ respectively.
They are 
\begin{align}\label{eq:wtau}
W_{\tau}&=\bigcup_{\substack{\tau \prec \tau_1 \prec \cdots \prec \tau_l, \\ l \geq 0, \tau_i \in \scrP(\tilde{\Sigma}')}} \conv \lb \lc a_\tau, a_{\tau_1}, \cdots, a_{\tau_l} \rc \rb
=\bigcup_{\substack{\tau \prec \tau_1 \prec \cdots \prec \tau_l, \\ l \geq 0, \tau_i \in \scrP(\tilde{\Sigma}')}} \rint \lb \conv \lb \lc a_{\tau_1}, \cdots, a_{\tau_l} \rc \rb \rb  \\
V_{\tau}&=\lb W_{\tau} + \cone_\bT \lb \bigcup_{i=1}^r \beta_i^\ast \lb \mu_{\tau} \rb \rb
 \rb 
 \cap X(f_1, \cdots, f_r) \subset X_{C_{\tau}}(\bT) \subset X_{\Sigma'}(\bT).
\end{align}
In the union in \eqref{eq:wtau}, $\tau_1$ may coincide with $\tau$.

\begin{lemma}\label{lm:vtau2}
Let $\sigma \in \scrP_{X(f_1, \cdots, f_r)}$ be a polyhedron such that $\sigma \cap N_\bR \neq \emptyset$, 
and $\tilde{\sigma} \in \scrP^{\check{h}}_{\Delta^\ast}=\scrP_{X(f_1, \cdots, f_r)^c}$ be the polyhedron of \pref{pr:t-sigma} for $\sigma$.
For any polyhedron $\tau \in \scrP(\tilde{\Sigma}')$, one has
\begin{align}\label{eq:vtaus}
V_{\tau}^\circ \cap \sigma
&=
\left\{
\begin{array}{ll}
\lb W_{\tau}^\circ \cap \tilde{\sigma} \rb + \cone_\bT \lb \bigcup_{i \in I_\sigma} \beta_i^\ast \lb \mu_{\tau} \rb \rb & \tau \subset \tilde{\sigma} \\
\emptyset & \mathrm{otherwise} \\
\end{array}
\right. \\ \label{eq:vtaus'}
V_{\tau} \cap \sigma
&=
\left\{
\begin{array}{ll}
\lb W_{\tau} \cap \tilde{\sigma} \rb + \cone_\bT \lb \bigcup_{i \in I_\sigma} \beta_i^\ast \lb \mu_{\tau} \rb \rb & \tau \subset \tilde{\sigma} \\
\emptyset & \mathrm{otherwise,} \\
\end{array}
\right.
\end{align}
where $I_\sigma$ is the one defined in \eqref{eq:Isig'}.
\end{lemma}
\begin{proof}
We write $\tilde{\sigma}=\rotatebox[origin=c]{180}{$\beta$} (F_1)+F_2$ with $\lb F_1, F_2 \rb \in \scrR^{\check{h}}_{\Delta^\ast}$.
When $\tau$ is not contained in $\tilde{\sigma}$, there exist some $j \in \lc1, \cdots, r \rc$ and $m_j \in \lb \Delta_j \cap M \rb \setminus \lc 0 \rc$ such that $\check{h} \lb m_j \rb+\la m_j, \bullet \ra$ attains the minimum of $f_i$ on $\tilde{\sigma}$, but not on $\rint \lb \tau \rb$.
We can see that $\check{h} \lb m_j \rb+\la m_j, \bullet \ra$ does not attain the minimum of $f_i$ also on $W_{\tau}$.
By \pref{pr:t-sigma}, we can also see that $\check{h} \lb m_j \rb+\la m_j, \bullet \ra$ attains the minimum on $\sigma$.
On the other hand, we have 
\begin{align}\label{eq:MbDb}
\la M_j \lb \theta (\tau) \rb, \beta_i^\ast \lb \mu_{\tau} \rb \ra=-\delta_{i, j}, \quad
\la \Delta_j, \beta_i^\ast \lb \mu_{\tau} \rb \ra \geq -\delta_{i, j}
\end{align}
by \eqref{eq:betat} and \eqref{eq:nabla} respectively.
By these, one can see that $\check{h} \lb m_j \rb+\la m_j, \bullet \ra$ does not attain the minimum of $f_i$ also on $V_{\tau}$.
Therefore, one can conclude $V_{\tau} \cap \sigma \neq \emptyset$ and $V_{\tau}^\circ \cap \sigma \neq \emptyset$.

Next, suppose $\tau \subset \tilde{\sigma}$.
Then one has $\mu_{\tau} \subset F_1$.
By taking the closure of \pref{eq:sigma2}, we obtain
\begin{align}
\sigma = \tilde{\sigma} + \cone_\bT \lb \bigcup_{i \in I_\sigma} \beta_i^\ast \lb F_1 \rb \rb.
\end{align}
We can see from this that the right hand sides of \eqref{eq:vtaus} and \eqref{eq:vtaus'} are contained in the left hand sides.
We show the opposite inclusions.
Since the monomial $0$ attains the minimum of the polynomial $f_i$ on $\sigma$ for $i \nin I_\sigma$,
we can see from \eqref{eq:MbDb} that we have
\begin{align}
V_{\tau}^\circ \cap \sigma \subset W_{\tau}^\circ + \cone_\bT \lb \bigcup_{i \in I_\sigma} \beta_i^\ast \lb \mu_{\tau} \rb \rb, \quad
V_{\tau} \cap \sigma \subset W_{\tau} + \cone_\bT \lb \bigcup_{i \in I_\sigma} \beta_i^\ast \lb \mu_{\tau} \rb \rb.
\end{align}
Every element $n_0 \in V_{\tau}^\circ \cap \sigma$ (resp. $n_0 \in V_{\tau} \cap \sigma$) can be written as $n_0=n_1+n_2$ with $n_1 \in W_{\tau}^\circ$ (resp. $n_1 \in W_{\tau}$) and $n_2 \in \cone_\bT \lb \bigcup_{i \in I_\sigma} \beta_i^\ast \lb \mu_{\tau} \rb \rb$.
We will show $n_1 \in \tilde{\sigma}$ by checking that the set of the indices of monomials of $f_i$ attaining the minimum at $n_1$ contains $\scrM_i \lb \tilde{\sigma} \rb=\scrM_i \lb \sigma \rb \cup \lc 0 \rc$.
It is obvious that the monomial $0$ attains the minimum at $n_1$.
We can also see again from \eqref{eq:MbDb} that adding the element $n_2 \in \cone_\bT \lb \bigcup_{i \in I_\sigma} \beta_i^\ast \lb \mu_{\tau} \rb \rb$ to $n_1$ does not change the set of monomials of $f_i$ attaining the minimum except for the monomial $0$.
Since $n_0 \in \sigma$, the set of the indices of monomials of $f_i$ attaining the minimum at $n_0$ contains $\scrM_i \lb \sigma \rb$.
Hence, the set of the indices of monomials of $f_i$ attaining the minimum at $n_1$ also contains $\scrM_i \lb \sigma \rb$.
Thus we obtain $n_1 \in \tilde{\sigma}$, and can conclude that in \eqref{eq:vtaus} and \eqref{eq:vtaus'}, the left hand sides are also contained in the right hand sides.
\end{proof}

\begin{lemma}\label{lm:restriction2}
When $X_{\tau_1}^\circ \cap X_{\tau_2}^\circ \neq \emptyset$ for polyhedra $\tau_1, \tau_2 \in \scrP(\tilde{\Sigma}')$, the restrictions to $X_{\tau_1}^\circ \cap X_{\tau_2}^\circ$ of the contractions $\delta_{\tau_1} \colon X_{\tau_1}^\circ \to U_{\tau_1}$, $\delta_{\tau_2} \colon X_{\tau_2}^\circ \to U_{\tau_2}$ of \eqref{eq:del-tau} coincide with each other.
\end{lemma}
\begin{proof}
One has $X_\tau^\circ=\bigcup_{\tau' \prec \tau} V_{\tau', \tau}^\circ$, and the restriction of the contraction $V_{\tau'}^\circ=V_{\tau', \tau'}^\circ \to W_{\tau', \tau'}^\circ=W_{\tau'}^\circ$ \eqref{eq:del-tau'} 
to $V_{\tau', \tau}^\circ \subset V_{\tau', \tau'}^\circ$ coincides with the contraction $V_{\tau', \tau}^\circ \to W_{\tau', \tau}^\circ$ \eqref{eq:del-tau'}.
Hence, it suffices to show that the contractions $\delta_{\tau_1}, \delta_{\tau_2}$ coincide on $V_{\tau_1}^\circ \cap V_{\tau_2}^\circ$.
Recall that local models of tropical contractions were originally constructed by gluing maps \eqref{eq:delt-tau0} between closed subsets $V_\tau, W_\tau$.
The contraction $\left. \delta_{\tau_j} \right|_{V_{\tau_j}^\circ}$ $(j=1, 2)$ naturally extends to a map $V_{\tau_j} \to W_{\tau_j}$ so that it agrees with \eqref{eq:delt-tau0}.
We write the extension $V_{\tau_j} \to W_{\tau_j}$ also as $\delta_{\tau_j}$.
We will show that $\delta_{\tau_1}$ and $\delta_{\tau_2}$ coincide on  $V_{\tau_1} \cap V_{\tau_2}$.
We have
\begin{align}
V_{\tau_1} \cap V_{\tau_2} = \bigcup_{\sigma} \bigcup_{C \in \Sigma'} V_{\tau_1} \cap V_{\tau_2} \cap \sigma \cap O_C(\bT),
\end{align}
where the union of $\sigma$ is taken over $\sigma \in \scrP_{X(f_1, \cdots, f_r)}$ such that $\sigma \cap N_\bR \neq \emptyset$.
We will show $\delta_{\tau_1}=\delta_{\tau_2}$ on $V_{\tau_1} \cap V_{\tau_2} \cap \sigma \cap O_C(\bT)$.
When $I_\sigma =\emptyset$, we have $\tilde{\sigma}=\sigma$, and $V_{\tau_1} \cap V_{\tau_2} \cap \sigma= W_{\tau_1} \cap W_{\tau_2} \cap \sigma$ by \pref{lm:vtau2}.
Since both $\delta_{\tau_1}$ and $\delta_{\tau_2}$ are the identity on there, we get $\delta_{\tau_1}=\delta_{\tau_2}$ on $V_{\tau_1} \cap V_{\tau_2} \cap \sigma$.
We suppose $I_\sigma \neq \emptyset$ in the following.
Since $V_{\tau_j} \cap \sigma= \emptyset$ unless $\tau_j \subset \tilde{\sigma}$ $(j=1, 2)$ by \pref{lm:vtau2}, we also assume $\theta (\tau_j) \prec \tilde{\sigma}$ $(j=1, 2)$ in the following.

We set $\scrP \ld \tilde{\sigma} \rd:= \lc \tau \in \scrP(\tilde{\Sigma}') \relmid \tau \subset \tilde{\sigma} \rc$.
Then we have $\bigcup_{\tau \in \scrP \ld \tilde{\sigma} \rd} \tau=\tilde{\sigma}$.
For $\tau \in \scrP \ld \tilde{\sigma} \rd$, we have $W_{\tau_j} \cap \tau = \emptyset$ unless $\tau \succ \tau_j$.
Hence, by \pref{lm:vtau2}, we have
\begin{align}
V_{\tau_j} \cap \sigma \cap O_C(\bT)&=\lc \lb W_{\tau_j} \cap \tilde{\sigma} \rb + \cone_\bT \lb \bigcup_{i \in I_\sigma} \beta_i^\ast \lb \mu_{\tau_j} \rb \rb \rc \cap O_C(\bT) \\ \label{eq:v-tau-sigma}
&=\bigcup_{\substack{\tau \in \scrP \ld \tilde{\sigma} \rd \\ \tau \succ \tau_j}} \lc \lb W_{\tau_j} \cap \tau \rb+ \cone_\bT \lb \bigcup_{i \in I_\sigma} \beta_i^\ast \lb \mu_{\tau_j} \rb \rb \rc \cap O_C(\bT).
\end{align}
We can see from \eqref{eq:O-emb} that this is non-empty if and only if
$\cone \lb \bigcup_{i \in I_\sigma} \beta_i^\ast \lb \mu_{\tau_j} \rb \rb \supset C$.
When this holds, \eqref{eq:v-tau-sigma} is equal to
\begin{align}
\bigcup_{\substack{\tau \in \scrP \ld \tilde{\sigma} \rd \\ \tau \succ \tau_j}} \lc \lb W_{\tau_j} \cap \tau \rb
+ \cone_\bT(C)
+\sum_{i \in I_\sigma} \bigcup_{k_{i, j} \in \bR_{\geq 0}} k_{i, j} \cdot \beta_i^\ast \lb \mu_{\tau_j} \rb \rc \cap O_C(\bT).
\end{align}
We try to show $\delta_{\tau_1} = \delta_{\tau_2}$ on the intersection
\begin{align}\label{eq:intersect2}
\bigcap_{j=1, 2} 
\lc \lb W_{\tau_j} \cap \tilde{\tau}_j \rb
+ \cone_\bT(C)
+\sum_{i \in I_\sigma} k_{i, j} \cdot \beta_i^\ast \lb \mu_{\tau_j} \rb \rc \cap O_C(\bT).
\end{align}
for polyhedra $\tilde{\tau}_j \in \scrP \ld \tilde{\sigma} \rd$ such that $\tilde{\tau}_j \succ \tau_j$ and $k_{i, j}\in \bR_{\geq 0}$ $(j=1, 2)$.
When there is a polyhedron $\xi \in \scrP(\tilde{\Sigma}')$ such that $\tau_1, \tau_2 \prec \xi \prec \tilde{\tau}_1, \tilde{\tau}_2$, it is essentially proved in \pref{lm:intersection} 
since $\overline{X_{\tau=\xi}^\circ}$ which contains \eqref{eq:intersect2} is isomorphic to (the closure of) a local model of tropical contractions by \pref{lm:domi-mono}, and we have $\beta_i^\ast \lb \mu_{\tau_j} \rb=\phi_{\xi, \Deltav_{\xi, i}} (\tau_j) \times \lc e_i \rc$ by \eqref{eq:phi-tau}.
In particular, we already know $\delta_{\tau_1} = \delta_{\tau_2}$ on \eqref{eq:intersect2} when $\tilde{\tau}_1=\tilde{\tau}_2$.

We consider the case where $\tilde{\tau}_1 \neq \tilde{\tau}_2$.
We write them as $\tilde{\tau}_j =\rotatebox[origin=c]{180}{$\beta$} (\mu_{\tilde{\tau}_j}) + \nu_{\tilde{\tau}_j}$.
We set
\begin{align}
C_j:= \cone \lb \mu_{\tilde{\tau_j}} \rb \times \lc 0 \rc + \cone \lb \nu_{\tilde{\tau_j}} \times \lc 1 \rc \rb.
\end{align}
They are cones in $\tilde{\Sigma}' \subset N_\bR \oplus \bR$.
There exists a linear function $\tilde{\phi} \colon N_\bR \oplus \bR \to \bR$ such that
$\tilde{\phi} >0$ on $C_1 \setminus \lb C_1 \cap C_2 \rb$, 
$\tilde{\phi} <0 $ on $C_2 \setminus \lb C_1 \cap C_2 \rb$, 
and $\tilde{\phi} = 0$ on $C_1 \cap C_2$.
We define $\phi \colon N_\bR \to \bR$ by $\phi(n):=\tilde{\phi} \lb (n, 0) \rb$ and $l :=\phi \lb (0, 1)\rb \in \bR$.
Since $\tilde{\phi} = 0$ on $C \subset \cap_{j=1, 2} \cone \lb \bigcup_{i \in I_\sigma} \beta_i^\ast \lb \mu_{\tau_j} \rb \rb \subset C_1 \cap C_2$, we have $\phi=0$ on $\vspan (C)$.
Hence, $\phi$ can be extended to a function on $X_C(\bT) \supset O_C \lb \bT \rb$.
One has
\begin{align}
\phi(n)=
\left\{ \begin{array}{ll}
\mathrm{positive}  & n \in \mu_{\tilde{\tau}_1} \setminus \lb \mu_{\tilde{\tau}_1} \cap \mu_{\tilde{\tau}_2} \rb \\
\mathrm{negative} & n \in \mu_{\tilde{\tau}_2} \setminus \lb \mu_{\tilde{\tau}_1} \cap \mu_{\tilde{\tau}_2} \rb \\
 0 & n \in \mu_{\tilde{\tau}_1} \cap \mu_{\tilde{\tau}_2},
\end{array} 
\right.
\quad
 \phi(n)=
\left\{ \begin{array}{ll}
\mathrm{greater\ than\ } -l  & n \in \nu_{\tilde{\tau}_1} \setminus \lb \nu_{\tilde{\tau}_1} \cap \nu_{\tilde{\tau}_2} \rb \\
\mathrm{less\ than\ } -l & n \in \nu_{\tilde{\tau}_2} \setminus \lb \nu_{\tilde{\tau}_1} \cap \nu_{\tilde{\tau}_2} \rb \\
-l & n \in \nu_{\tilde{\tau}_1} \cap \nu_{\tilde{\tau}_2}.
\end{array} 
\right.
\end{align}
Since $\tilde{\tau}_j =\lb \sum_{i=1}^r \beta_i^\ast (\mu_{\tilde{\tau_j}}) \rb + \nu_{\tilde{\tau}_j}$ and $\tilde{\tau}_1 \cap \tilde{\tau}_2=\lb \sum_{i=1}^r \beta_i^\ast \lb \mu_{\tilde{\tau}_1} \cap \mu_{\tilde{\tau}_2} \rb \rb+\lb \nu_{\tilde{\tau}_1} \cap \nu_{\tilde{\tau}_2} \rb$, we can see
\begin{align}\label{eq:phi-n1}
\phi(n)=
\left\{ \begin{array}{ll}
\mathrm{greater\ than\ } -l  & n \in \tilde{\tau}_1 \setminus \lb \tilde{\tau}_1 \cap \tilde{\tau}_2 \rb \\
\mathrm{less\ than\ } -l & n \in \tilde{\tau}_2 \setminus  \lb \tilde{\tau}_1 \cap \tilde{\tau}_2 \rb \\
-l & n \in \tilde{\tau}_1 \cap \tilde{\tau}_2.
\end{array} 
\right.
\end{align}
We also have
\begin{align}\label{eq:phi-n2}
\phi(n)=
\left\{ \begin{array}{ll}
\mathrm{greater\ than\ or\ equal\ to\ } 0 & n \in k_i \cdot \beta_i^\ast \lb \mu_{\tilde{\tau}_1} \rb \\
\mathrm{less\ than\ or\ equal\ to\ } 0 & n \in k_i \cdot \beta_i^\ast \lb \mu_{\tilde{\tau}_2} \rb.
\end{array} 
\right.
\end{align}
We can see from \eqref{eq:phi-n1} and \eqref{eq:phi-n2} that the intersection \pref{eq:intersect2} must be contained in
\begin{align}
\bigcap_{j=1, 2} 
\lc \lb W_{\tau_j} \cap \tilde{\tau}_1 \cap \tilde{\tau}_2 \rb
+ \cone_\bT(C)
+\sum_{i \in I_\sigma} k_{i, j} \cdot \beta_i^\ast \lb \mu_{\tau_j} \rb \rc \cap O_C(\bT).
\end{align}
We reduced to the case of $\tilde{\tau}_1 = \tilde{\tau}_2$.
We obtained the lemma.
\end{proof}

By \pref{lm:restriction2}, it turns out that we can construct the map
\begin{align}\label{eq:cy-contra}
\delta \colon \bigcup_{\tau \in \scrP(\tilde{\Sigma}')} X_\tau^\circ \to X(f_1, \cdots, f_r)^c
\end{align}
by gluing the maps $\delta_{\tau} \colon X_{\tau}^\circ \to U_{\tau}$ together.
By proving the following lemma, we can see that the map \eqref{eq:cy-contra} gives the tropical contraction of \pref{th:glcontr}.

\begin{lemma}
One has $X(f_1, \cdots, f_r)=\bigcup_{\tau \in \scrP(\tilde{\Sigma}')} X_\tau^\circ$.
\end{lemma}
\begin{proof}
It is obvious that the right hand side is contained in the left hand side.
Since $V_{\tau', \tau}^\circ \subset V_{\tau', \tau'}^\circ$ for $\tau' \prec \tau$, one has 
\begin{align}
\bigcup_{\tau \in \scrP(\tilde{\Sigma}')} X_\tau^\circ
=\bigcup_{\tau \in \scrP(\tilde{\Sigma}')} \bigcup_{\tau' \prec \tau} V_{\tau', \tau}^\circ
=\bigcup_{\tau \in \scrP(\tilde{\Sigma}')} \bigcup_{\tau' \prec \tau} V_{\tau', \tau'}^\circ
=\bigcup_{\tau \in \scrP(\tilde{\Sigma}')} V_{\tau}^\circ.
\end{align}
Hence, it suffices to show $\sigma \subset \bigcup_{\tau \in \scrP(\tilde{\Sigma}')} V_\tau^\circ \cap \sigma$ for all $\sigma \in \scrP_{X(f_1, \cdots, f_r)}$ such that $\sigma \cap N_\bR \neq \emptyset$.
By \pref{lm:vtau2}, we have 
\begin{align}
\bigcup_{\tau \in \scrP(\tilde{\Sigma}')} V_\tau^\circ \cap \sigma
&=\bigcup_{\tau \in \scrP \ld \tilde{\sigma} \rd} \lc \lb W_{\tau}^\circ \cap \tilde{\sigma} \rb + \cone_\bT \lb \bigcup_{i \in I_\sigma} \beta_i^\ast \lb \mu_\tau \rb \rb \rc \\
&=
\bigcup_{\tau \in \scrP \ld \tilde{\sigma} \rd}
\bigcup_{\substack{\tau' \in \scrP \ld \tilde{\sigma} \rd \\ \tau' \succ \tau}} 
\lc 
\lb W_{\tau}^\circ \cap \tau' \rb+ 
 \cone_\bT \lb \bigcup_{i \in I_\sigma} \beta_i^\ast \lb \mu_\tau \rb \rb
\rc,
\end{align}
where $\scrP \ld \tilde{\sigma} \rd:= \lc \tau \in \scrP(\tilde{\Sigma}') \relmid \tau \subset \tilde{\sigma} \rc$.
This is equal to
\begin{align}\label{eq:tech}
\bigcup_{\tau \in \scrP \ld \tilde{\sigma} \rd}
\bigcup_{\substack{\tau' \in \scrP \ld \tilde{\sigma} \rd \\ \tau' \succ \tau}}  
\lc \lb W_{\tau} \cap \tau' \rb + \cone_\bT \lb \bigcup_{i \in I_\sigma} \beta_i^\ast \lb \mu_\tau \rb \rb \rc
\end{align}
since one has
\begin{align}\label{eq:wboundary}
\lb W_{\tau} \cap \tau' \rb \setminus \lb W_{\tau}^\circ \cap \tau' \rb=\bigcup_{\substack{\tau \prec \tau_1 \prec \cdots \prec \tau_l \prec \tau', \\ \tau_i \in \scrP(\tilde{\Sigma}'), \tau_1 \neq \tau}}
\rint \lb \conv \lb \lc a_{\tau_1}, \cdots, a_{\tau_l} \rc \rb \rb
\end{align}
by \eqref{eq:wtau} and $\rint \lb \conv \lb \lc a_{\tau_1}, \cdots, a_{\tau_l} \rc \rb \rb$ in \pref{eq:wboundary} is contained in $W_{\tau_1}^\circ \cap \tau'$, and $\beta_i^\ast \lb \mu_\tau \rb \subset \beta_i^\ast \lb \mu_{\tau_1} \rb$.
Hence, one has
\begin{align}
\bigcup_{\tau \in \scrP(\tilde{\Sigma}')} V_\tau^\circ \cap \sigma \cap N_\bR
&=
\bigcup_{\tau \in \scrP \ld \tilde{\sigma} \rd}
\bigcup_{\substack{\tau' \in \scrP \ld \tilde{\sigma} \rd \\ \tau' \succ \tau}}  
\lc 
\lb W_{\tau} \cap \tau' \rb+ \cone \lb \bigcup_{i \in I_\sigma} \beta_i^\ast \lb \mu_\tau \rb \rb \rc \\ \label{eq:tech'}
&=
\bigcup_{\tau \in \scrP \ld \tilde{\sigma} \rd}
\bigcup_{\substack{\tau' \in \scrP \ld \tilde{\sigma} \rd \\ \tau' \succ \tau}}
\bigcup_{k_i \in \bR_{\geq 0}}
\lc
\lb W_{\tau} \cap \tau' \rb+ 
\sum_{i \in I_\sigma}  
k_i \cdot \beta_i^\ast \lb \mu_\tau \rb
\rc.
\end{align}
We use \pref{lm:tech}.2 for this.
Substitute $\Deltav_0=\Deltav=\tau'$, $I=\lc i \in I_\sigma \relmid k_i \neq 0 \rc$, $\Deltav_i=k_i \cdot \beta_i^\ast \lb \mu_{\tau'} \rb$, and $t=1$ in \pref{lm:tech}.
It turns out that \eqref{eq:tech'} is equal to 
\begin{align}\label{eq:srhs}
\bigcup_{k_i \in \bR_{\geq 0}} 
\bigcup_{\tau' \in \scrP \ld \tilde{\sigma} \rd}
\lb 
\tau' + \sum_{i \in I_\sigma} k_i \cdot \beta_i^\ast \lb \mu_{\tau'}\rb 
\rb.
\end{align}
Let $\lb F_1, F_2 \rb \in \scrR^{\check{h}}_{\Delta^\ast}$ be the element such that $\tilde{\sigma}=\rotatebox[origin=c]{180}{$\beta$} (F_1)+F_2$.
Then by \pref{eq:sigma2}, we have
\begin{align}\label{eq:slhs}
\sigma \cap N_\bR = \tilde{\sigma} + \cone \lb \bigcup_{i \in I_\sigma} \beta_i^\ast \lb F_1 \rb \rb
=\bigcup_{k_i \in \bR_{\geq 0}} \lb \tilde{\sigma} + \sum_{i \in I_\sigma} k_i \cdot \beta_i^\ast \lb F_1 \rb \rb.
\end{align}
If we show
\begin{align}\label{eq:last}
\bigcup_{\tau' \in \scrP \ld \tilde{\sigma} \rd} 
\lb 
\tau' + \sum_{i \in I_\sigma} k_i \cdot \beta_i^\ast \lb \mu_{\tau'}\rb 
\rb
=\tilde{\sigma} + \sum_{i \in I_\sigma} k_i \cdot \beta_i^\ast \lb F_1 \rb,
\end{align}
then it implies $\sigma \cap N_\bR = \bigcup_{\tau \in \scrP \ld \tilde{\sigma} \rd} V_\tau^\circ \cap \sigma \cap N_\bR$.
Since $\bigcup_{\tau \in \scrP \ld \tilde{\sigma} \rd} V_\tau^\circ \cap \sigma$ which is equal to \eqref{eq:tech} is closed, one can get  
\begin{align}
\sigma 
= \overline{\sigma \cap N_\bR}
=\overline{\bigcup_{\tau \in \scrP \ld \tilde{\sigma} \rd} V_\tau^\circ \cap \sigma \cap N_\bR} 
\subset \bigcup_{\tau \in \scrP \ld \tilde{\sigma} \rd} V_\tau^\circ \cap \sigma
\end{align}
and conclude $\sigma \subset \bigcup_{\tau \in \scrP(\tilde{\Sigma}')} V_\tau^\circ \cap \sigma$.
We will try to show \eqref{eq:last} in the following.

It is obvious that the left hand side of \eqref{eq:last} is contained in the right hand side.
We will prove the opposite inclusion.
Take an arbitrary point $n \in \tilde{\sigma} + \sum_{i \in I_\sigma} k_i \beta_i^\ast \lb F_1 \rb$.
Since 
\begin{align}\label{eq:til-sigma2}
\tilde{\sigma} + \sum_{i \in I_\sigma} k_i \cdot \beta_i^\ast \lb F_1 \rb
=\sum_{i \in I_\sigma} (k_i+1) \cdot \beta_i^\ast \lb F_1 \rb 
+\sum_{i \nin I_\sigma} \beta_i^\ast \lb F_1 \rb
+F_2,
\end{align}
we have $(n, 1) \in \cone (F_1) \times \lc 0 \rc + \cone (F_2 \times \lc 1 \rc) \in \tilde{\Sigma}$.
Since $\tilde{\Sigma}'$ is a refinement of $\tilde{\Sigma}$, there exists a polyheron $\tau_0=\rotatebox[origin=c]{180}{$\beta$} (\mu_{\tau_0}) + \nu_{\tau_0} \in \scrP \ld \tilde{\sigma} \rd$ such that $(n, 1) \in \cone (\mu_{\tau_0}) \times \lc 0 \rc + \cone (\nu_{\tau_0} \times \lc 1 \rc) \in \tilde{\Sigma}'$.
We will check 
$
n \in \tau_0 + \sum_{i \in I_\sigma} k_i \cdot \beta_i^\ast \lb \mu_{\tau_0}\rb,
$
from which we can get the inclusion that we want.

By \eqref{eq:til-sigma2}, the element $(n, 1)$ can be written as $(n, 1)=(n_0, 0)+(n_1, 1)$ with $n_0 \in \sum_{i \in I_\sigma} (k_i+1) \cdot \beta_i^\ast \lb F_1 \rb 
+\sum_{i \nin I_\sigma} \beta_i^\ast \lb F_1 \rb$ and $n_1 \in F_2$.
Since $\tilde{\varphi}_i \lb (n, 1)\rb=\tilde{\varphi}_i \lb (n_0, 0)\rb+\tilde{\varphi}_i \lb (n_1, 1)\rb=\varphi_i \lb n_0 \rb$, we have
\begin{align}
\tilde{\varphi}_i \lb (n, 1)\rb
=
\left\{ \begin{array}{ll}
    k_i + 1 & i \in I_\sigma \\
    1 & i \nin I_\sigma.
  \end{array} 
\right.
\end{align}
We can also write $(n,1)=(n', 0)+(n'', 1)$ with $n' \in \cone (\mu_{\tau_0})$ and $n'' \in \nu_{\tau_0}$.
Since
$\tilde{\varphi}_i \lb (n, 1)\rb=\tilde{\varphi}_i \lb (n', 0)\rb+\tilde{\varphi}_i \lb (n'', 1)\rb=\varphi_i \lb n' \rb$, we have
\begin{align}\label{eq:varphi1}
\varphi_i \lb n' \rb
=
\left\{ \begin{array}{ll}
    k_i + 1 & i \in I_\sigma \\
    1 & i \nin I_\sigma.
  \end{array} 
\right.
\end{align}
Since $n' \in \cone(\mu_{\tau_0})$, we can also write $n'=\sum_{j} r_j n_j$, where $n_j$ is a vertex of $\mu_{\tau_0}$ and $r_j \in \bR_{>0}$.
When $n_j \in \beta_k^\ast \lb \mu_{\tau_0} \rb$, we have $\varphi_i(n_j)=\delta_{i, k}$.
Hence, one has
\begin{align}\label{eq:varphi2}
\varphi_i \lb n' \rb=\sum_{\substack{j \ \mathrm{s.t.} \\ n_j \in \beta_i^\ast \lb \mu_{\tau_0} \rb}} r_j.
\end{align}
By \eqref{eq:varphi1} and \eqref{eq:varphi2}, we get
\begin{align}
\sum_{\substack{j \ \mathrm{s.t.} \\ n_j \in \beta_i^\ast \lb \mu_{\tau_0} \rb}} r_j n_j
\in
\left\{ \begin{array}{ll}
    (k_i + 1) \cdot \beta_i^\ast \lb \mu_{\tau_0} \rb & i \in I_\sigma \\
    \beta_i^\ast \lb \mu_{\tau_0} \rb & i \nin I_\sigma.
  \end{array} 
\right.
\end{align}
By using \eqref{eq:lattice-pts} and \eqref{eq:beta_i}, we obtain
\begin{align}
n'
=\sum_{j} r_j n_j
=\sum_{i=1}^r \sum_{\substack{j \ \mathrm{s.t.} \\ n_j \in \beta_i^\ast \lb \mu_{\tau_0} \rb}} r_j n_j
\in \sum_{i \in I_\sigma} (k_i+1) \cdot \beta_i^\ast \lb \mu_{\tau_0}\rb 
+\sum_{i \nin I_\sigma} \beta_i^\ast \lb \mu_{\tau_0}\rb.
\end{align}
This implies
\begin{align}
n=n'+n'' \in 
\sum_{i \in I_\sigma} (k_i+1) \cdot \beta_i^\ast \lb \mu_{\tau_0}\rb 
+\sum_{i \nin I_\sigma} \beta_i^\ast \lb \mu_{\tau_0}\rb
+\nu_{\tau_0}
=
\tau_0 + \sum_{i \in I_\sigma} k_i \cdot \beta_i^\ast \lb \mu_{\tau_0}\rb.
\end{align}
We proved the lemma.
\end{proof}

\begin{example}\label{eg:contraction}
Consider the case $d=2, r=1$, and
\begin{align}
\Delta&=\nabla^\ast=\conv \lb \lc (3, -1, -1), (-1, 3, -1), (-1, -1, 3), (-1,-1,-1) \rc \rb \subset M_\bR \cong \bR^3 \\
\Delta^\ast&=\nabla=\conv \lb \lc (1, 0, 0), (0, 1, 0), (0, 0, 1), (-1,-1,-1) \rc \rb \subset N_\bR \cong \bR^3.
\end{align}
Suppose further that $\Sigma'=\Sigma, \Sigmav'=\Sigmav$, $\check{h}=\check{\varphi}$, and $\tilde{\Sigma}'=\tilde{\Sigma}$.
One has $B_\nabla^{\check{h}}=X(f_1)^c=\partial \nabla$, and $\scrP(\tilde{\Sigma}')=\scrP_{X(f_1)^c}$ consists of all proper faces of $\nabla$.
\pref{fg:surf} shows the tropical hypersurface $X \lb f_1 \rb \subset X_\Sigma (\bT)$ and a tropical contraction $\delta \colon X(f_1) \to B_\nabla^{\check{h}}$ of \pref{th:glcontr} around $\tau=\conv \lb \lc (1, 0, 0), (0, 1, 0) \rc \rb \prec \nabla$.
Black points in \pref{fg:surf} show discriminant points $\Gamma(\tilde{\Sigma}') \subset B^{\check{h}}_\nabla$.

\begin{figure}[hbtp]
\hspace{5cm}
\begin{center}
	\includegraphics[scale=0.8]{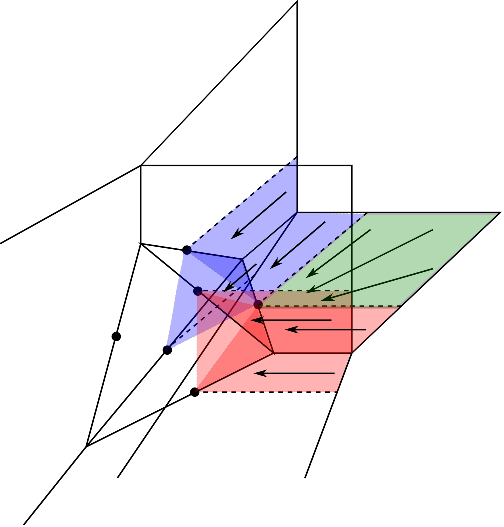}
	\caption{A tropical contraction of the tropical hypersurface $X(f_1)$}
	\label{fg:surf}
	\end{center}
\end{figure}
\end{example}

\begin{remark}\label{rm:cy-good}
When the intersections $\Sigma' \cap \partial \Delta^\ast$ and $\Sigmav' \cap \partial \nabla^\ast$ consist only of elementary simplices, $\Sigma'$ and $\Sigmav'$ give a maximal projective crepant partial resolution of the toric varieties associated with $\Sigma$ and $\Sigmav$ respectively \cite[Theorem 2.2.24]{MR1269718}, and the IAMS $(B^{\check{h}}_\nabla, \scrP(\tilde{\Sigma}'))$ is simple in the sense of Gross--Siebert program \cite[Theorem 3.16]{MR2198802}.
In this case, for any polyhedron $\tau =\rotatebox[origin=c]{180}{$\beta$} (\mu_\tau) + \nu_\tau \in \scrP(\tilde{\Sigma}')$, the polytope $\mu_\tau$ is an elementary simplex.
Since $\Deltav_{\tau, i}$ $(1 \leq i \leq r)$ coincide with faces of $\mu_\tau$ up to translation, we can see that $\sum_{i=1}^r T \lb \Deltav_{\tau, i} \rb$ is the internal direct sum of $\lc T \lb \Deltav_{\tau, i} \rb \rc_{i \in \lc 1, \cdots, r \rc}$.
Similarly, we can also see that $\sum_{i=1}^r T \lb \Delta_{\tau, i} \rb$ is the internal direct sum of $\lc T \lb \Delta_{\tau, i} \rb \rc_{i \in \lc 1, \cdots, r \rc}$.
$\Deltav_{\tau, i}, \Delta_{\tau, i}$ are the polytopes that define the local contraction around $\tau$, and we can see that the tropical contraction $\delta \colon X(f_1, \cdots, f_r) \to B_\nabla^{\check{h}}$ of \pref{th:glcontr} is good in the sense of \pref{df:contraction}.
One can also see by the same observation that the tropical contraction $\delta$ is very good if the intersections $\Sigma' \cap \partial \Delta^\ast$ and $\Sigmav' \cap \partial \nabla^\ast$ consist only of standard simplices.
In this case, the IAMS $(B^{\check{h}}_\nabla, \scrP(\tilde{\Sigma}'))$ is also very simple in the sense of \pref{df:simple} by \pref{lm:length1}.

On the other hand, as mentioned in \pref{rm:hyp}, when $r=1$ and we contract a tropical hypersurface, the conditions of being good and very good are obviously satisfied, and the tropical contraction $\delta$ is always very good.
\end{remark}

\section{Tropical contractions and (co)homology groups}\label{sc:coh}

\subsection{Tropical homology groups of IAMS}

In this subsection, we define tropical homology groups and tropical Borel--Moore homology groups of IAMS following the idea of \cite{MR4637248} to use Verdier dual.
We also define them as singular and cellular homology groups imitating the definitions for tropical varieties in \cite[Section 2]{MR3330789}.
Let $B$ be an IAMS, and $\check{\Lambda}$ be the locally constant sheaf of integral cotangent vectors on the smooth part $\iota \colon B_0 \hookrightarrow B$ of $B$.
We set $\Gamma:=B \setminus B_0$.

\begin{definition}\label{df:trophom}
For $Q=\bZ, \bQ, \bR$ and integers $p, q \geq 0$, the \emph{$(p,q)$-th tropical Borel--Moore homology group} $H_{p,q}^{\mathrm{BM}} \lb B, Q \rb$ and 
the \emph{$(p,q)$-th tropical homology group with compact support} $H_{p,q} \lb B, Q \rb$ of $B$ are defined by
\begin{align}
H_{p,q}^{\mathrm{BM}} \lb B, Q \rb&:=R^{-q} \Gamma R \cHom \lb \iota_\ast \bigwedge^{p} \check{\Lambda} \otimes Q, \omega_B \rb \\
\quad H_{p,q} \lb B, Q \rb&:=R^{-q} \Gamma_c R \cHom \lb \iota_\ast \bigwedge^{p} \check{\Lambda} \otimes Q, \omega_B \rb,
\end{align}
where $\omega_B \in D^b(\bZ_B)$ is the dualizing complex of $B$.
\end{definition}

\begin{definition}\label{df:fine}
A \textit{fine polyhedral structure} $\scrF$ of an IAMS $B$ is a finite family $\lc \tau \rc_{\tau \in \scrF}$ of closed subsets $\tau \subset B$ such that 
\begin{itemize}
\item Each $\tau \in \scrF$ is equipped with a homeomorphism $\phi_\tau$ to a rational polyhedron $\tilde{\tau}$ in $N_\bR$, and hence has the set of faces defined as inverse images of faces of $\tilde{\tau}$. 
\item $\scrF$ forms a complex, i.e., $\scrF$ satisfies \pref{cd:complex}.
\item $B=\bigcup_{\tau \in \scrF} \tau$, and  there is a subcomplex $\scrD \subset \scrF$ such that $\Gamma=\bigcup_{\tau \in \scrD} \tau$.
\item For any $\tau \in \scrF \setminus \scrD$, there is a chart $\psi_i \colon U_i \to N_\bR$ of $B_0$ such that $U_i \supset \bigcup_{\sigma \succ \tau} \phi_\sigma^{-1} \lb \rint(\tilde{\sigma}) \rb$ and $\psi_i  \circ \phi_\sigma^{-1} |_{\rint(\tilde{\sigma})}$ is an integral affine isomorphism onto its image for any $\sigma \succ \tau$.
\end{itemize}
For every integer $k \geq 0$, we also set $\scrF(k):=\lc \sigma \in \scrF \relmid \dim (\sigma)=k \rc$.
\end{definition}

\begin{example}\label{eg:fine}
In Construction \ref{construction}, we constructed an IAMS $B$ starting with a rational polytopal complex $\scrP$.
We chose a point $a_\tau \in \rint \lb \tau \rb$ for every $\tau \in \scrP$, and considered the associated subdivision $\widetilde{\scrP}$ \eqref{eq:subdivision} of $\scrP$.
If we take a rational point $a_\tau$ for every $\tau \in \scrP$, then $\widetilde{\scrP}$ becomes a fine polyhedral structure of $B$.
\end{example}

Let $\lb B, \scrF \rb$ be an IAMS equipped with a fine polyhedral structure.

\begin{definition}\label{df:singular}
Let $\Delta^q$ denote the standard $q$-simplex.
We say that a singular $q$-simplex $\delta \colon \Delta^q \to B$ \emph{respects the polyhedral structure $\scrF$} if for each face $\Delta' \prec \Delta^q$ there exists a polyhedron $\tau \in \scrF$ such that $\delta \lb \rint(\Delta') \rb \subset \rint(\tau)$.
A \emph{tropical $(p, q)$-simplex with respect to the polyhedral structure $\scrF$} is a pair $(\delta, s)$ of a singular $q$-simplex $\delta$ respecting $\scrF$ and $s \in \Hom \lb \lb \iota_\ast \left. \bigwedge^{p}\check{\Lambda} \rb \right|_{\delta \lb \Delta^q \rb}, \bZ_{\delta \lb \Delta^q \rb} \rb$.
Let $S_{p,q}(B; \scrF)$ denote the free abelian group generated by tropical $(p,q)$-simplices with respect to $\scrF$.
Pulling back a tropical $(p,q)$-simplex $(\delta, s)$ via the inclusion of the $i$-th facet $\epsilon^q_i \colon \Delta^{q-1} \hookrightarrow \Delta^q$ yields a tropical $(p,q-1)$-simplex 
$\partial_{p,q,i}(\delta, s)= 
\lb \delta \circ \epsilon^q_i, 
\left. s \right|_{\delta \lb \epsilon^q_i (\Delta^{q-1}) \rb}
\rb$. 
Extending $\partial_{p,q,i}$ by linearity and taking alternating sums, one defines the differentials 
\begin{align}
\partial_{p,q} :=\sum_{i=0}^q (-1)^i \partial_{p,q,i}
\end{align}
and obtains a chain complexes $\lb S_{p,\bullet}(B; \scrF), \partial_{p,\bullet} \rb$.
We call its homology groups
\begin{align}
H_{p, q}^{\mathrm{sing}} \lb B; \scrF \rb :=H_q \lb \lb S_{p,\bullet}(B; \scrF), \partial_{p,\bullet} \rb \rb
\end{align}
the \emph{singular tropical homology groups of $B$ with respect to $\scrF$}. 
We also define $S_{p, q}^\mathrm{lf} \lb B; \scrF \rb$ to be the abelian group of formal infinite sums of elements of $S_{p,q}(B; \scrF)$ with the condition that locally only finitely many tropical simplices have non-zero coefficients.
We call its homology groups
\begin{align}
H_{p, q}^{\mathrm{lf}} \lb B; \scrF \rb :=H_q \lb \lb S_{p,\bullet}^{\mathrm{lf}} (B; \scrF), \partial_{p,\bullet} \rb \rb
\end{align}
the \emph{locally finite tropical homology groups of $B$ with respect to $\scrF$}.
\end{definition}

\begin{proposition}{\rm(cf.~\cite[Theorem 4.20]{MR4637248})}\label{pr:sing}
Let $B$ be an IAMS equipped with a fine polyhedral structure $\scrF$. 
Then there are natural isomorphisms
\begin{align}
H_{p, q}^{\mathrm{lf}} \lb B; \scrF \rb \cong H^{\mathrm{BM}}_{p,q}(B, \bZ), \quad 
H_{p, q}^{\mathrm{sing}} \lb B; \scrF \rb \cong H_{p,q}(B, \bZ).
\end{align}
\end{proposition}
The analogous statement to the above one for rational polyhedral spaces is shown in \cite[Theorem 4.20]{MR4637248}.
The set of relative interiors of elements in the fine polyhedral structure $\scrF$ defines an admissible stratification on $B$ in the sense of \cite[Definition A.3]{MR4637248}.
Furthermore, since the restriction of the sheaf $\iota_\ast \bigwedge^{p} \check{\Lambda}$ to every stratum is locally free of finite rank, one can also use \cite[Proposition A.9]{MR4637248}.
Thus the above statement can also be proved in exactly the same way as \cite[Theorem 4.20]{MR4637248}.
We omit to repeat the proof.

\begin{corollary}
The homology groups $H_{p, q}^{\mathrm{lf}} \lb B, \scrF \rb$ and $H_{p, q}^{\mathrm{sing}} \lb B, \scrF \rb$ do not depend on the choice of the fine polyhedral structure $\scrF$ of $B$.
\end{corollary}
\begin{proof}
Since we do not use polyhedral structures $\scrF$ for defining $H^{\mathrm{BM}}_{p,q}(B, \bZ)$ and $H_{p,q}(B, \bZ)$, the claim is obvious from \pref{pr:sing}.
\end{proof}

Consider the one-point compactification $B \cup \lc \infty \rc$ of $B$.
When $\lc \infty \rc \cup \lc \rint \lb \tau \rb \rc_{\tau \in \scrF}$ gives a regular CW complex structure of $B \cup \lc \infty \rc$, the space $B$ equipped with the partition $\lc \rint \lb \tau \rb \rc_{\tau \in \scrF}$ forms a \emph{cell complex} in the sense of \cite{MR2634247}.
We refer the reader to \cite[Section 4.1]{MR3259939} for its definition.
The fine polyhedral structure $\scrF$ also forms the category whose objects are the polyhedra in $\scrF$ and morphisms are given by
\begin{eqnarray}
\Hom \lb \tau, \sigma \rb:=
\left\{
\begin{array}{ll}
\id & \tau \prec \sigma \\
\emptyset & \mathrm{otherwise} \\
\end{array}
\right.
\end{eqnarray}
for $\tau, \sigma \in \scrF$.
Let $\scrF^\mathrm{op}$ denote its opposite category.
Consider the functor from $\scrF^\mathrm{op}$ to the category of $\bZ$-modules assigning 
$F_p^\bZ(\tau):=\Hom \lb \lb \iota_\ast \left. \bigwedge^{p}\check{\Lambda} \rb \right|_{\rint(\tau)}, \bZ_{\rint(\tau)} \rb$ 
to each polyhedron $\tau \in \scrF$, and 
the extension map
$
e_{\tau, \sigma} \colon F_p^\bZ(\sigma) \to F_p^\bZ(\tau)
$
induced by the inclusion $\iota_\ast \left. \bigwedge^{p}\check{\Lambda} \right|_{\rint(\tau)} \hookrightarrow \iota_\ast \left. \bigwedge^{p}\check{\Lambda} \right|_{\rint(\sigma)}$
to each pair $\tau \prec \sigma$.
It forms a cellular cosheaf (cf.~e.g.~\cite[Definition 4.1.6]{MR3259939}), which will be denoted by $F_p^\bZ$.
We consider the cosheaf homology group and the Borel--Moore cosheaf homology group (cf.~e.g.~\cite[Section 6]{MR3259939}) of the cellular cosheaf $F_p^\bZ$.
They are the homology groups of the chain complexes 
$\lb C_{p, \bullet} \lb B; \scrF \rb,  \partial_{p, \bullet} \rb$ and $\lb C_{p, \bullet}^{\mathrm{BM}} \lb B; \scrF \rb,  \partial_{p, \bullet} \rb$ defined by
\begin{align}
C_{p, q} \lb B; \scrF \rb&:= \bigoplus_{\tau \in \scrF(q): \mathrm{compact}} F_p^\bZ(\tau), \quad
C_{p, q}^{\mathrm{BM}} \lb B; \scrF \rb:= \bigoplus_{\tau \in \scrF(q)} F_p^\bZ(\tau)
 \\ \label{eq:cboundary}
\partial_{p, q}&:=\sum_{\substack{\tau \prec \sigma \in \scrF(q) \\ \dim \tau=q-1}} \ld \tau : \sigma \rd e_{\tau, \sigma} 
\colon C_{p, q}^{(\mathrm{BM})} \lb B; \scrF \rb \to C_{p, q-1}^{(\mathrm{BM})} \lb B; \scrF \rb.
\end{align}
Here $\ld \tau : \sigma \rd$ is the sign determined as follows:
We first fix an orientation of every polyhedron $\tau \in \scrF$.
For each pair $\tau \prec \sigma$, we define $\ld \tau : \sigma \rd:=1$ if the orientation of $\tau$ coincides with the orientation of $\partial \sigma$ induced by the orientation of $\sigma$, and $\ld \tau : \sigma \rd:=-1$ otherwise.

\begin{definition}\label{df:cellular}
We write the cosheaf homology group and the Borel--Moore cosheaf homology group of the cellular cosheaf $F_p^\bZ$ as 
\begin{align}
H_{p, q}^{\mathrm{cell}} \lb B; \scrF \rb:=H_q \lb \lb C_{p,\bullet} (B; \scrF), \partial_{p,\bullet} \rb \rb, \quad
H_{p, q}^{\mathrm{cBM}} \lb B; \scrF \rb:=H_q \lb \lb C_{p,\bullet}^{\mathrm{BM}} (B; \scrF), \partial_{p,\bullet} \rb \rb
\end{align}
and call them the \emph{cellular tropical homology group} and the \emph{cellular tropical Borel--Moore homology group of $B$} respectively.
\end{definition}

\begin{proposition}{\rm(cf.~\cite[Proposition 2.2]{MR3330789}, \cite[Remark 2.8]{MR3894860})}\label{pr:cell}
Let $B$ be an IAMS equipped with a fine polyhedral structure $\scrF$. 
Then there are natural isomorphisms
\begin{align}
H_{p, q}^{\mathrm{cell}} \lb B; \scrF \rb
\cong
H_{p, q}^{\mathrm{sing}} \lb B; \scrF \rb, \quad
H_{p, q}^{\mathrm{cBM}} \lb B; \scrF \rb
\cong
H_{p, q}^{\mathrm{lf}} \lb B; \scrF \rb.
\end{align}
\end{proposition}
The analogous statement to the above one for singular/cellular tropical homology groups of tropical varieties is shown in \cite[Proposition 2.2]{MR3330789}.
There is also an explanation for generalizing \cite[Proposition 2.2]{MR3330789} to Borel--Moore homology groups in \cite[Remark 2.8]{MR3894860}.
Since in the proof of \cite[Proposition 2.2]{MR3330789}, we only use the condition that the coefficient is constant on every cell, the above statement can also be proved in the same way as \cite[Proposition 2.2]{MR3330789}.
We omit to repeat the proof also for this.

\begin{remark}\label{rm:sheaf-homology}
There is another approach by Ruddat \cite{MR4347312} to homology theory for IAMS using sheaf homology.
We refer the reader to \cite[Section V\hspace{-.1em}I-12]{MR1481706} for the definition of (singular) sheaf homology.
In general, the sheaf homology $H_q \lb B, \iota_\ast \bigwedge^p \Lambda \rb$ considered in \cite{MR4347312} is not isomorphic to the homology group $H_{p,q}(B, \bZ)$ or $H_{p,q}^{\mathrm{BM}} \lb B, \bZ \rb$ of \pref{df:trophom}.
For instance, consider a closed integral affine disk $B$ of dimension $2$ with two focus-focus singularities.
Let $\lc e_1, e_2 \rc$ be a basis of the integral tangent space at a smooth point in $B$, and $k \geq 1$ be an integer. 
Assume that the monodromy invariant subspaces of those two singular points are generated by $e_1+ke_2$ and $e_1$ respectively.
The disk $B$ is shown on the left in \pref{fg:2ff}.
The dots are the singular points, and the dotted lines show the monodromy invariant subspaces.
\begin{figure}[htbp]
\begin{center}
\includegraphics[scale=0.6]{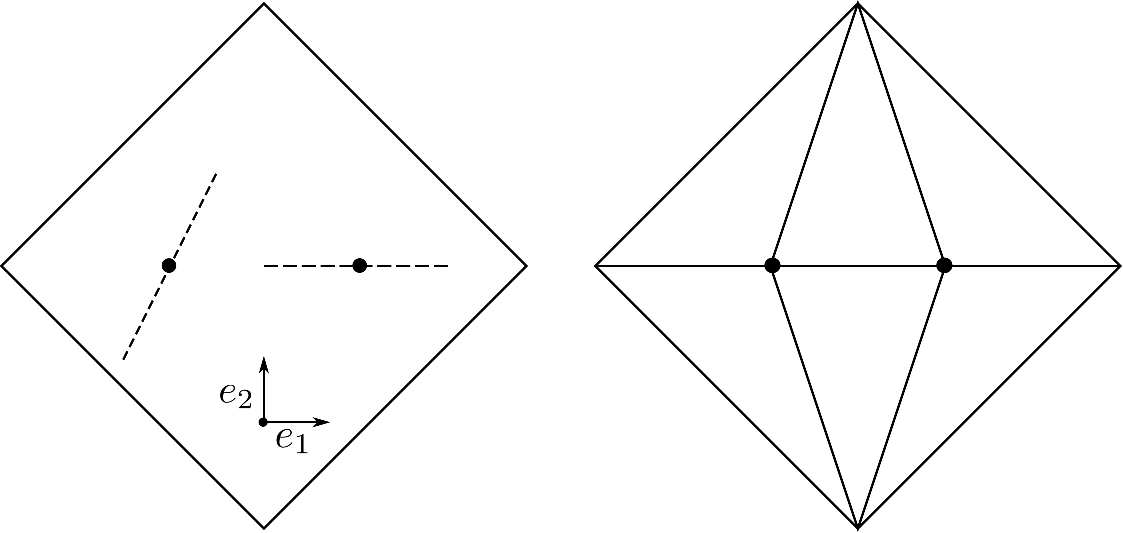}
\end{center}
\caption{The disk $B$ with two focus-focus singularities and its cellular decomposition}
\label{fg:2ff}
\end{figure}
The sheaf homology group $H_1 \lb B, \iota_\ast \Lambda \rb$ of $B$ is isomorphic to $\bZ/k\bZ$.
See \cite[Introduction 7, Remark 9]{MR4347312} for how to compute it.
On the other hand, one can also see $H_{1,1}(B, \bZ)=H_{1,1}^{\mathrm{BM}} \lb B, \bZ \rb=0$, for instance, by taking a cellular decomposition of $B$ shown on the right in \pref{fg:2ff} and computing its celluar homology group.
Notice that the coefficient group $F_1^\bZ$ is $\left. \lb \bZ e_1 \oplus \bZ e_2 \rb \middle/ \bZ \lb e_1+ke_2 \rb \right.$ and $\left. \lb \bZ e_1 \oplus \bZ e_2 \rb \middle/ \bZ e_1 \right.$ at those two singular points respectively, and is $\bZ e_1 \oplus \bZ e_2$ on any other cells.
\end{remark}

\begin{example}\label{eg:1-cycle}
Consider an integral affine disk $B$ of dimension $2$ with two focus-focus singularities $p_1, p_2$ whose monodromy invariant subspaces are parallel.
Let $\lc e_1, e_2 \rc$ be a basis of the integral tangent space at a smooth point in $B$ again, and assume that the monodromy invariant subspaces are generated by $e_1$.
We recall two types of tropical $1$-cycles considered in \cite{MR4347312}:
One is called a \emph{goggle cycle} in \cite[Figure 0.2]{MR4347312}.
It is shown on the left side of \pref{fg:goggle}.
It consists of two loops going around the singular point and a path connecting those two loops, which are equipped with coefficient as shown in the figure.
It is a \emph{tropical $1$-cycle} in the sense of \cite[Definition 1.2]{MR4179831}.
It defines a cycle class in the singular sheaf homology group $H_1\lb B, \iota_\ast \Lambda \rb$.
Tropical $1$-cycles are used to compute period integrals of toric degenerations in \cite{MR4179831}.
The other one is called \emph{Symington's (\cite{MR2024634}) relative cycle} in \cite[Figure 0.7]{MR4347312}.
It is shown on the right side of \pref{fg:goggle}.
It is a path with coefficient $e_1$ that connects two singular points $p_1, p_2$.
\begin{figure}[htbp]
\begin{center}
\includegraphics[scale=0.75]{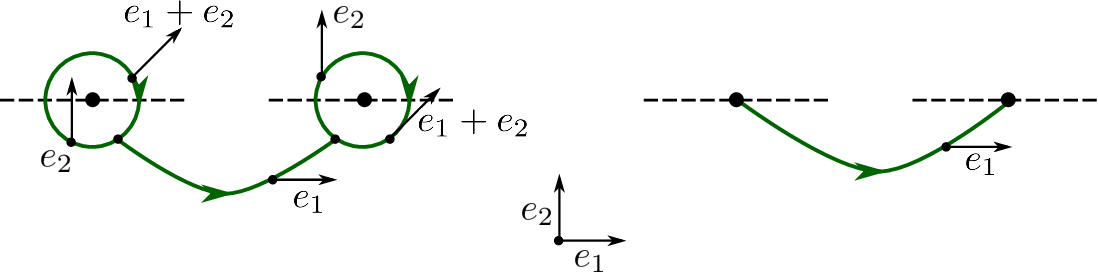}
\end{center}
\caption{A goggle cycle and Symington's relative cycle}
\label{fg:goggle}
\end{figure}
Since the stalks of $\iota_\ast \Lambda$ at the focus-focus singularities $p_1, p_2$ are isomorphic to $\bZ e_1 \subset \bZ e_1 \oplus \bZ e_2$ and the boundary of Symington's relative cycle is not zero, it does not define a cycle class in the sheaf homology group $H_1\lb B, \iota_\ast \Lambda \rb$.
However, one can think of it as a cycle defining a class in the relative homology group $H_1 \lb B, \Gamma, \iota_\ast \Lambda \rb$. (This is why it is called a relative cycle.)
This type of relative cycles is used for constructing deformations of tropical polarized K3 surfaces in \cite{MR21}.
As explained in \cite[Figure 0.7]{MR4347312}, the natural map $H_1\lb B, \iota_\ast \Lambda \rb \to H_1 \lb B, \Gamma, \iota_\ast \Lambda \rb$ sends the class defined by the goggle cycle in $H_1\lb B, \iota_\ast \Lambda \rb$ to the class defined by Symington's relative cycle in $H_1 \lb B, \Gamma, \iota_\ast \Lambda \rb$.

Both Symington's relative cycle and the goggle cycle also define a non-trivial cycle class in the tropical homology group $H_{1, 1}(B, \bZ)$.
Notice that we do not need to consider the relative homology group for this since the coefficient group 
$\Hom \lb \lb \iota_\ast \left. \check{\Lambda} \rb \right|_{p_i}, \bZ_{p_i} \rb$ at the focus-focus singularities $p_1, p_2$ is $\lb \bZ e_1 \oplus \bZ e_2 \rb/\bZ e_1$ and the boundary of Symington's relative cycle also becomes zero.
These two cycle classes also differ only by a $1$-boundary and define the same class in $H_{1, 1}(B, \bZ)$.
\end{example}

\begin{example}\label{eg:2-cycle}
Let $\lb B, \scrP, \phi \rb$ be a tropical conifold in the sense of \cite[Definition 6.3]{MR3228462}.
A \emph{tropical $2$-cycle} $(S, j, v)$ on $\lb B, \scrP, \phi \rb$ in the sense of \cite[Definition 7.2]{MR3228462} defines an element of $H_{1,2} \lb B, \bZ \rb$ as follows:
We choose a rational point $a_\tau \in \rint \lb \tau \rb$ for every $\tau \in \scrP$, and consider the associated subdivision $\widetilde{\scrP}$ \eqref{eq:subdivision} of $\scrP$.
As mentioned in \pref{eg:fine}, $\widetilde{\scrP}$ becomes a fine polyhedral structure of $B$ in the sense of \pref{df:fine}.
Take a simplicial subdivision $\tilde{S}$ of $S$ so that the restriction of the map $j$ to each simplex respects the polyhedral structure $\widetilde{\scrP}$.
For each singular $2$-simplex $\left. j \right|_{\Delta^2} \colon \Delta^2 \to B$ $(\Delta^2 \in \tilde{S})$, the restriction of the integral parallel vector field $v$ to the image of $\rint \lb \Delta^2 \rb$ by the map $j$ defines an element of $\Hom \lb \lb \iota_\ast \left. \check{\Lambda} \rb \right|_{j \lb \Delta^2 \rb}, \bZ_{j \lb \Delta^2 \rb} \rb \cong \left. \Lambda \right|_{j \lb \rint \lb \Delta^2 \rb \rb}$.
We obtain the following tropical $(1, 2)$-simplex in the sense of \pref{df:singular}:
\begin{align}
s:= \sum_{\Delta^2 \in \tilde{S}} 
\lb \left. j \right|_{\Delta^2} \colon \Delta^2 \to B, \left. v \right|_{j \lb \rint \lb \Delta^2 \rb \rb} \rb
 \in S_{1,2} \lb B; \widetilde{\scrP} \rb.
\end{align}
By the map $j$, any edge $\Delta^1 \in \tilde{S}$ in the boundary of $S$ is sent to the discriminant $\Gamma$ of $B$ (the condition (i) of \cite[Definition 7.2]{MR3228462}).
The condition (viii) of \cite[Definition 7.2]{MR3228462} requires that the annihilator of $v$ in the cotangent space at a point in a small neighborhood of $j(\Delta^1)$ coincides with the cotangent subspace that is invariant with respect to the monodromy operator around $j(\Delta^1)$.
Since the monodromy invariant cotangent subspace coincides with $\lb \iota_\ast \left. \check{\Lambda} \rb \right|_{j(\rint \lb \Delta^1 \rb)}$, one has $v=0$ as an element of the coefficient group
\begin{align}
\Hom \lb \lb \iota_\ast \left. \check{\Lambda} \rb \right|_{j(\Delta^1)}, \bZ_{j(\Delta^1)} \rb
\cong
\Hom \lb \lb \iota_\ast \left. \check{\Lambda} \rb \right|_{j(\rint \lb \Delta^1 \rb )}, \bZ_{j(\rint \lb \Delta^1 \rb)} \rb
.
\end{align}
Hence, the boundary $\partial_{2,1} (s)$ does not have support on $j(\partial S)$.
Furthermore, the condition (vii) of \cite[Definition 7.2]{MR3228462} (the \emph{balancing condition} along interior edges in $S$) ensures that coefficients of interior edges in $\partial_{2,1} (s)$ are also all $0$.
Therefore, we have $\partial_{2,1} (s)=0$ and $s$ defines an element of $H_{1,2}^{\mathrm{sing}} \lb B; \widetilde{\scrP} \rb$.

As mentioned in \cite[Introduction 4]{MR4347312}, a tropical $2$-cycle $(S, j, v)$ also defines an element of the relative sheaf homology group $H_2 \lb B, \Gamma, \iota_\ast \Lambda \rb$, although it does not define an element of the sheaf homology group $H_2 \lb B, \iota_\ast \Lambda \rb$ as well as Symington's relative cycle mentioned in \pref{eg:1-cycle}.
\end{example}

\subsection{Comparison of tropical (co)homology groups}

In this subsection, we prove \pref{cr:cohisom} and \pref{cr:wave}.
Let $\delta \colon V \to B$ be a tropical contraction.

\begin{proof}[Proof of \pref{cr:cohisom}]
In the following, $\Gamma_\Phi$ denotes the global section functor $\Gamma$ or the functor of global sections with compact support $\Gamma_c$.
Since the tropical contraction $\delta \colon V \to B$ is proper, one has $\delta_!=\delta_\ast$.
If we assume that $\delta$ is a good (resp. very good) tropical contraction in the sense of \pref{df:contraction}, then we have
\begin{align}\label{eq:lems1}
R^q \Gamma_\Phi \lb V, \scF^p_Q \rb
&\cong R^q \Gamma_\Phi \lb B, R \delta_\ast \scF^p_Q \rb \\ \label{eq:proper1} 
&\cong R^q \Gamma_\Phi \lb B, \iota_\ast \bigwedge^{p} \check{\Lambda} \otimes_\bZ Q \rb
\end{align}
and 
\begin{align}\label{eq:proper2}
R^{-q} \Gamma_\Phi R \cHom \lb \scF^p_Q, \omega_V \rb
& \cong R^{-q} \Gamma_\Phi R \cHom \lb R \delta_\ast \scF^p_Q, \omega_B \rb \\ \label{eq:lems2}
& \cong R^{-q} \Gamma_\Phi R \cHom \lb \iota_\ast \bigwedge^{p} \check{\Lambda} \otimes_\bZ Q, \omega_B \rb
\end{align}
for $Q=\bQ$ (resp. $Q=\bZ$), where $\omega_V \in D^b(\bZ_V)$ and $\omega_B \in D^b(\bZ_B)$ are the dualizing complexes of $V$ and $B$.
Here we use $\delta_!=\delta_\ast$ for \eqref{eq:lems1} and \pref{eq:proper2}.
We also use \pref{th:local-cont}.2 and \pref{th:local-cont}.3 for \pref{eq:proper1} and \eqref{eq:lems2}.
\end{proof}

\begin{proof}[Proof of \pref{cr:wave}]
The map \eqref{eq:natu} is given by
\begin{align}\label{eq:lems4}
R^q \Gamma_V \scW_p^Q \cong 
R^q \Gamma_B R \delta_\ast \scW_p^Q \cong 
R^q \Gamma_B \delta_\ast \scW_p^Q =
R^q \Gamma_B \delta_\ast \cHom \lb \scF^p_\bZ, Q \rb \\ \label{eq:lems5}
\to R^q \Gamma_B \cHom \lb \delta_\ast \scF^p_\bZ, \delta_\ast Q \rb=
R^q \Gamma_B \cHom \lb \iota_\ast \bigwedge^p \check{\Lambda}, Q \rb=
R^q \Gamma_B \lb \iota_\ast \bigwedge^p \Lambda \otimes_\bZ Q \rb,
\end{align}
where we use \pref{th:local-cont}.4 in \eqref{eq:lems4} and \pref{th:local-cont}.2 in \eqref{eq:lems5}.

Suppose that the tropical contraction $\delta$ is good.
Recall that the eigenwave of $V$ and the radiance obstruction of $B$ are the extension classes of the exact sequences \eqref{eq:expseq} and \eqref{eq:exaff2} respectively (see \pref{sc:radiance} and \pref{pr:eigext}).
Consider the tensor products of these exact sequences with $\bQ$
\begin{align}\label{eq:q-tensor1}
0 \to \bR \to \iota_\ast \mathrm{Aff}_{B_0} \otimes_\bZ \bQ \to \iota_\ast \check{\Lambda} \otimes_\bZ \bQ \to 0 \\ \label{eq:q-tensor2}
0 \to \bR \to \mathrm{Aff}_V \otimes_\bZ \bQ \to \scF^1_\bQ \to 0.
\end{align}
The extension classes of \eqref{eq:exaff2} and \eqref{eq:expseq} coincide with the extension classes of \eqref{eq:q-tensor1} and \eqref{eq:q-tensor2} respectively.
We further consider the following diagram:
\begin{align}
\begin{CD}
\Hom \lb \scF^1_\bQ, \scF^1_\bQ \rb  @>\partial>> R^1 \Gamma_V R \cHom \lb \scF^1_\bQ, \bR \rb @<<< R^1 \Gamma_V \cHom \lb \scF^1_\bQ, \bR \rb\\
@|                     @V\cong VV  @V\cong VV \\
\Gamma_B \delta_\ast \cHom \lb \scF^1_\bQ, \scF^1_\bQ \rb  @. R^1 \Gamma_B R \delta_\ast R \cHom \lb \scF^1_\bQ, \bR \rb @. R^1 \Gamma_B \delta_\ast \cHom \lb \scF^1_\bQ, \bR \rb \\
@VVV                    @VVV  @VVV \\
\Gamma_B \cHom \lb \delta_\ast \scF^1_\bQ, \delta_\ast \scF^1_\bQ \rb  @. R^1 \Gamma_B R \cHom \lb \delta_\ast \scF^1_\bQ, R \delta_\ast \bR \rb @. R^1 \Gamma_B \cHom \lb \delta_\ast \scF^1_\bQ, \bR \rb \\
@V\cong VV                    @V\cong VV  @V\cong VV \\
\Hom \lb \iota_\ast \check{\Lambda} \otimes_\bZ \bQ, \iota_\ast \check{\Lambda} \otimes_\bZ \bQ \rb  @>\partial>>  R^1 \Gamma_B R \cHom \lb \iota_\ast \check{\Lambda} \otimes_\bZ \bQ, \bR \rb @<<< R^1 \Gamma_B \cHom \lb \iota_\ast \check{\Lambda} \otimes_\bZ \bQ, \bR \rb\\
\end{CD}
\end{align}
Here the maps $\partial$ are the connecting homomorphisms of the long exact sequences arising from $\Hom \lb \scF_\bQ^1, \bullet \rb$ and $\Hom \lb \iota_\ast \check{\Lambda} \otimes_\bZ \bQ, \bullet \rb$.
It is obvious that the right square commutes.
When $\delta$ is good, it turns out by \pref{th:local-cont} that for an injective resolution $0 \to I^\bullet \to J^\bullet \to K^\bullet \to 0$ of \eqref{eq:q-tensor2}, the pushforward $0 \to \delta_\ast I^\bullet \to \delta_\ast J^\bullet \to \delta_\ast K^\bullet \to 0$ becomes an injective resolution of \eqref{eq:q-tensor1}, and it is straightforward to check that the left square in the above diagram also commutes.
The composition of the rightmost vertical maps corresponds to the map \eqref{eq:natu} with $p=q=1$ and $Q=\bR$.
The composition of the leftmost vertical maps sends $\id \in \Hom \lb \scF^1_\bQ, \scF^1_\bQ \rb$ to $\id \in \Hom \lb \iota_\ast \check{\Lambda} \otimes_\bZ \bQ, \iota_\ast \check{\Lambda} \otimes_\bZ \bQ \rb$.
Since the extension classes of \eqref{eq:q-tensor1} and \eqref{eq:q-tensor2} are the images of these $\id$ by the horizontal maps, the map \eqref{eq:natu} with $p=q=1$ and $Q=\bR$ sends the eigenwave of $V$ to the radiance obstruction of $B$.
\end{proof}

\section{Technical lemmas}\label{sc:lem}

\subsection{Lemmas on the map $\phi_{\Deltav', \Deltav}$}

Let $\Deltav, \Deltav' \subset N_\bR$ be rational polytopes such that the normal fan $\Sigmav' \subset M_\bR$ of $\Deltav'$ is a refinement of the normal fan $\Sigmav \subset M_\bR$ of $\Deltav$.
Consider the map $\phi_{\Deltav', \Deltav} \colon \scrP_{\Deltav'} \to \scrP_{\Deltav}$ of \eqref{eq:phi}, where $ \scrP_{\Deltav},  \scrP_{\Deltav'}$ are the sets of faces of $\Deltav, \Deltav'$.
 
\begin{lemma}\label{lm:TT}
The following hold:
\begin{enumerate}
\item For $\sigma \in \scrP_{\Deltav}$, one has
\begin{align}\label{eq:TM}
T \lb \sigma \rb = \lb \overline{\scM(\sigma, \Deltav)} \rb^\perp,
\end{align}
where $T (\sigma)$ and $\scM(\sigma, \Deltav)$ are the ones defined in \eqref{eq:T} and \eqref{eq:scM} respectively.
\item 
For $\sigma \in \scrP_{\Deltav'}$, one has 
\begin{align}\label{eq:TT}
T \lb \phi_{\Deltav', \Deltav} (\sigma) \rb \subset T (\sigma).
\end{align}
In particular, we have $T \lb \Deltav \rb \subset T \lb \Deltav' \rb$.
\end{enumerate}
\end{lemma}

\begin{proof}
It is obvious that the left hand side of \eqref{eq:TM} is contained in the right hand side.
Since the dimensions of both sides are $\dim \sigma$, we get \eqref{eq:TM}.

Concerning \eqref{eq:TT}, one has $\overline{\scM \lb \sigma, \Deltav' \rb} \subset \overline{\scM \lb \phi_{\Deltav', \Deltav} (\sigma), \Deltav \rb}$.
By taking the annihilators of these and using \eqref{eq:TM} for $\sigma$ and $\phi_{\Deltav', \Deltav} (\sigma)$, we obtain \eqref{eq:TT}.
By substituting $\Deltav'$ to $\sigma$ in \eqref{eq:TT}, we also get  $T \lb \Deltav \rb \subset T \lb \Deltav' \rb$.
\end{proof}

For a face $\sigma \prec \Deltav'$, we set $\tilde{\sigma}:=\phi_{\Deltav', \Deltav} (\sigma) \prec \Deltav$.
\begin{lemma}\label{lm:face-phi}
The normal fan $\Sigmav_\sigma$ of $\sigma$ is a refinement of the normal fan $\Sigmav_{\tilde{\sigma}}$ of $\tilde{\sigma}$, and one can consider the map $\phi_{\sigma, \tilde{\sigma}} \colon \scrP_{\sigma} \to \scrP_{\tilde{\sigma}}$ of \eqref{eq:phi} for $\sigma, \tilde{\sigma}$.
For any face $\tau \prec \sigma$, one has
$
\phi_{\sigma, \tilde{\sigma}} \lb \tau \rb=\phi_{\Deltav', \Deltav} \lb \tau \rb.
$
\end{lemma}
\begin{proof}
First, we check that the fan $\Sigmav_\sigma$ is a refinement of the fan $\Sigmav_{\tilde{\sigma}}$.
One has
\begin{align}
\Sigmav_\sigma&=\lc C+ \vspan \lb \delta_{\Deltav'} (\sigma) \rb \relmid C \in \Sigmav', C \succ \delta_{\Deltav'} (\sigma)
\rc \\
\Sigmav_{\tilde{\sigma}}&=
\lc C'+ \vspan \lb \delta_{\Deltav} (\tilde{\sigma}) \rb \relmid C' \in \Sigmav, C' \succ \delta_{\Deltav} (\tilde{\sigma})
\rc,
\end{align}
where $\delta_{\Deltav'}, \delta_{\Deltav}$ are the maps of \eqref{eq:poly-fan} for $\Deltav', \Deltav$.
Let $C \in \Sigmav'$ be an arbitrary cone such that $C \succ \delta_{\Deltav'} (\sigma)$.
Take the minimal cone $C' \in \Sigmav$ containing $C$.
Since $C' \supset C \succ \delta_{\Deltav'} (\sigma)$ and $\delta_{\Deltav} (\tilde{\sigma})$ is the minimal cone in $\Sigmav$ containing $\delta_{\Deltav'} (\sigma)$, one has $C' \succ \delta_{\Deltav} (\tilde{\sigma})
$.
Hence, $C'+ \vspan \lb \delta_{\Deltav} (\tilde{\sigma}) \rb \in \Sigmav_{\tilde{\sigma}}$.
Furthermore, one also has $C+ \vspan \lb \delta_{\Deltav'} (\sigma) \rb \subset C'+ \vspan \lb \delta_{\Deltav} (\tilde{\sigma}) \rb$ since $C \subset C'$ and $\delta_{\Deltav'} (\sigma) \subset \delta_{\Deltav} (\tilde{\sigma})$.
Hence, the fan $\Sigmav_\sigma$ is a refinement of the fan $\Sigmav_{\tilde{\sigma}}$.

The latter statement can be checked as follows:
For a face $\tau \prec \sigma$, take a point $m_0 \in \scM(\tau, \Deltav')$.
Since we have \eqref{eq:ss2} and $\phi_{\Deltav', \Deltav} \lb \tau \rb \prec \phi_{\Deltav', \Deltav} \lb \sigma \rb=:\tilde{\sigma}$, one can get
\begin{align}
\phi_{\Deltav', \Deltav} \lb \tau \rb&=\lc n \in \Deltav \relmid -\la m_0, n \ra=-\inf_{n' \in \Deltav} \la m_0, n' \ra \rc \\
&=\lc n \in \tilde{\sigma} \relmid -\la m_0, n \ra=-\inf_{n' \in \tilde{\sigma}} \la m_0, n' \ra \rc.
\end{align}
Since $m_0 \in \scM(\tau, \Deltav') \subset \scM(\tau, \sigma)$, we can see again from \eqref{eq:ss2} that this is equal to $\phi_{\sigma, \tilde{\sigma}} \lb \tau \rb$.
\end{proof}

Consider the Minkowski sum $\Deltav_+:=\sum_{i \in I} \Deltav_i$.
Then the normal fan of $\Deltav_+$ is a subdivision of the normal fan $\Sigmav_i$ of $\Deltav_i$ for all $i \in I$.
We also consider the map $\phi_{\Deltav_+, \Deltav_i} \colon \scrP_{\Deltav_+} \to \scrP_{\Deltav_i}$ of \eqref{eq:phi} for $\Deltav_+, \Deltav_i$.
\begin{lemma}{\rm(cf.~e.g.~\cite[Theorem 3.1.2]{Wei07})}\label{lm:minksum}
For any face $\tau \in \scrP_{\Deltav_+}$ of $\Deltav_+$, one has 
\begin{align}
\tau=\sum_{i \in I} \phi_{\Deltav_+, \Deltav_i} (\tau).
\end{align}
\end{lemma}
\begin{proof}
Take a point $m_0 \in \scM(\tau, \Deltav_+)$.
Then we have
\begin{align}\label{eq:ss'}
\tau&=\lc n \in \Deltav_+ \relmid -\la m_0, n \ra=-\inf_{n' \in \Deltav_+} \la m_0, n' \ra \rc \\ \label{eq:mink-decomp}
&=\sum_{i \in I} \lc n \in \Deltav_i \relmid -\la m_0, n \ra=-\inf_{n' \in \Deltav_i} \la m_0, n' \ra \rc \\ \label{eq:ss2'}
&=\sum_{i \in I} \phi_{\Deltav_+, \Deltav_i} (\tau).
\end{align}
The equalities \eqref{eq:ss'} and \eqref{eq:ss2'} are due to \eqref{eq:ss} and \eqref{eq:ss2} respectively.
For the equality \eqref{eq:mink-decomp}, see, for instance, \cite[Theorem 3.1.2]{Wei07} and its proof.
\end{proof}

\subsection{A lemma on the set $W_\tau$}

Let $\Deltav$ be a rational polytope in $N_\bR$, and $\Deltav_0$ be its face.
Let further $I$ be a finite set, and $\lc \Deltav_i \rc_{i \in I}$ be lattice polytopes in $N_\bR$.
We suppose that the normal fan $\Sigmav_0$ of $\Deltav_0$ is a subdivision of the normal fan $\Sigmav_i$ of $\Deltav_i$ for all $i \in I$.
We consider the map $\phi_{\Deltav_0, \Deltav_i} \colon \scrP_{\Deltav_0} \to \scrP_{\Deltav_i}$ of \eqref{eq:phi} for $\Deltav_0, \Deltav_i$.
For each face $\tau \in \scrP_{\Deltav}$ of $\Deltav$, we choose a point $a_\tau \in \rint(\tau)$.
For a face $\tau \in \scrP_{\Deltav_0}$ of $\Deltav_0$ and a face $F \in \scrP_{\Deltav}$ of $\Deltav$ such that $F \succ \Deltav_0$, we define the subset $W_\tau \lb F \rb \subset F$ by
\begin{align}\label{eq:WtF}
W_\tau(F):=\bigcup_{\substack{\tau \prec \tau_1 \prec \cdots \prec \tau_l \prec F, \\ l \geq 0}} \conv \lb \lc a_\tau, a_{\tau_1}, \cdots, a_{\tau_l} \rc \rb.
\end{align}
Let further $t \in \bR$ be a positive number, and consider the family of subsets
\begin{align}
\lc \widetilde{W}_\tau^t (F):=t \cdot W_\tau(F) + \sum_{i \in I} \phi_{\Deltav_0, \Deltav_i} \lb \tau \rb \rc_{\tau \in \scrP_{\Deltav_0}}.
\end{align}

\begin{lemma}\label{lm:tech}
For any positive number $t \in \bR$, the following hold:
\begin{enumerate}
\item  For $\tau_1, \tau_2 \in \scrP_{\Deltav_0}$, we have $\bigcap_{j=1, 2} \widetilde{W}_{\tau_j}^{t}(F) \neq \emptyset$ if and only if $\bigcap_{j=1, 2} \lb \sum_{i \in I} \phi_{\Deltav_0, \Deltav_i} \lb \tau_j \rb \rb \neq \emptyset$.
Furthermore, when $\bigcap_{j=1, 2} \lb \sum_{i \in I} \phi_{\Deltav_0, \Deltav_i} \lb \tau_j \rb \rb \neq \emptyset$, one has 
\begin{align}\label{eq:tech1}
\bigcap_{j=1, 2} \widetilde{W}_{\tau_j}^{t} (F)=
t \cdot \lb \bigcap_{j=1, 2} W_{\tau_j}(F) \rb + \bigcap_{j=1, 2} \lb \sum_{i \in I} \phi_{\Deltav_0, \Deltav_i} \lb \tau_j \rb \rb.
\end{align}
\item When $\Deltav_0=\Deltav$, one has
\begin{align}\label{eq:tech2}
\bigcup_{\tau \in \scrP_{\Deltav_0}} \widetilde{W}_{\tau}^t (\Deltav_0)=
t \cdot \Deltav_0 + \sum_{i\in I} \Deltav_i.
\end{align}
\end{enumerate}
\end{lemma}

\begin{example}
Suppose $\dim N_\bR=2$, $I= \lc 1 \rc$ and $t=1$.
Consider the case where
\begin{align}
\Deltav_1 = \conv \lb \lc (0, 0), (0, 1), (1, 0) \rc \rb, \quad \Deltav_0 =\Deltav= \conv \lb \lc (4, -2), (-2, 4), (0, -2), (-2, 0) \rc \rb.
\end{align}
\pref{fg:W} shows the polytope $\Deltav_0$ and the subsets $W_\tau:=W_\tau(\Deltav_0), \widetilde{W}_\tau^t:=\widetilde{W}_{\tau}^{t} (\Deltav_0)$.
The faces of $\Deltav_0$ shown in the figure are
\begin{align}
v:=\lc (0, -2) \rc, \tau_1:=\conv \lb \lc (4, -2), (0, -2) \rc \rb, \tau_2:=\conv \lb \lc (0, -2), (-2, 0) \rc \rb.
\end{align}
\begin{figure}[htbp]
\begin{center}
\includegraphics[scale=0.8]{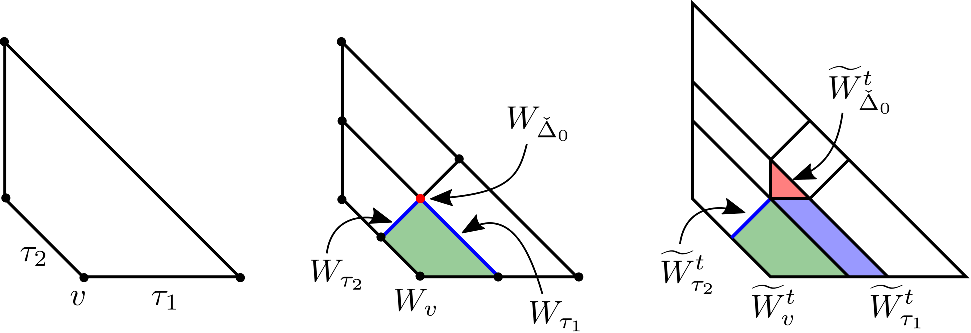}
\end{center}
\caption{The polytope $\Deltav_0$ and the subsets $W_\tau, \widetilde{W}_\tau^t$}
\label{fg:W}
\end{figure}
\end{example}

\begin{proof}[Proof of \pref{lm:tech}]
First, we consider the case where $\Deltav_0=\Deltav=F$.
We write $W_\tau(\Deltav_0)$ and $\widetilde{W}_{\tau}^{t} (\Deltav_0)$ as $W_\tau$ and $\widetilde{W}_{\tau}^{t}$ respectively for short.
We prove the first statement of the lemma by induction on the dimension of $\Deltav_0$.
When $\dim \Deltav_0=0$, all polytopes $\Deltav_0, \Deltav_i$ are a single point, 
and the statement is obvious in this case.
Assume that the statement holds when $\dim \Deltav_0=k$.
We will show that it holds also when $\dim \Deltav_0=k+1$.
Suppose $\dim \Deltav_0=k+1$ in the following.
Translating the polytope $\Deltav_0$ (resp. $\Deltav_i$ $(i \in I)$) by some vector $n \in N_\bR$ just causes the translation by the vector $t \cdot n$ (resp. $n$) to $\widetilde{W}_{\tau_j}^{t} (F)$ and the right hand side of \eqref{eq:tech1}.
Hence, we may assume that the point $a_{\Deltav_0} \in \rint \lb \Deltav_0 \rb$ that we choose for $\tau=\Deltav_0$ is $0 \in N_\bR$ and $0 \in \rint \lb \Deltav_i \rb$ for all $i \in I$ without loss of generality.
Furthermore, since we have $\Deltav_0 \subset T(\Deltav_0)$ and $\Deltav_i \subset T(\Deltav_i) \subset T(\Deltav_0)$ by \pref{lm:TT}.2, we may also assume that $\dim \Deltav_0$ coincides with the dimension of the ambient space $N_\bR$ by replacing $N_\bR$ with $T(\Deltav_0)$ if necessary.
By this procedure, the normal fans $\Sigmav_0$ and $\Sigmav_i$ $(i \in I)$ are replaced with their quotients by the subspace $\lb T(\Deltav_0) \rb^\perp \subset M_\bR$.
(Note that all cones in the normal fans $\Sigmav_0$ or $\Sigmav_i$ $(i \in I)$ contain $\lb T(\Deltav_0) \rb^\perp$.)

First, we will show
\begin{align}\label{eq:ladec0}
t \cdot \Deltav_0 + \sum_{i\in I} \Deltav_i
=\lb \sum_{i\in I} \Deltav_i \rb \sqcup 
\bigsqcup_{t' \in \lb 0, t\rd} \partial 
\lb t' \cdot \Deltav_0 + \sum_{i\in I} \Deltav_i \rb.
\end{align}
We set $\Deltav_1:=\sum_{i\in I} \Deltav_i$, and let $\check{h}_0, \check{h}_1 \colon M_\bR \to \bR$ be the support functions \eqref{eq:supp} of $\Deltav_0, \Deltav_1$.
Let further $A \subset M$ be the set of primitive generators of $1$-dimensional cones in $\Sigmav_0$.
$\dim \Deltav_0=\dim N_\bR$ implies that all cones in $\Sigmav_0$ have $\lc 0 \rc$ as their face.
Since the normal fan $\Sigmav_1$ of $\Deltav_1$ is coarser than $\Sigmav_0$ and $\check{h}_0, \check{h}_1$ are strictly convex with respect to $\Sigmav_0, \Sigmav_1$, we have
\begin{align}\label{eq:b-del-1}
\Deltav_i&= \lc n \in N_\bR \relmid \min_{m \in M_\bR} \lc \check{h}_i (m) + \la m, n \ra \rc \geq 0 \rc\\ \label{eq:del-i}
&= \lc n \in N_\bR \relmid \min_{m \in A} \lc \check{h}_i (m) + \la m, n \ra \rc \geq 0 \rc \\ \label{eq:b-del-0}
\partial \Deltav_0&= \lc n \in N_\bR \relmid \min_{m \in A} \lc \check{h}_i (m) + \la m, n \ra \rc =0 \rc,
\end{align}
where $i=0, 1$.
Here \eqref{eq:b-del-1} is due to \eqref{eq:supp2}.
Furthermore, since the normal fan of $\lb t' \cdot \Deltav_0+\Deltav_1 \rb$ with $t' \in \lb 0, t \rd$ is the common refinement of $\Sigmav_0$ and $\Sigmav_i$ $(i \in I)$, which coincides with $\Sigmav_0$, we also have
\begin{align}\label{eq:s-del}
t' \cdot \Deltav_0+\Deltav_1&= \lc n \in N_\bR \relmid \min_{m \in A} \lc \lb t' \cdot \check{h}_0+\check{h}_1 \rb (m) + \la m, n \ra \rc \geq 0 \rc\\ \label{eq:b-del}
\partial \lb t' \cdot \Deltav_0+\Deltav_1 \rb&= \lc n \in N_\bR \relmid \min_{m \in A} \lc \lb t' \cdot \check{h}_0+\check{h}_1 \rb (m) + \la m, n \ra \rc = 0 \rc.
\end{align}
\eqref{eq:s-del} is also due to \eqref{eq:supp2}.
Since the polytope $\Deltav_0$ contains $0 \in N_\bR$ in its interior, one can get
\begin{align}\label{eq:hvp}
\check{h}_0 (m) > 0,  \forall m \in A
\end{align}
by \eqref{eq:del-i} and \eqref{eq:b-del-0}.
We can see from \eqref{eq:del-i}, \eqref{eq:b-del}, and \eqref{eq:hvp} that the sets $\Deltav_1$, $\partial \lb t' \cdot \Deltav_0+\Deltav_1 \rb$, and $\partial \lb t'' \cdot \Deltav_0+\Deltav_1 \rb$ with $t' \neq t''$ are all disjoint.
Furthermore, in \eqref{eq:ladec0}, the right hand side is contained in the left hand side since $\Deltav_1=0+\Deltav_1 \subset t \cdot \Deltav_0 +\Deltav_1$ and $\partial \lb t' \cdot \Deltav_0 +\Deltav_1 \rb \subset \lb t' \cdot \Deltav_0 +\Deltav_1 \rb \subset \lb t \cdot \Deltav_0 +\Deltav_1 \rb$.
We check the opposite inclusion.
Take an element $n_0 \in \lb t \cdot \Deltav_0+\Deltav_1 \rb$.
Then we have $\min_{m \in A} \lc \lb t \cdot \check{h}_0+\check{h}_1 \rb (m) + \la m, n_0 \ra \rc \geq 0$ by \eqref{eq:s-del}.
From \eqref{eq:hvp}, we can see that we have $\min_{m \in A} \lc \check{h}_1(m) + \la m, n_0 \ra \rc \geq 0$ or there exists $t' \in \lb 0, t \rd$ such that $\min_{m \in A} \lc \lb t' \cdot \check{h}_0+\check{h}_1 \rb (m) + \la m, n_0 \ra \rc = 0$.
In the former case, we have $n_0 \in \Deltav_1$ by \eqref{eq:del-i}.
In the latter case, we have $n_0 \in \partial \lb t' \cdot \Deltav_0+\Deltav_1 \rb$ by \eqref{eq:b-del}.
Hence, the element $n_0$ is contained in the right hand side of \eqref{eq:ladec0} in either case.
We obtained \eqref{eq:ladec0}.

By using \pref{lm:minksum} to $\lb t' \cdot \Deltav_0 + \sum_{i\in I} \Deltav_i \rb$, one can get
\begin{align}\label{eq:mink-b}
\partial \lb t' \cdot \Deltav_0 + \sum_{i\in I} \Deltav_i \rb
=
\bigcup_{\substack{\sigma \prec \lb t' \cdot \Deltav_0 + \sum_{i\in I} \Deltav_i \rb \\ \dim \sigma=k}} \sigma
=
\bigcup_{\substack{\sigma \in \scrP_{\Deltav_0} \\\dim \sigma=k}} \lb t' \cdot \sigma + \sum_{i\in I} \phi_{\Deltav_0, \Deltav_i} \lb \sigma\rb \rb.
\end{align}
By combining this and \eqref{eq:ladec0}, we get 
\begin{align} \label{eq:ladec}
t \cdot \Deltav_0 + \sum_{i\in I} \Deltav_i
=\lb \sum_{i\in I} \Deltav_i \rb \sqcup 
\bigsqcup_{t' \in \lb 0, t\rd} \bigcup_{\substack{\sigma \in \scrP_{\Deltav_0}  \\\dim \sigma=k}} \lb t' \cdot \sigma + \sum_{i\in I} \phi_{\Deltav_0, \Deltav_i} \lb \sigma\rb \rb.
\end{align}
For each face $\tau \in \scrP_{\Deltav_0}$ such that $\tau \neq \Deltav_0$, we also set
\begin{align}
V_\tau:=\bigcup_{\substack{\tau \prec \tau_1 \prec \cdots \prec \tau_l \neq \Deltav_0 \\ l \geq 0, \tau_i \in \scrP_{\Deltav_0} }} \conv \lb \lc a_\tau, a_{\tau_1}, \cdots, a_{\tau_l} \rc \rb.
\end{align}
Since $\conv \lb \lc a_\tau, a_{\tau_1}, \cdots, a_{\tau_l}, a_{\Deltav_0}=0 \rc \rb = \bigsqcup_{s \in \ld 0, 1 \rd} s \cdot \conv \lb \lc a_\tau, a_{\tau_1}, \cdots, a_{\tau_l} \rc \rb$, one has
\begin{align}\label{eq:WV}
t \cdot W_\tau=\bigsqcup_{t' \in \ld 0, t\rd}  t' \cdot V_\tau.
\end{align}
From this, we get
\begin{align}\label{eq:wtt}
\widetilde{W}_\tau^{t}=t \cdot W_\tau + \sum_{i \in I} \phi_{\Deltav_0, \Deltav_i} \lb \tau \rb
=\bigsqcup_{t' \in \ld 0, t\rd}  t' \cdot V_\tau 
+ \sum_{i \in I} \phi_{\Deltav_0, \Deltav_i} \lb \tau \rb.
\end{align}
Here we have 
\begin{align}
t' \cdot V_\tau + \sum_{i \in I} \phi_{\Deltav_0, \Deltav_i} \lb \tau \rb 
\subset \bigcup_{\substack{\sigma \succ \tau, \sigma \in \scrP_{\Deltav_0} \\ \dim \sigma =k}}
\lb t' \cdot \sigma + \sum_{i \in I} \phi_{\Deltav_0, \Deltav_i} \lb \sigma \rb \rb
\subset \partial \lb t' \cdot \Deltav_0 + \sum_{i\in I} \Deltav_i \rb
\end{align}
by \eqref{eq:mink-b}.
Hence, for faces $\tau_1, \tau_2 \in \scrP_{\Deltav_0}$ and a $k$-dimensional face $\sigma \in \scrP_{\Deltav_0}$, one can see from \eqref{eq:ladec} and \eqref{eq:wtt} that we have
\begin{align}
\widetilde{W}_{\tau_j}^{t} \cap \lb \sum_{i\in I} \Deltav_i \rb
&=\lb \sum_{i \in I} \phi_{\Deltav_0, \Deltav_i} \lb \tau_j \rb \rb \\
\widetilde{W}_{\tau_j}^{t} \cap \lb t' \cdot \sigma + \sum_{i\in I} \phi_{\Deltav_0, \Deltav_i} \lb \sigma \rb \rb
&=
\lb t' \cdot V_{\tau_j} + \sum_{i \in I} \phi_{\Deltav_0, \Deltav_i} \lb \tau_j \rb \rb
\cap
\lb t' \cdot \sigma + \sum_{i\in I} \phi_{\Deltav_0, \Deltav_i} \lb \sigma \rb \rb \\
&=
\left\{
\begin{array}{ll}
t' \cdot \lb V_{\tau_j} \cap \sigma \rb+ \sum_{i \in I} \phi_{\Deltav_0, \Deltav_i} \lb \tau_j \rb & \sigma \succ \tau_j \\
\emptyset & \mathrm{otherwise},
\end{array}
\right.
\end{align}
where $j=1 ,2$ and $t' \in \lb 0, t \rd$.
Therefore, one has
\begin{align}\label{eq:0l}
\bigcap_{j=1, 2} \widetilde{W}_{\tau_j}^{t} \cap \lb \sum_{i\in I} \Deltav_i \rb
&=\bigcap_{j=1, 2} \lb \sum_{i \in I} \phi_{\Deltav_0, \Deltav_i} \lb \tau_j \rb \rb \\ \label{eq:tl}
\bigcap_{j=1, 2} \widetilde{W}_{\tau_j}^{t} \cap \lb t' \cdot \sigma + \sum_{i\in I} \phi_{\Deltav_0, \Deltav_i} \lb \sigma \rb \rb
&=
\left\{
\begin{array}{ll}
\bigcap_{j=1, 2} \lc
t' \cdot \lb V_{\tau_j} \cap \sigma \rb+ \sum_{i \in I} \phi_{\Deltav_0, \Deltav_i} \lb \tau_j \rb
\rc & \sigma \succ \tau_1, \tau_2 \\ 
\emptyset & \mathrm{otherwise}. \\
\end{array}
\right.
\end{align}
Suppose $\sigma \succ \tau_1, \tau_2$.
We use the induction hypothesis for the right hand side of \eqref{eq:tl}.
By \pref{lm:face-phi}, one can substitute $\sigma$ and $\phi_{\Deltav_0, \Deltav_i} \lb \sigma \rb$ to $\Deltav_0$ and $\Deltav_i$ in the statement of \pref{lm:tech}.1.
It turns out that the right hand side of \eqref{eq:tl} is non-empty if and only if $\bigcap_{j=1, 2} \lb \sum_{i \in I} \phi_{\Deltav_0, \Deltav_i} \lb \tau_j \rb \rb \neq \emptyset$, and when $\bigcap_{j=1, 2} \lb \sum_{i \in I} \phi_{\Deltav_0, \Deltav_i} \lb \tau_j \rb \rb \neq \emptyset$, it is equal to
\begin{align}\label{eq:tl2}
t' \cdot \lb V_{\tau_1} \cap V_{\tau_2} \cap \sigma \rb+ \bigcap_{j=1, 2} \lb \sum_{i \in I} \phi_{\Deltav_0, \Deltav_i} \lb \tau_j \rb \rb.
\end{align}
By combining this, \eqref{eq:0l}, \eqref{eq:ladec} and \eqref{eq:WV}, we obtain
\begin{align}
\bigcap_{j=1, 2} \widetilde{W}_{\tau_j}^{t} 
&=
\bigsqcup_{t' \in \ld 0, t\rd} 
\bigcup_{\substack{\sigma \in \scrP_{\Deltav_0} \\ \dim \sigma=k}}
\lc 
t' \cdot \lb V_{\tau_1} \cap V_{\tau_2} \cap \sigma \rb + 
\bigcap_{j=1, 2} \lb \sum_{i \in I} \phi_{\Deltav_0, \Deltav_i} \lb \tau_j \rb \rb
\rc \\
&=t \cdot \lb W_{\tau_1} \cap W_{\tau_2} \rb + \bigcap_{j=1, 2} \lb \sum_{i \in I} \phi_{\Deltav_0, \Deltav_i} \lb \tau_j \rb \rb.
\end{align}
We obtained \eqref{eq:tech1}.

Next, we prove the second statement.
We will do it by induction on the dimension of $\Deltav_0$ again.
This statement is also obvious when $\dim \Deltav_0=0$.
Assume that the statement holds when $\dim \Deltav_0=k$.
We will show that it holds also when $\dim \Deltav_0=k+1$.
Suppose $\dim \Deltav_0=k+1$.
It is obvious that in \eqref{eq:tech2}, the left hand side is contained in the right hand side.
We will prove the opposite inclusion.
Take an element $n \in \lb t \cdot \Deltav_0 + \sum_{i\in I} \Deltav_i \rb$.
By \eqref{eq:ladec}, we have either $n \in \lb \sum_{i\in I} \Deltav_i \rb$ or $n \in \lb t' \cdot \sigma + \sum_{i\in I} \phi_{\Deltav_0, \Deltav_i} \lb \sigma \rb \rb$ for some $k$-dimensional face $\sigma \in \scrP_{\Deltav_0}$ and $t' \in \lb 0, t \rd$.
When $n \in \lb \sum_{i\in I} \Deltav_i \rb$, consider $\widetilde{W}_{\tau}^t$ with $\tau=\Deltav_0$, which is
\begin{align}
\widetilde{W}_{\Deltav_0}^t=t \cdot W_{\Deltav_0} + \sum_{i \in I} \phi_{\Deltav_0, \Deltav_i} \lb \Deltav_0 \rb
=t \cdot \lc 0 \rc + \sum_{i\in I} \Deltav_i
=\sum_{i\in I} \Deltav_i.
\end{align}
Hence, we have $n \in \widetilde{W}_{\Deltav_0}^t \subset \bigcup_{\tau \in \scrP_{\Deltav_0}} \widetilde{W}_{\tau}^t$.
Next suppose $n \in \lb t' \cdot \sigma + \sum_{i\in I} \phi_{\Deltav_0, \Deltav_i} \lb \sigma \rb \rb$.
One can apply the induction hypothesis for the polytopes $\sigma, \lc \phi_{\Deltav_0, \Deltav_i} \lb \sigma \rb \rc_{i \in I}$ again by \pref{lm:face-phi}.
We obtain
\begin{align}
n \in \lb t' \cdot \sigma + \sum_{i\in I} \phi_{\Deltav_0, \Deltav_i} \lb \sigma \rb \rb
=\bigcup_{\tau \prec \sigma} \lc t' \cdot \lb W_{\tau} \cap \sigma \rb + \sum_{i \in I} \phi_{\Deltav_0, \Deltav_i} \lb \tau \rb \rc.
\end{align}
Since we also have
$t' \cdot \lb W_{\tau} \cap \sigma \rb \subset t' \cdot V_{\tau} \subset t \cdot W_{\tau}$, we get
\begin{align}
n \in \bigcup_{\tau \prec \sigma} \lb t \cdot W_{\tau} + \sum_{i \in I} \phi_{\Deltav_0, \Deltav_i} \lb \tau \rb \rb
=\bigcup_{\tau \prec \sigma} \widetilde{W}_{\tau}^t 
\subset \bigcup_{\tau \in \scrP_{\Deltav_0}} \widetilde{W}_{\tau}^t.
\end{align}
We obtained \eqref{eq:tech2}.

Lastly, we prove the first statement in the case where we do not necessary have $\Deltav_0=\Deltav=F$.
For a face $F \in \scrP_{\Deltav}$ such that $F \succ \Deltav_0$, take a sequence of faces $\Deltav_0=F_0 \prec F_1 \prec \cdots \prec F_k=F$ such that $\dim F_{i+1}=\dim F_i+1$.
The claim has been shown for $F_0=\Deltav_0$.
We assume that the claim holds for $F_i$, and show it also for $F_{i+1}$.
We may assume $a_{F_{i+1}}=0$ by translation of the polytope $\Deltav$.
Then for any face $\tau \in \scrP_{\Deltav_0}$, we have
\begin{align}\label{eq:tWF}
t \cdot W_{\tau}(F_{i+1})&=\bigsqcup_{t' \in \ld 0, t \rd} t' \cdot W_{\tau} (F_i) \\
\widetilde{W}_{\tau}^{t} (F_{i+1})&=\bigsqcup_{t' \in \ld 0, t \rd} t' \cdot W_{\tau} (F_i)
+ \sum_{i \in I} \phi_{\Deltav_0, \Deltav_i} \lb \tau_j \rb.
\end{align}
For $t_1, t_2 \in \ld 0, t \rd$, two subsets $\lb t_j \cdot W_{\tau_j} (F_i)
+ \sum_{i \in I} \phi_{\Deltav_0, \Deltav_i} \lb \tau_j \rb \rb$ with $j=1, 2$ do not intersect if $t_1 \neq t_2$.
Hence, we have
\begin{align}\label{eq:WF}
\bigcap_{j=1, 2} \widetilde{W}_{\tau_j}^{t} (F_{i+1})=
\bigcup_{t' \in \ld 0, t \rd} 
\bigcap_{j=1, 2} \lb t' \cdot W_{\tau_j} (F_i)
+ \sum_{i \in I} \phi_{\Deltav_0, \Deltav_i} \lb \tau_j \rb \rb
=\bigcup_{t' \in \ld 0, t \rd} \bigcap_{j=1, 2} \widetilde{W}_{\tau_j}^{t'} (F_{i}).
\end{align}
By hypothesis, this is non-empty if and only if $\bigcap_{j=1, 2} \lb \sum_{i \in I} \phi_{\Deltav_0, \Deltav_i} \lb \tau_j \rb \rb \neq \emptyset$, and if this holds, then \eqref{eq:WF} is equal to
\begin{align}\label{eq:WF2}
\bigcup_{t' \in \ld 0, t \rd} 
\lc t' \cdot \lb \bigcap_{j=1, 2} W_{\tau_j}(F_i) \rb + \bigcap_{j=1, 2} \lb \sum_{i \in I} \phi_{\Deltav_0, \Deltav_i} \lb \tau_j \rb \rb \rc.
\end{align}
Here one has $\bigcap_{j=1, 2} W_{\tau_j}(F_i)=W_{\tau_0}(F_i)$, where $\tau_0 \in \scrP_{\Deltav_0}$ is the minimal face such that $\tau_0 \succ \tau_1, \tau_2$.
By this and \eqref{eq:tWF}, we can see that \eqref{eq:WF2} is equal to
\begin{align}
t \cdot W_{\tau_0}(F_{i+1}) + \bigcap_{j=1, 2} \lb \sum_{i \in I} \phi_{\Deltav_0, \Deltav_i} \lb \tau_j \rb \rb
=
t \cdot \lb \bigcap_{j=1, 2} W_{\tau_j}(F_{i+1}) \rb + \bigcap_{j=1, 2} \lb \sum_{i \in I} \phi_{\Deltav_0, \Deltav_i} \lb \tau_j \rb \rb.
\end{align}
We obtained the claim for $F_{i+1}$.
Thus we can conclude it also for any face $F \succ \Deltav_0$.
\end{proof}

\bibliographystyle{amsalpha}
\bibliography{bibs}

\end{document}